\documentclass[reqno]{amsart}

\usepackage[utf8]{inputenc}
\usepackage[T1]{fontenc}
\usepackage[a4paper,top=3cm, bottom=3cm, left=3cm, right=3cm]{geometry}

\usepackage{amsfonts}
\usepackage{amsthm}
\usepackage{amsmath}
\usepackage{amssymb}
\usepackage{slashed}
\usepackage{mathtools}
\usepackage[shortlabels]{enumitem}
\usepackage{musicography}

\usepackage{tikz}
\usepackage{tikz-3dplot}
\usepackage{float}

\usepackage{url}
\usepackage{bookmark}

\newtheorem{theorem}{Theorem}[section]
\newtheorem{defn}[theorem]{Definition}

\newtheorem{lemma}[theorem]{Lemma}

\newtheorem{proposition}[theorem]{Proposition}

\newtheorem{manualassumptioninner}{Assumption}
\newenvironment{manualassumption}[1]{%
  \renewcommand\themanualassumptioninner{#1}%
  \manualassumptioninner
}{\endmanualassumptioninner}

\theoremstyle{definition}
\newtheorem{example}[theorem]{Example}

\theoremstyle{definition}

\def\Xint#1{\mathchoice
   {\XXint\displaystyle\textstyle{#1}}%
   {\XXint\textstyle\scriptstyle{#1}}%
   {\XXint\scriptstyle\scriptscriptstyle{#1}}%
   {\XXint\scriptscriptstyle\scriptscriptstyle{#1}}%
   \!\int}
\def\XXint#1#2#3{{\setbox0=\hbox{$#1{#2#3}{\int}$}
     \vcenter{\hbox{$#2#3$}}\kern-.5\wd0}}

\def\dashint{\Xint-}

\newcommand{\ice}{^{\operatorname{ice}}}
\newcommand{\ce}{\operatorname{ce}}

\newcommand{\backface}{\operatorname{bkf}}
\newcommand{\frontface}{\operatorname{ff}}
\newcommand{\leftface}{\operatorname{lf}}
\newcommand{\rightface}{\operatorname{rf}}
\newcommand{\bottomface}{\operatorname{bf}}
\newcommand{\timeface}{\operatorname{tf}}
\newcommand{\leftbackface}{\operatorname{lbf}}
\newcommand{\rightbackface}{\operatorname{rbf}}
\newcommand{\sideface}{\operatorname{sf}}
\newcommand{\cl}{\operatorname{cl}}
\newcommand{\diag}{\operatorname{diag}}
\newcommand{\edgebackface}{\operatorname{ebkf}}
\newcommand{\cornerbackface}{\operatorname{cbkf}}
\newcommand{\cornerfrontface}{\operatorname{cff}}
\newcommand{\corneredgeface}{\operatorname{cef}}
\newcommand{\edgefrontface}{\operatorname{eff}}

\numberwithin{equation}{section}

\usepackage{subfiles}

\title{Index Theory on Incomplete Cusp Edge Spaces}

\author{Jayson Liu}
\address{School of Mathematics and Statistics, University of Melbourne, Melbourne, Victoria, Australia}
\email{jayson.liu@unimelb.edu.au}  

\begin{document}

\begin{abstract}
    We study Dirac-type operators on incomplete cusp edge spaces with invertible boundary families. In particular, we construct the heat kernel for the associated Laplace-type operator and prove that the Dirac operators are essentially self-adjoint and Fredholm on their unique self adjoint domain. Using the asymptotics of the heat kernel and a generalisation of Getzler's rescaling argument we establish an index formula for these operators including a signature formula for the Hodge-de Rham operator on Witt incomplete cusp edge spaces.
\end{abstract}

\maketitle

\tableofcontents

\section{Introduction}

In this work we study Dirac-type operators on incomplete Riemannian singular spaces with non-isolated cusp singularities. We construct, under a Witt-type assumption, the heat kernel for the Laplace-type operator given their square and we use this heat kernel to prove the essential self-adjointness of the Dirac operators themselves. We also construct the Green's operator, establishing a natural Fredholm problem. We analyse the asymptotics of the heat kernel and use them to find an index formula for a class of Dirac-type operators which includes the spin Dirac operator. This builds on the work of Gell-Redman and Swoboda \cite{ice} who constructed the heat kernel for the Hodge Laplacian on Witt incomplete cusp edge spaces. This general approach goes back to Melrose \cite{melrose} and is closely related to the heat kernel analysis in \cite{cone}\cite{leschcones}\cite{pseudo} in the conical singularities setting. See below for further discussion of related literature.

We consider singular Riemannian manifolds for which the underlying topological space is a smoothly stratified space with a single singular stratum, so locally is just a bundle of cones on a smooth manifold at the singular stratum. We work on a resolved space which is obtained by separating the collapsed fibres at the singular locus obtaining a manifold with boundary $M$ whose boundary $\partial M$ is the total space of a fibre bundle 
\begin{align}
    Z\to \partial M \to Y 
\end{align}
where $Z$, the link of the singular space, is a closed manifold. 

These spaces often come with some natural Riemannian structure which has a particular type of structured collapse in the metric approaching the singular locus. In particular, we work with an incomplete cusp edge metric which near the boundary can be written in the form
\begin{align}
    g=dx^{2}+x^{2k}g_{\partial M/Y}+\phi^{\star}g_{Y}+e
\end{align}
where $x$ is the distance to the boundary, $g_{\partial M/Y}$ restricts to a metric on each fibre, $g_{y}$ is a metric on the base and $e$ is an error term and $k\geq 2$. 

The motivation for studying these types of metrics comes from compactified Riemann moduli spaces $(\overline{\mathcal{M}}_{g},g_{\operatorname{WP}})$, in particular one of its natural metrics $g_{\operatorname{WP}}$, the Weil-Petersson metric. The singular locus of $\overline{\mathcal{M}}_{g}$ is the union of $\lfloor\frac{g}{2}\rfloor+1$ normal crossing divisors and near an intersection of divisors $D_{J}=D_{i_{0}}\cap\ldots\cap D_{i_{j}}$ the Weil-Petersson metric takes the form \cite{wolpert}\cite{yamada}
\begin{align}
    g_{\operatorname{WP}}=4\pi^{3}\sum_{i=0}^{j}(d\rho_{i}^{2}+\rho^{6}d\theta_{i}^{2})+g_{D_{J}}+e.
\end{align}
Thus, for $k=3$, an incomplete cusp edge metric models the Weil-Petersson metric near the interior of a single divisor.

As formalised by Cheeger \cite{cheegerspectral}, the study of Dirac operators on singular spaces, and their index theory first requires analysis of their self-adjoint extensions and natural Fredholm problems. We first start with the question of self-adjointness. On a closed manifold, an $n$th order symmetric elliptic differential operators is self-adjoint, Fredholm operators on their natural domain $H^{n}(M)\to L^{2}(M)$. Moreover on complete Riemannian manifolds, it is well known that the Hodge Laplacian \cite{gaffney} and spin Dirac operator are essentially self-adjoint, however on incomplete Riemannian manifolds this is not necessarily true and in general there could be many self-adjoint extensions.

We consider a Dirac-type operator which is symmetric on the dense domain $C^{\infty}_{c}(M)$ of $L^{2}$ of smooth functions compactly supported in the interior of $M$. There is a minimal closed extension $\mathcal{D}_{\text{min}}$ and a maximal closed extension $\mathcal{D}_{\text{max}}$ defined by
\begin{align}
    \begin{split}
    \mathcal{D}_{\min}&=\{u\in L^{2}:\exists u_{n}\in C^{\infty}_{c} \text{ such that } u_{n}\to u, \slashed{\partial}u_{n} \text{ converges in } L^{2}\} \\
    \mathcal{D}_{\text{max}}&=\{u\in L^{2}(M,E): \slashed{\partial}u\in L^{2}(M,E)\}.
    \end{split}
\end{align}
The operator will be essentially self-adjoint if $\mathcal{D}_{\operatorname{min}}=\mathcal{D}_{\operatorname{max}}$. For incomplete cusp edge spaces, there is an induced family of Dirac operators on the boundary given by $x^{k}\slashed{\partial}|_{\partial M}$ which we denote by $\slashed{\partial}_{Z}$ and questions of self-adjointness depend on the properties of this family. We consider the case where the boundary family is invertible for which we obtain our first main result, see Section \ref{essentialsa}.
\begin{theorem}\label{introtheoremmain1}
     If the Dirac operator $\slashed{\partial}$ has an invertible boundary family of operators $\slashed{\partial}_{Z}$ then $\slashed{\partial}$ and $\slashed{\partial}^{2}$ are essentially self-adjoint. They have a discrete spectrum of eigenvalues which satisfy Weyl asymptotics and have eigenfunctions vanishing to smooth and vanishing to infinite order at $\partial M$. 
\end{theorem}

Having obtained the self-adjoint domain $\mathcal{D}$ we have our second main result, see Theorem \ref{greenfunction}.
 \begin{theorem}\label{introtheoremmain2}
     If the Dirac operator $\slashed{\partial}$ has an invertible boundary family then there exists $Q\in x^{k}\Psi^{-1}_{\ce}(M)$ such that $\slashed{\partial}Q=\operatorname{Id}-R$ and $Q\slashed{\partial}=\operatorname{Id}-\tilde{R}$ such that the remainders $R,\tilde{R}\in\Psi^{-\infty}(M)$ have kernels which vanish to infinite order at all faces. The operator $\slashed{\partial}$ is Fredholm on $\mathcal{D}$ and $\mathcal{D}=x^{k}H_{\ce}^{1}(M;\mathcal{S})$. 
\end{theorem}

We also prove an index theorem for Dirac-type operators on incomplete cusp edge spaces. The Dirac-type operators for which we have proven a Fredholm mapping property act on sections of a Clifford module which are $\mathbb{Z}_{2}$-graded so we can consider the restriction of $\slashed{\partial}$ to the even and odd sections of its self-adjoint domain 
\begin{align}
    \slashed{\partial}^{\pm}\colon\mathcal{D}^{\pm}=\mathcal{D}\cap L^{2}(M;E^{\pm})\to L^{2}(M;E^{\mp}).
\end{align}
We define the index of the Dirac operator to be the index of the even part $\slashed{\partial}^{+}$. 

As we have shown that $\slashed{\partial}$ is essentially self adjoint and the heat kernel $e^{-t\slashed{\partial}^{2}}$ is a trace class operator, we can use the McKean-Singer argument so we have 
\begin{align}
    \operatorname{ind}(\slashed{\partial}^{+})=\lim_{t\to 0}\operatorname{Str}(e^{-t\slashed{\partial}^2}).
\end{align}
On a singular space, the short-time limit of the supertrace can have contributions from the singular locus. For example, on an incomplete edge space \cite{iedge}, it was shown that a spin Dirac operator which satisfies the geometric Witt condition $\operatorname{Spec}(\slashed{\partial}_{Z})\cap (-\frac{1}{2},\frac{1}{2})=\varnothing$ where $\slashed{\partial}_{Z}$ is an induced family of Dirac operators on the fibers of $\partial M$ is essentially self-adjoint and Fredholm with index
\begin{align}\label{introcone}
    \operatorname{ind}(\slashed{\partial}^{+})=\int_{M}\hat{A}(M)+\int_{Y}\hat{A}(Y)\left(-\hat{\eta}({\partial_{Z}})+\int_{Z}T\hat{A}(\nabla^{v^{+}},\nabla^{pt})\right).
\end{align}
Here, we have the usual contribution by $\hat{A}(M)$ from the interior and a contribution from the boundary which is the singular locus of the incomplete edge space.
         
For the non-isolated cusp with trivial boundary kernel, we prove the following formula, see Theorems \ref{spindiracindex} and \ref{twistedtrangression}.
\begin{theorem}
        Let $(M,g)$ be incomplete cusp edge space $E$ be an ice Clifford module on M and $\nabla^{E}$ a Clifford connection which satisfies assumptions \ref{assumption2},\ref{assumption3} (found at the end of Section 3) and \ref{assumption4} (at the beginning of section 11). If the Dirac operator $\slashed{\partial}_{E}$ has a boundary family with trivial kernel, then the index of the $\slashed{\partial}_{E}^{+}$ is given by
        \begin{align}
            \operatorname{ind}(\slashed{\partial}_{E}^{+})=\int_{M}\mathcal{A}(M)\operatorname{ch}'(E)-\int_{Y}\mathcal{A}(Y)\operatorname{ch}'(E_{B})\tilde{\eta}(\slashed{\partial}_{E,\partial M/Y})
        \end{align}
        where $\tilde{\eta}(\slashed{\partial}_{E,\partial M/Y})$ is the normalised Bismut-Cheeger eta form for the family $\slashed{\partial}_{E,\partial M/Y}$ on $\partial M$.
         \end{theorem}
The case where the induced boundary family has non-trivial kernel is more complicated and we do not consider it in this work. For the spin Dirac operator, we obtain formula similar to \eqref{introcone} without the third transgression term which vanishes for the incomplete cusp edge. For an incomplete cusp edge, one can show that this trangression form actually vanishes for the spin Dirac operator. This formula could also have been obtained using the methods of \cite{iedge} in this case.

Following \cite{iedge}, we also prove that in some casaes index of the spin Dirac operator is an obstruction to the existence of positive scalar curvature incomplete cusp edge metric, see Section \ref{pscalarcurvature}.
\begin{theorem}
    Let (M,g) be a spin incomplete cusp edge space. Suppose either
    \begin{enumerate}
        \item $\operatorname{dim}(\partial M/Y)\geq 2$, the scalar curvature of $g$ is non-negative in a neighbourhood of the boundary and positive at least at one point sufficiently close to the boundary.

        \item $\operatorname{dim}(\partial M/Y)=1$ and the spin structure on $M$ is the lift of a spin structure on the associated space $\tilde{M}$ with fibres collapsed at the boundary (given the smooth structure identifying the family of cones as a family of disks).
    \end{enumerate}
    Then the induced boundary family has trivial kernel and the index of the spin Dirac operator vanishes.
\end{theorem}
Finally, we study the signature operator on Witt incomplete cusp edge spaces. Although the Hodge-de Rham operator satisfies all the assumptions we placed on the Clifford modules for the above theorems, the boundary family always has non-trivial kernel corresponding to the Hodge cohomology of the closed link $Z$ so the above theorem does not apply. However, using results in \cite{ice} we are also able to prove the following signature formula, see Theorem \ref{signature}.
\begin{theorem}
    Let $(M,g)$ be a Witt incomplete cusp edge space with $k\geq 3$. Then the $L^{2}$-signature is given by
    \begin{align}
        \operatorname{sgn}_{L^{2}}(M,g)=\int_{M}\mathcal{L}(M)-\int_{Y}\mathcal{L}(Y)\tilde{\eta}(\slashed{\partial}_{\Lambda(\partial M/Y)}).
    \end{align}
\end{theorem}
We now describe the methods we use to prove our main theorems. To prove essential self-adjointness we construct the heat kernel for the generalised Laplacian $\slashed{\partial}^{2}$. The heat kernel $H$ is the fundamental solution to the heat equation $(\partial_{t}+\slashed{\partial}^{2})H=0$ and is a smooth section of the large homomorphism bundle $\operatorname{HOM}(E)$ over $M^{\circ}\times M^{\circ}\times [0,\infty)$ which satisfies
\begin{align}
    \int_{M}H(t,x,x')u(x')\operatorname{dVol}\xrightarrow[L^{2}]{t\to 0} u(x)
\end{align}
for all $u\in L^{2}$. In fact, we will construct the heat kernel on a larger space $M^{2}_{\operatorname{heat}}$ which is a manifold with corners obtained from $M\times M\times [0,\infty)$ by a sequence of three quasihomogeneous radial blowups of different submanifolds, each of which produces a new boundary hypersurface. For any smooth ice vector field $W$ on $M$, which are generated by $\partial_{x},\partial_{y_{\alpha}}$ and $x^{-k}\partial_{z_{i}}$ near the boundary, thought of as a vector field on $M^{2}\times[0,\infty)$ acting on the left factor, the lift of $t^{\frac{1}{2}}W$ to the heat space $M^{2}_{\operatorname{heat}}$ is smooth in the interior and has well defined restrictions (at the faces $\frontface$ and $\timeface$, see figure \ref{heatspace}) so the lift of the heat operator $t(\partial_{t}+\slashed{\partial}^{2})$ also has a well defined restriction to these faces, called the normal operator.

This allows us to construct a parametrix for the heat equation on $M^{2}_{\operatorname{heat}}$ that is polyhomogeneous on $M^{2}_{\operatorname{heat}}$ but not on $M^{2}\times[0,\infty)$. Removing the error term we complete the construction of the heat kernel, see Theorem \ref{heatkernelexistence} and Lemma \ref{htdmin}.
\begin{theorem}
    Let $(M,g)$ be an incomplete cusp edge space and $E$ an ice Clifford module with Clifford connection $\nabla^{E}$. If the Dirac operator $\slashed{\partial}$ has an invertible boundary family of operators $\slashed{\partial}_{Z}$ then 
    there exists a distribution $H\in\mathcal{A}_{\text{phg}}(M^{2}_{\text{heat}};\beta^{\star}(\operatorname{HOM}(\mathcal{S})))$ such that $t(\partial_{t}+\slashed{\partial}^{2})H=0$ and the operator $H_{t}$ extends to a compact linear operator on $L^{2}(M;\mathcal{S})$ such that $H_{t}s\to s$ in $L^{2}$ at $t\to 0$. For $t>0$ and $s\in L^{2}(M;\mathcal{S})$, $H_{t}^{*}s\in\mathcal{D}_{\min}$.
\end{theorem}
From this theorem and a general argument (see section 6) we obtain Theorem \ref{introtheoremmain1}.

Having found a unique self-adjoint domain $\mathcal{D}$, we then construct a Green's operator for $\slashed{\partial}\colon \mathcal{D}\to L^{2}$. Although the operators we are studying are naturally described using these incomplete cusp edge structures and this suffices for the construction of the heat kernel, it is more convenient to consider a closely related class of differential operators called the cusp edge differential operators. The cusp edge tangent bundle $^{\ce}TM$ is defined by the space of smooth vector fields generated by $x^{k}\partial_{x},x^{k}\partial_{y_{\alpha}}$ and $\partial_{z_{i}}$. The main difference from the incomplete cusp edge vector fields is that this is a Lie subalgebra of the space of smooth vector fields and we can define the cusp edge differential operators as elements of its universal enveloping algebra. In our case, similar to the wedge operators in \cite{pseudo}, the incomplete cusp edge differential operators that we consider are more naturally defined in terms of the cusp edge differential operators. This allows us to use the general procedure of Melrose \cite{melrose} where we construct a double space $M^{2}_{\ce}$ by blowing up $M^{2}$ where this Lie algebra of singular vector fields is resolved and generates a subbundle of the tangent bundle of $M^{2}_{\ce}$ which is transverse to the lifted diagonal up to the blown up face. From there we can define a small cusp edge pseudodifferential calculus $\Psi_{\ce}^{\star}$ as distributions on $M^{2}_{\ce}$ conormal to the lifted diagonal for which we can prove mapping and composition properties by constructing a triple space and we can define cusp edge ellipticity using the symbol mapping.

 The operators $x^{k}\slashed{\partial}$ and $x^{2k}\slashed{\partial}^{2}$ are cusp edge elliptic operators so they have small parametrices in $\Psi^{\star}_{\ce}$ which are inverses up to a residual remainder in $\Psi^{-\infty}_{\ce}$. Residual operators are usually not compact with the normal operator at the blown up face being the obstruction. Here we again use the assumption of invertible boundary family to solve away the normal operator to establish Theorem \ref{introtheoremmain2}.

To obtain the index formula, we need to study the short time asymptotics of the heat kernel. For a trace-class integral operator $A$ with smooth kernel $K$ on a smooth manifold, the trace is equal to the integral its restriction to the diagonal
\begin{align}
    \operatorname{Tr}(A)=\int_{M}\operatorname{tr}K(x,x)\operatorname{dVol}.
\end{align}
Equivalently, this can be written as the pushforward of the pointwise trace of the heat kernel restricted to the diagonal. Similarly, we can express the supertrace of the heat kernel by the pushforward of the restriction to the lifted diagonal in the heat space of the pointwise supertrace to $[0,\infty)_{t}$. We can then determine the small time asymptotics of the supertrace from the asymptotics of the heat kernel at each of the boundary hypersurfaces of $M^{2}_{\operatorname{heat}}$ in the preimage of $t=0$, see Section \ref{pushforward}. The asymptotics at the time face $\timeface$ which is the preimage of the $t=0$ diagonal gives rise to the usual $\hat{A}$ contribution while there is another contribution at the front face $\frontface$ which lives above the fibre diagonal in the corner which gives the contribution from the singular locus.

While the contribution at $\timeface$ is standard, to determine the contribution at $\frontface$ is, we need to determine the coefficients of the asymptotic expansion $H$ at $\text{ff}$ which will allow us to find an explicit expression for the pointwise supertrace of $H$ restricted to the intersection to the diagonal with $\text{ff}$. To prove Theorem \ref{introtheoremmain2}, we use a generalisation of the rescaling argument introduced by Getzler \cite{getzler86}
 at $\frontface$ and we see that the boundary contribution comes from a higher order term in the asymptotic expansion of the rescaled heat kernel. That is, while for the standard Getzler rescaling, for example at the diagonal in the interior, the contributing term is the supertrace of the leading order term of the rescaled heat kernel. The analagous term at the boundary for an incomplete cusp appears to contribute a singular term in the asymptotic expansion of the supertrace. However, it turns out that the supertrace of the leading term vanishes and the next non-vanishing term, which is $k-1$ orders higher, is exactly the term which produces the constant term in the supertrace asymptotics.

For the signature theorem, using the heat kernel constructed in \cite{ice}, we can still use the same methods to determine the contribution to the supertrace from $\frontface$ and $\timeface$. In this case, the heat kernel has non-trivial asymptotics at another boundary hypersurface $\backface$ 
 (see figure \ref{heatspace}) which could also contribute to the supertrace. By using another Getlzer rescaling argument on the normal operator acting on fibre harmonic forms at this face, we show that in fact the contribution at $\backface$ vanishes.

\bigskip

The study of the spectral theory of elliptic operators on the Riemann moduli spaces was initiated by Ji, Mazzeo, M\"uller and Vasy in \cite{jmmv} where they proved that the scalar Laplacian for the Weil-Petersson metric is essentially self-adjoint has discrete spectrum and its eigenvalues satisfies Weyl asymptotics. Gell-Redman and Swoboda \cite{ice} proved the essential self-adjointess, discrete spectrum and Weyl asymptotics for the Hodge-Laplacian on incomplete cusp edge spaces satisfying the Witt condition, which included boundaries with circle fibres as in the case of the Riemann moduli spaces near the interior of a single divisor. They used techniques of geometric microlocal analysis, constructing the heat kernel as a polyhomogeneous distribution on a blown up heat space, which our construction of the heat kernel for Dirac operators with invertible boundary family is based on.

This geometric approach to the analysis of the heat kernel of elliptic operators is based on Melrose \cite{melrose} construction on non-compact complete Riemannian manifolds with asymptotically cylindrical ends (b-metrics). The first use of these techniques for incomplete singular spaces, where the question of self-adjointness extensions appears, is in the work of Mooers \cite{cone} on spaces with isolated conical singularities. Mazzeo and Vertman in \cite{mazzeovertman} constructed the heat kernel for the Friedrich's extension of the Laplacian for spaces with non-isolated conic singularites which are called incomplete edge spaces. See also \cite{vertmanalgebraic}\cite{vertmanbiharmonic}\cite{vertmanmapping}\cite{piazzavertman} for the construction of heat kernels on incomplete edge spaces. For the analysis of the mapping and Fredholm properties of singular differential operators using the microlocal approach generalising of the pseudodifferential calculus and for example \cite{mazzeo}\cite{grieserhunsicker2}. The boundary-valued problems for first order elliptic wedge operators, which includes Dirac operators on cone edge spaces, has been studied in \cite{krainermendoza}.

Our resolved spaces most closely mimic Gell-Redman and Swoboda \cite{ice} for the heat kernel and Mazzeo-Melrose \cite{melrosemazzeo} and Grieser-Hunsicker \cite{grieserhunsicker1} for the Green's operator. These papers also use an iterated blow-up construction needed to deal with the inhomogeneities of the operators.

Besides microlocal geometric techniques, there is a lot of previous work studying the spectral properties of elliptic operators on spaces with conical singularities. For example Cheeger \cite{cheegerspectral}\cite{cheegerspectralindex}\cite{cheegerhodge} on the Hodge and spectral theory on spaces with conical singularities, Br\"uning and Seeley \cite{bruningseeleyresolventcone}\cite{bruningseeleyresolvent2} on the heat kernel and resolvent of regular singular operators which includes the Laplacian on cones and Lesch \cite{leschcones} on the heat kernel and index theory of these operators.

The heat equation approach to index theory was used by Atiyah, Patodi and Singer in \cite{aps1} for their index theorem for Dirac operators on manifolds with boundary with APS boundary conditions where they also introduced the eta invariant of an elliptic operator. The study of index theory of Dirac operators in incomplete singular Riemannian spaces goes back to the work of Cheeger on the signature and Gauss-Bonnet theorem on spaces with conical singularities which was extended by Chou \cite{chou},\cite{chou2} to an index theorem for the spin Dirac operator. Bismut and Cheeger calculated the adiabatic limit of the reduced eta invariant of the spin Dirac operator on a fibration of compact spin manifolds with invertible vertical family as an integral involving the Bismut-Cheeger eta form. This was generalised by Dai \cite{daiadiabatic}, to families whose kernel form a vector bundle who also studied the signature operator. In 
\cite{BismutCheeger90a},\cite{BismutCheeger90b} Bismut and Cheeger proved an index theorem for families of Dirac operators on even-dimensional manifolds with boundaries by finding a formula for a family with conical singularities and taking the limit as the cone length goes to infinity. Atiyah and Lebrun \cite{atiyahlebrun} found formulas for the Euler characteristic and signature on 4-manifolds with an incomplete edge along an embedded surface. In \cite{iedge}, Albin and Gell-Redman proved an index formula for the spin Dirac operator manifolds with non-isolated conical singularities satisfying the Witt condition by considering the limit of the formula obtained on the manifold $M_{\epsilon}=\{x\geq \epsilon\}$ and using the formula for the adiabatic limit of the eta invariant. They generalised this in \cite{pseudo} to stratified spaces with iterated non-isolated conical singularities, this time using heat kernel and Getzler rescaling methods. 

There are many related works on index theory on complete spaces with fibred boundaries for example Vaillant \cite{vaillant} and Albin-Rochon \cite{albinrochonindex}. For Hodge theory on spaces with fibred boundary we mention \cite{almphodge}\cite{almpsignature}\cite{mullerhodge}\cite{hhmgravitational}.

\subsection{Acknowledgements}
I would like to thank my supervisor, Jesse Gell-Redman, for his guidance and support. This work was supported by the Australian Government Research Training Project Scholarship. This research was supported in part by the Australian Research Council grant DP21010324.

\section{Incomplete cusp edge spaces}

Let $M$ be a manifold with boundary $\partial M$ which is the total space of a fibre bundle $Z\to \partial M\to Y$. Let $x\geq 0$ be a smooth function on $M$ such that $x^{-1}(0)=\partial M$ and $dx\neq 0$ on $\partial M$ which we call a boundary defining function for $\partial M$. The existence of such a function is equivalent to the existence of a collar neighbourhood of the boundary, that is a neighbourhood $\mathcal{U}$ which is diffeomorphic to $\partial M\times [0,\epsilon)$.

Let $g_{Y}$ be a metric on $Y$ and $g_{\partial M/Y}$ a symmetric 2-tensor on $\partial M$ so that $\phi^{\star}g_{Y}+g_{\partial M/Y}$ is a submersion metric on $\partial M$ ($\phi$ gives an isometry between the horizontal distribution and $TY$). We consider metrics on $M$ which can be written in a neighbourhood of the boundary in the form
\begin{align}\label{1}
    g_{0}=dx^{2}+x^{2k}g_{\partial M/Y}+\phi^{\star}g_{Y}
\end{align}
for some boundary defining function $x$. For $k\geq 2$, we call such a metric a \textbf{product-type incomplete cusp edge metric}. 

The space of vector fields which are smooth on the interior and generated over $C^{\infty}(M)$ in a neighbourhood of the boundary by vector fields of the form
\begin{align}
    \partial_{x},\partial_{y_{\alpha}},x^{-k}\partial_{z_{i}}
\end{align}
where $y_{\alpha}$ is a lift of coordinates on the base and $z_{i}$ restrict to coordinates on the fibre is a projective $C^{\infty}(M)$ module thus defines a vector bundle called the \textbf{incomplete cusp edge (ice) tangent bundle} which we denote $\ice TM$. 

The space of smooth sections of the dual, $\ice T^{\star}M$, is identified with the space of one forms which are generated near the boundary by
\begin{align}\label{3}
    dx,dy_{\alpha},x^{k}dz_{i}.
\end{align}
In terms of these vector bundles, metrics of the form \eqref{1} define smooth metrics on the incomplete cusp edge tangent bundle, that is a smooth section the symmetric square of $^{\text{ice}}T^{\star}M$ which is non-degenerate on each fibre.

More generally we consider any metric $g$ on $^{\text{ice}}TM$ which satisfies
\begin{align}\label{4}
    g(W_{1},W_{2})-g_{0}(W_{1},W_{2})=O(x^{k})
\end{align}
for some product type metric $g_{0}$, all ice vector fields $W_{1},W_{2}$ and some polyhomogeneous (see appendix for definition) error term. Such a metric is called an \textbf{exact incomplete cusp edge metric} and the pair $(M,g)$ is called an \textbf{incomplete cusp edge space}. For any incomplete cusp edge space, we will fix a choice of boundary defining function which we denote $x$ and associated product-type metric which we will denote $g_{0}$. (If $k=1$ and the error is $O(x^2)$, we have a totally geodesic wedge metric of depth 1 \cite{pseudo}). For now, we will mostly consider product type metrics.

We define the ice differential operators $\operatorname{Diff}^{k}_{\operatorname{ice}}(M,E)$ of order $r$ to be those which locally near the boundary have the form
\begin{align}\label{icediff}
    P=\sum_{i+|\alpha|+|\beta|\leq r}a_{i\alpha\beta}(x,y,z)(x^{-k}\partial_{z})^{\beta}\partial_{y}^{\alpha}\partial_{x}^{i}
\end{align}
for some smooth sections $a_{i\alpha\beta}$ of $\operatorname{End}(E)$. Note that this is not an algebra of operators, in particular it is not the algebra of differential operators generated by the ice vector fields since for example $\partial_{x}(x^{-k}\partial_{z})=x^{-k}\partial_{z}\partial_{x}-kx^{-k-1}\partial_{z}$ is not of this form.

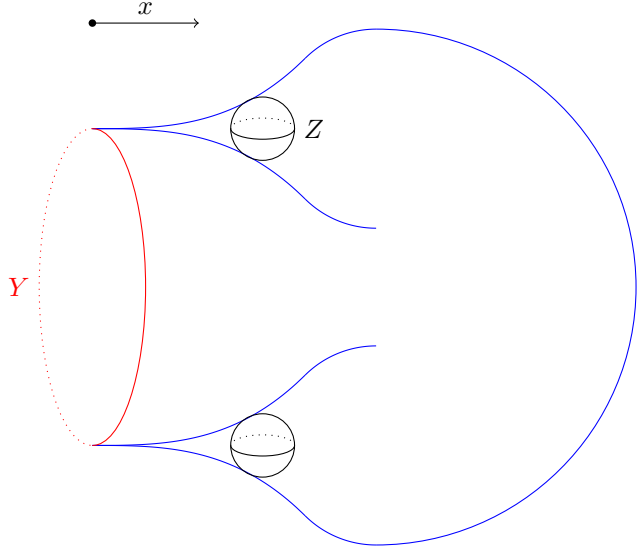
\begin{figure}[H]
\centering
    \begin{tikzpicture}[scale=1.4]
        \draw[->] (0,1) -- (1,1) node[pos=0,circle,fill,inner sep=1pt] {} node[pos=0.5,above] {$x$};
        
        \draw[scale=2, domain=0:1, smooth, variable=\x, blue] plot ({\x}, {(1/3)*\x^(3)});
        \draw[scale=2, domain=0:1, smooth, variable=\x, blue] plot ({\x}, {-(1/3)*\x^(3)});
        \draw[scale=2, domain=-135:-90, smooth, variable=\x, blue] plot ({4/3+(sqrt(2)/3)*cos(\x)}, {(sqrt(2)/3)*(sin(\x)}); 
        \draw[scale=2, domain=90:135, smooth, variable=\x, blue] plot ({4/3+(sqrt(2)/3)*cos(\x)}, {(sqrt(2)/3)*(sin(\x)});

        \draw[scale=2, domain=0:1, smooth, variable=\x, blue] plot ({\x}, {-1.5+(1/3)*\x^(3)});
        \draw[scale=2, domain=0:1, smooth, variable=\x, blue] plot ({\x}, {-1.5-(1/3)*\x^(3)});
       \draw[scale=2, domain=-135:-90, smooth, variable=\x, blue] plot ({4/3+(sqrt(2)/3)*cos(\x)}, {-1.5+(sqrt(2)/3)*(sin(\x)}); 
        \draw[scale=2, domain=90:135, smooth, variable=\x, blue] plot ({4/3+(sqrt(2)/3)*cos(\x)}, {-1.5+(sqrt(2)/3)*(sin(\x)}); 
        
         \draw[dotted,red] (0,0) arc(90:270:0.5cm and 1.5cm) node[pos=0.5,left] {$Y$};
         \draw[red] (0,-3) arc(-90:90:0.5cm and 1.5cm);

         \draw (1.9,0) arc(0:360:0.3cm) node[pos=0,right] {$Z$};
         \draw (1.3,0) arc(-180:0:0.3cm and 0.1cm);
         \draw[dotted] (1.9,0) arc(0:180:0.3cm and 0.1cm);

         \draw (1.9,-3) arc(0:360:0.3cm);
         \draw (1.3,-3) arc(-180:0:0.3cm and 0.1cm);
         \draw[dotted] (1.9,-3) arc(0:180:0.3cm and 0.1cm);

         \draw[blue] ({8/3},{-3-2*sqrt(2)/3}) arc(-90:90:{1.5+2*sqrt(2)/3}); 
    \end{tikzpicture}
    \caption{$\hat{M}=M/$\text{fibres collapsed over} $\partial M$} \label{fig:ice space}

 \end{figure}
In a neighbourhood of the boundary $\mathcal{U}=[0,\epsilon)\times\partial M$ the metric is of the form \eqref{1} where $\phi^{\star}g_{Y}+g_{\partial M/Y}$ is a submersion metric on $\partial M$. On $\partial M$ with respect to this submersion metric, given a local orthonormal frame $U_{i}$ of $Y$, there is a unique lift $\tilde{U}_{i}$ which we can extend to a local orthonormal frame on $\partial M$ with the inclusion of any choice of orthonormal frame $V_{j}$ of the vertical bundle. This then extends to a local orthonormal frame given by $\partial_{x}$, $\tilde{U}_{i}$, $x^{-k}V_{j}$ on $\mathcal{U}$ with respect the product type metric. Thus we have an orthonormal splitting 
\begin{align}\label{6}
    ^{\text{ice}}T\mathcal{U}=\langle \partial_{x}\rangle\oplus x^{-k}T\partial M/Y\oplus \phi^{\star}TY
\end{align}
each summand spanned by $\partial_{x}$, $\tilde{U}_{i}$, $x^{-k}V_{j}$ respectively. Let $\boldsymbol{v}$ and $\boldsymbol{h}$ be the orthogonal projections onto the second and third summand and $\boldsymbol{v}_{+}$ the projection onto the first two summands.

Using this orthonormal frame, we have the following commutators
\begin{align}
    [\partial_{x},\tilde{U}]&=0 \\
    [\partial_{x},x^{-k}V]&=-kx^{-k-1}V\in x^{-1}C^{\infty}(\mathcal{U},x^{-k}T(\partial M/Y))  \\
    [x^{-k}V_{1},x^{-k}V_{2}]&=x^{-2k}[V_{1}.V_{2}]\in x^{-k}C^{\infty}(\mathcal{U},x^{-k}T(\partial M/Y)) \\
    [x^{-k}V_{1},\tilde{U}]&=x^{-k}[V,\tilde{U}]\in C^{\infty}(\mathcal{U},x^{-k}T(\partial M/Y)) \\
    [\tilde{U_{1}},\tilde{U_{2}}]&\in x^{k}C^{\infty}(\mathcal{U},x^{-k}T(\partial M/Y))+C^{\infty}(\mathcal{U},\phi_{Y}^{\star}TY).
\end{align}
Note that by \eqref{4}, the difference between this orthonormal frame for $g_{0}$ and the orthonormal frame with respect to $g$ we get by parallel transporting with respect to $\partial_{x}$ or applying Gram-Schmidt is $O(x^{k})$.

We can define a connection $\nabla$ on smooth sections of $^\text{ice}TM$ by the Koszul formula which we will also call the Levi-Civita connection.
\begin{equation}\label{11}
    \begin{split}
    2g(\nabla_{W_{0}}W_{1},W_{2})=& W_{0}g(W_{1},W_{2})-W_{1}g(W_{2},W_{0})+W_{2}g(W_{0},W_{1}) \\ &+g([W_{0},W_{1}],W_{2})-g([W_{0},W_{2}],W_{1})-g([W_{1},W_{2}],W_{0}).
    \end{split}
\end{equation}

To describe the asymptotics of this connection, we make use of two other tensors. The curvature of the fibration
\begin{align}
    \mathcal{R}^{\phi}(\tilde{U}_{1},\tilde{U}_{2})=\boldsymbol{v}([\tilde{U}_{1},\tilde{U}_{2}]).
\end{align}
By  Frobenius' theorem, this is also the obstruction to finding coordinates $y_{i}$ complementary to coordinates $x,z_{j}$ where $z_{j}$ restrict to coordinates on the fibres such that the coordinate vector fields $\partial_{y_{i}}$ span the horizontal bundle. Thus in general, if
\begin{align}
    U_{i}=a_{ij}\partial_{y_{j}}
\end{align}
then its lift to the horizontal bundle will be of the form
\begin{align}
    \tilde{U}_{i}=a_{ij}\partial_{y_{j}}+b_{ij}(x^{-k}\partial_{z_{j}}).
\end{align}
Recall that $g_{0}$ is a smooth section of the symmetric product of $\ice T^{\star}M$ so in coordinates it is given by products of the terms \eqref{3}
\begin{align}
    g_{0}=dx^{2}+
    \begin{bmatrix}
        dy^{\alpha} & x^{k}dz^{i}
    \end{bmatrix}
    \begin{bmatrix}
        \left(g_{Y}\right)_{\alpha\beta} & h_{\alpha j} \\
        h_{i\beta } & \left(g_{\partial M/Y}\right)_{ij}
    \end{bmatrix}
    \begin{bmatrix}
        dy^{\beta} \\
        x^{k}dz^{j}
    \end{bmatrix}.
\end{align}
Now, if we consider the $h_{\alpha j}dy^{\alpha}(x^{k}dz^{j})$ terms in the product-type metric we have
\begin{equation}
    \begin{split}
        \phi^{\star}g_{Y}(\partial_{y_{i}},\partial_{z_{j}})&=g_{Y}(\phi_{\star}\partial_{y_{i}},\phi_{\star}\partial_{z_{j}}) \\
    &=g_{Y}(\partial_{y_{i}},0)=0.
    \end{split}
\end{equation}
Thus the only contribution to these mixed terms are $O(x^{k})$ which come from the $x^{k}g_{\partial M/Y}$ term. Now $\tilde{U}_{i}$ is orthogonal to the vertical bundle so denoting the restriction of $g_{\partial M/Y}$ to a fibre over $y$ by $g_{Z,y}$ we have
\begin{equation}
    \begin{split}
        0=g_{\text{ice}}(\tilde{U}_{i},x^{-k}\partial_{z_{l}})=b_{ij}(g_{Z,y})_{jl}+O(x^{k}).
    \end{split}
\end{equation}
Hence the term $b_{ij}(g_{Z,y})_{jl}$ is $O(x^{k})$ for all $l$. Since $g_{Z,y}$ has smooth non-vanishing coefficients we have that $b_{ij}=O(x^{k})$ so we can write
\begin{align}\label{16}
        \tilde{U}_{i}=a_{ij}\partial_{y_{j}}+\tilde{b}_{ij}x^{k}(x^{-k}\partial_{z_{j}})
\end{align}
for some different smooth coefficients $\tilde{b}_{ij}$.

We also have the second fundamental form
\begin{align}
    \mathcal{S}^{\phi}(V_{1},V_{2})=\boldsymbol{h}(\nabla^{\partial M}_{V_{1}}V_{2})
\end{align}
where $\boldsymbol{h}$ is the projection onto $\phi_{Y}^{\star}TM$ and $\nabla^{\partial M}$ is the Levi-Civita connection on the boundary induced by the submersion metric
\begin{align}
    g_{\partial M}= g_{\partial M/Y}+\phi_{Y}^{\star}g_{Y}.
\end{align}
Finally there is the family of connections $\nabla^{\partial M/Y}$ induced by the metric on each fibre and $\phi_{Y}^{\star}\nabla^{Y}$ the pullback connection of the Levi-Civita connection on $(Y,g_{Y})$. 

This, together with the Koszul formula and the fact that $[V,\cdot]$ is vertical for all vertical fields $V$, we can calculate the asymptotics of the connection as follows.

\begin{lemma} Let $V_{i}\in C^{\infty}(\mathcal{U},T\partial M/Y)$, $U_{j}\in C^{\infty})Y,TY)$ and $\tilde{U}_{j}$ their respective lifts then
\begin{equation}\label{20}
    \begin{split}
        g_{\text{ice}}(\nabla_{V_{i}}\partial_{x},x^{-k}V_{j})&=kx^{k-1}g_{\partial M/Y)}(V_{i},V_{j}) \\
        g_{\text{ice}}(\nabla_{V_{i}}x^{-k}V_{j},\partial_{x})&=-kx^{k-1}g_{\partial M/Y)}(V_{i},V_{j}).
    \end{split}
\end{equation}
The combinations with $\partial_{x}$ which do not appear above all vanish. The combinations without $\partial_{x}$ are given in the following table
\begin{center}
    \begin{tabular}{|c||c|c|}
        \hline
             $g_{\text{ice}}(\cdot,\cdot)$ & $x^{-k}V_{3}$ & $\tilde{U}_{3}$ \\ 
                \hline\hline
            $\nabla_{V_{1}}x^{-k}V_{2}$ & $g_{\partial M/Y}(\nabla^{\partial M/Y}_{V_{1}}V_{2},V_{3})$ & $x^{k}\phi_{Y}^{\star}g_{Y}(\mathcal{S}^{\phi}(V_{1},V_{2}),\tilde{U}_{3})$ \\ 
                \hline
            $\nabla_{\tilde{U}}x^{-k}V$ & $\phi_{Y}^{\star}g_{Y}(\mathcal{S}^{\phi}(V,V_{3}),\tilde{U})-g_{\partial M}([\tilde{U},\tilde{V}_{3}],V) $  & $-\frac{x^{k}}{2}g_{\partial M/Y}(\mathcal{R}^{\phi}(\tilde{U},\tilde{U}_{3}),V)$ \\ 
                \hline
            $\nabla_{V_{1}}\tilde{U}$ & $-x^{k}\phi_{Y}^{\star}g_{Y}(\mathcal{S}^{\phi}(V,V_{3}),\tilde{U})$ & $-\frac{x^{2k}}{2}g_{\partial M/Y}(\mathcal{R}^{\phi}(\tilde{U},\tilde{U}_{3}),V)$ \\ 
                \hline
            $\nabla_{\tilde{U}_{1}}\tilde{U}_{2}$ & $\frac{x^{k}}{2}g_{\partial M/Y}(\mathcal{R}^{\phi}(\tilde{U}_{1},\tilde{U}_{2}),V)$ & $g_{Y}(\nabla_{V_{1}}V_{2},V_{3})$ \\ 
        \hline
    \end{tabular}
\end{center}
\end{lemma}

\begin{proof}
We make repeated use of the Koszul formula \eqref{11}. If any one of the $W_{i}$ is $\partial_{x}$ then the only possible non-zero terms come from the commutators $[\partial_{x},x^{-k}V]=-kx^{-k-1}V$ which forces the other two vector fields to be vertical to get something non-zero and we can take them to be one of the $x^{-k}V_{i}$ by expanding with respect to the orthonormal frame $V_{i}$. If $W_{0}=\partial_{x}$ then the first two commutator terms cancel each other and the last one is $0$. If $W_{1}$ or $W_{2}=\partial_{x}$ then the first and third terms are equal and give \eqref{20}.

If $W_{0}=V_{1}$, $W_{1}=x^{-k}V_{2}$ and $W_{2}=x^{-k}V_{3}$ then \eqref{11} reduces to the Koszul formula for the Levi-Civita connection on each fibre with metric $g_{Z_{y}}$. If $W_{2}=\tilde{U}_{3}$ then by the definition of $\mathcal{S}^{\phi}$, we see that
\begin{align}
    \begin{split}
        2g(\nabla_{W_{0}}W_{1},\tilde{U}_{3})&=  V_{1}g(x^{-k}V_{2},\tilde{U}_{3})-x^{-k}V_{2}g(\tilde{U}_{3},V_{1})+\tilde{U}_{3}g(V_{1},x^{-k}V_{2}) \\ &\quad +g([V_{1},x^{-k}V_{2}],\tilde{U}_{3})-g([V_{1},\tilde{U}_{3}],x^{-k}V_{2})-g([x^{-k}V_{2},\tilde{U}_{3}],V_{1}) \\
        &=x^{k}\tilde{U}_{3}g_{\partial M}(V_{1},V_{2})-x^{k}g_{\partial M}([V_{1},\tilde{U}_{3}],V_{2})-x^{k}g_{\partial M}([V_{2},\tilde{U}_{3}],V_{1}) \\
        &= 2x^{k}g_{\partial M}(\nabla^{\partial M}_{V_{1}}V_{2},\tilde{U}_{3})=2x^{k}g_{\partial M}(\boldsymbol{h}\nabla^{\partial M}_{V_{1}}V_{2},\tilde{U}_{3}) \\
        &=2x^{k}\phi_{Y}^{\star}g_{Y}(\mathcal{S}^{\phi}(V_{1},V_{2}),\tilde{U}_{3}).
    \end{split}
\end{align}
If $W_{0}=\tilde{U}$, $W_{1}=x^{-k}V$ and $W_{3}=x^{-k}V_{3}$ then
\begin{align}
    \begin{split}
        2g(\nabla_{\tilde{U}}x^{-k}V,x^{-k}V_{3})&=  \tilde{U}g(x^{-k}V,x^{-k}V_{3})+g([\tilde{U},x^{-k}V],x^{-k}V_{3})-g([\tilde{U},x^{-k}\tilde{V}_{3}],x^{-k}V) \\
        &=2g_{\partial M}(\nabla_{V}V_{3},\tilde{U}_{3})+g_{\partial M}([V,\tilde{U}],V_{3})+g_{\partial M}([V_{3},\tilde{U}_{3}],V) \\
        &\quad +g_{\partial M}([\tilde{U},V],V_{3})-g_{\partial M}([\tilde{U},\tilde{V}_{3}],V) \\
        &=2\phi_{Y}^{\star}g_{Y}(\mathcal{S}^{\phi}(V,V_{3}),\tilde{U})-2g_{\partial M}([\tilde{U},\tilde{V}_{3}],V). 
    \end{split}
\end{align}
For $W_{3}=\tilde{U}_{3}$
\begin{align}
    \begin{split}
        2g(\nabla_{\tilde{U}}x^{-k}V,\tilde{U}_{3})&=-g([\tilde{U},\tilde{U}_{3}],x^{-k}V) \\
        &=-x^{k}g(x^{-k}\boldsymbol{v}[\tilde{U},\tilde{U}_{3}],x^{-k}V)=-x^{k}g_{\partial M/Y}(\mathcal{R}^{\phi}(\tilde{U},\tilde{U}_{3}),V). 
    \end{split}
\end{align}
For the last four combinations we have
\begin{align}
    \begin{split}
        g(\nabla_{v_{1}}\tilde{U},x^{-k}V_{3})&=-g(\nabla_{V_{1}}x^{-k}V_{3},\tilde{U}) \\
            &=-x^{k}\phi_{Y}^{\star}g_{Y}(\mathcal{S}^{\phi}(V,V_{3}),\tilde{U}) \\
        2g(\nabla_{V_{1}}\tilde{U},\tilde{U}_{3})&=-x^{k}g([\tilde{U},\tilde{U}_{3}],x^{-k}V_{1}) \\
            &=-x^{2k}g_{\partial M/Y}(\mathcal{R}^{\phi}(\tilde{U},\tilde{U}_{3}),V_{1}) \\
        g(\nabla_{\tilde{U}_{1}}\tilde{U}_{2},x^{-k}V_{3})&=-g([\tilde{U_{1}},\tilde{U}_{2}],x^{-k}V_{3}) \\
            &=x^{k}g_{\partial M/Y}(\mathcal{R}^{\phi}(\tilde{U}_{1},\tilde{U}_{2}),V_{1}) \\
        g(\nabla_{\tilde{U}_{1}}\tilde{U}_{2},\tilde{U}_{3})&=g_{Y}(\nabla_{V_{1}}V_{2},V_{3}).
    \end{split}
\end{align}
\end{proof}

This shows that we have a well defined connection
    \begin{align}
        \nabla\colon C^{\infty}(M,^\text{ice}TM)\to C^{\infty}(M,^\text{ice}TM\otimes T^{\star}M).
    \end{align}

By \eqref{4}, the difference between the operator defined using the incomplete cusp edge metric and an associated product-type metric is a smooth $\operatorname{End}(\ice TM)$-valued 1-form which is $O(x^{k-1})$, that is, it is equal to $x^{k-1}\omega$ where $\omega$ is smooth up to the boundary. For the curvature, if we denote the connection for the product type metric as $\nabla^{0}$ then
\begin{align}
    \nabla_{V_{1}}\nabla_{V_{2}}V_{3}-\nabla_{V_{1}}^{0}\nabla_{V_{2}}^{0}V_{3}=(\nabla_{V_{1}}-\nabla_{V_{1}}^{0})\nabla_{V_{2}}V_{3}-\nabla_{V_{1}}^{0}(\nabla_{V_{2}}^{0}-\nabla_{V_{2}})V_{3}.
\end{align}
So the difference in the terms $g(R(V_{1},V_{2})V_{3},V_{4})$ are at least $O(x^{k-2})$. 

\begin{lemma}\label{1.3}
Let $N\in\mathcal{C}^{\infty}(M,TM)$, $W_0\in C^{\infty}(M,TM)$ tangent to the fibres which satisfies $W_{0}=\boldsymbol{v}W_{0}+O(x^{k})$ and $W_{1},W_{2}\in C^{\infty}(M,^{\text{ice}}TM)$ then
\begin{enumerate}
    \item $\begin{aligned}[t]
                 g_{\text{ice}}(R_{\text{ice}}(N,W_{0})W_{1},W_{2})=&g_{\text{ice}}(R_{\text{ice}}(N,\boldsymbol{v}W_{0})\boldsymbol{v}_{+}W_{1},\boldsymbol{v}_{+}W_{2}) \\
    &+N(x)kx^{k-1}\phi_{Y}^{\star}g_{Y}(\mathcal{S}^{\phi}(\boldsymbol{v}W_{0},x^{k}\boldsymbol{v}W_{1}),\boldsymbol{h}W_{2})\\
    &-N(x)kx^{k-1}\phi_{Y}^{\star}g_{Y}(\mathcal{S}^{\phi}(\boldsymbol{v}W_{0},x^{k}\boldsymbol{v}W_{2}),\boldsymbol{h}W_{1})+O(x^{k}). 
        \end{aligned}$
    \\
    Note that $x^{k}\boldsymbol{v}W_{2}$ is a non vanishing smooth vector field if $\boldsymbol{v}W_{2}$ is a non-vanishing ice vector field.
    
    \item Let $U_{N},U_{i}$ be the vector fields on $Y$ whose horizontal lifts are equal to $\boldsymbol{h}W_{N}$, $\frac{1}{x^{k}}\boldsymbol{h}W_{0}$, $\boldsymbol{h}W_{i}$ respectively then for $2\leq j\leq 2k-1$ we have
    \begin{equation}\label{27}
        \begin{split}
    g_{\text{ice}}(\nabla_{N}R_{\text{ice}}(N,W_{0})\boldsymbol{h}W_{1},\boldsymbol{h}W_{2})&=N(x)x^{k-1}g_{\text{ice}}(R_{\text{ice}}(\boldsymbol{h}N,\frac{1}{x^{k}}\boldsymbol{h}W_{0})\boldsymbol{h}W_{1},\boldsymbol{h}W_{2})+O(x^{2k-2}) \\
    &=N(x)x^{k-1}g_{Y}(R_{Y}(U_{N},U_{0})U_{1},U_{2}) \\
    &\quad+x^{k}N(g_{Y}(R_{Y}(U_{N},U_{0})U_{1},U_{2})) \\
    &\quad-k^{2}N(x)^{2}x^{2k-2}g_{\partial M/Y}(\mathcal{R}^{\phi}(\boldsymbol{h}W_{1},\boldsymbol{h}W_{2}),\boldsymbol{v}W_{0}) \\
    &\quad+O(x^{2k-1}) \\
    g_{\text{ice}}(\nabla_{N}^{j}R_{\text{ice}}(N,W_{0})\boldsymbol{h}W_{1},\boldsymbol{h}W_{2})&=N(x)N^{j-1}(x^{k-1}g_{Y}(R_{Y}(U_{N},U_{0})U_{1},U_{2})) \\
    &\quad-k^{2}N(x)^{2}N^{j-1}(x^{2k-2}g_{\partial M/Y}(\mathcal{R}^{\phi}(\boldsymbol{h}W_{1},\boldsymbol{h}W_{2}),\boldsymbol{v}W_{0})) \\
    &\quad+O(x^{2k-j}).
        \end{split}
    \end{equation}
\end{enumerate}
\end{lemma}

\begin{proof}
\begin{enumerate}
    \item Since $W_{0}$ satisfies $W_{0}=\boldsymbol{v}(W_{0})+O(x^{k})$, it suffices to consider the case $W_{0}=\boldsymbol{v}W_{0}$ since $R_{\text{ice}}$ is a tensor. Making this replacement leaves an error of $O(x^{k})$.
    
    Now split the terms up into the three cases corresponding to the splitting $W_{i}=\boldsymbol{h}W_{i}+\boldsymbol{v}_{+}W_{i}$
    \begin{enumerate}[(a)]
        \item For the horizontal term we have
            \begin{equation}
                \begin{split}
                g_{\text{ice}}(R_{\text{ice}}(N,\boldsymbol{v}W_{0})\boldsymbol{h}W_{1},\boldsymbol{h}W_{2})=&g_{\text{ice}}(\nabla_{N}\nabla_{\boldsymbol{v}W_{0}}\boldsymbol{h}W_{1},\boldsymbol{h}W_{2}) \\
                &-g_{\text{ice}}(\nabla_{\boldsymbol{v}W_{0}}\nabla_{N}\boldsymbol{h}W_{1},\boldsymbol{h}W_{2})-g_{\text{ice}}(\nabla_{[N,\boldsymbol{v}W_{0}]}\boldsymbol{h}W_{1},\boldsymbol{h}W_{2}).
                \end{split}
            \end{equation}
            By the asymptotics of the connection, the first term is 
            \begin{equation}
                \begin{split}
                g_{\text{ice}}(\nabla_{N}\nabla_{\boldsymbol{v}W_{0}}\boldsymbol{h}W_{1},\boldsymbol{h}W_{2})=&-\frac{1}{2}N(x)x^{2k-1}g_{\text{ice}}(\sum_{i}g_{\partial M/Y}(\mathcal{R}^{\phi}(\boldsymbol{h}W_{1},\tilde{U}_{i}),\boldsymbol{v}W_{0})\tilde{U}_{i},\boldsymbol{h}W_{2}) \\
                &+O(x^{2k}).
                \end{split}
            \end{equation}
            For the third term the commutator is vertical so this term is also $O(x^{2k})$ and for the second term we have
            \begin{equation}
                \begin{split}
                    g_{\text{ice}}(\nabla_{\boldsymbol{v}W_{0}}\nabla_{N}\boldsymbol{h}W_{1},\boldsymbol{h}W_{2})&=x^{k}\phi_{Y}^{\star}g_{Y}(\mathcal{S}^{\phi}(\boldsymbol{v}W_{0},\boldsymbol{v}\nabla_{N}\boldsymbol{h}W_{1}),\boldsymbol{h}W_{2})+O(x^{2k}).
                \end{split}
            \end{equation}
            Here $\boldsymbol{v}\nabla_{N}\boldsymbol{h}W_{1}$ is also $O(x^{k})$ so the second term is $O(x^{2k})$ as well.
        
        \item For the mixed term, since the curvature evaluated on vector fields tangent to the boundary preserves the splitting we have
        \begin{align}
             g_{\text{ice}}(R_{\text{ice}}(N,\boldsymbol{v}W_{0})\boldsymbol{h}W_{1},\boldsymbol{v}_{+}W_{2})&=N(x)g_{\text{ice}}(R_{\text{ice}}(\partial_{x},\boldsymbol{v}W_{0})\boldsymbol{h}W_{1},\boldsymbol{v}_{+}W_{2})+O(x^{k}) \\
             g_{\text{ice}}(R_{\text{ice}}(N,\boldsymbol{v}W_{0})\boldsymbol{v}_{+}W_{1},\boldsymbol{h}W_{2})&=N(x)g_{\text{ice}}(R_{\text{ice}}(\partial_{x},\boldsymbol{v}W_{0})\boldsymbol{v}_{+}W_{1},\boldsymbol{h}W_{2})+O(x^{k}).
        \end{align}
        For the first term since the commutator $[\partial_{x},\boldsymbol{v}W_{0}]$ and $\nabla_{\partial_{x}}\boldsymbol{h}W_{1}$ vanish we are left with a single term
        \begin{equation}
            \begin{split}
            N(x)g_{\text{ice}}(&\nabla_{\partial_{x}}\nabla_{\boldsymbol{v}W_{0}}\boldsymbol{h}W_{1},\boldsymbol{v}_{+}W_{2})=N(x)\partial_{x}g_{\text{ice}}(\nabla_{\boldsymbol{v}W_{0}}\boldsymbol{h}W_{1},\boldsymbol{v}_{+}W_{2})\\
            &=-N(x)kx^{k-1}g_{\text{ice}}(\phi_{Y}^{\star}g_{Y}(\sum_{j=1}^{f}\mathcal{S}^{\phi}(\boldsymbol{v}W_{0},V_{j}),\boldsymbol{h}W_{1})x^{-k}V_{j},\boldsymbol{v}_{+}W_{2}) \\
            &=-N(x)kx^{k-1}\phi_{Y}^{\star}g_{Y}(\mathcal{S}^{\phi}(\boldsymbol{v}W_{0},x^{k}\boldsymbol{v}W_{2}),\boldsymbol{h}W_{1}).
            \end{split}
        \end{equation}
        Here the first equality uses that $\nabla$ is a metric connection for $g_{ice}$ and the second uses the asymptotics of the connection. For the second term, vanishing of the commutator and term with inner $\nabla_{\partial_{x}}$ leaves us with the following term
        \begin{equation}
            \begin{split}
                N(x)g_{\text{ice}}(\nabla_{\partial_{x}}\nabla_{\boldsymbol{v}W_{0}}\boldsymbol{v}_{+}W_{1},\boldsymbol{h}W_{2}).
            \end{split}
        \end{equation}
        The $\partial_{x}$ part of $\boldsymbol{v}_{+}W_{1}$ vanishes in this term since after taking the $\nabla_{\boldsymbol{v}W_{0}}$ derivative we are left with $x^{k}$ times an (ice) vertical vector field thus the term vanishes after taking $\nabla_{\partial_{x}}$ on this vertical vector field or evaluating it against the horizontal vector field in the metric. This leaves the vertical part for which we have
        \begin{equation}
            \begin{split}
                N(x)g_{\text{ice}}(\nabla_{\partial_{x}}\nabla_{\boldsymbol{v}W_{0}}\boldsymbol{v}W_{1},\boldsymbol{h}W_{2})=N(x)kx^{k-1}\phi_{Y}^{\star}g_{Y}(\mathcal{S}^{\phi}(\boldsymbol{v}W_{0},x^{k}\boldsymbol{v}W_{1}),\boldsymbol{h}W_{2}).
            \end{split}
        \end{equation}
    \end{enumerate}
    The remaining term is $g_{\text{ice}}(N,\boldsymbol{v}W_{0})\boldsymbol{v}_{+}W_{1},\boldsymbol{v}_{+}W_{2})$ hence we get the first part of the lemma.
    
    \item We have
    \begin{equation}
        \begin{split}
            g_{\text{ice}}(\nabla_{N}&R_{\text{ice}}(N,W_{0})\boldsymbol{h}W_{1},\boldsymbol{h}W_{2})= \\
            &g_{\text{ice}}(\nabla_{N}(R_{\text{ice}}(N,W_{0})\boldsymbol{h}W_{1}),\boldsymbol{h}W_{2})-g_{\text{ice}}(R_{\text{ice}}(N,W_{0})\nabla_{N}\boldsymbol{h}W_{1},\boldsymbol{h}W_{2}).
        \end{split}
    \end{equation}
    The second term is $O(x^{k})$ by the first part of the lemma since $\nabla_{N}\boldsymbol{h}W_{1}$ is a smooth ice vector field (If $W_{0}-\boldsymbol{v}W_{0}$ then the error is $O(x^{2k-1})$ by first part of lemma proof part (a) and asymptotics of the connection). For the first term using that $\nabla$ is a metric connection we have
    \begin{equation}
        \begin{split}
        g_{\text{ice}}(\nabla_{N}&(R_{\text{ice}}(N,W_{0})\boldsymbol{h}W_{1}),\boldsymbol{h}W_{2})= \\
        &Ng_{\text{ice}}(R_{\text{ice}}(N,W_{0})\boldsymbol{h}W_{1},\boldsymbol{h}W_{2})-g_{\text{ice}}(R_{\text{ice}}(N,W_{0})\boldsymbol{h}W_{1},\nabla_{N}\boldsymbol{h}W_{2}).
        \end{split}
    \end{equation}
    Since $\nabla_{N}\boldsymbol{h}W_{2}$ is a smooth ice vector field the second term here also vanishes by the first part of the lemma. Since $W_{0}$ satisfies $W_{0}=x^{k}W_{0}'+\boldsymbol{v}W_{0}$ for some smooth vector field $W_{0}'$, for the first summand 
    \begin{equation}\label{38}
        \begin{split}
        Ng_{\text{ice}}&(R_{\text{ice}}(N,x^{2k}W_{0}')\boldsymbol{h}W_{1},\boldsymbol{h}W_{2})= \\
        &N(x)x^{k-1}g_{\text{ice}}(R_{\text{ice}}(N,W_{0}')\boldsymbol{h}W_{1},\boldsymbol{h}W_{2})+x^{k}Ng_{\text{ice}}(R_{\text{ice}}(N,W_{0}')\boldsymbol{h}W_{1},\boldsymbol{h}W_{2}).
        \end{split}
    \end{equation}
    Any $\partial_{x}$ part of $N$ and $W_{0}'$ produces $O(x^{k})$ terms by the first part of the lemma, and for any vertical part of $N$ we have $R(\boldsymbol{v}N,W'_{0})=-R(W'_{0},\boldsymbol{v}N)$ with these two vector fields satisfying the first part of the lemma so this term is also $O(x^{k})$. Including the factor of $x^{2k-1}$ which appears in \eqref{38}, the $\partial_{x}$ and vertical parts of $N$ thus contribute a factor which is $O(x^{2k-1})$.  This leaves us with only horizontal terms which leaves the term which appears in the second part of the lemma \eqref{27}.
    
    Finally, for the term with the second summand $\boldsymbol{v}W_{0}$
     \begin{equation}
        \begin{split}
        Ng_{\text{ice}}&(R_{\text{ice}}(N,\boldsymbol{v}W_{0})\boldsymbol{h}W_{1},\boldsymbol{h}W_{2}) \\
        &=-Ng_{\text{ice}}(R_{\text{ice}}(\boldsymbol{h}W_{1},N)\boldsymbol{v}W_{0}+R_{\text{ice}}(\boldsymbol{v}W_{0},\boldsymbol{h}W_{1})N,\boldsymbol{h}W_{2}) \\
        &=N(g_{\text{ice}}(R_{\text{ice}}(N,\boldsymbol{h}W_{1})\boldsymbol{v}W_{0},\boldsymbol{h}W_{2})-g_{\text{ice}}(R_{\text{ice}}(\boldsymbol{v}W_{0},\boldsymbol{h}W_{1})N,\boldsymbol{h}W_{2})) \\
        &=N(g_{\text{ice}}(R_{\text{ice}}(N,\boldsymbol{h}W_{1})\boldsymbol{v}W_{0},\boldsymbol{h}W_{2})-g_{\text{ice}}(R_{\text{ice}}(N,\boldsymbol{h}W_{2})\boldsymbol{v}W_{0},\boldsymbol{h}W_{1})) \\
        &=N(x^{k})(g_{\text{ice}}(R_{\text{ice}}(N,\boldsymbol{h}W_{1})x^{-k}\boldsymbol{v}W_{0},\boldsymbol{h}W_{2})-g_{\text{ice}}(R_{\text{ice}}(N,\boldsymbol{h}W_{2})x^{-k}\boldsymbol{v}W_{0},\boldsymbol{h}W_{1})).
        \end{split}
    \end{equation}
    The first equality follow from the Bianchi identity, the second by skew-symmetry of the curvature, the third by interchange symmetry of the curvature.
    
    Simplifying these terms we have
    \begin{equation}
        \begin{split}
            g_{\text{ice}}&(R_{\text{ice}}(N,\boldsymbol{h}W_{1})x^{-k}\boldsymbol{v}W_{0},\boldsymbol{h}W_{2}) \\
            &=g_{\text{ice}}(R_{\text{ice}}(N(x)\partial_{x},\boldsymbol{h}W_{1})x^{-k}\boldsymbol{v}W_{0},\boldsymbol{h}W_{2}) \\
            &=N(x)g_{\text{ice}}(\nabla_{\partial_{x}}\nabla_{\boldsymbol{h}W_{1}}x^{-k}\boldsymbol{v}W_{0},\boldsymbol{h}W_{2}) \\
            &=N(x)\partial_{x}g_{\text{ice}}(\nabla_{\boldsymbol{h}W_{1}}x^{-k}\boldsymbol{v}W_{0},\boldsymbol{h}W_{2}) \\
            &=-\frac{k}{2}N(x)x^{k-1}g_{\partial M/Y}(\mathcal{R}^{\phi}(\boldsymbol{h}W_{1},\boldsymbol{h}W_{2}),\boldsymbol{v}W_{0})+O(x^{k}).
        \end{split}
    \end{equation}
    The first equality holds since the curvature asymptotically preserves the splitting for vector fields tangent to the boundary, the second since the other two terms in the curvature vanish, the third using that $\nabla$ is a metric connection and the final equality uses the asymptotics of the connection.
    
\end{enumerate}
\end{proof}

\section{Clifford modules and Dirac operators}

Let $(M,g)$ be an incomplete cusp edge space. We define the \textbf{incomplete edge Clifford bundle} $\ice\mathbb{C}l(M):=\mathbb{C}\otimes Cl(\ice TM)$ whose fibre over a point is the complexified Clifford algebra of the incomplete cusp edge tangent space $\ice T_{p}M$. 

Let $E=E_{0}\oplus E_{1}$ be a $\mathbb{Z}_{2}$-graded complex vector bundle on $M$ with Hermitian metric $g_{E}$ such that the odd and even parts are orthogonal and $\nabla^{E}$ a connection on $E$ compatible with $g_{E}$. We define $E$ to be an \textbf{incomplete cusp edge Clifford module} if it has a graded action of $\mathbb{C}l\ice(M)$ such that
\begin{align}\label{25}
    g_{E}(\operatorname{cl}(\theta)\cdot,\cdot)=-g_{E}(\cdot,\operatorname{cl}(\theta)\cdot)).
\end{align}
The connection $\nabla^{E}$ is an \textbf{incomplete cusp edge Clifford connection} if for all $V\in C^{\infty}(M;TM)$
\begin{align}\label{37}
    [\nabla^{E}_{V},\operatorname{cl}(\theta)]u=\operatorname{cl}(\nabla_{V}\theta)u.
\end{align}
Here $\nabla$ is the connection on the $\ice\mathbb{C}l(M)$ induced by the Levi-Civita connection of $M$. 

Let $E$ be an Clifford module, then the Dirac operator on $E$ as the composition
\begin{align}
    C^{\infty}(M^{\circ},E) \overset{\nabla^{E}}{\longrightarrow} C^{\infty}(M^{\circ},T^{\star}M\otimes E) \overset{\operatorname{cl}}{\longrightarrow} C^{\infty}(M^{\circ},E).
\end{align}
If $E$ is an ice Clifford module then the Dirac operator extends to an operator on smooth sections of $E$ up to the boundary of $M$. In a local orthonormal frame in a neighbourhood of the boundary the operator is given by
\begin{align}\label{47}
    \slashed{\partial}_{E}=\operatorname{cl}(\partial_{x})\nabla^{E}_{\partial_{x}}+\sum_{i}\operatorname{cl}(x^{-k}V_{i})\nabla^{E}_{x^{-k}V_{i}}+\sum_{j}\operatorname{cl}(U_{j})\nabla^{E}_{\tilde{U}_{j}}.
\end{align}
For the spin Dirac operator we will simply write $\slashed{\partial}$ rather than $\slashed{\partial}_{\mathcal{S}}$. 

Now suppose $M$ is spin and consider the connection on the spinor bundle. Given an orthonormal frame of the tangent bundle $X_{i}$, we have the Christoffel symbols given by
\begin{align}
    \nabla_{X_{i}}X_{j}=\sum_{k}\gamma_{ij}^{k}X_{k}.
\end{align}
Given a local section of the spin bundle covering the section of the frame bundle given by $X_{i}$, a local section of the spinor bundle is given by a map to the spinor space on $\mathbb{R}^{n}$ and the action of the connection on the spinor bundle is
\begin{align}\label{52}
    \nabla^{\mathcal{S}}_{X_{i}}s=X_{i}s+
    \frac{1}{4}\sum_{jk}\gamma_{ij}^{k}\operatorname{cl}(X_{j})\operatorname{cl}({X_{k}})s.
\end{align}
Now consider the restriction of $\mathcal{S}$ to a collar neighbourhood $\mathcal{U}=\partial M\times [0,\epsilon)$ of the boundary. We can identify the restriction of $\mathcal{S}$ to this collar neighbouhood with the pullback bundle of $\mathcal{S}|_{Z}$ by the projection $\pi\colon Z\times [0,\epsilon)\to Z$ by parallel transport along the vector field $\partial_{x}$. 

Denoting by $\ice T\partial M$ the sum of the horizontal and vertical bundles, We can identify the restriction $\ice T\partial M|_{\partial M}$ with $T\partial M$ with metric $g_{\partial M}=\phi^{\star}g_{Y}+g_{\partial M/Y}$ using the map $x^{-k}V_{i}\to V_{i}$ and $\tilde{U}_{\alpha}\to\tilde{U}_{\alpha}$. Recalling that the Clifford bundle over the boundary $\mathbb{C}l(\ice T\partial M|_{\partial M})$ can be identified with the positive part of the Clifford bundle $\mathbb{C}l^{+}(\ice TM|_{\partial M})$ by a map which is given on terms of Clifford degree 1 by
\begin{align}
    \cl(x^{-k}V)\mapsto -\cl(\partial_{x})\cl(x^{-k}V).
\end{align}
This identification induces an identification of the (ungraded) spinor bundle $\mathcal{S}_{\partial M}$ (on $(\partial M,g_{\partial M})$) with $\mathcal{S}^{+}|_{\partial M}$. We also identify $\mathcal{S}^{-}$ with $\mathcal{S}_{\partial M}$ using the isomorphism $\cl(\partial_{x})\colon \mathcal{S}^{+}\to\mathcal{S}^{-}$.

Under these identifications, the restriction of the connection $\nabla^{\mathcal{S}}$ to $\mathcal{S}|_{\partial M}\simeq \mathcal{S}_{\partial M}\oplus\mathcal{S}_{\partial M}$ is given by the sum of the connection $\nabla^{\mathcal{S}_{\partial M}}$ on the two factors. The connection $\nabla^{\mathcal{S}}|_{\partial M}$ acting on the odd part of the spinor bundle is identified with $\nabla^{\mathcal{S}_{\partial M}}$ using the identification of the even and odd parts by $\cl(\partial_{x})$ we have $-\cl(\partial_{x})\nabla^{\mathcal{S}}\cl(\partial_{x})=\nabla^{\mathcal{S}}+O(x^{k-1})$ using the asymptotics of the connection and that $\nabla^{\mathcal{S}}$ is a Clifford connection.

The restriction of $\mathcal{S}|_{\partial M}\to \partial M$ to any fibre $Z_{y}=\phi^{-1}(y)$ is a Clifford module for the vertical Clifford bundle $\mathbb{C}l((x^{-k}T\partial M/Y)|_{Z_{y}})$ which by the above is identified with $\mathcal{S}_{\partial M}\oplus\mathcal{S}_{\partial M}|_{Z_{y}}$ with a $\mathbb{C}l(TZ)$-Clifford action. The restriction of the connection $\nabla^{\mathcal{S}}$ to the fibre $Z_{y}$ is a Clifford connection for $\mathbb{C}l((x^{-k}T\partial M/Y)|_{Z_{y}})$ whose action locally is given by the first line of \eqref{vertical connection}. We will write $\nabla^{\partial M/Y}$ for this family of connections as well as the corresponding family on $\mathcal{S}_{\partial M}\oplus\mathcal{S}_{\partial M}$. Note that if the base and $T(\partial Z/Y)$ are spin (so in particular the fibres are spin), then we have $\mathcal{S}_{\partial M}\simeq \mathcal{S}_{Z}\otimes \mathcal{S}_{B}$ where $\mathcal{S}_{B}$ is trivial on each fibre and this family of connections is exactly the family of spin connections on $\mathcal{S}_{Z}$.

\begin{lemma}\label{2.3}
    Let $M$ be an incomplete cusp edge space with spin structure and $\partial_{x},\tilde{U}_{i},V_{j}$ a local orthonormal frame in the collar neighbourhood of the boundary $\mathcal{U}$. Then given a local section of the spinor bundle covering this frame which locally identifies sections of the spinor bundle with maps to the spinor space as described above, the action of the Dirac operator $\slashed{\partial}_{0}$ for $g_{0}$ is given by
    \begin{align}\label{spindirac0}
        \slashed{\partial}_{0}=\operatorname{cl}(\partial_{x})\partial_{x}+\frac{kf}{2x}\operatorname{cl}(\partial_{x})+\sum_{j}\operatorname{cl}(\tilde{U}_{j})\tilde{U}_{j}+x^{-k}\sum_{i}\operatorname{cl}(V_{i})\nabla^{\partial M/Y}_{V_{i}}+B.
    \end{align}
    where $B\in C^{\infty}(\mathcal{U},\operatorname{End}(\mathcal{S}))$ has odd Clifford degree and of the form $B_{0}+x^{k}B_{1}$ where $B_{i}$ commutes with $\nabla_{\partial_{x}}$.
    For $g$ we have
    \begin{align}
        \slashed{\partial}=\slashed{\partial}_{0}+x^{k-1}E+x^{k}P.
    \end{align}
    where $E\in C^{\infty}(\mathcal{U},\operatorname{End}(\mathcal{S}))$ and $P\in\operatorname{Diff}_{\ice}^{1}$.
\end{lemma}

\begin{proof}
    We have
    \begin{align}
        \slashed{\partial}_{0}=\operatorname{cl}(\partial_{x})\nabla^{\mathcal{S}}_{\partial_{x}}+\sum_{j}\operatorname{cl}(\tilde{U}_{j})\nabla^{\mathcal{S}}_{\tilde{U}_{j}}+\sum_{i}\operatorname{cl}(\tilde{x^{-k}V}_{j})\nabla^{\mathcal{S}}_{x^{-k}V_{j}}.
    \end{align}
    Given a local section of the spinor bundle covering the given section of the orthonormal frame bundle, using the asymptotics of the connection $\nabla^{\mathcal{S}}_{\partial_{x}}=\partial_{x}$  and for the horizontal derivative  $\nabla^{\mathcal{S}}_{\tilde{U}}=\tilde{U}+B'$ for some endomorphism $B'$ which from the asymptotics of the connection has terms which are both $O(1)$ and $O(x^{k})$, have Clifford degree 1 or 3 and commute with $\nabla_{\partial_{x}}$ up to the factor of $x^{k}$. For the vertical derivative
    \begin{align}\label{vertical connection}
        \begin{split}
        \nabla_{V_{i}}^{\mathcal{S}}=V_{i}&+\sum_{ml}g_{\partial M/Y}(\nabla^{\partial M/Y}_{V_{i}}V_{k},V_{l})\operatorname{cl}(x^{-k}V_{k})\operatorname{cl}(x^{-k}V_{l}) \\
        &+\frac{1}{4}\sum_{m}kx^{k-1}(-\operatorname{cl}(x^{-k}V_{k})\operatorname{cl}(\partial_{x})+\operatorname{cl}(\partial_{x})\operatorname{cl}(x^{-k}V_{k}))+O(x^{k}). 
        \end{split}
    \end{align}
    Now acting on this by $x^{-k}\operatorname{cl}(V_{i})$ gives
    \begin{align}
        x^{-k}\operatorname{cl}(V_{i})\nabla_{V_{i}}^{\mathcal{S}}=x^{-k}\operatorname{cl}(V_{i})\nabla_{V_{i}}^{\partial M/Y}+\frac{k}{2x}\operatorname{cl}(\partial_{x})+B_{i}.
    \end{align}
    For $g$, we use the fact that there exists a local orthonormal frame $\xi,f_{j},e_{i}$ of $\ice TM$ which differs from $\partial_{x},\tilde{U}_{j},x^{-k}V_{i}$ by an error which is $O(x^{k})$. Hence $\cl(\xi)\nabla_{\xi}=\cl(\partial_{x})\nabla_{\partial_ x}+x^{k}E'\nabla_{\partial_{x}}$ for some section $E'$ of $\operatorname{End}(\mathcal{S})$ and similarly for the other vector fields so from which the error terms give the $x^{k}P$ for some $P\in\operatorname{Diff}^{1}_{\ice}$. Also the difference between the connection induced by $g_{0}$ and $g$ is $x^{k-1}\omega$ for some smooth $\operatorname{End}(\mathcal{S})$ valued 1-form so this error produces a term $x^{k-1}E$ for some smooth section $E$ of $\operatorname{End}(E)$. 
\end{proof}

Since on each fibre $\nabla^{\partial M/Y}$ is a $\mathbb{C}l(TZ)$ connection, the corresponding term in the above lemma is actually just the Dirac operator on the fibre hence is globally defined on each $Z$. Thus we will write
\begin{align}
    \slashed{\partial}_{Z}=\sum_{i}\operatorname{cl}(V_{i})\nabla^{\partial M/Y}_{V_{i}}.
\end{align}
for this family and $\slashed{\partial}_{Z,y}$ for the restriction to the fibre $\phi^{-1}(y)$. Using the identification of the Clifford bundle of the boundary with the positive part of the Clifford bundle on $M$, the family of spin Dirac operators $\slashed{\partial}_{\partial M/Y}$ on $\mathcal{S}_{Z}$ is identified with $\mp\cl(\partial_{x})\slashed{\partial}_{Z}$ upon restriction to $\mathcal{S}^{\pm}$.

\begin{lemma}\label{squarespindirac}
    The action of the square of the spin Dirac operator is given by
    \begin{align}
        \begin{split}
             \slashed{\partial}^{2}_{0}&=-\left(\partial_{x}+\frac{fk}{2x}\right)^{2}-kx^{-k-1}\operatorname{cl}(\partial_{x})\slashed{\partial}_{Z}-x^{-2k}\slashed{\partial}_{Z}^{2}+\left(\sum_{j}\cl(\tilde{U}_{j})+B\right)^{2} \\
             &\quad +x^{k-1}\cl(\partial_{x})B_{1}+x^{-k}\left(\slashed{\partial}_{Z}B+B\slashed{\partial}_{Z}\right) \\
             \slashed{\partial}^{2}&=\slashed{\partial}^{2}_{0}+\left(\slashed{\partial}_{Z}P+P\slashed{\partial}_{Z}\right)+x^{-1}Q.
        \end{split}
    \end{align}
    where $Q\in \operatorname{Diff}^{2}$, $P\in \operatorname{Diff}_{\ice}^{1}$ from Lemma \ref{2.3} and $B_{1}\in C^{\infty}(\mathcal{U},\operatorname{End}(\mathcal{S}))$ from Lemma \ref{2.3}.
\end{lemma}

\begin{proof}
    The expression for $\slashed{\partial}^{2}_{0}$ follows from \eqref{spindirac0} of the preceding lemma. The $\cl(\partial_{x})\partial_{x}$ acting on the $x^{-k}$ produces the $-kx^{-k-1}\operatorname{cl}(\partial_{x})\slashed{\partial}_{Z}$ term otherwise it anticommutes with the last three terms of \eqref{spindirac0}. Similarly, the other mixed terms which do not appear are due to the anticommuation of those terms.

    For $\slashed{\partial}^{2}$, since all the terms of \eqref{spindirac0} except $x^{-k}\slashed{\partial}_{Z}$ are in fact smooth up to the boundary, their compositions with the error terms $P$ and $E$ are also smooth hence can be included in the error term $x^{-1}Q$ above. This leaves the terms containing $\slashed{\partial}_{Z}$ of which the compositions with $x^{k}P$ give the anticommutator term while the compositions with $x^{k-1}E$ are of the form $x^{-1}Q$ where $Q\in \operatorname{Diff}^{2}$.
\end{proof}

If we have a twisting of the spinor bundle by a Hermitian bundle $E$ with connection $\nabla^{E}$, then in a local coordinates its connection form may have terms which are $O(x^{n})$ for any non-negative $n$ in which case the endomorphism terms from $\nabla_{x^{-k}V}^{E}$ can also be $O(x^{-m})$ for $-k\leq m\leq 1$. This also holds for arbitrary Clifford modules on (not necessarily spin) incomplete cusp spaces as all this holds locally where we can always choose a spin structure.

For the above lemma, when we have a twisting of the spinor bundle, terms of the form $\nabla^{E}_{\partial_{x}}$ and $\nabla^{E}_{\tilde{U}}$ do not change the form of $\slashed{\partial}_{E}$ given in the lemma as the endomorphism terms are all smooth up to the boundary. However, by the above comments, the $\nabla_{x^{-k}V}^{E}$ terms can contribute extra endomorpisms of the form $x^{-j}B$ where $j\leq k$. So for a general twisting we have the above lemma holds with the following form for the Dirac operator. 
\begin{align}\label{diraclemmatwisting}
        \slashed{\partial}_{E}=\operatorname{cl}(\partial_{x})\partial_{x}+\frac{kfx^{k-1}}{2}\operatorname{cl}(\partial_{x})+\sum_{j}\operatorname{cl}(\tilde{U}_{j})\tilde{U}_{j}+x^{-k}\sum_{i}\operatorname{cl}(x^{-k}V_{i})\nabla^{E,\partial M/Y}_{V_{i}}+\sum_{j=0}^{k-1} x^{-k+j}B_{j}.
\end{align}
where the $B_{j}$ are smooth sections of $\operatorname{End}(E)$ and $\nabla^{E,\partial M/Y}$ is the family of connections obtained by restriction $\nabla^{E}$ to each fibre. We will write $\slashed{\partial}_{E,Z}$ for the associated family of Dirac operators which appears in the above expression. In analogy with the case of the spin Dirac operator, we will define the Clifford module $E_{\partial M}=E^{+}|_{\partial M}\to \partial M$ with Clifford action given by the identification of Clifford bundles defined above. 

For the Dirac operator $\nabla^{E}_{\partial_{x}}$ of a Clifford module to have the same form as in the above lemma, which is needed in the construction of the heat kernel, we must make a restriction on the allowed connections.

Let $E$ be an ice Clifford module with ice Clifford connection $\nabla^{E}$. Identifying $E$ with the pullback of the projection onto $Z$ in a collar neighourhood as described above for the spinor bundle we define a connection on the collar neighbourhood
\begin{align}
    \nabla^{E,0}=\partial_{x}\otimes dx+\tilde{\nabla}^{E}.
\end{align}
where $\tilde{\nabla}^{E}$ is the pullback of the restriction of $\nabla^{E}$ to the boundary. Then the difference $\tilde{\omega}^{E}=\nabla^{E}-\tilde{\nabla}^{E}$ is an endomorphism valued one-form.

If the induced family of Dirac operator on the boundary $\slashed{\partial}_{Z}$ has trivial kernel, then we will see that we do not need to make any assumptions on the properties of $\tilde{\omega}^{E}$ if we wish to construct the heat kernel for $\slashed{\partial}_{E}^{2}$. However, for the calculation of the index formula we will need to make the following assumption although the formula will end up holding in more generality.
\begin{manualassumption}{1}\label{assumption1}
    $\tilde{\omega}^{E}=x^{k-1}\omega^{E}$ where $\omega^{E}$ is bounded up to $Z$.
\end{manualassumption}
This assumption would also be needed in the case of non-trivial kernel however, we will only briefly comment on this case later.

We can construct a Clifford module with Clifford connection satisfying this assumption by twisting the spinor bundle with any Hermitian vector bundle with connection which satisfies the same assumption.

\begin{example}
    The Hodge-de Rham operator $d+\delta$ on the incomplete cusp edge form bundle $\ice\Lambda(M)$ satisfies Assumption 1 and the induced boundary operator has non-trivial kernel. 

     Recall that locally we have $\ice\Lambda T^{\star}M\otimes \mathbb{C}\simeq \mathcal{S}\otimes \mathcal{S}^{\star}$ and under this identification we have $\nabla^{M}=\nabla^{\mathcal{S}}\otimes \operatorname{Id}_{\mathcal{S}^{\star}}+\operatorname{Id}_{\mathcal{S}}\otimes\nabla^{\mathcal{S}^{\star}}$. Denoting the Clifford action on $\Lambda T^{\star}M\otimes\mathbb{C}$ by $\cl_{\Lambda}$ then the Clifford action on the left factor of $\mathcal{S}\otimes\mathcal{S}^{\star}$ is equal to $\cl_{\Lambda}$. On the other hand, the dual factor $\mathcal{S}^{\star}$ also has a Clifford module structure whose corresponding action on the ice exterior bundle is given by $\cl_{\mathcal{S}^{\star}}(e)=\epsilon(e)+\iota(e)$. As a Clifford module $\ice\Lambda T^{\star}M$ is isomorphic to $\mathbb{C}l\ice(M)$ with Clifford action given by left Clifford multiplication which we denote $\cl_{L}$. The Clifford action on the dual factor becomes right Clifford multiplication $\cl_{R}$ under this isomorphism.

     The connection $\nabla^{\mathcal{S}^{\star}}$ is locally given by the same expression and for the spin connection $\nabla^{\mathcal{S}}$ with the Clifford action replaced by its dual
     \begin{align}
        \nabla^{\mathcal{S}^{\star}}_{X_{i}}s=X_{i}s+
    \frac{1}{4}\sum_{jk}\gamma_{ij}^{k}\operatorname{cl}_{\mathcal{S}^{\star}}(X_{j})\operatorname{cl}_{\mathcal{S}^{\star}}(X_{k})s.
    \end{align}
    Thus the local expression for the Levi-Civita connection which we take to be acting on $\mathbb{C}l\ice(M)$ can be written as
    \begin{align}
        \nabla_{X_{i}}s=X_{i}s+
    \frac{1}{4}\sum_{jk}\gamma_{ij}^{k}(\operatorname{cl}_{L}(X_{j})\operatorname{cl}_{L}({X_{k}})+\operatorname{cl}_{R}(X_{j})\operatorname{cl}_{R}({X_{k}}))s.
    \end{align}
    So locally  the $\partial_{x}$ and horizontal derivatives have the same form as in Lemma \ref{spindirac0} and the vertical derivatives have the same form with an extra term
    \begin{align}
        \begin{split}
        \nabla_{V_{i}}^{\mathcal{S}}=V_{i}&+\sum_{ml}g_{\partial M/Y}(\nabla^{\partial M/Y}_{V_{i}}V_{k},V_{l})\operatorname{cl}(x^{-k}V_{k})\operatorname{cl}(x^{-k}V_{l}) \\
        &+\frac{1}{4}\sum_{m}kx^{k-1}(-\operatorname{cl}_{L}(x^{-k}V_{k})\operatorname{cl}_{L}(\partial_{x})+\operatorname{cl}_{L}(\partial_{x})\operatorname{cl}_{L}(x^{-k}V_{k})) \\
        &+\frac{1}{4}\sum_{m}kx^{k-1}(-\operatorname{cl_{R}}(x^{-k}V_{k})\operatorname{cl}_{R}(\partial_{x})+\operatorname{cl}_{R}(\partial_{x})\operatorname{cl}_{R}(x^{-k}V_{k}))+O(x^{k}).
        \end{split}
    \end{align}
     The Hodge-de Rham operator is equal to the Dirac operator for this Clifford connection so we see that it is given locally by
\begin{align}\label{hodgederhamexplicit}
    \begin{split}
        d+\delta&=\operatorname{cl}_{\Lambda}(\partial_{x})\nabla^{M}_{\partial_{x}}+\sum_{\alpha}\operatorname{cl}_{\Lambda}(\tilde{U_{\alpha}})\nabla^{M}_{\tilde{U}_{\alpha}}+x^{-k}\sum_{i}\operatorname{cl}_{\Lambda}(x^{-k}V_{i})\nabla^{M}_{V_{i}} \\
            &=\operatorname{cl}_{L}(\partial_{x})\partial_{x}+\sum_{\alpha}\operatorname{cl}_{\Lambda}(\tilde{U}_{\alpha})\tilde{\nabla}^{M}_{\tilde{U}_{\alpha}}+x^{-k}\sum_{i}\operatorname{cl}_{L}(x^{-k}V_{i})\tilde{\nabla}^{M}_{V_{i}} \\
            &\quad +\frac{kf}{2x}\operatorname{cl}_{L}(\partial_{x})+\frac{k}{2x}\sum_{i}\operatorname{cl}_{L}(V_{i})\operatorname{cl}_{R}(\partial_{x})\operatorname{cl}_{R}(V_{i})+O(x^{k}).
    \end{split}
\end{align}
Every term here is in fact globally defined on the restriction of $^{\text{ice}}\Lambda T^{\star}M\otimes \mathbb{C}$ to $Z\times [0,\epsilon)$ identified with the pullback of $\Lambda T^{\star}M\otimes \mathbb{C}|_{Z}$ by the projection onto the first factor (the identification given by $(x^{-k}V_{i})_{x=0}\mapsto((x^{-k}V_{i})_{x=a})$ for $a\in(0,\epsilon)$ or equivalently parallel transport by $\partial_{x}$) Here as before $\tilde{\nabla}^{M}$ denotes the pullback of the restriction of the connection to the boundary. Consider the last two to terms of \eqref{hodgederhamexplicit}
\begin{align}\label{hodgederhamexplicit2}
    \operatorname{cl}_{L}(\partial_{x})\left(\frac{kf}{2x}-\frac{k}{2x}\sum_{i}\operatorname{cl}_{L}(\partial_{x})\operatorname{cl}_{L}(x^{-k}V_{i})\operatorname{cl}_{R}(\partial_{x})\operatorname{cl}_{R}(x^{-k}V_{i})\right).
\end{align}
In fact, using the identification $^{\text{ice}}\Lambda T^{\star}M\otimes \mathbb{C}\simeq\mathbb{C}l(^{\text{ice}}T^{\star}M)$ we see that if $\cl(x^{-k}V_{i_{1}})\ldots\cl(x^{-k}V_{i_{n}})\in \mathbb{C}l(^{\text{ice}}T^{\star}M)$ is a product of $n$ distinct orthonormal vertical vector fields then
\begin{align}
    \begin{split}
    \operatorname{cl}_{L}(\partial_{x})&\operatorname{cl}_{L}(x^{-k}V_{i})\operatorname{cl}_{R}(\partial_{x})\operatorname{cl}_{R}(x^{-k}V_{i})(\cl(x^{-k}V_{i_{1}})\ldots\cl(x^{-k}V_{i_{n}})) \\
    &=\pm\cl(x^{-k}V_{i_{1}})\ldots\cl(x^{-k}V_{i_{n}}).
    \end{split}
\end{align}
where we get a $+$ if the product does not contain $V_{i}$ and $-$ otherwise. Similarly, the result is unchanged by the addition of any horizontal vectors. If we include a factor of $\cl(\partial_{x})$ at the front then the signs are reversed. So \eqref{hodgederhamexplicit2} acts as $\operatorname{cl}_{L}(\partial_{x})kx^{-1}\mathbf{N}$ on $a$ where $\mathbf{N}$ is the vertical Clifford degree of $a$ (or the vertical number operator on $^{\text{ice}}\Lambda T^{\star}M\otimes \mathbb{C}$) and acts as $\operatorname{cl}_{L}(\partial_{x})kx^{-1}(f-\mathbf{N})$ on $\operatorname{cl}(\partial_{x})a$.

Thus we see that the term in brackets actually commutes with the Hodge Laplacian, in particular it preserves the kernel and its orthogonal complement. 
\end{example}

For an arbitrary Clifford module satisfying assumption 1 consider the section $\rho$ of $\operatorname{End}(E)$ defined by
\begin{align}\label{2.16}
    \operatorname{cl}_{E}(\partial_{x})\rho=\sum_{i}\operatorname{cl}_{E}(x^{-k}V_{i})\omega^{E}(V_{i})-\operatorname{cl}_{E}(\partial_{x})\frac{kf}{2}.
\end{align}
This is the endomorphism which appears for a general Clifford module in place of the number operator term which appeared for the Hodge-de Rham operator. Using the local identification $E\simeq \mathcal{S}\otimes W$ we can write $\slashed{\partial}^{E}$ in the form of \eqref{hodgederhamexplicit} where in the final term $\operatorname{cl}_{R}(\partial_{x})\operatorname{cl}_{R}(V_{i})$ is replaced by a section of $\operatorname{End}_{\mathbb{C}l(^{\text{ice}}T^{\star}M)}(E)$. In particular, we see that $\rho$ is also a section of $\operatorname{End}_{\mathbb{C}l(^{\text{ice}}T^{\star}M)}(E)$. For the Hodge-de Rham operator $\operatorname{cl}(\partial_{x})\rho$ is exactly \eqref{hodgederhamexplicit}. In general, unlike for the Hodge-de Rham operator $\rho$ will not commute with the boundary Dirac operator.

We will also use the following assumption used in the calculation of an index formula for Clifford modules.
\begin{manualassumption}{2}\label{assumption2}
    $\rho(0)$ commutes with $\cl(\partial_{x})$ and $\slashed{\partial}_{E,Z}^{2}$ hence preserves the orthogonal decomposition $K\oplus K^{\perp}$ where $K=\operatorname{ker}(\slashed{\partial}^{E\prime}_{Z})$.
\end{manualassumption}

\begin{manualassumption}{3}\label{assumption3}
    $\rho(0)$ commutes with $\nabla^{E}$.
\end{manualassumption}

\section{Heat space}

In this section we will construct the heat kernels for $\slashed{\partial}_{E}^2$ for Clifford connections which satisfy assumption 1 and whose Dirac operators have an invertible induced boundary family of operators. We construct the heat kernel on a manifold with corners produced from $M^{2}\times[0,\infty)$ by consecutive blowups on which the heat kernel will be a polyhomogeneous distribution, see Figure \ref{heatspace}. We will only consider the case of a product type metric as this is sufficient for obtaining an index formula for more general metrics however the construction of the heat kernel can also be extended straighforwardly to exact ice metrics.

First, we recall some facts about manifolds with corners and quasi-homogeneous blowup. A \textbf{manifold with corners} (mwc) $M$ is a topological manifold with a smooth atlas of charts modelled on open subsets of $\mathbb{R}^{k}_{+}\times\mathbb{R}^{n-k}$ such that all boundary hypersurfaces are embedded. 

Locally in a neighbourhood a point we have coordinates $x_{1},\ldots,x_{k},y_{1},\ldots.y_{n-k}$ where $x_{i}\geq 0$. A boundary hypersurface is the closure of a connected component of the set of points which are locally given by the vanishing of a single boundary coordinate function $x_{i}$. The closures of the connected components of the set of points given by the vanishing of $l$ coordinate functions $x_{i}$ are faces of codimension $l$. The union of all boundary hypersurfaces is the boundary $\partial M$ of $M$

A \textbf{p-submanifold} $Y\subset X$ is an embedded manifold with corners which is locally given by $x_{1}=\ldots=x_{\alpha}=0=y_{1}=y_{\beta}$ for some $\alpha\leq k$ and $\beta\leq n-k$. An interior p-submanifold is locally given by the vanishing of some $y_{j}$ otherwise it is contained in some boundary hypersurface and we call it a boundary p-submanifold.

A quasi-homogeneous blowup of order $a$ of a boundary p-submanifold requires a choice of interior extension $\tilde{Y}$ of the boundary p-submanifold $Y$ which is an interior p-submanifold such that $Y=\tilde{Y}\cap\partial M$. In fact, two interior extensions $\tilde{Y},\tilde{Y}'$ will define the same quasi-homogeneous blowup of order $a$ if they "agree to order $a$" at $Y$ which means that for any choice of coordinates $x_{i},y_{j}$  such that $\tilde{Y}$ is given by the vanishing of some $y_{1},\ldots,y_{\alpha}$ and $Y$ the vanishing of $x_{1}\ldots, x_{\alpha},y_{1},\ldots,y_{\beta}$ then $\tilde{Y}'$ is locally given by the vanishing of
\begin{align}
    y'_{i}=\sum_{|\alpha|=a} (x')^{\alpha}F_{\alpha}(x',y_{\beta+1},\ldots,y_{n-k})
\end{align}
where $x'=(x_{1},\ldots,x_{\alpha})$ for some smooth function $F_{\alpha}$.

Given a boundary submanifold $Y$ and interior extension $\tilde{Y}$ and local coordinates in a neighbourhood $W$ of $\tilde{Y}$ such that $\tilde{Y}$ is given by the vanishing of some $y_{1},\ldots,y_{\alpha}$ and $Y$ the vanishing of $x_{1}\ldots, x_{\alpha},y_{1},\ldots,y_{\beta}$ we define the quasihomogeneous blowup of order a as follows. We define the weighted radial function $r_{a}\colon\mathbb{R}^{\alpha}_{+}\times\mathbb{R}^{\beta}\to\mathbb{R}$ by
\begin{align}
    r_{a}(x',y')=(x_{1}^{2a}+\ldots +x_{\alpha}^{2a}+y_{1}^{2}+\ldots + y_{\beta}^{2a})^{\frac{1}{2a}}
\end{align}
and the weighted sphere in $\mathbb{R}^{\alpha}\times\mathbb{R}^{\beta}$ by
\begin{align}
    S^{+}_{a}=\{(\omega,\nu)\in\mathbb{R}_{+}^{\alpha}\times\mathbb{R}^{\beta}|r(\omega,\nu)=1\}.
\end{align}
For any subset $U''\subset\mathbb{R}^{k-\alpha}_{+}\times\mathbb{R}^{n-k-\beta}$ with $U=r_{a}^{-1}([0,\epsilon))\times U''\subset W$ we define the blow up $[U,\tilde{Y}\cap U]_{a}:=S^{+}_{a}\times[0,\epsilon)\times U''$ together with a blow-down map
\begin{align}
    \begin{split}
    \beta_{a,U}\colon[U,\tilde{Y}\cap U]_{a}&\to U \\
    (\omega,\nu,R,x'',y'')&\mapsto (R\omega,R^{a}\nu,x'',y'').
    \end{split}
\end{align}
It can be shown that a diffeomorphism of $U$ lifts to a diffeomorphism of the blowup (for example \cite{melroseDA}[Chapter 5]) so these locally defined blowups glue together to an invariantly defined manifold with corners $[X,Y]_{a}$ together with a blow-down map $\beta_{a}\colon[X,Y]_{a}\to X$. The blowup space has a new boundary hypersurface given by $F_{a}=\overline{\beta^{-1}_{a}(Y)}$.

For $i\leq\alpha$ the coordinate $x_{i}$ on $U$ lifts to $R\omega_{i}$ which is a boundary defining function for $F_{a}$ away from $\omega_{i}=0$. Furthermore, away from $\omega_{i}=0$, then we have the equalities
\begin{align}
    x_{i},\quad X_{j}=\frac{x_{j}}{x_{i}}=\frac{\omega_{j}}{\omega_{i}},\quad Y_{k}=\frac{y_{k}}{x_{i}^{a}}=\frac{\nu_{k}}{\omega_{i}^{a}}
\end{align}
so $x_{i},X_{j},Y_{k}$ define are coordinates away from $\omega_{i}=0$ since the latter are weighted projective coordinates on the blown up space.

Similarly, for $i\leq\beta$ we have $y_{i}^{\frac{1}{a}}=R\nu_{i}^{\frac{1}{a}}$ which is a boundary defining function for $F_{a}$ away from $\nu_{i}=0$. Similarly, we have the weighted projective coordinates away from $\nu_{i}=0$ given by
\begin{align}
    y_{i}^{\frac{1}{a}},\quad X_{j}==\frac{x_{j}}{y_{i}^{\frac{1}{a}}}=\frac{\omega_{i}}{\nu_{i}^{\frac{1}{a}}}, \quad Y_{k}=\frac{y_{k}}{y_{i}}=\frac{\nu_{k}}{\nu_{i}}
\end{align}

Now we describe the construction of the heat space in \cite{ice}. First we blow up $t=0$ parabolically with the map $M^{2}\times [0,\infty)_{\tau}\to M^{2}\times [0,\infty)_{t}$ given by $t=\tau^{2}$, so that $\tau=t^{\frac{1}{2}}$ is smooth. Starting with $M^{2}\times [0,\infty)_{\tau}$, we have the $\tau=0$ fibre diagonal in the corner given by
\begin{align}
    \mathcal{B}_{0}=\partial M\times_{\phi_{Y}}\partial M\times \{0\}=\{(p,q,0)\in\partial M\times\partial M\times \{0\}\colon\phi_{Y}(p)=\phi_{Y}(q)\}.
\end{align}
Given coordinates $x,y,z$ on $M$ in a neighbourhood of $\partial M$ where $x$ is a boundary defining function, $y$ coordinates on the base and $z$ coordinates on the fibres, we have induced coordinates near the $\tau=0$ corner on $M^{2}\times[0,\infty)_{\tau}$ given by $x,y,z,x',y',z',\tau$ where the primed coordinates are for the right factor in $M^{2}$. In these coordinates the fibre diagonal is given by
\begin{align}
    \mathcal{B}_{0}=\{x=x'=\tau=0,y=y',z,z'\}.
\end{align}
Now we do a radial blow up of $\mathcal{B}_{0}$ . This produces a mwc $M^{2}_{\text{heat},1}$ with a blow down map $\beta_{1}:M^{2}_{\text{heat},1}\to M^{2}\times[0,\infty)$. This mwc has a boundary hypersurface which we call the back face $\backface=\overline{\beta_{1}^{-1}(\mathcal{B}_{0})}$ which away from the lift of the $x'=0$ face, we can give projective coordinates
\begin{align}\label{50}
    x', \quad T=\frac{t^{\frac{1}{2}}}{x'}, \quad \eta=\frac{y-y'}{x'}, \quad s=\frac{x}{x'}, \quad y',z,z'.
\end{align}
Similarly we can use projective coordinates with respect the $x$ coordinate away from the $x=0$ face which is the same as \eqref{50} interchanging $x'$ with $x$.
\begin{align}
    x, \quad T^*=\frac{t^{\frac{1}{2}}}{x}, \quad \eta^*=\frac{y-y'}{x}, \quad s^*=\frac{x'}{x}, \quad y',z,z'.
\end{align}

Now we describe the second blowup. Consider the neighbourhood of corner at $t=0$ given by $\mathcal{U}^{2}\times\{0\}$. Let $\mathcal{U}^{\circ}=\partial M\times (0,\epsilon)$ and take the fibre diagonal of the map $\pi\colon\mathcal{U}^{\circ}\to Y\times (0,\epsilon)$ inside $\mathcal{U}^{2}$. Let $\mathcal{B}_{1}$ be the intersection of the closure of the lift of $\mathcal{U}\times_{\pi}\mathcal{U}$ with $\backface$. In projective coordinates with respect to $x'$ this submanifold is given by
\begin{align}
    \mathcal{B}_{1}=\{T=0,s=1\}.
\end{align}
We blow up this submanifold quasihomogeneously of order $k-1$ produce a manifold with corners $M^{2}_{\text{heat,2}}$. We will denote the blowdown maps by $\beta_{2\to 1}:M^{2}_{\text{heat},2}\to M^{2}_{\text{heat},1}$ and $\beta_{2}=\beta_{1}\circ\beta_{2\to 1}$. This blowup produces a new boundary hypersurface which we will denote $\frontface$ which is given by $\overline{\beta_{2\to 1}^{-1}(\mathcal{B}_{1})}$.

At $\text{ff}$ we can use projective coordinates with respect to $x'$ which are valid away from the lift of $\text{bf}\setminus\mathcal{B}_{0}$. These are given by
\begin{align}\label{61}
    x', \quad \tilde{T}=\frac{T}{(x')^{k-1}}=\frac{t^{\frac{1}{2}}}{(x')^{k}}, \quad \tilde{\eta}=\frac{\eta}{(x')^{k-1}}=\frac{y-y'}{(x')^{k}}, \quad \tilde{s}=\frac{s-1}{(x')^{k-1}}=\frac{x-x'}{(x')^{k}}, \quad y',z,z'.
\end{align}

Finally, we get the heat space $M^{2}_{\text{heat}}$ by taking the blowup of the lift of the $t=0$ diagonal. At the intersection $\text{tf}\cap\text{ff}$ we have coordinates
\begin{align}\label{tfff}
    \tilde{T},\quad \tilde{E}=\frac{\tilde{\eta}}{\tilde{T}}=\frac{y-y'}{t^{\frac{1}{2}}}\quad \tilde{S}=\frac{\tilde{s}}{\tilde{T}}=\frac{x-x'}{t^{\frac{1}{2}}},\quad \tilde{Z}=\frac{z-z'}{\tilde{T}}=(x')^{k}{\frac{z-z'}{t^{\frac{1}{2}}}},\quad x',y',z'.
\end{align}
Let $\beta\colon M^{2}_{\text{heat}}\to M^{2}\times [0,\infty)$ be the composition of the blowdown maps. We also have the remaining boundary hypersurfaces
\begin{align}
    \begin{split}
    &\leftface=\overline{(\beta^{-1}(\partial M\times \{0\}\times[0,\infty))^{\circ})}, \quad \rightface=\overline{(\beta^{-1}(\{0\} \times \partial M\times[0,\infty))^{\circ})}, \\
    &\bottomface=\overline{(\beta^{-1}(M\times M\times \{0\}\setminus \diag\times \{0\})^{\circ})}.
    \end{split}
\end{align}

\section{Heat kernel}

Let $E$ be an ice Clifford module with Clifford connection $\nabla^{E}$. We define the big homomorphism bundle $\operatorname{HOM}(E)$ on $M^{2}\times [0,\infty)$ to be the vector whose fibre over a point $(a,b,t)$ is $\operatorname{Hom}(E_{b},E_{a})$. We will construct the heat kernel as a polyhomogeneous section of the bundle $\beta^{\star}\operatorname{HOM}(E)$.

Let $\rho_{\frontface},\rho_{\backface}$ be boundary defining functions for $\frontface$ and $\backface$. If $A\in C^{\infty}(M^{2}_{\operatorname{heat}};\beta^{\star}\operatorname{HOM}(E))$ is smooth and vanishing to infinite order at all faces except at $\frontface,\backface,\timeface$ then at $\frontface$, due to rapid decay at the side faces we can write
\begin{align}
    \rho_{\frontface}^{a}A=(x')^{a}A'
\end{align}
where $A'$ is smooth and vanishing to infinite order at the same faces. The same holds at $\backface$ so we can define the \textbf{normal operator} for $A\in \rho_{\frontface}^{a}\rho_{\backface}^{b}C^{\infty}(M^{2}_{\operatorname{heat}};\beta^{\star}\operatorname{HOM}(E))$ by
\begin{align}
    \begin{split}
        \mathcal{N}_{\frontface,a}&=((x')^{-a}A)|_{\frontface} \\
        \mathcal{N}_{\backface,b}&=((x')^{-b}A)|_{\backface}
    \end{split}
\end{align}
which are smooth sections of the restriction of $\operatorname{HOM}(E)$ to the respective faces. For an ice differential operator $P$ of order $n$, $\tau^{n}P$ taken to act on the left factor lifts to an operator on $M^{2}_{\operatorname{heat}}$ mapping $C^{\infty}(M^{2}_{\operatorname{heat}};\beta^{\star}\operatorname{HOM}(E))$ to itself so we can define the normal operator by
\begin{align}
    \begin{split}
        \mathcal{N}_{\frontface,a}(\tau^{n} PA)&=\mathcal{N}_{\frontface,a}( P)\mathcal{N}_{\frontface,a}(A) \\
        \mathcal{N}_{\backface,b}(\tau^{n} PA)&=\mathcal{N}_{\backface,b}( P)\mathcal{N}_{\backface,b}(A)
    \end{split}
\end{align}
We will see that the lift of $t\partial_{t}$ also preserves these spaces so we can define the normal operator for this operator in the same way
\begin{align}
        \mathcal{N}_{\frontface,a}(t\partial_{t}A)&=\mathcal{N}_{\frontface}( t\partial_{t})\mathcal{N}_{\frontface,a}(A).
\end{align}
If the operator $\slashed{\partial}_{E}$ the boundary family has non-trivial kernel forms a vector bundle $\mathcal{K}$, then we can also consider sections of $C^{\infty}(M^{2}_{\operatorname{heat}};\beta^{\star}\operatorname{HOM}(E))$ whose normal operator at $\backface$ is a section of $\mathcal{K}\otimes\mathcal{K}^{\star}$ and define the corresponding normal operator $\mathcal{N}_{\backface,b,\mathcal{K}}(\slashed{\partial}^{E})$ on these sections.
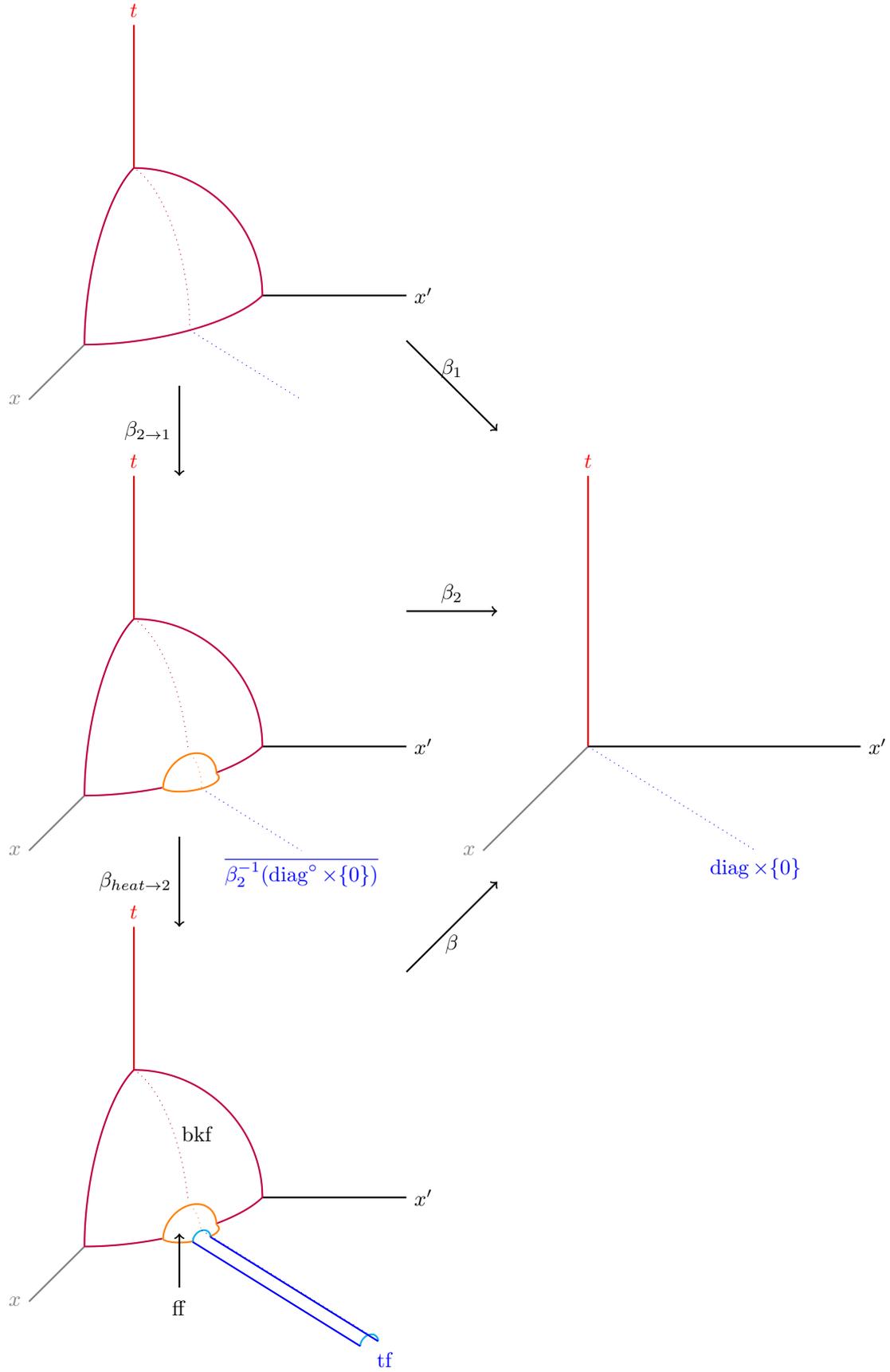
\begin{figure}[H]
\centering
    \begin{tikzpicture}[scale=1.5,rotate around x=90,rotate around z=90,z={(0,0,-1)},y={(0,-1,0)}] 
        \draw[gray, thick] ({sqrt(2)},0,0) -- (3,0,0) node[pos=1,left] {$x$};
        \draw[black, thick] (0,{sqrt(2)},0) -- (0,3,0) node[pos=1,right] {$x'$};
        \draw[red, thick] (0,0,{sqrt(2)}) -- (0,0,3) node[pos=1,above] {$t$};
        \draw[blue, dotted] (1,1,0) -- (3,3,0) node[pos=1,below] {};
        \draw[domain=0:90, smooth, variable=\x, purple, thick] plot ({sqrt(2)*cos(\x)}, {sqrt(2)*(sin(\x)}, {0}); 
        \draw[domain=0:90, smooth, variable=\x, purple, thick] plot ({sqrt(2)*cos(\x)}, {0}, {sqrt(2)*(sin(\x)}); 
        \draw[domain=0:90, smooth, variable=\x, purple, thick] plot ({0}, {sqrt(2)*cos(\x)}, {sqrt(2)*(sin(\x)}); 
        \draw[domain=0:90, smooth, variable=\x, purple, dotted] plot ({cos(\x)}, {cos(\x)}, {sqrt(2)*(sin(\x)}); 

        \draw[->,black,thick] (0,3,-0.5) -- (0,4,-1.5) node[pos=0.5,above] {$\beta_{1}$};


        \draw[gray, thick] ({sqrt(2)},0,-5) -- (3,0,-5) node[pos=1,left] {$x$};
        \draw[black, thick] (0,{sqrt(2)},-5) -- (0,3,-5) node[pos=1,right] {$x'$};
        \draw[red, thick] (0,0,{sqrt(2)-5}) -- (0,0,-2) node[pos=1,above] {$t$};
        \draw[blue, dotted] ({1+0.3*(1/sqrt(2))},{1+0.3*(1/(sqrt(2)))},-5) -- (3,3,-5) node[pos=1,below] {$\overline{\beta^{-1}_{2}(\diag^{\circ}\times\{0\})}$};
        \draw[domain=0:33, smooth, variable=\x, purple, thick] plot ({sqrt(2)*cos(\x)}, {sqrt(2)*(sin(\x)}, {-5}); 
        \draw[domain=58.1:90, smooth, variable=\x, purple, thick] plot ({sqrt(2)*cos(\x)}, {sqrt(2)*(sin(\x)}, {-5}); 
        \draw[domain=0:90, smooth, variable=\x, purple, thick] plot ({sqrt(2)*cos(\x)}, {0}, {(sqrt(2)*(sin(\x))-5}); 
        \draw[domain=0:90, smooth, variable=\x, purple, thick] plot ({0}, {sqrt(2)*cos(\x)}, {(sqrt(2)*(sin(\x))-5}); 
        \draw[domain=15:90, smooth, variable=\x, purple, dotted] plot ({cos(\x)}, {cos(\x)}, {(sqrt(2)*(sin(\x))-5}); 
        \draw[domain=-45:135, smooth, variable=\x, orange, thick] plot ({1+0.3*cos(\x)}, {1+0.3*(sin(\x)}, {-5}); 
        \draw[domain=0:180, smooth, variable=\x, orange, thick] plot ({1+0.3*(1/(sqrt(2)))*cos(\x)}, {1-0.3*(1/(sqrt(2)))*cos(\x)}, {(0.3*sin(\x))-5}); 
        \draw[domain=0:90, smooth, variable=\x, orange, dotted] plot ({1+0.3*(1/(sqrt(2)))*sin(\x)}, {1+0.3*(1/(sqrt(2)))*sin(\x)}, {(0.3*cos(\x))-5}); 

        \draw[->,black,thick] (0,3,-3.5) -- (0,4,-3.5) node[pos=0.5,above] {$\beta_{2}$};


        \draw[gray, thick] (0,5,-5) -- (3,5,-5) node[pos=1,left] {$x$};
        \draw[black, thick] (0,5,-5) -- (0,8,-5) node[pos=1,right] {$x'$};
        \draw[red, thick] (0,5,-5) -- (0,5,-2) node[pos=1,above] {$t$};
        \draw[blue, dotted] (0,5,-5) -- (3,8,-5) node[pos=1,below] {$\operatorname{diag}\times \{0\}$};

        \draw[gray, thick] ({sqrt(2)},0,-10) -- (3,0,-10) node[pos=1,left] {$x$};
        \draw[black, thick] (0,{sqrt(2)},-10) -- (0,3,-10) node[pos=1,right] {$x'$};
        \draw[red, thick] (0,0,{sqrt(2)-10}) -- (0,0,-7) node[pos=1,above] {$t$};
        
        \draw[domain=0:33, smooth, variable=\x, purple, thick] plot ({sqrt(2)*cos(\x)}, {sqrt(2)*(sin(\x)}, {-10}); 
        \draw[domain=58.1:90, smooth, variable=\x, purple, thick] plot ({sqrt(2)*cos(\x)}, {sqrt(2)*(sin(\x)}, {-10}); 
        \draw[domain=0:90, smooth, variable=\x, purple, thick] plot ({sqrt(2)*cos(\x)}, {0}, {(sqrt(2)*(sin(\x))-10}); 
        \draw[domain=0:90, smooth, variable=\x, purple, thick] plot ({0}, {sqrt(2)*cos(\x)}, {(sqrt(2)*(sin(\x))-10}); 
        \draw[domain=15:90, smooth, variable=\x, purple, dotted] plot ({cos(\x)}, {cos(\x)}, {(sqrt(2)*(sin(\x))-10}); 

        \draw[domain=-45:29, smooth, variable=\x, orange, thick] plot ({1+0.3*cos(\x)}, {1+0.3*(sin(\x)}, {-10}); 
        \draw[domain=63:135, smooth, variable=\x, orange, thick] plot ({1+0.3*cos(\x)}, {1+0.3*(sin(\x)}, {-10}); 
        \draw[domain=0:180, smooth, variable=\x, orange, thick] plot ({1+0.3*(1/(sqrt(2)))*cos(\x)}, {1-0.3*(1/(sqrt(2)))*cos(\x)}, {(0.3*sin(\x))-10}); 
        \draw[domain=0:75, smooth, variable=\x, orange, dotted] plot ({1+0.3*(1/(sqrt(2)))*sin(\x)}, {1+0.3*(1/(sqrt(2)))*sin(\x)}, {(0.3*cos(\x))-10}); 

        \draw[domain=0:180, smooth, variable=\x, cyan, thick] plot ({1+0.3*(1/(sqrt(2)))+0.1*(1/(sqrt(2)))*cos(\x)}, {1+0.3*(1/(sqrt(2)))-0.1*(1/(sqrt(2)))*cos(\x)}, {(0.1*sin(\x))-10}); 

        \draw[domain=0:180, smooth, variable=\x, cyan, thick] plot ({3+1+0.3*(1/(sqrt(2)))+0.1*(1/(sqrt(2)))*cos(\x)}, {3+1+0.3*(1/(sqrt(2)))-0.1*(1/(sqrt(2)))*cos(\x)}, {(0.1*sin(\x))-10});
        
        \draw[blue, dotted] ({1+0.3*(1/sqrt(2))},{1+0.3*(1/(sqrt(2)))},0.1-10) -- (4.2,4.2,0.1-10) node[pos=1.1,below] {$\operatorname{\timeface}$}; 
        \draw[blue, thick] ({1+0.3*(1/(sqrt(2)))+0.1*(1/(sqrt(2)))}, {1+0.3*(1/(sqrt(2)))-0.1*(1/(sqrt(2)))}, {-10}) -- ({1+0.3*(1/(sqrt(2)))+0.1*(1/(sqrt(2)))+3}, {1+0.3*(1/(sqrt(2)))-0.1*(1/(sqrt(2)))+3}, {-10}) ; 
        \draw[blue, thick] ({1+0.3*(1/(sqrt(2)))-0.1*(1/(sqrt(2)))}, {1+0.3*(1/(sqrt(2)))+0.1*(1/(sqrt(2)))}, {-10}) -- ({1+0.3*(1/(sqrt(2)))-0.1*(1/(sqrt(2)))+3}, {1+0.3*(1/(sqrt(2)))+0.1*(1/(sqrt(2)))+3}, {-10}) ; 

        \node at (0,0.7,-9.3) {$\backface$};

        \draw[->,black,thick] (0,3,1.-8.5) -- (0,4,1.-7.5) node[pos=0.5,below] {$\beta$};

        \draw[->,black,thick] (0,0.5,-1) -- (0,0.5,-2) node[pos=0.5,left] {$\beta_{2\to 1}$};

        \draw[->,black,thick] (0,0.5,-6) -- (0,0.5,-7) node[pos=0.5,left] {$\beta_{heat\to 2}$};

        \draw[->,black,thick] (0,0.5,-11) -- (0,0.5,-10.4) node[pos=-0.1,below] {$\frontface$};

    \end{tikzpicture}
    \caption{Sequence of blown up spaces for $M^{2}_{\operatorname{heat}}$.} \label{fig:blowup}
    \label{heatspace}
\end{figure}
Note that since multiplication by $\tau$ commutes with the lift of an ice differential operator, the normal operator of a composition is just the composition of the normal operators. We could similarly define normal operators at $\timeface$ however we will solve away directly to infinite order there.

For Dirac operators $\frontface$ and $\backface$ we have the following.
\begin{lemma}
    The normal operators on the heat space of the Dirac operator $\slashed{\partial}_{E}$ on an incomplete cusp edge space do not depend on $a$ and $b$ and $\text{ff}$ locally in the base in coordinates \eqref{61} is
    \begin{align}
        \mathcal{N}_{\text{ff}}(\tau\slashed{\partial}_{E})=\tilde{T}\operatorname{cl}(\partial_{x})\partial_{\tilde{s}}+\tilde{T}\sum_{ij}\operatorname{cl}(a_{ij}\partial_{y_{i}})a_{ij}(0,y')\partial_{\tilde{\eta}_{i}}+\tilde{T}\slashed{\partial}_{Z,y'}.
    \end{align}
    Under assumption \ref{assumption1}, if the induced boundary family has non-trivial kernel which forms a vector bundle, then the normal operator at $\text{bf}$ in coordinates \eqref{61} acting on fibre harmonic sections is
        \begin{align}\label{bf normal}
        \mathcal{N}_{\backface}(\tau\slashed{\partial}_{E})=T\operatorname{cl}(\partial_{x})\partial_{s}+T\sum_{ij}\operatorname{cl}(a_{ij}\partial_{y_{i}})a_{ij}(0,y')\partial_{\eta_{i}}+T\left(\frac{fk}{2s}\operatorname{cl}(\partial_{x})+B_{0}|_{x=0}\right).
    \end{align}
\end{lemma}

\begin{proof}
    \begin{enumerate}
        \item Using Lemma \ref{2.3} and \ref{diraclemmatwisting}, 
                \begin{align}\label{80}
        \tau\slashed{\partial}_{E}=\operatorname{cl}(\partial_{x})\tau\partial_{x}+\frac{kf\tau}{2x}\operatorname{cl}(\partial_{x})+\sum_{j}\operatorname{cl}(\tilde{U}_{j})\tau\tilde{U}_{j}+\frac{\tau}{x^{k}}\sum_{i}\operatorname{cl}(x^{-k}V_{i})\nabla^{\partial M/Y}_{V_{i}}+\sum_{j=0}^{k-1}\tau x^{-k+j}B_{j}.
            \end{align}
            First note that the lift of all these terms commute with multiplication by $x'$ so the normal operators are independent of $a$ and $b$. In projective coordinates near $\text{ff}$ with boundary defining function $x'$ we have $x=x'(\tilde{s}(x')^{k-1}+1)$ so $\tau x^{-k+j}=(x')^{k}x^{-k+j}\tilde{T}$ vanishes at $\text{ff}$ so the normal operators of the lift of the second and last term vanish when restricted to $\text{ff}$. Note that this does not depend on the form of the connection. The lift of the first term is $T\cl(\partial_{x})\partial_{\tilde{s}}$ which restricts to $\frontface$ to give the corresponding normal operator.
            
            For the horizontal vector, recall that in local coordinates for some smooth functions $a_{ij}$ and $b_{ij}$ we have
            \begin{align}                       
            \tilde{U}_{i}=a_{ij}(x,y)\partial_{y_{j}}+b_{ij}(x,y)x^{k}(x^{-k}\partial_{z_{j}}).
            \end{align}
            Since $x^{k}(x^{-k}\partial_{z_{j}})$ lifts to $\partial_{z_{j}}$ the lift of the second term vanishes at $\text{ff}$ while $\tau\partial_{y_{i}}$ lifts to $\tilde{T}\partial_{\tilde{\eta_{i}}}$ thus we get the corresponding horizontal term. Finally, the vertical term is exactly the Dirac operator on the fibre with metric which we defined earlier \eqref{vertical connection} after identifying the vertical ice tangent bundle with the tangent bundle of the fibre by the map $x^{-k}V_{i}\mapsto V_{i}$.

            For a general exact ice metric with polyhomogeneous error term of order $O(x^{k})$ we also have extra terms $x^{k-1}E$ and $x^{k}P$ for some smooth endomorphism $E$ and ice-differential operator $P$ by Lemma \ref{2.3}. But the lift of both these terms vanish by the same calculations as above.
            
        \item For $\text{bf}$, by the same arguments as for $\frontface$ the $\partial_{x}$ and horizontal terms contribute the first and second terms in \eqref{bf normal} while $\frac{kf\tau}{2x}\operatorname{cl}(\partial_{x})$ gives the last term since $T=\frac{\tau}{x}$. Since we are considering the action on the bundle of fibre harmonic forms, the vertical term does not contribute. 

        In general the last term would lift to $T(sx')^{k-j-1}B_{j}$ so we would get singular terms without assumption 1. With assumption 1, the only remaining term is exactly $B_{0}$ which contributes the final term in \eqref{bf normal}.

        The polyhomogeneous error terms for a general exact ice metric clearly vanish in this case at $\backface$ also.
        
    \end{enumerate}
\end{proof}
Now we take the square of the expressions above to get the normal operators of $\tau^{2}\slashed{\partial}_{E}^{2}$. Every term for both normal operators anticommute with each other with the exception of the $\partial_{s}$ and $\frac{1}{s}$ terms in the normal operator at $\backface$ and the $B_{0}$ term with all the other terms. Hence these are the only terms which contributes any cross terms to the normal operator of the square.
\begin{lemma}\label{normaloperator}
    The normal operators on the heat space of the heat operator associated with the Dirac operator on an incomplete cusp edge space with spin structure is
        \begin{align}
        \mathcal{N}_{\text{ff}}(t(\partial_{t}+\slashed{\partial}_{E}^{2}))=\frac{1}{2}\tilde{T}\partial_{\tilde{T}}+\tilde{T}^{2}\left(-\partial_{\tilde{s}}^{2}+\Delta_{\tilde{\eta}',y'}+\slashed{\partial}_{Z,y'}^{2}\right).
    \end{align}
    The normal operator at $\text{bf}$ in coordinates \eqref{61} acting on fibre harmonic sections is
        \begin{align}
        \mathcal{N}_{\text{bkf}}((t(\partial_{t}+\slashed{\partial}_{E}^{2}))=\frac{1}{2}T\partial_{T}+T^{2}\left(-\partial_{s}^{2}-\frac{fk}{s}\partial_{s}+\Delta_{\eta,y'}+\left(\frac{fk}{2}-\frac{f^{2}k^{2}}{4}\right)\frac{1}{s^{2}}+R\right).
    \end{align}
    for some first order translation invariant differential operator $R$ (in the coordinates $(s,\eta)$) which depend on $B_{0}$.
\end{lemma}

    Now we will assume that the boundary family has trivial kernel. In this case, the normal operator of the heat kernel at $\text{ff}$ should be 
    \begin{align}
        \mathcal{N}_{\text{ff}}(H)=(4\pi)^{\frac{b+1}{2}}\tilde{T}^{-(b+1)}e^{-\frac{s^{2}}{4\tilde{T}^{2}}}e^{-\frac{|\tilde{\eta}|}{4\tilde{T}^{2}}}e^{-\tilde{T}^{2}\slashed{\partial}_{Z,y'}}.
    \end{align}
    In particular, the kernel of $e^{-\tilde{T}^{2}\slashed{\partial}_{Z}^{2}}$ vanishes exponentially as $T\to \infty$ and we can construct a heat kernel which will be rapidly vanishing at $\backface$.

    We first construct a distribution which solves the heat equation to infinite order at $\timeface$ and to leading order at $\frontface$ and $\backface$ while vanishing to infinite order at all other faces while satisfying the initial condition. We use this distribution to obtain the heat kernel by removing the error term via Neumann series.

\begin{lemma}
    There exists $K_{1}\in\mathcal{A}_{\text{phg}}(M^{2}_{\operatorname{heat}};\beta^{\star}(\operatorname{HOM}(\mathcal{S})))$ with
    \begin{equation}
        \begin{split}\label{indexsets}
            \mathcal{E}(\leftface)=\mathcal{E}(\rightface)=\mathcal{E}(\backface)=\mathcal{E}(\bottomface)&=\varnothing \\
            \inf(\mathcal{E}(\frontface))&=-nk \\
            \mathcal{E}(\timeface)&=-n
        \end{split}
    \end{equation}
    such that $Q:=t(\partial_{t}+\slashed{\partial}_{E}^{2})K_{1}\in\mathcal{A}_{\text{phg}}(M^{2}_{\text{heat}};\beta^{\star}(\operatorname{HOM}(\mathcal{E})))$ with $\mathcal{E}(\frontface)=-nk+1$ and $\mathcal{E}(-)=\varnothing$ at all other boundary hypersurfaces and defines an operator $K_{1}(t)$ on $L^{2}(M,E)$ such that $K_{1}(t)s\to s$ in $L^{2}$ as $t\to 0$ and is supported in $[0,T]$ for some $T>0$. 
\end{lemma}

\begin{proof}
    Let $w_{i}$ be local coordinates on the interior of $M$ (which we can take to be $x,y,z$ near $\partial M$), then we have projective coordinates in a neighbourhood of $\timeface$ away from $\bottomface$ given by
    \begin{align}
        W_{i}=\frac{w_{i}-w_{i}'}{\tau},w_{i}',\tau.
    \end{align}
    Now with respect to a local orthonormal frame $U_{i}$ we have
    \begin{align}
        \nabla^{E}_{W_{i}}=U_{i}+E
    \end{align}
    where $E$ is a smooth endomorphism.
    
    Thus we have that $\tau\cl(\partial_{W_{i}})\nabla_{W_{i}}$ lifts to $\cl(\partial_{w_{i}})\partial_{W_{i}}+\tau E'$ and so
    \begin{align}\label{diracsquaretf}
        t\slashed{\partial}_{E}^{2}=\Delta_{g(p)}+\tau F
    \end{align}
    where $F$ is some first order differential operator. We also have
    \begin{align}\label{partialttf}
        t\partial_{t}=\frac{1}{2}(\tau\partial_{\tau}-R)
    \end{align}
    where $R=W_{i}\partial_{W_{i}}$ restricts to the radial vector field on each fibre $\timeface\to\operatorname{diag}$. Let $s\in C^{\infty}(M,\mathcal{S})$ then using the coordinates
    \begin{align}
        W_{i}=\frac{w_{i}-w_{i}'}{\tau},w_{i},\tau
    \end{align}
    which are valid near $\timeface$ away from $\bottomface$ and $\frontface$ the action of $K_{1}$ is given by
    \begin{align}
        \begin{split}
            K_{1}s&=\int K(\frac{w-w'}{\tau},w',\tau)s(w')dw' \\
            &=\tau^{n}\int K(W,w-\tau W,\tau)s(w-\tau W)dW.
        \end{split}
    \end{align}
    Thus we see that the condition that $K_{1}s\to s$ as $t\to 0$ implies that
    \begin{align}\label{ictf}
        \tau^{n}\int K(W,w,0)dW=\operatorname{Id}_{E}
    \end{align}
    Now we want to solve for the coefficients of the asymptotic expansion of the heat kernel at $\timeface$. Thus we solve for the coefficients of expansion of the form
    \begin{align}
        \tau^{-n}\sum_{j=0}^{\infty}\tau^{j}b_{j}
    \end{align}
    subject to \eqref{ictf} such that the $b_{j}$ vanish to infinite order at $\backface$. In particular by \eqref{diracsquaretf} and \eqref{partialttf}, $b_{0}$ must satisfy
    \begin{align}
        \left(\frac{-n}{2}-R+\Delta_{g(p)}\right)b_{0}=0.
    \end{align}
    This has a unique solution satisfying \eqref{ictf} given by 
    \begin{align}
        b_{0}=(4\pi)^{-n/2}e^{-\lVert w\rVert^{2}_{g(p)}/4}\operatorname{Id}
    \end{align}
    We can then solve for all the remaining coefficients inductively. Near $\timeface\cap\frontface$, using coordinates \eqref{tfff} we have an expansion of the form
    \begin{align}
            T^{-n}x^{-kn}\sum_{j=0}^{\infty}T^{j}x^{kj}b_{j}'
    \end{align}
    and
    \begin{align}
        b_{0}'=(4\pi)^{-n/2}e^{-\lVert (\tilde{S},\tilde{E},\tilde{Z})\rVert^{2}_{g'(p)}/4}\operatorname{Id}_{E}.
    \end{align}
    where $g'$ is the metric on $M$ induced from $g$ by the identification of $\ice TM$ with $TM$ by $\partial_{x}\to\partial_{x}$, $\tilde{U}\to\tilde{U}$ and $x^{-k}V\to V$. In particular, as $x'\to 0$ we have
    \begin{align}
        b_{0}\to(4\pi)^{-n/2}e^{-| \tilde{S}|^{2}/4}e^{-\lVert \tilde{E}\rVert^{2}_{g_{Y}(p)}/4}e^{-\lVert \tilde{Z})\rVert^{2}_{g_{\partial M/Y}(p)}/4}\operatorname{Id}
    \end{align}
    On $\frontface$, we want the normal operator of the heat kernel to satisfy the equation
    \begin{align}
        \frac{1}{2}\tilde{T}\partial_{\tilde{T}}+\tilde{T}^{2}\left(-\partial_{\tilde{s}}^{2}+\Delta_{\tilde{\eta}',y'}+\slashed{\partial}_{Z,y'}^{2}\right)\mathcal{N}_{\frontface,nk}(H)=0
    \end{align}
    such that $\mathcal{N}_{\frontface,nk}(H)$ converges to $\delta_{\tilde{s}=0}\delta_{\tilde{\eta}=0}\operatorname{Id}_{Z}$. Thus we want to take
    \begin{align}
        \mathcal{N}_{\text{ff}}(H)=(4\pi)^{\frac{b+1}{2}}\tilde{T}^{-(b+1)}e^{-\frac{s^{2}}{4\tilde{T}^{2}}}e^{-\frac{|\tilde{\eta}|}{4\tilde{T}^{2}}}e^{-\tilde{T}^{2}\slashed{\partial}_{Z,y'}}.
    \end{align}
    The coefficients $b_{0}$ and $\mathcal{N}_{\frontface}(H)$ are smooth of $\timeface$ and $\frontface$ respectively as they are heat kernels of a smooth family of generalised Laplacians. Near $\timeface\cap\frontface$ in coordinates $\eqref{tfff}$, the leading term in the asymptotic expansion of $e^{-\tilde{T}^{2}\slashed{\partial}_{Z,y'}}$ as $\tilde{T}\to 0$ is $(4\pi)^{-f/2}e^{-\lVert \tilde{Z})\rVert^{2}_{g_{\partial M/Y}(p)}/4}\operatorname{Id}$ so we see that $(\tilde{T}^{n}\mathcal{N}_{\text{ff}}(H))|_{\timeface\cap\frontface}=b_{0}|_{\timeface\cap\frontface}$.

    We can now construct the distribution $K_{1}$. Let $\rho_{\timeface},\rho_{\frontface}$ be boundary defining functions for the respective boundary hypersurfaces. Since $(\tilde{T}^{n}\mathcal{N}_{\text{ff}}(H))|_{\timeface\cap\frontface}=b_{0}|_{\timeface\cap\frontface}$ and vanish to infinite order at the respective intersections with $\rightface,\leftface,\bottomface$ there exists a smooth $K_{1}'$ such that $K_{1}'|_{\timeface}=b_{0}$, $K_{1}'|_{\frontface}=\tilde{T}^{n}\mathcal{N}_{\text{ff}}(H)$ with no other terms in the asymptotic expansion at $\frontface,\timeface$ and vanishing to infinite order at all other faces.

     Now let $K_{1}''$ any smooth function with asymptotic expansion at $\timeface$
     \begin{align}
         (x')^{kn}\sum_{j=1}^{\infty}\tau^{j}b_{j}.
     \end{align}
     Since the $b_{j}$ vanish to infinite order at $\bottomface$, $K_{1}''$ can be taken to vanish to infinite order at all other faces except $\timeface,\frontface$. Now take
     \begin{align}
         K_{1}=\tilde{T}^{-n}(x')^{-nk}K_{1}'+\tau^{-n}(x')^{-nk}K_{1}''.
     \end{align}
     Multiplying by any cut-off function supported in $t<T$ does not affect any of the asymptotic properties of this distribution.
\end{proof}

Using the distribution $K_{1}$ we can construct the heat kernel by solving away the error term. For this we need to define the composition of our integral kernels as convolution operators. Consider $A,B\in C^{\infty}_{c}(M;\operatorname{HOM}(E))$ acting as convolution operators in the time variable, then the integral kernel of the composition of these operators is given by
\begin{align}
    A\circ B=\int_{M}\int_{0}^{t}A(w,w',t')B(w',w'',t-t')\operatorname{dVol}_{w'}dt'.
\end{align}
Equivalently we can write this as
\begin{align}
    A\circ B\mu=(\pi_{C})_{\star}((\pi_{L}^{\star}A)(\pi_{R}^{\star}B)\mu')
\end{align}
for appropriate density factors $\mu$ and $\mu'$ and the projections are defined by
\begin{align}
    \begin{split}
        \pi_{L}\colon M^{3}\times W &\to M^{2}\times [0,\infty)\\
        (w_{1},w_{2},w_{3},t,t')&\mapsto (w_{1},w_{2},t)\\
        \pi_{R}\colon M^{3}\times W &\to M^{2}\times [0,\infty)\\
        (w_{1},w_{2},w_{3},t,t')&\mapsto (w_{2},w_{3},t-t')\\
        \pi_{C}\colon M^{3}\times W &\to M^{2}\times [0,\infty)\\
        (w_{1},w_{2},w_{3},t,t')&\mapsto (w_{1},w_{3},t)
    \end{split}
\end{align}
where W is defined as
\begin{align}
    W=\{(t,t')\in [0,\infty)^{2}:t>t'\}
\end{align}
The pushforward in this case is simply integration along the fibres which for a given point $(w_{1},w_{3},t)$ is a copy of $M\times [0,t)$.

To extend this composition to polyhomogeneous sections on the heat space $M^{2}_{\operatorname{heat}}$, a heat triple space $M^{3}_{\operatorname{heat}}$ is constructed by blow-up of $M^{3}\times W$ together with projections
\begin{align}
    \pi_{L,C,R}\colon M^{3}_{\operatorname{heat}}\to M^{2}_{\operatorname{heat}}
\end{align}
such that $\pi_{C}$ is a b-fibration.

We refer to \cite{ice} for the construction of this space. We can then define the composition of convolution operators with polyhomogeneous integral kernels for appropriate index sets $\mathcal{E}_{i}$ and they satisfy the following property
\begin{proposition}[\cite{ice}]
    Let $A_{i}\in\mathcal{A}^{\mathcal{E}_{i}}_{\operatorname{phg}}(M^{2}_{\operatorname{heat}})$ with index sets satisfying
    \begin{align}
      \begin{split}
            \mathcal{E}^{i}(\leftface)=\mathcal{E}_{i}(\rightface)=\mathcal{E}_{i}(\backface)=\mathcal{E}_{i}(\bottomface)=\mathcal{E}_{i}(\timeface)&=\varnothing \\
            \inf(\mathcal{E}_{i}(\frontface))&=-nk-2k+a_{i}. 
        \end{split}
    \end{align}
    for some $a_{i}\geq 0$.
    Then the composition of convolution operators $A_{3}=A_{1}\circ A_{2}\in \mathcal{A}^{\mathcal{E}_{3}}_{\operatorname{phg}}(M^{2}_{\operatorname{heat}})$ with index set $\mathcal{E}_{3}$ satisfying
     \begin{align}
      \begin{split}
            \mathcal{E}_{3}(\leftface)=\mathcal{E}_{3}(\rightface)=\mathcal{E}_{3}(\backface)=\mathcal{E}_{3}(\bottomface)=\mathcal{E}_{3}(\timeface)&=\varnothing \\
            \inf(\mathcal{E}_{3}(\frontface))&=-nk-2k+a_{1}+a_{2}.
        \end{split}
    \end{align}
\end{proposition}
Using this we can now construct the heat kernel.

\begin{theorem}\label{heatkernelexistence}
    There exists a distribution $H\in\mathcal{A}_{\text{phg}}(M^{2}_{\text{heat}};\beta^{\star}(\operatorname{HOM}(\mathcal{S})))$ satisfying \eqref{indexsets} such that $t(\partial_{t}+\slashed{\partial}_{E}^{2})H=0$ and the operator $H_{t}$ extends to a compact linear operator on $L^{2}(M;E)$ such that $H_{t}s\to s$ in $L^{2}$ at $t\to 0$.
\end{theorem}

\begin{proof}
    Let $R=t(\partial_{t}+\slashed{\partial}_{E}^{2})K_{1}$ considered as a convolution operator and then regarding $K_{1}$ itself as a convolution operator for $\varphi\in C^{\infty}_{0}(M\times[0,\infty),E)$ we have
    \begin{align}
    \begin{split}
        (\partial_{t}+\slashed{\partial}_{E}^{2})K_{1}\star \varphi&=(\partial_{t}+\slashed{\partial}_{E}^{2})\int_{M}\int_{0}^{t}K_{1}(w,w',t-s)\varphi(w',s)ds\operatorname{dVol}_{w'} \\
        &=\int_{M}K_{1}(w,w',t-s)\varphi(w',s)\operatorname{dVol}_{w'}|_{s=t} \\
        &\quad +\int_{M}\int_{0}^{t}(\partial_{t}+\slashed{\partial}_{E}^{2})K(w,w',t-s)\varphi(w',s)ds\operatorname{dVol}_{w'} \\
        &=\varphi(w',t)+\int_{M}\int_{0}^{t}t^{-1}R(w,w',t-s)\varphi(w',s)ds\operatorname{dVol}_{w'}.
        \end{split}
    \end{align}
    where the second equality follows from the fundamental theorem of calculus and the third by the fact that $K_{1}$ converges strongly to the identity as $t\to 0$ and the definition of $R$. In particular we have
    \begin{align}
         (\partial_{t}+\slashed{\partial}_{E}^{2})K_{1}\star=\operatorname{Id}+t^{-1}R.
    \end{align}
    
    From this calculation, we also see that  $(\partial_{t}+\slashed{\partial}_{E}^{2})K=0$ is equivalent to $(\partial_{t}+\slashed{\partial}^{E})^{2})K\star=\operatorname{Id}$. So to cancel off the error term $t^{-1}$, we consider the Neumann series
    \begin{align}
        R'=\sum_{i=1}^{\infty}(-1)^{i}(t^{-1}R)^{i}
    \end{align}
    for which $\operatorname{Id}+R'$ formally inverts $\operatorname{Id}+t^{-1}R$.
    
    Since the lift of $t$ is given by vanishes to second order at $\timeface,\backface$ and to order $2k$ and $\frontface$, the index sets for $t^{-1}R$ satisfy
    \begin{align}
      \begin{split}\label{errorindexsets}
            \mathcal{E}^{i}(\leftface)=\mathcal{E}(\rightface)=\mathcal{E}(\backface)=\mathcal{E}(\bottomface)=\mathcal{E}(\timeface)&=\varnothing \\
            \inf(\mathcal{E}(\frontface))&=-nk-2k+1.
        \end{split}
    \end{align}
    So by the previous proposition, $(t^{-1}R)^{i}$ is polyhomogeneous with the same index sets except for
    \begin{align}
            \inf(\mathcal{E}_{i}(\frontface))=-nk-2k+2i-\epsilon.
    \end{align}
    We take $R'$ to be a Borel sum of this Neumann series so that
    \begin{align}
        E_{j}=R'-\sum_{i=1}^{j}(-1)^{i}(t^{-1}R)^{i}
    \end{align}
    has index set which satisfy the same properties as those for $\mathcal{E}_{j}$ described above. It follows that 
    \begin{align}
         (\partial_{t}+\slashed{\partial}_{E}^{2})K_{1}\star(\operatorname{Id}+R')=\operatorname{Id}+E'
    \end{align}
    where $E'$ vanishes to infinite order at all faces, thus its integral kernel on $M^{2}\times [0,\infty)$ is also smooth and vanishing to infinite order at all faces. Moreover, again by the above proposition $K_{1}\star(\operatorname{Id}+R')$ satisfy the same index set properties as $K_{1}$.
    
    Since $M^{2}$ is compact, for any $\alpha>0$ for small $t$ we have $R''=O(t^{\alpha})$ so for all $t$ there exists $C$ such that $\lVert E'\rVert\leq Ct^{\alpha}$ since $R''$ is supported in $t<T'$ for some $T'$ where $\lVert\cdot\rVert$ denotes the pointwise norm. We can now take the Neumann series $R''$ of $E'$. If $\lVert A_{i}\rVert\leq C_{i}t^{\alpha_{i}}$ then
    \begin{align}
        \begin{split}
         \lVert A_{i}\circ B_{i}&\rVert=\lVert\int_{M}\int_{0}^{t}A(w,w',t')B(w',w'',t-t')\operatorname{dVol}_{w'}dt'\rVert \\
         &\leq C_{1}C_{2}\operatorname{Vol}(M)\frac{t^{\alpha_{1}+\alpha_{2}+1}}{\alpha_{1}+\alpha_{2}+1}
        \end{split}
    \end{align}
    inductively applying this to $(R'')^{j}$ we have the pointwise estimate
    \begin{align}
        \lVert (R'')^{j}\rVert\leq C\frac{t^{j\alpha+j}}{j!}.
    \end{align}
    Since derivatives of $R''$ also vanishes to infinite order at all faces similar estimates hold for $\nabla^{n}(R'')^{j}$ so $(R')^{j}$. 
    
    Hence the series converges in $C^{\infty}$ to a smooth section vanishing to infinite order at all faces so finally taking 
    \begin{align}
        H=K_{1}\star(\operatorname{Id}+R')\star(\operatorname{Id}+R')
    \end{align}
    we have that
    \begin{align}
        (\partial_{t}+\slashed{\partial}_{E})^{2})K\star=\operatorname{Id}
    \end{align}
    so $H$ is our desired heat kernel and it satisfies the same leading order asymptotics as $K_{1}$.
    
    Since $H_{t}$ and $\partial H_{t}$ are smooth and vanish to infinite order at $\leftface\rightface$ for $t>0$ they satisfy
    \begin{align}
        \int \lVert H_{t}(p,p')\rVert^{2}_{\operatorname{End}(\mathcal{S})}\operatorname{dVol}(p)\operatorname{dVol}(p')<\infty.
    \end{align}
    Hence $H_{t},\partial_{t}H_{t}$ are compact on $L^{2}(M;E)$.
\end{proof}

\section{Essential self-adjointness}\label{essentialsa}

Recall that given an unbounded operator $P$ defined on $C^{\infty}_{c}(M,E)\subset L^{2}(M;\mathcal{S})$ we have the closed domains 
\begin{align}
    \mathcal{D}_{\max}(P)&=\{u\in L^{2}:Pu\in L^{2}\} \\
    \mathcal{D}_{\min}(P)&=\{u\in L^{2}:\exists u_{n}\in C^{\infty}_{c} \text{ such that } u_{n}\to u, Pu_{n} \text{ converges in } L^{2}\}.
\end{align}
Given a domain $\mathcal{D}$, its closure $\bar{\mathcal{D}}$ is defined to be its closure in the graph norm $\lVert u\rVert_{\Gamma}=\lVert u\rVert_{L^{2}}+\lVert Pu\rVert_{L^{2}}$. The adjoint of a domain $\mathcal{D}$ is defined
\begin{align}
    \mathcal{D}^{\star}=\{u\in L^{2}: \exists v\in L^{2},\quad \forall w\in\mathcal{D}, \quad\langle Pu,v\rangle=\langle u,w \rangle\}.
\end{align}
The maximal and minimal domains are the largest and smallest closed extensions of $P$ and are the adjoint domains of each other. A domain is self-adjoint if it is equal to its adjoint. We say $P$ is \textbf{essentially self-adjoint} if it has a unique self-adjoint extension which is equivalent to
\begin{align}
    \mathcal{D}_{\max}(P)=\mathcal{D}_{\min}(P).
\end{align}
We need the following to show that $\slashed{\partial}_{E}$ is essentially self-adjoint.

\begin{lemma}\label{htdmin}
    For $t>0$ and $s\in L^{2}(M;E)$, $H_{t}^{*}s\in\mathcal{D}_{\min}$.
\end{lemma}

\begin{proof}
    For $t>0$, the kernel $H_{t}^{*}$ is smooth on $M\times M$ and vanishes to infinite order at $\leftface$ and $\rightface$. Thus $H_{t}^{*}s$ is smooth and vanishes to infinite order at $\partial M$. So if suffices to show that if $\varphi$ is smooth and vanishing to infinite order at $\partial M$ then $\varphi\in\mathcal{D}_{\min}$. 

    Let $\psi$ be any smooth cutoff function with $\chi(x)=0$ for $x\leq 1$ and $\chi(x)=1$ for $x\geq 2$ and take $\varphi_{n}=\chi(nx)\varphi$ which converges to $\varphi$ in $L^{2}$. We recall the local expression for the Dirac operator in a neighbourhood of the boundary
    \begin{align}
        \slashed{\partial}=\operatorname{cl}(\partial_{x})\partial_{x}+\frac{kf}{2x}\operatorname{cl}(\partial_{x})+\sum_{j}\operatorname{cl}(\tilde{U}_{j})\tau\tilde{U}_{j}+\frac{1}{x^{k}}\sum_{i}\operatorname{cl}(x^{-k}V_{i})\nabla^{\partial M/Y}_{V_{i}}+\sum_{j=0}^{k-1} x^{-k+j}B_{j}.
            \end{align}
    In particular we see that
    \begin{align}
        \slashed{\partial}(\chi(nx)\varphi)=n\chi'(nx)\varphi+\chi(nx)\slashed{\partial}\varphi.
    \end{align}
    Since $\chi(nx)\slashed{\partial}\varphi$ converges to $\slashed{\partial}\varphi$ in $L^{2}$ it remains to show that $n\chi'(nx)\varphi$ converges to $0$ in $L^{2}$ but
    \begin{align}
        \begin{split}
            \lVert n\chi'(nx)\varphi\rVert_{L^{2}}&=\int_{M} n^{2}|\chi'(nx)|^{2}\lVert\varphi\rVert^{2}_{g}\operatorname{dVol} \\
            &\leq Cn^{2}\int_{\frac{1}{n}\leq x\leq\frac{2}{n}}\lVert\varphi\rVert^{2}_{g}x^{kf}dxdydz.
        \end{split}
    \end{align}
    Since $\varphi$ vanishes to infinite order it follows that for any $a>0$
    \begin{align}
            \lVert n\chi'(nx)\varphi\rVert_{L^{2}}\leq Cn^{-a}.
    \end{align}
\end{proof}

We now prove a general lemma which allows us to show the existence of a self-adjoint extension given a heat operator $H_{t}$ which together with its adjoint satisfy certain mapping properties. This will allow us to prove essential self-adjointness for the trivial kernel case and allow us to find certain self-adjoint domains in the case of non-trivial kernel.

\begin{lemma}
    Let $\mathcal{H}$ be a Hilbert space, $\mathcal{D}$ be a dense subspace. Let $P$ be a symmetric operator on $\mathcal{H}$ with domain $\mathcal{D}$. Suppose $H_{t}$ is a smooth family of bounded operators on $\mathcal{H}$ which satisfy
    \begin{enumerate}
        \item $(\partial_{t}+P)H_{t}=0$
        \item $H_{t}u\to u$ as $t\to 0$ for all $u\in\mathcal{H}$
        \item $H_{t},H_{t}^{*}\colon \mathcal{H}\to\mathcal{D}$.
    \end{enumerate}
    Let $\bar{\mathcal{D}}$ be the graph closure of $\mathcal{D}$ in $\mathcal{D}_{\max}(P)$ and $\bar{\mathcal{D}}^{\star}$ be the adjoint domain and assume $\mathcal{D}\subset\bar{\mathcal{D}}^{\star}$. Then $\bar{\mathcal{D}}=\bar{\mathcal{D}}^{\star}$ hence $(P,\mathcal{D})$ is essentially self adjoint and $H_{t}$ is self adjoint. 
\end{lemma}

\begin{proof}
    Let $u\in\bar{\mathcal{D}}^{\star}$ then by assumption $H_{t}u\in\mathcal{D}$. To show that $u\in \bar{\mathcal{D}}$ it suffices to show that $H_{t}u$ converges to $u$ in the graph norm since $\bar{\mathcal{D}}$ is the closure of $\mathcal{D}$ in the graph norm. Since by assumption $H_{t}u\to u$ in $\mathcal{H}$ and $u\in\bar{\mathcal{D}}^{\star}\subset\mathcal{D}_{\max}(P)$ we have that $Pu\in \mathcal{H}$, thus it remains to show that $PH_{t}u\to Pu$ in $\mathcal{H}$.

    Let $v\in \mathcal{H}$ then $\langle H_{t}^{*}Pu,v\rangle=\langle Pu,H_{t}v\rangle$. Then $\langle Pu,H_{t}v\rangle=\langle u,PH_{t}v\rangle$ since $H_{t}v\in\bar{\mathcal{D}}$ and $u\in\bar{\mathcal{D}}^{\star}$. 
    Since $PH_{t}=-\partial_{t}H_{t}$ on $\mathcal{H}$ we have $\langle H_{t}^{\star}Pu,v\rangle=-\langle \partial_{t}H_{t}^{\star}u,v\rangle=\langle PH_{t}^{\star}u,v\rangle$. This holds for all $v\in\mathcal{H}$ so $PH_{t}^{\star}u=H_{t}^{\star}Pu$. Now repeat the same argument with $H_{t}$ replacing $H_{t}^{\star}$ which now works since we know that $PH_{t}^{\star}=-\partial_{t}H_{t}$ since for all $v\in\mathcal{H}$
    \begin{align}
        \begin{split}
            \langle \partial_{t}H^{\star}_{t}u,v\rangle&=\langle u,\partial_{t}H_{t}v\rangle \\
            &=\langle u,PH_{t}v\rangle \\
            &=\langle Pu,H_{t}v\rangle \\
            &=\langle H_{t}^{\star}Pu,v\rangle \\
            &=\langle PH_{t}^{\star}u,v\rangle.
        \end{split}
    \end{align}
    Thus we have that $PH_{t}u=H_{t}Pu$ for all $u\in\bar{\mathcal{D}}^{\star}$ which converges in $\mathcal{H}$ to $u$ as $t\to 0$.
\end{proof}

\begin{theorem}
    $\slashed{\partial}_{E}^{2}$ is essentially self adjoint and has discrete spectrum.
\end{theorem}

\begin{proof}
    By Theorem \ref{heatkernelexistence} and Lemma \ref{htdmin}, the heat kernel satisfies all the properties of the previous lemma with $\mathcal{D}=\mathcal{D}_{\min}$. Since $H_{t}$ are compact self-adjoint operators, they have discrete spectrum hence $\slashed{\partial}_{E}^{2}$ also has discrete spectrum.
\end{proof}

\begin{theorem}
    $\slashed{\partial}_{E}$ is essentially self adjoint and has discrete spectrum.
\end{theorem}

\begin{proof}
    Since $\slashed{\partial}_{E}^{2}$ is essentially self-adjoint with domain $\mathcal{D}(\slashed{\partial}_{E}^{2})=\mathcal{D}_{\min}=\mathcal{D}_{\max}$ and has discrete spectrum, $L^{2}(M;\mathcal{S})$ has an orthonormal basis of eigenfunctions $\varphi_{n}$ with eigenvalues $\lambda_{n}^{2}$ of finite multiplicity which vanish to infinite order at $\partial M$. Since $\slashed{\partial}_{E}\varphi_{n}\in\dot{C}(M,\mathcal{S})$ we have $\slashed{\partial}_{E}^{2}(\slashed{\partial}_{E}\varphi_{n})=\slashed{\partial}_{E}(\slashed{\partial}_{E}^{2}\varphi_{n})=\lambda_{n}^{2}\slashed{\partial}_{E}\varphi_{n}$ so $\slashed{\partial}_{E}$ has a well defined restriction to $\lambda_{n}^{2}$-eigenspace which is self adjoint. Hence there is an orthonormal basis of eigenfunctions $\psi_{n}$ with eigenvalues $\mu_{n}$ of $\slashed{\partial}_{E}$ so we get a unique self adjoint extension of $\slashed{\partial}_{E}$ with domain containing $\mathcal{D}(\slashed{\partial}_{E}^{2})$ given by
    \begin{align}
        \mathcal{D}(\slashed{\partial}_{E})=\left\{u\in L^{2}(M;E)|\sum_{n}\mu_{n}\langle u,\psi_{n}\rangle\psi_{n}\in L^{2}(M;E)\right\}.
    \end{align}
    Now supposed $\mathcal{D}'$ is a self-adjoint extension for $\slashed{\partial}_{E}$. Then there is an induced self-adjoint domain for $\slashed{\partial}^{2}$ given by $\mathcal{D}'(\slashed{\partial}_{E}^{2})=\{u\in \mathcal{D}'|\slashed{\partial}_{E}u\in\mathcal{D}'\}$. But $\slashed{\partial}_{E}^{2}$ is essentially self adjoint so $\mathcal{D}'(\slashed{\partial}_{E}^{2})=\mathcal{D}(\slashed{\partial}_{E}^{2})$. Thus $\mathcal{D}'=\mathcal{D}'(\slashed{\partial}_{E})$ since it is the unique self adjoint extension of $\mathcal{D}'(\slashed{\partial}_{E})$ containing $\mathcal{D}(\slashed{\partial}_{E}^{2})$.
\end{proof}

Now we briefly describe what happens in the case of non-trivial kernel in the case of an isolated cusp. We will also take the spin Dirac operator for simplicity. In this case the boundary family reduces to a single boundary operator which we denote $\slashed{\partial}_{Z}$ and we denote the kernel of this operator by $\mathcal{K}$.

In this case the construction proceeds as usual at $\timeface,\frontface$ however the normal operator at $\frontface$ will no longer vanish to infinite order near $\backface$ and we need to consider what happens there. The spin Dirac operator will also not be essentially self- adjoint and this is reflected in the non-essential self-adjointness to the normal operator on $\backface$.

From Lemma \ref{normaloperator} for the spin Dirac operator
 \begin{align}
      \frac{1}{2}T\partial_{T}+T^{2}\left(-\partial_{s}^{2}-\frac{fk}{s}\partial_{s}+\left(\frac{fk}{2}-\frac{f^{2}k^{2}}{4}\right)\frac{1}{s^{2}}\right)\mathcal{N}_{\text{bf}}(H)=0.
\end{align}
Conjugating by $s^{-\frac{kf}{2}}$ this reduces to 
 \begin{align}
      (\frac{1}{2}T\partial_{T}-T^{2}\partial_{s}^{2})A=0.
\end{align}
The operator $-\partial_{s}^{2}$ on $[0,\infty)$ is not essentially self adjoint. Taking values in a one-dimensional complex vector space it has many self-adjoint extensions corresponding to Dirichlet, Neumann and Robin boundary conditions at $s=0$ while taking values in a finite dimensional vector space there are many more self-adjoint extensions.

For our case, where the normal operator takes values in the kernel $\mathcal{K}$ only some choices of boundary conditions at $s=0$ lead to a heat kernel for which we can use the same arguments as above to construct a self-adjoint domain for $\slashed{\partial}$.

\section{Green's operator and Fredholm property}

In this section we construct the Green's operator for $\slashed{\partial}$ and show that it is a Fredholm operator. All assumptions will be the same as in the previous section.

First we construct the double space on which the Green's operator will be a polyhomogeneous distribution. Let $\mathcal{B}_{0}\subset M^{2}$ be the fibre diagonal in the corner
\begin{align}
    \mathcal{B}_{0}=\partial M\times_{\phi_{Y}}\partial M=\{(p,q)\in\partial M\times\partial M\times \colon\phi_{Y}(p)=\phi_{Y}(q)\}.
\end{align}
Blowing up $\mathcal{B}_{0}$ radially, we get a manifold with corners $M^{2}_{1}$ with boundary hypersurfaces $\leftface_{1},\rightface_{1}$ and $\backface_{1}$. In a neighbourhood of $\backface$ and away from $\rightface$ we have projective coordinates
\begin{align}\label{x'coordM21}
    x',\quad s=\frac{x}{x'},\quad \eta=\frac{y-y'}{x'},\quad z,z'
\end{align}
and similarly swapping $x'$ with $x'$ away from $\leftface$.
As in the construction of the heat space, we blow up once more at $B_{1}\subset M^{2}_{1}$, the intersection of the closure of $\mathcal{U}\times_{\pi}\mathcal{U}$ with $\backface$ which in the above projective coordinates is given by
\begin{align}
    \mathcal{B}_{1}=\{s=1,x'=0\}.
\end{align}
We obtain a manifold with corners $M^{2}_{\ce}$ with blowdown map $\beta\colon M^{2}_{\ce}\to M^{2}$ and boundary hypersurfaces $\leftface=\overline{\beta^{-1}(\leftface_{1})},\leftface=\overline{\beta^{-1}(\leftface_{1})}$ and using coordinates \eqref{x'coordM21}
\begin{align}\label{doublespacedirac}
    \begin{split}
        \leftbackface&=\overline{\beta^{-1}(\{s>1,x'=0\})} \\
        \rightbackface&=\overline{\beta^{-1}(\{s<1,x'=0\})} \\
        \frontface&=\overline{\beta^{-1}(\{s=1,x'=0\})}.
    \end{split}
\end{align}
In addition we define $\operatorname{diag}_{\ce}=\overline{\beta^{-1}(\operatorname{diag})}$ where $\operatorname{diag}$ is the diagonal in $M^{2}$.

We define the cusp edge ($\ce$-)vector fields
\begin{align}
    \mathcal{V}_{ce}=\{W\in C^{\infty}(M,TM):Wf=O(x^{k}) \text{ for all } f\in C^{\infty}_{\Phi}(M)\}
\end{align}
where 
\begin{align}
    C^{\infty}_{\Phi}(M)=\{f\in C^{\infty}(M):f|_{\partial M}=\phi_{Y}^{\star}g \text{ for some } g\in C^{\infty}(Y)\}.
\end{align}
This is a Lie algebra of vector fields which is related to the ice-vector fields by $\mathcal{V}_{\ice}=x^{-k}\mathcal{V}_{\ce}$ which is not a Lie algebra with respect to the usual bracket. As for the ice-vector fields, we define the cusp edge tangent bundle $^{ce}TM$ whose smooth sections are canonically identified with the $\ce$-vector fields. Locally near the boundary, it is generated over $C^{\infty}(M)$ by 
\begin{align}
    x^{k}\partial_{x},x^{k}\partial_{y_{i}},\partial_{z}.
\end{align}
We define the cusp edge cotangent bundle $^{\ce}T^{\star}M$ to be the dual which is locally generated near the boundary by $x^{-k}dx,x^{-k}dy_{i},dz$.
\begin{figure}[H]
\centering
    \begin{tikzpicture}[scale=1.5,rotate around x=90,rotate around z=90,z={(0,0,-1)},y={(0,-1,0)}]
        \draw[black, thick] (0,{sqrt(2)},0) -- (0,3,0) node[pos=1,right] {$x$};
        \draw[red, thick] (0,0,{sqrt(2)}) -- (0,0,3) node[pos=1,above] {$x'$};
         \draw[domain=0:33, smooth, variable=\x, purple, thick] plot ({0}, {sqrt(2)*cos(\x)}, {sqrt(2)*(sin(\x)}); 
         \draw[domain=57:90, smooth, variable=\x, purple, thick] plot ({0}, {sqrt(2)*cos(\x)}, {sqrt(2)*(sin(\x)}); 
         \draw[domain=-50:140, smooth, variable=\x, orange, thick] plot ({0}, {(sqrt(2)*cos(45))+(0.3*cos(\x))}, {(sqrt(2)*(sin(45))+(0.3*sin(\x)}); 


        \draw[->,black,thick] (0,3,2.5) -- (0,4,3) node[pos=0.5,above] {$\beta$};


        \draw[black, thick] (0,{sqrt(2)},4) -- (0,3,4) node[pos=1,right] {$x$};
        \draw[red, thick] (0,0,{sqrt(2)+4}) -- (0,0,7) node[pos=1,above] {$x'$};
        \draw[domain=0:90, smooth, variable=\y, purple, thick] plot ({0}, {(sqrt(2)*cos(\y))}, {(sqrt(2)*(sin(\y))+4}); 


        \draw[->,black,thick] (0,3,5.5) -- (0,4,5.5) node[pos=0.5,above] {$\beta_{1}$};


        \draw[black, thick] (0,5,4) -- (0,8,4) node[pos=1,right] {$x$};
        \draw[red, thick] (0,5,4) -- (0,5,7) node[pos=1,above] {$x'$};

        \draw[->,black,thick] (0,1,2.5) -- (0,1,3) node[pos=0.5,right] {$\beta_{\ce\to 1}$};

        \node at (0,1.35,1.35) {$\frontface$};
        \node at (0,0.5,1) {$\backface$};
        \node at (0,-0.2,{0.5*(sqrt(2)+3)}) {$\leftface$};
        \node at (0,{0.5*(sqrt(2)+3)},-0.2) {$\rightface$};
        
    \end{tikzpicture}
    \caption{Sequence of blown up spaces for $M^{2}_{\ce}$.} \label{fig:blowup2}
\end{figure}
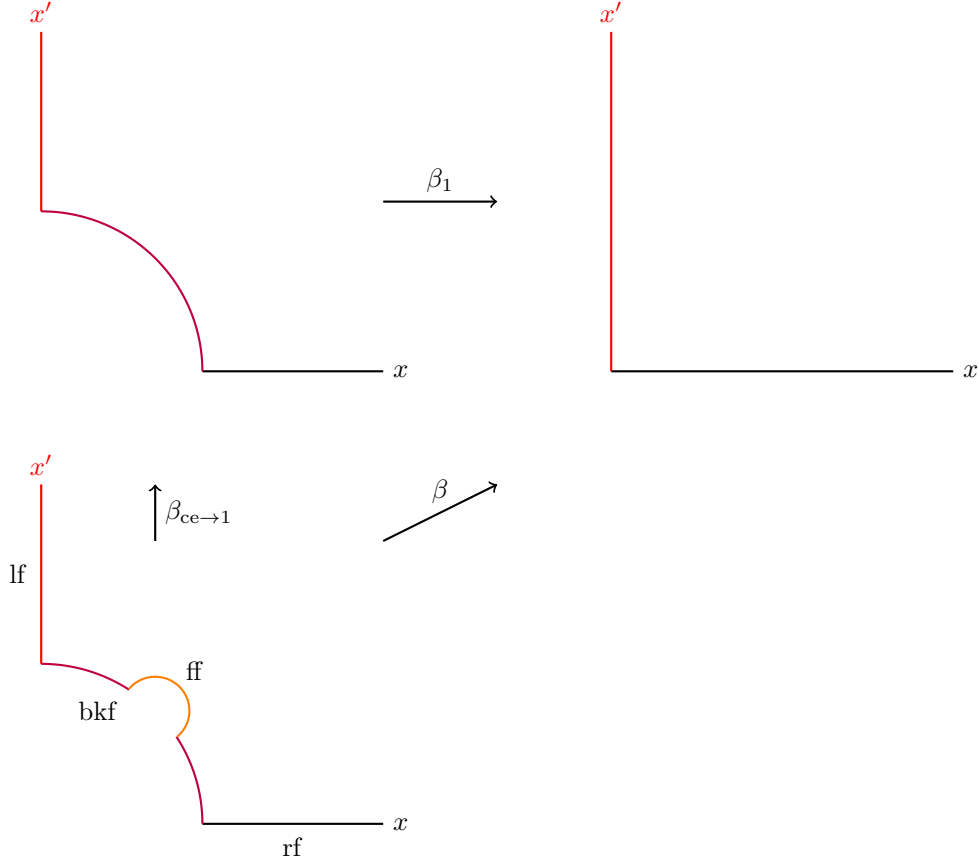
We define the cusp edge differential operators $\operatorname{Diff}_{\ce}^{\star}(M;E)$ as those locally generated near the boundary by $\mathcal{V}_{\ce}$, that is of the form
\begin{align}\label{cediff}
    P=\sum_{i+|\alpha|+|\beta|<r}a_{i\alpha\beta}(x,y,z)(x^{k}\partial_{x})^{i}(x^{k}\partial_{y})^{\alpha}\partial_{z}^{\beta}
\end{align}
where $a_{i\alpha\beta}(x,y,z)\in C^{\infty}(M;\operatorname{End}(E))$. We define the incomplete cusp edge differential operators by $\operatorname{Diff}^{l}_{\operatorname{ice}}(M,E)=x^{-l}\operatorname{Diff}_{\ce}^{\star}(M;E)$.

In a neighbourhood of $\frontface$ we have projective coordinates with respect to $x'$
\begin{align}\label{projectivedouble}
    x',\quad \tilde{s}=\frac{s-1}{(x')^{k-1}}=\frac{x-x'}{(x')^{k}},\quad \tilde{\eta}_{j}=\frac{y_{j}-y'_{j}}{(x')^{k}},\quad y_{j}',z_{i},z'_{i}.
\end{align}
The double space $M^{2}_{\ce}$ resolves this Lie algebra of singular vector fields as follows.
\begin{lemma}
    The lift of the $\ce$-vector fields $\beta^{\star}\mathcal{V}_{\ce}(M)$ to $M^{2}_{\ce}$ are tangent to the front face $\frontface$ and the kernel of the restriction to the front face is the $\ce$-vector fields which vanish at $\partial M$. Moreover, we have canonical isomorphisms
    \begin{align}
        \begin{split}
            N\diag_{\ce}\simeq {^{\ce}TM} \\
            N^{\star}\diag_{\ce}\simeq {^{\ce}T^{\star}M}.
        \end{split}
    \end{align}
\end{lemma}

\begin{proof}
    Using the projective coordinates \eqref{projectivedouble} we have the following lifts
    \begin{align}\label{liftdouble}
        \begin{split}
            \beta^{\star}x^{k}\partial_{x}&=[(x')^{k-1}\tilde{s}+1]^{k}\partial_{\tilde{s}} \\
            \beta^{\star}x^{k}\partial_{y_{j}}&=[(x')^{k-1}\tilde{s}+1]^{k}\partial_{\tilde{\eta_{j}}} \\
            \beta^{\star}\partial_{z_{i}}&=\partial_{z_{i}}
        \end{split}
    \end{align}
    from which the first two statements follow.   The isomorphism $^{\ce}TM\to N\diag_{\ce}$ is given by the composition
    \begin{align}
        W_{p}\mapsto (\beta^{\star}W_{p})|_{\diag_{\ce}}\mapsto [(\beta^{\star}W_{p}|_{\diag_{\ce}}]\in N\diag_{\ce}
    \end{align}
    which we see is an isomorphism from \eqref{liftdouble}.
\end{proof}
Now by analogous calculations as for the heat space we find that lift of the operators of $\slashed{\partial}$ and $\slashed{\partial}^{2}$ in a neighbourhood of $\frontface$ in these coordinates are given by
\begin{align}
    \begin{split}
        \slashed{\partial}&=(x')^{-k}\left(\cl(\partial_{x})\partial_{s}+\sum_{ij}\cl(a_{ij}\partial_{y_{i}})a_{ij}(0,y')\partial_{\tilde{\eta}_{i}}+(\tilde{s}(x')^{k-1}+1)^{k}\slashed{\partial}^{Z,y'} \right. \\
        &\quad \left. +(x')^{k-1}(\tilde{s}(x')^{k-1}+1)^{-1}\frac{kf}{2}\cl(\partial_{x})+(x)^{k}R\right) \\
        \slashed{\partial}^{2}&=(x')^{-2k}\left(-\partial_{s}^{2}+\Delta_{\tilde{\eta},y'}+(\slashed{\partial}^{Z,y'})^{2} +x^{k-1}R'\right).
    \end{split}
\end{align}
In these coordinates we also have
\begin{align}\label{deltaff}
    \begin{split}
        \delta(x-x')\delta(y-y')\delta(z-z')&=(x')^{-k(b+1)}\delta(\tilde{s})\delta(\tilde{\eta})\delta(z-z').
    \end{split}
\end{align}
Let $\rho_{\frontface}$ be a boundary defining function for $\frontface$ and $\rho_{\sideface}$ the product of boundary defining functions for the remaining faces. Set $\Omega_{R,\ce}=\rho_{\frontface}^{-k(b+1)}\beta^{\star}_{R}\Omega$ then we define the small ice calculus of pseudodifferential operators acting between sections of vector bundles $E$ and $F$ by
\begin{align}
    \Psi^{r}_{\ce}(M)=\rho_{\sideface}^{\infty}I^{r}(M^{2}_{\ice},\diag_{ice};\operatorname{HOM}(E,F)\otimes\Omega_{R,\ce})
\end{align}
We include the weighted density factor so that the operators $\operatorname{Diff}_{ce}(M)$ are contained in $\Psi_{\ce}^{r}$ as we have the factor $(x')^{k(b+1)}$ in the kernel \eqref{deltaff} of the identity. 

We define the residual operators
\begin{align}
    \Psi^{-\infty}_{\ce}(M)=\bigcap_{r}\Psi^{r}_{\ce}(M)=\rho_{\sideface}^{\infty}C^{\infty}(M^{2}_{\ice};\operatorname{HOM}(E,F)\otimes\Omega_{R,\ce}).
\end{align}
From the conormal singularity at $\diag_{ice}
$, we have the symbol map
\begin{align}
    \sigma_{r,\ce}\colon\Psi^{r}_{\ce}(M)\to S^{\{r\}}(N^{\star}\diag_{\ce}; \Omega_{\text{fibre}}(N^{\star}\diag_{\ice})\otimes p^{\star}(\Omega_{R,\ce}|_{\diag_{\ce}})\otimes p^{\star}(\operatorname{HOM}(E,F)|_{\diag_{\ice}})).
\end{align}
Here we have $p\colon N^{\star}\diag_{\ice}\to\diag_{\ice}$ and $\Omega_{\text{fibre}}$ is the bundle of densities which are translation invariant along the fibres. Now $N^{\star}\diag_{\ice}$ is canonically isomorphic to $^{\ce}T^{\star}M$ and the restriction $\omega_{R,\ce}$ is canonically isomorphic to the cusp edge density bundle $^{\ce} \Omega$. Hence we see that
\begin{align}
    \Omega_{\text{fibre}}(N^{\star}\diag_{\ice})\otimes p^{\star}(\Omega_{R,\ce}|_{\diag_{\ce}})\simeq \Omega_{\text{fibre}}(^{\ce} T^{\star}M)\otimes p^{\star}(^{\ce}\Omega)
\end{align}
This is a vector bundle over $^{\ce} T^{\star}M$ which has a canonical trivialisation given in local coordinates near the boundary by
\begin{align}
    d\xi d\upsilon d\zeta\frac{dxdydz}{x^{k}x^{kb}}
\end{align}
where $\xi,\upsilon,\zeta$ are the coordinates on $^{\ce} T^{\star}M$ induced by $x,y,z$. We also have the canonical isomorphism $\operatorname{HOM}(E,F)|_{\diag_{\ice}}\simeq \operatorname{Hom}(E,F)\to M$.

Thus we can canonically take the symbol map as
\begin{align}
    \sigma_{r,\ce}\colon\Psi^{r}_{\ce}(M)\to S^{\{r\}}(^{\ce}T^{\star}M;\pi^{\star}\operatorname{Hom}(E,F)).
\end{align}
In particular, if $\tilde{\sigma},\tilde{E}_{j},\zeta_{i}$ are coordinates on $^{\ce}T^{\star}M$ dual to the coordinates on $^{\ce}TM$ induced by $x,y,z$ then the principal symbol of a $\ce$-differential operator of the form \eqref{cediff} is
\begin{align}
    \sigma_{r}(P)=\sum_{i+|\alpha|+|\beta|=r}a_{i\alpha\beta}(x,y,z)(i\tilde{\sigma})^{i}(i\tilde{E})^{\alpha}(i\zeta)^{\beta}.
\end{align}
We say that $A\in\Psi^{r}_{\ce}(M)$ is $\ce$-elliptic if $\sigma_{r}(A)$ is invertible. That is, there exists $b\in S^{r}$ such that $\sigma_{r}(A)b-1\in S^{-1}$.
\begin{lemma}\label{principalsymbol}
    $x^{k}\slashed{\partial}$ is an elliptic ce-differential operator.
\end{lemma}

\begin{proof}
    At any fixed point $p$ near the boundary, choose coordinates $y,z$ such that $x^{k}\partial_{X},\partial_{y_{j}},\partial_{z_{i}}$ are an orthonormal basis of $^{\ce}TM$ at $p$. Then the principal symbol at $p$ in these coordinates of $x^{k}\slashed{\partial}$ is
    \begin{align}
        \sigma_{1}(x\slashed{\partial})=i\cl(\partial_{x})\tilde{\sigma}+\sum_{i}i\cl(\partial_{y_{j}})\tilde{E}_{j}+\sum_{i}i\cl(x^{-k}\partial_{z_{i}})\zeta_{i}.
    \end{align}
    Note that at this point, it is not equal Clifford multiplication by a $\ce$-(co)tangent vector since the Clifford multiplications which appear here are $\operatorname{ice}$-(co)tangent vectors. However, we can make the identification $\ice TM \ni W\mapsto x^{k}W\in {^{\ce}TM}$ which is an isometry if we take the metric $x^{2k}g$ on $^{\ce}TM$. This induces an isomorphism of their Clifford bundles thus defines a $\ce$-Clifford action on the ($\ice$)-spinor bundle $\mathcal{S}$ where $\cl_{ce}(W)\mapsto\cl_{\ice}(x^{-k}W)$. Then identifying tangent and cotangent bundle using the metric, we can write the principal symbol as
    \begin{align}
        \sigma_{1}(x\slashed{\partial})(\xi)=i\cl_{\ce}(\xi)
    \end{align}
    where $\xi\in {^{\ce}T^{\star}M}$. 
\end{proof}

This pseudodifferential calculus satisfies the usual properties which we briefly recall. There is a short exact sequence involving the symbol map
\begin{align}
    0\to \Psi^{r-1}_{\ce}(M)\to\Psi^{r}_{\ce}(M)\to S^{\{r\}}(^{\ce}T^{\star}M;\pi^{\star}\operatorname{Hom}(E,F))\to 0
\end{align}
and composition
\begin{align}
    \Psi^{r}_{\ce}(M)\circ \Psi^{s}_{\ce}(M)&\subset \Psi^{r+s}_{\ce}(M) \\
    \sigma_{r+s}(A\circ B)&=\sigma_{r}(A)\sigma_{s}(B)
\end{align}
Given a sequence of operators $P_{i}\in\Psi_{\ce}^{r-1}$ we have an asymptotic sum $P\in\Psi^{r}_{\ce}$ unique modulo $\Psi^{\infty}_{\ce}$ which satisfies
\begin{align}
    P-\sum_{j=0}^{k-1}P_{i}\in\Psi_{\ce}^{r-k}.
\end{align}

From these three properties, it follows that if $A\in\Psi^{r}_{\ce}(M)$ is elliptic then there exists $B\in\Psi^{-r}_{\ce}(M)$ such that
\begin{align}
    A\circ B-I, B\circ A-I\in \Psi^{-\infty}_{\ce}(M).
\end{align}
We call $B$ a small elliptic parametrix.

On a closed manifold, the residual operators are smoothing operators and thus compact on $L^{2}$, whereas on a manifold with boundary there is an obstruction to the compactness of the residual terms given by the non-vanishing of the kernel of $A$ to $\frontface$.

Consider again the lift of the kernel of the identity to a neighbourhood of the front face with the right density factor. Here we use projective coordinates with $x$ as the boundary defining function
\begin{align}\label{deltaffdens}
        \beta^{\star}(\delta(x-x')\delta(y-y')\delta(z-z')dx'dy'dz')&=x^{-k(b+1)}\delta(\tilde{s})\delta(\tilde{\eta})\delta(z-z')\beta^{\star}(dx'dy'dz').
\end{align}
Multiplying by the lift of a non-vanishing left density factor $\mu_{L}=adxdydz$ where $a$ is some non-vanishing smooth function and using the canonical identification $\beta^{\star}\Omega_{L}\otimes\beta^{\star}\Omega_{R}\simeq\beta^{\star}\Omega(M^{2})$ we have
\begin{align}\label{deltaffdens2}
        \beta^{\star}(\delta(x-x')\delta(y-y')\delta(z-z')dx'dy'dz')\beta^{\star}\mu_{L}&=\delta(\tilde{s})\delta(\tilde{\eta})\delta(z-z')d\tilde{s}d\tilde{\eta}dz'adxdyzdz.
\end{align}
Thus multiplying by $\beta^{\star}\mu_{L}$ we can identify the resulting distribution as taking values in $\Omega(M^{2}_{\ce})$. Using these coordinates, Identifying $\frontface^{\circ}\to Y$ as with the bundle $\mathbb{R}_{\tilde{s}}\times\mathbb{R}_{\tilde{\eta}}^{b}\times\partial M\times_{\phi_{Y}}\partial M\to Y$ (invariantly we would identify the first two factors with the fibre of the bundle $^{\ce}N\partial M\to Y$ of $\ce$-vector fields which vanish as standard vector fields over $\partial M$), the $d\tilde{s}d\tilde{\eta}dz'$ can be identified with a right (in the $\partial M\times_{\phi_{Y}}\partial M$ factor) fibre density over $Y$ which we will denote as $\Omega_{\operatorname{fib},R}(\frontface)$.

Thus we identify the restriction of $\Omega(M^{2}_{\ce})|_{\frontface}$ with $\Omega_{\operatorname{fib},R}(\frontface)\otimes \beta^{\star}\Omega_{L}$ so we can divide by $\beta^{\star}\mu_{L}$ to get a distributional section $\delta(\tilde{s})\delta(\tilde{\eta})\delta(z-z')d\tilde{s}d\tilde{\eta}dz'$ of $\Omega_{\operatorname{fib},R}(\frontface)$ over $\frontface$ which is conormal as $\diag_{\ce}\cap\frontface$. As the restriction of the identity operator to the $\frontface$, the action of this operator is given by convolution in the $\tilde{s}.\tilde{\eta}$ variables.

We define the space of suspended pseudodifferential operators at $\frontface$ by
\begin{align}
    \Psi^{r}_{\operatorname{sus}}(\frontface)=\rho^{\infty}_{\sideface}I^{r}(\frontface,\diag_{\ce}\cap\frontface;\operatorname{HOM}(E,F)\otimes\Omega_{\operatorname{fib},R}(\frontface)).
\end{align}
which act as convolution operators in the $\tilde{s},\tilde{\eta}$ variables. Thus in general, if $A\in \Psi_{\ce}^{r}(M)$ we define the normal operator $\mathcal{N}_{\frontface}(A)$ at $\frontface$ to the element of $\Psi^{r}_{\operatorname{sus}}(\frontface)$ whose Schwartz kernel is given by
\begin{align}
    (K_{A}\beta^{\star}\mu_{L})|_{\frontface}/\beta^{\star}\mu_{L}.
\end{align}
Since in projective coordinates with $x$ as the boundary defining function we have
\begin{align}\label{xdefiningfunctionlifts}
    \begin{split}
        \beta^{\star}x^{k}\partial_{x}&=\partial_{\tilde{s}}-kx^{k-1}\tilde{s}\partial_{\tilde{s}}+x^{k}\partial_{x} \\
        \beta^{\star}x^{k}\partial_{y_{j}}&=\partial_{\tilde{\eta_{j}}}+x^{k}\partial_{y_{j}} \\
        \beta^{\star}\partial_{z_{i}}&=\partial_{z_{i}}.
    \end{split}
\end{align}
we see from the expression of the kernel of the identity operator that the normal operator of the left lift of a differential operator on $M$ is given by the restriction of that operator to $\frontface$. The normal operator satisfies
\begin{align}
    \mathcal{N}_{\frontface}(A\circ B)=\mathcal{N}_{\frontface}(A)\circ \mathcal{N}_{\frontface}(B).
\end{align}
There is a short exact sequence
\begin{align}
    0\to x\Psi^{r}_{\ce}(M)\to\Psi^{r}_{\ce}(M)\to\Psi^{r}_{\operatorname{sus}}(\frontface)\to 0
\end{align}
If $A\in x^{a}\Psi^{r_{1}}$ and $A\in x^{b}\Psi^{r_{2}}$  then their composition $A\circ B\in x^{a+b}\Psi^{r_{1}+r_{2}}$. Finally we can define weighted Sobolev spaces $x^{a}H_{\ce}^{l}(M,E)$ such that if 
$A\in x^{b}\Psi^{r}_{\ce}$ then
\begin{align}
    A\colon x^{a}H_{\ce}^{l}(M,E)\to x^{a+b}H_{\ce}^{l+r}(M,E)
\end{align}
defines a bounded map for all $a,b,r,l$. We elaborate more on the cusp edge pseudodifferential calculus and some of its properties in the appendix.

\begin{lemma}\label{solvingnormaloperator}
    The normal operator of $x^{k}\slashed{\partial}_{E}$ at $\frontface$ is 
    \begin{align}
        \mathcal{N}_{\frontface,\ce}(x^{k}\slashed{\partial}_{E})=\cl(\partial_{x})\partial_{\tilde{s}}+\sum_{ij}\cl(a_{ij}\partial_{y_{i}})a_{ij}(0,y')\partial_{\tilde{\eta}_{i}}+\slashed{\partial}_{Z,y'}.
    \end{align}
    There exists $\mathcal{N}_{\frontface,\ce}(Q')\in \Psi^{-1}_{\operatorname{sus}}(\frontface)$ such that
    \begin{align}\label{doublespaceffinverse}
         \mathcal{N}_{\frontface,\ce}(x^{k}\slashed{\partial}_{E})\mathcal{N}_{\frontface,\ce}(Q')=\mathcal{N}_{\frontface,\ce}(Q')\mathcal{N}_{\frontface,\ce}(x^{k}\slashed{\partial}_{E})=\delta(\tilde{s})\delta(\tilde{\eta})\delta(z-z')d\tilde{s}d\tilde{\eta}dz'.
    \end{align}
\end{lemma}

\begin{proof}
    The expression for the normal operator follows from \eqref{xdefiningfunctionlifts} and doing the analogous calculation as for the normal operator for the heat space.

    Denote the second term by $\slashed{\partial}_{\tilde{\eta},y'}$ and $\Delta_{\tilde{\eta},y'}=\slashed{\partial}_{\tilde{\eta},y'}^{2}$ which is just the Laplacian on $\mathbb{R}^{b}_{\tilde{\eta}}$  with respect to the metric $g_{Y}(y')$ tensored with the identity $\operatorname{Id}_{E}$. Then the square of the normal operator is
    \begin{align}
        \mathcal{N}_{\frontface,\ce}(x^{k}\slashed{\partial}_{E})^{2}=-\partial_{\tilde{s}}^{2}+\Delta_{\tilde{\eta},y'}+\slashed{\partial}_{Z,y'}^{2}.
    \end{align}
    First we will find $A$ such that
    \begin{align}
        \mathcal{N}_{\frontface,\ce}(x^{k}\slashed{\partial})_{E}^{2}A=\delta(\tilde{s})\delta(\tilde{\eta})\delta(z-z').
    \end{align}
    Since $\mathcal{N}_{\frontface,\ce}(x^{k}\slashed{\partial})^{2}$ is elliptic in $\Psi^{2}(\mathbb{R}_{\tilde{s}}\times\mathbb{R}_{\tilde{\eta}}\times Z;E)$ we take $A'\in \Psi^{-2}(\mathbb{R}_{\tilde{s}}\times\mathbb{R}_{\tilde{\eta}}\times Z;E)$ which using a smooth cutoff function near the diagonal we can take to have kernel that supported in a neighbourhood of the lifted diagonal
    \begin{align}
       \mathcal{N}_{\frontface,\ce}(x^{k}\slashed{\partial}_{E})^{2}A'=\operatorname{Id}-R'
    \end{align}
     where $R'$ has smooth kernel and is also supported in a neighbourhood of the lifted diagonal. Restricting the kernel of $A'$ to $\tilde{s}'=0,\tilde{\eta}'=0$ we get a distribution $\tilde{A}$ on $\mathbb{R}_{\tilde{s}}\times\mathbb{R}_{\tilde{\eta}}\times Z^{2}$ compactly supported in a neighbourhood of $\{\tilde{s}=0,\tilde{\eta}=0\}$ which a conormal singularity of the correct order at $\{\tilde{s}=0,\tilde{\eta}=0,z=z'\}$ such that 
     \begin{align}
        \mathcal{N}_{\frontface,\ce}(x^{k}\slashed{\partial}_{E})^{2}\tilde{A}=\delta(\tilde{s})\delta(\tilde{\eta})\delta(z-z')-\tilde{R}
    \end{align}
    where $R$ is smooth and compactly supported in a neighbourhood of $\{\tilde{s}=0,\tilde{\eta}=0,z=z'\}$. 
    
     By assumption $\slashed{\partial}_{Z,y'}^{2}$ has trivial kernel, so taking $\varphi_{i,y'}$ to be a basis on eigenfunctions of $\slashed{\partial}_{Z,y'}^{2}$ with eigenvalue $\lambda^{2}>0$, we can expand the kernel $R_{y'}$ of $R$ restricted to the fibre above $y'$
     \begin{align}
         \begin{split}
             R_{y'}&=\sum_{ij}a_{ij}(\tilde{s},\tilde{\eta})\varphi_{i}(z)\varphi_{j}(z') \\
             a_{ij}(\tilde{s},\tilde{\eta})&=\int\langle\tilde{R}(\tilde{s},\tilde{\eta})\varphi_{j},\varphi_{i}(z)\rangle_{y'}dVol_{y'}(z).
        \end{split}
     \end{align}
     This series converges in $\mathcal{S}(\mathbb{R}_{\tilde{s}}\times\mathbb{R}^{b}_{\tilde{\eta}}\colon C^{\infty}(Z^{2}\colon \operatorname{HOM(\mathcal{S})}))$. and the coefficients $a_{ij}$ are compactly supported in a fixed neighbourhood of $\tilde{s}=0,\tilde{\eta}=0$. Thus to cancel away this error term we must solve
      \begin{align}
        \mathcal{N}_{\frontface,\ce}(x^{k}\slashed{\partial}_{E})^{2}\tilde{B}=\tilde{R}.
    \end{align}
    Expanding the kernel of $\tilde{B}$
    \begin{align}\label{errorcancel}
             B_{y'}&=\sum_{ij}b_{ij}(\tilde{s},\tilde{\eta})\varphi_{i}(z)\varphi_{j}(z').
     \end{align}
     We must solve
     \begin{align}
         (-\partial_{\tilde{s}}^{2}+\Delta_{\eta,y'}+\lambda_{i}^{2})b_{ij}(\tilde{s},\tilde{\eta})=a_{ij}(\tilde{s},\tilde{\eta}).
     \end{align}
     This has a unique solution in $\mathcal{S}(\mathbb{R}_{\tilde{s}}\times\mathbb{R}_{\tilde{\eta}}^{b})$ given by
     \begin{align}
         b_{ij}(\tilde{s},\tilde{\eta})=(2\pi)^{-b-1}\int\frac{e^{i(\tilde{\sigma}\tilde{s}+\tilde{E}\cdot\tilde{\eta})}\tilde{a}_{ij}(\tilde{\sigma},\tilde{E})}{\tilde{\sigma}^{2}+\tilde{E}^{2}+\lambda_{i}^{2}}d\tilde{\sigma}d\tilde{E}.
     \end{align}
     This converges to a smooth function which vanishes rapidly in $\tilde{s},\tilde{\eta}$.

     Now $\tilde{A}+\tilde{B}$ is a right inverse for the operator  $\mathcal{N}_{\frontface,\ce}(x^{k}\slashed{\partial}_{E})^{2}$ on the fibre over $y'$. Since $  \mathcal{N}_{\frontface,\ce}(x^{k}\slashed{\partial}_{E})^{2}$ is a smooth family of elliptic operators there is a smooth family of parametrices $A_{y'}$ such that
     \begin{align}
         \mathcal{N}_{\frontface,\ce}(x^{k}\slashed{\partial}_{E})^{2}A_{y'}=\operatorname{Id}-R_{y'}
     \end{align}
     where $R_{y'}$ is a smooth family of smoothing operators. In particular, the family of kernels $B_{y'}$ constructed using this smooth family of remainders is also smooth so that we can take $\tilde{A}+\tilde{B}$ is a smooth over $Y$. We also note that $\tilde{A}+\tilde{B}$ is a left inverse since $\mathcal{N}_{\frontface,\ce}(x^{k}\slashed{\partial}_{E})^{2}$ is self adjoint. Now we define
     \begin{align}
             \mathcal{N}_{\frontface,\ce}(Q'):=\mathcal{N}_{\frontface,\ce}(x^{k}\slashed{\partial}_{E})(\tilde{A}+\tilde{B})=(\tilde{A}+\tilde{B})  \mathcal{N}_{\frontface,\ce}(x^{k}\slashed{\partial}_{E})^{2}\tilde{B}
     \end{align}
     which by construction is in $\Psi^{-1}_{\operatorname{sus}}(\frontface)$ and satisfies \eqref{doublespaceffinverse}.
\end{proof}

Note that for $k=1$, we construct the Green's operator on the edge double space with a single blowup which in our case is $\backface$ blowup. In this case, the normal operator is given by
\begin{align}
    \mathcal{N}_{\backface,\operatorname{e}}(\slashed{\partial})=\cl(\partial_{x})\partial_{s}+\frac{f}{2s}\cl(\partial_{x})+\sum_{ij}\cl(a_{ij}\partial_{y_{i}})a_{ij}(0,y')\partial_{\tilde{\eta}_{i}}+\frac{1}{s}\slashed{\partial}_{Z,y'}.
\end{align}
In this case, having a trivial kernel is not sufficient for this operator to be essentially self-adjoint. The spectrum of the family of operators must be disjoint from $(-\frac{1}{2},\frac{1}{2})$ which is called the geometric Witt condition and the Green's operator will have non-trivial asymptotics at both $\rightface$ and $\leftface$ depending on the spectrum of the family.

Using the inverted normal operator and elliptic parametrix, we can now construct a parametrix for $\slashed{\partial}_{E}$.

\begin{theorem}\label{greenfunction}
    There exists $Q\in x^{k}\Psi^{-1}_{\ce}(M)$ such that $\slashed{\partial}_{E}Q=\operatorname{Id}-R$ and $Q\slashed{\partial}_{E}=\operatorname{Id}-\tilde{R}$ such that the remainders $R,\tilde{R}\in\Psi^{-\infty}_{\ce}(M)$ have kernels which vanish to infinite order at all faces. The operator $\slashed{\partial}_{E}$ is Fredholm on $\mathcal{D}$ and $\mathcal{D}=x^{k}H_{\ce}^{1}(M;E)$. 

    There exists a generalised inverse $G\in \Psi^{-1}_{\ce}(M;E)$ such that
    \begin{align}\label{generalisedinverse}
        G\slashed{\partial}_{E}=\operatorname{Id}-\pi_{K^{H}},\quad \slashed{\partial}_{E}G=\operatorname{Id}-\pi_{K^{L}}
    \end{align}
    where $\pi_{K^{H}},\pi_{K^{L}}\in\Psi^{-\infty}_{\ce}(M;\mathcal{S})$ are the orthogonal projections onto the kernel on $\mathcal{D}$ and $L^{2}$ respectively which have kernels which vanish to infinite order at all faces.
\end{theorem}

\begin{proof}
    Let $Q'_{1}\in\Psi^{-1}_{\ce}(M)$ be such that its normal operator is $\mathcal{N}_{\frontface,\ce}(Q')$ then 
    \begin{align}
        R_{1}=\operatorname{Id}-(x^{k}\slashed{\partial})Q'_{1}\in \Psi^{0}_{\ce}(M), \quad R_{1}'=\operatorname{Id}-Q'_{1}(x^{k}\slashed{\partial})\in \Psi^{0}_{\ce}(M)
    \end{align}
    vanish to first order at $\frontface$. Now we cancel off the conormal singularity of the remainder. Let $P\in \Psi^{-1}_{\ce}(M)$ be an elliptic parametrix for $x^{k}\slashed{\partial}_{E}$ so that
    \begin{align}
        R_{2}=\operatorname{Id}-(x^{k}\slashed{\partial}_{E})P,\quad R_{2}'=\operatorname{Id}-P(x^{k}\slashed{\partial}_{E})\in \Psi^{-\infty}_{\ce}(M).
    \end{align}
    Then $PR_{1}\in \Psi^{-1}_{\ce}(M)$ satisfies
    \begin{align}
        (x^{k}\slashed{\partial}_{E})(Q_{1}'+PR_{1})=\operatorname{Id}-R_{2}R_{1}
    \end{align}
    In particular the remainder $R_{3}=R_{2}R_{1}\in\Psi^{-\infty}_{\ce}(M)$ and vanishes to first order at $\frontface$ since $R_{1}$ does. Finally we take the asymptotic sum
    \begin{align}
        R_{4}\sim \sum_{i=1}^{\infty}R_{3}^{i}
    \end{align}
    then
    \begin{align}
         (x^{k}\slashed{\partial}_{E})(Q_{1}'+PR_{1})(\operatorname{Id}+R_{4})=\operatorname{Id}+R
    \end{align}    
    where this final remainder $R\in\Psi^{-\infty}_{\ce}(M)$ and vanishes to infinite order at all faces. Thus we have right parametrix $Q=(Q_{1}'+PR_{1})(\operatorname{Id}+R)\in\Psi^{-1}_{\ce}(M)$. Similarly, we can construct a left parametrix $Q_{L}$ which differs from $Q_{L}$ by an element in $\Psi^{-\infty}_{\ce}(M)$ which vanishes to all orders at all faces. In particular, any one of these can be taken to be a left and right parametrix.

    Since $(x)^{k}x^{-k}$ is bounded and non-vanishing in a neighbourhood of $\frontface\cap\diag_{\ce}$, the operator $\slashed{\partial}_{E}x^{k}=x^{-k}(x^{k}\slashed{\partial}_{E})x^{k}$ is also elliptic hence has a left and right parametrix $\tilde{Q}$ as above. Thus we have
    \begin{align}
        \slashed{\partial}_{E}x^{k}\tilde{Q}=\operatorname{Id}-\tilde{R}\quad Qx^{k}\slashed{\partial}_{E}=\operatorname{Id}-R
    \end{align}
    Thus the operator $\slashed{\partial}$ has left and right parametrices themselves given by $x^{k}\tilde{Q}$ and $Qx^{k}$ which again differ by a residual term vanishing to infinite order at $\frontface$. 

    Since $Q\colon x^{a}H^{s}_{\ce}\to x^{a}H^{s+1}_{\ce}$ is bounded and $R$ is bounded between any weighted Sobolev space since it has smooth kernel which vanishes to infinite order at all faces, if $u\in\mathcal{D}$ then $f=\slashed{\partial}_{E}u\in L^{2}$ so
    \begin{align}
        u=Q(x^{k}f)+Ru\in x^{k}H_{\ce}^{1}
    \end{align}
    so we have a bounded inclusion $\mathcal{D}\subset x^{k}H_{\ce}^{1}$. On the other hand, since $x^{k}\slashed{\partial}_{E}\colon x^{k}H_{\ce}^{1}\to x^{k}L^{2}$ we have that $\slashed{\partial}$ is a bounded map from $x^{k}H^{1}_{\ce}\to L^{2}$ hence $x^{k}H^{1}_{\ce}\subset\mathcal{D}$. So we have the equality $\mathcal{D}=x^{k}H_{\ce}^{1}$. The remainders are compact operators $\mathcal{D}\to L^{2}$ hence we have that $\slashed{\partial}$ is Fredholm.

    If $K,R$ are the kernel and range of $\slashed{\partial}_{E}$ then $\slashed{\partial}_{E}$ is invertible from $K^{\perp}\to R$ so we can define the generalised inverse $G$ to be the inverse of this map on $R$ extended to $0$ on $R^{\perp}$ so $G$ satisfies \eqref{generalisedinverse}. Since the kernel is smooth, vanishing to infinite order and $\slashed{\partial}$ is symmetric $\dot{C}^{\infty}(M;\mathcal{S})$ we have the orthogonal decomposition
    \begin{align}
        \dot{C}^{\infty}(M;\mathcal{S})=K\oplus (\operatorname{ran}(\slashed{\partial}_{E})\cap \dot{C}^{\infty}(M;\mathcal{S}))
    \end{align}
    Since this domain is dense in $L^{2}$ and the range is closed, it follows that $K=R^{\perp}$. Since $\dim{\ker}=n<\infty$, if $\varphi_{i}$ is an orthonormal basis for the kernel then the kernel of the projection is
    \begin{align}
        K_{\pi_{K}}(x,x')=\sum_{i=1}^{n}\varphi_{i}(x)\varphi_{i}(x')
    \end{align}
    which is smooth and vanishing to infinite order at all faces.

    It remains to show that $G\in\Psi_{\ce}^{-1}(M;\mathcal{S})$. Using $\operatorname{Id}=\slashed{\partial}Q+R=Q\slashed{\partial}+R'$ can write
    \begin{align}
        \begin{split}
             G&=Q\slashed{\partial}_{E}G+R'G=Q(\operatorname{Id}+\pi_{K^{L}})+R'G \\
             &=Q(\operatorname{Id}+\pi_{K^{L}})+R'(G\slashed{\partial}_{E}Q+GR) \\
             &=Q(\operatorname{Id}+\pi_{K^{L}})+R'(\operatorname{Id}+\pi_{K_{H}})Q+R'GR.
        \end{split}
    \end{align}
    The first two terms are in $\Psi_{\ce}^{-1}(M;\mathcal{S})$ and the $R'GR$ term maps smooth functions to smooth functions hence has smooth kernel and since both $R$ and $R'$ vanish to infinite order at all faces so does $R'GR$ hence is also in $\Psi_{\ce}^{-1}(M;\mathcal{S})$.
\end{proof}

Note that the domain $x^{k}H^{1}_{\ce}(M;E)$ consists of distributions of the form $x^{k}u$ such that $x^{k}\partial_{x}u\in L^{2}$. To understand why the maximal domain should be of this form, consider the simplified case of $L^{2}([0,\infty);x^{k}dx)$ with the operator
\begin{align}
    P=\partial_{x}+\frac{\lambda}{x^{k}}
\end{align}
If $u\in L^{2}$ then $u\in \mathcal{D}_{\operatorname{max}}(P)$ if $Pu\in L^{2}$. Now the general solution to the equation $Pu=f$ is given by
\begin{align}
    u=e^{-\frac{\lambda}{-k+1}x^{-k+1}}\left(\int_{1}^{x}e^{\frac{\lambda}{-k+1}y^{-k+1}}f(y)dy+c\right).
\end{align}
Then
\begin{align}
    x^{k}\partial_{x}(x^{-k}u)=-kx^{-1}u+\partial_{x}u.
\end{align}
Now
\begin{align}
    \partial_{x}u=-\lambda x^{-k}e^{-\frac{\lambda}{-k+1}x^{-k+1}}\int_{1}^{x}e^{\frac{\lambda}{-k+1}y^{-k+1}}f(y)dy+f(x).
\end{align}
So $ x^{k}\partial_{x}(x^{-k})u\in L^{2}$ if $u\in x^{k}L^{2}$. First consider $\lambda<0$ then if $u$ is $L^{2}$ for some $c$, then it is in $L^{2}$ for all $c$ so set $c=0$. Note that if $f$ is compactly supported and $f(y)=y^{-k}$ for small $y$ then $u=-\lambda$ is in $L^{2}$ so for all $\alpha\geq -k$ we have that $f(y)=y^{\alpha}\in L^{2}$. In particular for small $x$
\begin{align}
    \begin{split}
    |u|&=|e^{-\frac{\lambda}{-k+1}x^{-k+1}}\int_{1}^{x}e^{\frac{\lambda}{-k+1}y^{-k+1}}y^{-k}y^{\alpha+k}dy \\
    &\leq \lambda x^{\alpha+k}.
    \end{split}
\end{align}
So $f\in L^{2}$ if and only if $x^{-k}u\in L^{2}$. This still holds for any smooth function in $L^{2}$ compactly supported in a neghbourhood of $0$. Since any $f\in L^{2}$ compactly supported in a neighbourhood of $0$ is in the closure the subset of smooth functions, the bound also shows that $x^{-k}u\in L^{2}$ in this case as well.

If $\lambda>0$ then at most one solution is in $L^{2}$ given by
\begin{align}
    u=e^{-\frac{\lambda}{-k+1}x^{-k+1}}\int_{0}^{x}e^{\frac{\lambda}{-k+1}y^{-k+1}}f(y)dy.
\end{align}
The same arguments go through in this case.

Since $\slashed{\partial}_{E}$ is Fredholm on its self-adjoint domain, the index of $\slashed{\partial}^{+}$
\begin{align}
        \operatorname{ind}(\slashed{\partial}^{+})=\dim(\ker{\slashed{\partial}^{+}})-\dim(\operatorname{coker}{\slashed{\partial}^{+}})=\dim(\ker{\slashed{\partial}^{+}})-\dim(\ker{\slashed{\partial}^{-}}).
\end{align}
is well defined. Similar to the case of a closed manifold, the index satisfies the McKean-Singer formula.

\begin{theorem}
    For all $t>0$, $\operatorname{ind}(\slashed{\partial}^{+})=\operatorname{Str}(e^{-t\slashed{\partial}^{2}})$
\end{theorem}

\begin{proof}
    This follows from the standard arguments used when $M$ is a closed manifold. Let $\varphi\in L^{2}(M;R^{\pm})$ be an eigenfunction of $\slashed{\partial}^{2}$ with eigenvalue $\lambda^{2}\neq 0$ which is smooth and vanishing to infinite order at $\partial M$. Then $\slashed{\partial}^{2}(\slashed{\partial}\varphi)=\lambda^{2}(\slashed{\partial})\varphi\in L^{2}(M;E^{\mp})$ so the even and odd $\lambda^{2}_{i}$ eigenspaces $V^{\mp}_{i}$ which have the same finite dimension. So
    \begin{align}
        \operatorname{Str}(e^{-t\slashed{\partial}^{2}})=\sum_{i}e^{-t\lambda_{i}^{2}}(\dim(V^{+}_{i})-\dim(V^{-}_{i}))=\dim(\ker{\slashed{\partial}^{+}})-\dim(\ker{\slashed{\partial}^{-}}).
    \end{align}
\end{proof}
Now consider the family of metrics $g_{t}=g_{0}+th$ where $h$ is the $O(x^{k})$ polyhomogenous error term of $g$ with respect to the product type metric $g_{0}$. Consider the odd dimensional manifold with corner $\tilde{M}=M\times [0,1]$ with metric $\tilde{g}=dt^{2}+g_{t}$. If $M$ is spin then $\tilde{M}$ is spinnable and we can choose a spin structure on $\tilde{M}$ such that its restriction to $M\times\{1\}$ is the spin structure on $M$. Since $\tilde{M}$ has odd dimension, the restriction of the associated spinor bundle $\tilde{\mathcal{S}}\to\tilde{M}$ to $M\times\{1\}$ can be identified with the spinor bundle on $(M,g)$. Similarly the restriction to $M\times\{t\}$ can be identified with the spinor bundle of the induced spin structure on $(M,g_{t})$ which we denote $\mathcal{S}_{t}\to M$.

Now all these bundles are diffeomorphic, in particular we can identify them all with $\mathcal{S}_{0}\to M$ by parallel transport by the vector field $\partial_{t}$  with respect to the connection $\tilde{\nabla}$ of $\tilde{g}$. This gives us an identification of the smooth sections $C^{\infty}(M;\mathcal{S}_{0})\simeq C^{\infty}(M;\mathcal{S}_{t})$ and we can regard the associated spin Dirac operators $\slashed{\partial}_{t}$ as a smooth family of operators on $C^{\infty}(M;\mathcal{S}_{0})$ in the sense that the family of endomorphisms $\nabla^{\mathcal{S},t}-\nabla^{\mathcal{S},0}$ and Clifford multiplication $\cl_{\ice,t}$ depend smoothly on $t$.

Let $\partial_{x},\tilde{U}_{j},V_{i}$ be a local orthonormal frame of $\ice TM|_{\partial M}$ with respect to $g_{0}$, hence for all $g_{t}$ since the error vanishes at the boundary. Furthermore, these vector fields are all parallel with respect to $\tilde{\nabla}_{\partial_{t}}$ so, identified as operators on $\mathcal{S}_{0}$, all the $\slashed{\partial}_{t}$ have the same normal operator as that of $\slashed{\partial}_{0}$. 

If $F_{t}\colon\ice TM_{0}\to \ice TM_{t}$ and $G_{t}\colon\mathcal{S}_{0}\to\mathcal{S}_{t}$ the isomorphisms induced by parallel transport then, making the identifications of $\ice TM$ with $^{\ce}TM$ described in \ref{principalsymbol}, the principal symbol is the endomorphism
\begin{align}
    \begin{split}
        \sigma_{1,\ce}(x^{k}\slashed{\partial}_{t})(\xi)(s)&=G^{-1}_{t}(i\cl_{\ce,t}(F_{t}(\xi))(G_{t}s)) \\
        \sigma_{2,\ce}((x^{k}\slashed{\partial}_{t})^{2})(\xi)(x)&=\lvert F_{t}(\xi)\rvert_{x^{2k}g_{t}}s
    \end{split}
\end{align}
where $\cl_{t}$ is Clifford multiplication with respect to the metric $g_{t}$ on $\mathcal{S}_{t}$. In particular, the $\slashed{\partial}_{t}$ are a family of $\ce$-elliptic first order differential operators with smoothly varying principal symbol.

Since $\slashed{\partial}_{t}$ have the same normal operators and smoothly varying principal symbol we can construct a smooth family of parametrices for $\slashed{\partial}_{t}$ just following the proof of Theorem \ref{greenfunction} and using a smooth family of small elliptic parametrices.

\begin{lemma}\label{deformtoproducttype}
    The operators $\slashed{\partial}_{t}$ are a smooth family of Fredholm operators on $x^{k}H_{\ce}^{1}(M;\mathcal{S})\to L^{2}$ so
    \begin{align}
        \operatorname{ind}(\slashed{\partial})=\operatorname{ind}(\slashed{\partial}_{0}).
    \end{align}
\end{lemma}

The situation is similar for a Clifford connection $\nabla^{E}$ on $\ice\mathbb{C}l(M)$-module $E$ for an arbitrary exact ice metric. Again consider $M\times[0,1]$ with metric $g_{0}+tg$ and take a Clifford module $\tilde{E}$ whose restriction $E_{M\times\{1\}}$ is isomorphic to $E$ as a $\mathbb{C}l\ice(M)$-module. Take any Clifford connection $\nabla^{\tilde{E}}$ on $\tilde{E}$ then its restriction to $M\times[0,1]$ is also a Clifford connection on $E$ is the difference
\begin{align}
    A=\nabla^{E}-\nabla^{\tilde{E}}|_{M\times[0,1]}
\end{align}
is a section of $\operatorname{End}_{\ice\mathbb{C}l(M)}$ which can be regarded as a section of $\operatorname{End}_{\ce\mathbb{C}l(M\times[0,1])}(\tilde{E})|_{M\times 1}$. Thus we can extend to a section $\tilde{A}$ of $\operatorname{End}_{\ice\mathbb{C}l(M\times[0,1])}(\tilde{E})$. Then the connection
\begin{align}
    \nabla^{\tilde{E}\prime}=\nabla^{\tilde{E}}+\tilde{A}
\end{align}
is also a $\ice\mathbb{C}l (M\times[0,1])$ connection on $E$ whose restriction to $M\times [0,1]$ is equal to $\nabla^{E}$. Now we are in the same situation as above and we can define a family of operator $\ce$-elliptic operators $x^{k}\slashed{\partial}_{E,t}$ on $C^{\infty}(M,E_{0})$ such that $x^{k}\slashed{\partial}_{E,1}$ can be identified with $x^{k}\slashed{\partial}_{E}$ acting on $C^{\infty}(M,E_{0})$ and $x^{k}\slashed{\partial}_{E,0}$ is a Dirac operator associated to a Clifford connection on a space with product-type metric. Repeating the same arguments, $\slashed{\partial}_{E,t}$ defines a smooth family of Fredholm operators on $x^{k}H_{\ce}^{1}(M;E)\to L^{2}(M,E)$ so have equal indices.

Now we have reduced the calculation of the index to Clifford modules on spaces with product type metrics and we can simplify the situation further. Recall that in this case we have defined a connection $\nabla^{E,0}=\partial_{x}\otimes dx+\tilde{\nabla}^{E}$ such and a section of $\operatorname{End}(E)$ near the boundary
\begin{align}
    \tilde{\omega}=\nabla^{E}-\nabla^{E,0}
\end{align}
which vanishes at the boundary and is polyhomogeneous. Locally, using an identification $E\simeq \mathcal{S}\otimes W$ and the asymptotics of the Levi-Civita connection which implies the action of the connection $\nabla^{S}$ is of the form
\begin{align}
        \nabla^{\mathcal{S}}&=  d+\tilde{A}+x^{k-1}B 
\end{align}
it follows that every term in the asymptotic expansion of $\tilde{\omega}$ of order less than $x^{k-1}$ is a section of $\operatorname{End}_{\ice\mathbb{C}l(M)}(E)$. Take any polyhomogeneous section $\tilde{\Omega}$ of this bundle whose asymptotic expansion matches that of $\tilde{\omega}$ up to order $x^{k-1}$ so that the endomorphism
\begin{align}
    \omega'=\tilde{\omega}-\tilde{\Omega}
\end{align}
is polyhomogenous and vanishes to order $O(x^{k-1})$. Moreover, the connection
\begin{align}
    \nabla^{E_{0}}=\nabla^{E,0}+\omega'
\end{align}
is a Clifford connection since it differs from $\nabla^{E}$ by a section of $\operatorname{End}_{\mathbb{C}l\ice(M)}(E)$.

Now consider the family of connections defined by
\begin{align}
    \nabla^{E,t}=\nabla^{E_{0}}+t\tilde{\Omega}
\end{align}
 which is equal to $\nabla^{E}$ at $t=1$. This gives us a family of Dirac operators $\nabla^{E}_{t}$ which have equal normal operators as $\frontface$. In the same way as for the family of operators $\slashed{\partial}_{t}$ we have the following.

\begin{lemma}\label{twistedreducetoproductlemma}
    Let $(M,g)$ be incomplete cusp edge space $E$ be an ice Clifford module on M and $\nabla^{E}$ a Clifford connection. Then if the boundary family has trivial kernel then $\slashed{\partial}_{E}$ is a Fredholm operator from $x^{k}H^{1}_{\ce}(M,E)\to L^{2}(M,E)$. 
    
    Moreover, there exists a Clifford module $E_{0}$ on the $(M,g_{0})$ for an associated product type metric $g_{0}$ with $E_{0}|_{\partial M}\simeq E|_{\partial M}$ as boundary Clifford modules. There is a Clifford connection $\nabla^{E_{0}}$ on $E_{0}$ which satisfies assumption \ref{assumption1} such that the Dirac operator $\slashed{\partial}_{E_{0}}$ associated to the connection $\nabla^{E_{0}}$ has boundary family which can be identified with that of $\slashed{\partial}_{E}$. 
    
    The operator $\slashed{\partial}_{E_{0}}$ is Fredholm on $x^{k}H^{1}_{\ce}(M,E_{0})\to L^{2}(M,E_{0})$ and  
    \begin{align}
        \operatorname{ind}(\slashed{\partial}_{E})=\operatorname{ind}(\slashed{\partial}_{E,0}).
    \end{align}
\end{lemma}
Thus Lemma allows us to reduce the calculation of an index formula to the case of product type metrics and Clifford connections satisfying assumption 1. We will say that the Clifford connection $\nabla^{E}$ satisfy assumptions 2 and 3 if the associated connection $\nabla^{E_{0}}$ constructed above satisfies these assumptions.
 
If we had an actual twisting of the spinor bundle $\mathcal{S}\otimes W$ then the situation can be made even simpler. We can modify the twisting connection as described above
\begin{align}
    \nabla^{W,0}=dx\otimes\partial_{x}+\tilde{\nabla}^{W}.
\end{align}
Then the connection
\begin{align}
    \nabla^{\tilde{S}}\otimes\operatorname{Id}+\operatorname{Id}\otimes\nabla^{W,0}
\end{align}
is a Clifford connection since only the twisting connection has been modified and the only vanishing endomorphism terms which appear in the local expression for this connection are those that come from the spin connection. So formally, the calculation of an index formula in this case is the same as for the spin Dirac operator.

Note that since the heat kernel construction also goes through for these operators, the McKean-Singer formula also holds for them. Thus, in principle, one can directly find an index formula for them by considering the small time asymptotics of the heat kernel however the above lemmas allow us to just consider the operators $\slashed{\partial}_{0}$ and $\slashed{\partial}_{E,0}$ whose heat kernels have much simpler small time asymptotics.

We also note that although we could have taken any connection which is equal to $dx\otimes\partial_{x}+\tilde{\nabla}^{E}$ in a neighbourhood of the boundary and still obtained a Dirac-type operator, that is one that squares to a generalised Laplacian, which has the same index, this connection would no longer be a Clifford connection so would no longer satisfy the Lichnerowicz formula and other properties which are needed in the rescaling proof of a local index formula.

\section{Pushforward formula}\label{pushforward}

We now calculate the contribution from $\text{ff}$ to the asymptotic expansion of $\operatorname{Str}(e^{-t\slashed{\partial}^{2}})$ at $t=0$ for a product-type metric. Let $\text{diag}_{h}$$=\overline{\beta^{-1}(\text{diag}\times (0,\infty))}$ then the composition
\begin{align}\label{3.18}
    \sigma: \text{diag}_{h}\to \text{diag}\times [0,\infty)\to [0,\infty)
\end{align}
is a b-map. Let $\mu$ be the pullback of $\operatorname{dVol}$ to $\text{diag}_{h}$ then heat trace is given by
\begin{align}\label{3.19}
    \operatorname{Tr}(e^{-t\slashed{\partial}_{M}^{2}})\frac{d\tau}{\tau}=\sigma_{\star}(H_{t}|_{\text{diag}_{h}}\mu\sigma^{\star}(\frac{d\tau}{\tau})).
\end{align}
Let $\text{sf}, \text{bf}^{d}, \text{ff}^{d}, \text{tf}^{d}$ be the boundary hypersurfaces of $\text{diag}_{h}$ which are its intersections with $\text{lf}\cap\text{rf}, \text{bf}, \text{ff}, \text{tf}$ respectively. At $\text{ff}^{d}$ and $\text{tf}^{d}$, away from $\text{bf}^{d}$, we have coordinates
\begin{align}\label{tracespacecoord}
    \tilde{T}=\frac{t^{\frac{1}{2}}}{x^{k}},x,y,z.
\end{align}
The pullback of the volume form as a b-density in these coordinates is then
\begin{align}\label{3.21}
    \mu\sigma^{\star}(\frac{d\tau}{\tau})=x^{kf+1}(1+O(x^{2k}))\frac{dx\operatorname{dVol_{Y}}\operatorname{dVol}_{Z_{y}}d\tilde{T}}{x\tilde{T}}
\end{align}
where $\operatorname{dVol}_{Z_{y}}:=\sqrt{\det(g_{Z_{y}})}dz$ restricts to the volume form on each fibre $(Z_{y},g_{Y,y})$ and $\operatorname{dVol}_{Y}:=\sqrt{\det(g_{Y})}dy$ is the pullback of the volume form from the base $(Y,g_{Y})$ and the $O(x^{2k})$ term is a smooth error which comes from the off block diagonal terms in the metric in these coordinates (i.e. between $\partial_{y},\partial_{z}$), so in particular, this error term does not appear when the base is a point.

Directly from the index sets of the heat kernel we have that the index sets of $\tilde{H}_{t}|_{\text{diag}_{h}}$ are
\begin{align}\label{3.22}
    \mathcal{E}(\text{bf}^{d})=\varnothing, \quad \mathcal{E}(\text{ff}^{d})=-nk+\mathbb{N}, \quad \mathcal{E}(\text{tf})=-n+\mathbb{N}.
\end{align}
Then considering $H_{t}|_{\text{diag}_{h}}\text{dVol}\frac{dt}{t}$ as a b-density, by \eqref{3.21} it has the following index sets.
\begin{align}\label{3.23}
    \mathcal{E}(\text{bf}^{d})=\varnothing, \quad \mathcal{E}(\text{ff}^{d})=-k(b+1)+1+\mathbb{N}, \quad \mathcal{E}(\text{tf})=-n+\mathbb{N}.
\end{align}
For the following Lemma, we define the finite part of integral the integral of $u(x)\frac{dx}{x}$ denoted by
\begin{align}
    \dashint_{0}^{\infty}u(x)\frac{dx}{x}
\end{align}
to be the constant term in the asymptotic expansion of $\int_{\epsilon}^{\infty}u(x)\frac{dx}{x}$ and $\epsilon\to 0$.

\begin{figure}[H]
\centering
    \begin{tikzpicture}[scale=1.5,rotate around x=90,rotate around z=90,z={(0,0,-1)},y={(0,-1,0)}]
        \draw[gray, thick] ({sqrt(2)},0,0) -- (3,0,0) node[pos=1,left] {$x$};
        \draw[black, thick] (0,{sqrt(2)},0) -- (0,3,0) node[pos=1,right] {$x'$};
        \draw[red, thick] (0,0,{sqrt(2)}) -- (0,0,3) node[pos=1,above] {$t$};

        \draw[domain=0:33, smooth, variable=\x, purple, thick] plot ({sqrt(2)*cos(\x)}, {sqrt(2)*(sin(\x)}, {0}); 
        \draw[domain=58.1:90, smooth, variable=\x, purple, thick] plot ({sqrt(2)*cos(\x)}, {sqrt(2)*(sin(\x)}, {0}); 
        \draw[domain=0:90, smooth, variable=\x, purple, thick] plot ({sqrt(2)*cos(\x)}, {0}, {sqrt(2)*(sin(\x)}); 
        \draw[domain=0:90, smooth, variable=\x, purple, thick] plot ({0}, {sqrt(2)*cos(\x)}, {sqrt(2)*(sin(\x)}); 
        \draw[domain=15:90, smooth, variable=\x, purple, dotted] plot ({cos(\x)}, {cos(\x)}, {sqrt(2)*(sin(\x)}); 

        \draw[domain=-45:29, smooth, variable=\x, orange, thick] plot ({1+0.3*cos(\x)}, {1+0.3*(sin(\x)}, {0}); 
        \draw[domain=63:135, smooth, variable=\x, orange, thick] plot ({1+0.3*cos(\x)}, {1+0.3*(sin(\x)}, {0}); 
        \draw[domain=0:180, smooth, variable=\x, orange, thick] plot ({1+0.3*(1/(sqrt(2)))*cos(\x)}, {1-0.3*(1/(sqrt(2)))*cos(\x)}, {0.3*sin(\x)}); 
        \draw[domain=0:75, smooth, variable=\x, orange, dotted] plot ({1+0.3*(1/(sqrt(2)))*sin(\x)}, {1+0.3*(1/(sqrt(2)))*sin(\x)}, {0.3*cos(\x)}); 

        \draw[domain=0:180, smooth, variable=\x, cyan, thick] plot ({1+0.3*(1/(sqrt(2)))+0.1*(1/(sqrt(2)))*cos(\x)}, {1+0.3*(1/(sqrt(2)))-0.1*(1/(sqrt(2)))*cos(\x)}, {0.1*sin(\x)}); 

        \draw[domain=0:180, smooth, variable=\x, cyan, thick] plot ({3+1+0.3*(1/(sqrt(2)))+0.1*(1/(sqrt(2)))*cos(\x)}, {3+1+0.3*(1/(sqrt(2)))-0.1*(1/(sqrt(2)))*cos(\x)}, {0.1*sin(\x)});
        
        \draw[blue, dotted] ({1+0.3*(1/sqrt(2))},{1+0.3*(1/(sqrt(2)))},0.1) -- (4.2,4.2,0.1) node[pos=1.1,below] {}; 
        \draw[blue, thick] ({1+0.3*(1/(sqrt(2)))+0.1*(1/(sqrt(2)))}, {1+0.3*(1/(sqrt(2)))-0.1*(1/(sqrt(2)))}, {0}) -- ({1+0.3*(1/(sqrt(2)))+0.1*(1/(sqrt(2)))+3}, {1+0.3*(1/(sqrt(2)))-0.1*(1/(sqrt(2)))+3}, {0}) ; 
        \draw[blue, thick] ({1+0.3*(1/(sqrt(2)))-0.1*(1/(sqrt(2)))}, {1+0.3*(1/(sqrt(2)))+0.1*(1/(sqrt(2)))}, {0}) -- ({1+0.3*(1/(sqrt(2)))-0.1*(1/(sqrt(2)))+3}, {1+0.3*(1/(sqrt(2)))+0.1*(1/(sqrt(2)))+3}, {0}) ; 

        \draw[->,black,thick] (0,4,1.5) -- (0,3,1.5) node[pos=0.5,above] {};

        \draw[black, thick] (0,{5+(cos(30))+(0.75)},0) -- (0,8,0) node[pos=1,right] {$x$};
        \draw[red, thick] (0,5,1.5) -- (0,5,3) node[pos=1,above] {$t$};
        \draw[domain=30:90, smooth, variable=\x, orange, thick] plot ({0}, {5+(cos(\x)}, {1.5*sin(\x)}); 
        \draw[domain=0:90, smooth, variable=\x, orange, thick] plot ({0}, {5+(cos(30))+(0.75*cos(\x)}, {0.75*sin(\x)}); 


        \filldraw[draw=red,fill=red!20, variable=\x, domain=0:78]
            ({cos(12)}, {cos(12)}, {sqrt(2)*(sin(12)}) plot ({cos(12+\x)}, {cos(12+\x)}, {sqrt(2)*(sin(12+\x)}) 
            -- (0,0,{sqrt(2)}) 
            -- (0,0,3)
            -- (3,3,3)
            -- (3,3,0.1)
            -- ({1+0.3*(1/sqrt(2))},{1+0.3*(1/(sqrt(2)))},0.1) 
            -- ({1+0.3*(1/(sqrt(2)))*sin(70)}, {1+0.3*(1/(sqrt(2)))*sin(70)}, {0.3*cos(70)}) 
            plot ({1+0.3*(1/(sqrt(2)))*sin(70-\x)}, {1+0.3*(1/(sqrt(2)))*sin(70-\x)}, {0.3*cos(70-\x)})
            -- (1, 1, 0.3);

    \end{tikzpicture}
    \caption{$\diag_{h}\subset M^{2}_{\operatorname{heat}}$} \label{fig:blowup3}
\end{figure}
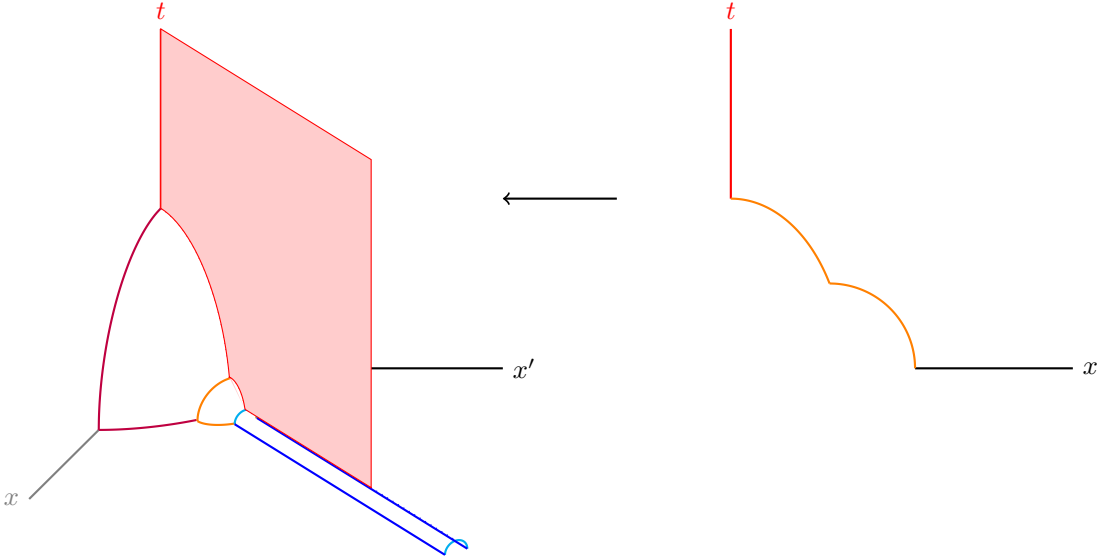

\begin{lemma}\label{lemma 3.3}
    Let $u=\tilde{u}\mu\sigma^{\star}(\frac{d\tau}{\tau})=:u'x^{kf+1}\frac{dx\operatorname{dVol_{Y}}\operatorname{dVol}_{Z_{y}}d\tilde{T}}{x\tilde{T}}$ be a polyhomogeneuous b-density on $\operatorname{diag}_{h}$ with index sets
    \begin{align}\label{3.24}
        \mathcal{E}(\text{bf}^{d})=\varnothing, \quad \mathcal{E}(\text{ff}^{d})=-k(b+1)+1+\mathbb{N}, \quad \mathcal{E}(\text{tf})=-n+\mathbb{N}.
    \end{align}
    Then (for $k>1$) if $\sigma_{\star}u=v\frac{d\tau}{\tau}$, then $v$ has asymptotic expansion at $\tau=0$
    \begin{align}\label{3.25}
        \sum_{i=0}^{n-1}a_{i}\tau^{-n+i}+\sum_{j=1}^{k(b+1)-1}b_{j}\tau^{-(b+1)+\frac{j}{k}}+c_{0}+\sum_{i=f+1}^{n-1}d_{i}\tau^{-n+i}\log(\tau)+c_{1}\log(\tau)+O(\tau^{\frac{1}{k}}).
    \end{align}
    The coefficients in \eqref{3.25} are given by
    \begin{align}\label{3.26}
        \begin{split}
            a_{i}&=\dashint_{M}u_{i,\operatorname{tf}^{h}}(x,y,z)\operatorname{dVol} \\
            b_{j}&=\frac{1}{k}\dashint_{0}^{\infty}\int_{\partial M}\tilde{T}^{b+1-\frac{j}{k}}u_{j,\operatorname{ff}^{h}}(\tilde{T},y,z)\operatorname{dVol_{Y}}\operatorname{dVol}_{Z_{y}}\frac{d\tilde{T}}{\tilde{T}} \\
            c_{0}&=\dashint_{M}u_{n,\operatorname{tf}^{h}}(x,y,z)\operatorname{dVol}+\frac{1}{k}\dashint_{0}^{\infty}\int_{\partial M}u_{k(b+1),\operatorname{ff}^{h}}(\tilde{T},y,z)\operatorname{dVol_{Y}}\operatorname{dVol}_{Z_{y}}\frac{d\tilde{T}}{\tilde{T}} \\
            c_{1}&=-\frac{1}{k}\int_{\partial M}u_{0,-kf-1}(y,z)\operatorname{dVol_{Y}}\operatorname{dVol}_{Z_{y}} \\
            d_{i}&=-\frac{1}{k}\int_{\partial M}u_{-n+i,-k(n+f-i)-1}(y,z)\operatorname{dVol_{Y}}\operatorname{dVol}_{Z_{y}} .
        \end{split}
    \end{align}
    Here $u_{i,\operatorname{tf}^{h}}(x,y,z)$ is the $-n+i$ coefficient in the asymptotic expansion of $u'$ at $\operatorname{tf}^{h}$ and $u_{j,\operatorname{ff}^{h}}(\tilde{T},y,z)$ is $-kf-k(b+1)+j-1=-kn-1+j$ ($u'$ is shifted by $-kf-1$ at $\operatorname{ff}^{d}$ by $\mu$) term in the asymptotic expansion of $u'$ in the coordinate $x$ in \eqref{3.18}. The $u_{\alpha,\beta}$ are the coefficients of $\tilde{T}^{\alpha}x^{\beta}$ in the mixed asymptotic expansion of $u'$ at $\frontface^{h}\cap\timeface^{h}$ in the coordinates.

    Moreover, for $0\leq i\leq f$ the integral for $a_{i}$ converges.
\end{lemma}

Note that the terms in the formulas for the coefficients are coefficients in the asymptotic expansion of $u'$ which includes the error term arising from the volume form.

We will need an extension of the above result in the case where $\mathcal{E}(\backface^{d})=\mathcal{E}(\frontface^{d})=\mathbb{N}$ for the heat kernel of the signature operator on a Witt incomplete cusp edge space. Near $\backface^{d}$ away from $\frontface^{d}$ using the coordinates
\begin{align}
    T=\frac{t^{\frac{1}{2}}}{x},x,y,z
\end{align}
and in a neighbourhood of the intersection $\backface^{d}\cap\frontface^{d}$ the coordinates
\begin{align}
    X=\frac{x}{T^{\frac{1}{k-1}}},S=T^{\frac{1}{k-1}},y,z
\end{align}
we can calculate similarly as we have done above and we see that
\begin{align}
    \begin{split}
     c_{0}=&\dashint_{M}u_{n,\operatorname{tf}^{h}}(x,y,z)\operatorname{dVol}+\frac{1}{k}\dashint_{0}^{\infty}\int_{\partial M}u_{k(b+1),\operatorname{ff}^{h}}(\tilde{T},y,z)\operatorname{dVol_{Y}}\operatorname{dVol}_{Z_{y}}\frac{d\tilde{T}}{\tilde{T}} \\
     &+\dashint_{0}^{\infty}\int_{\partial M}u_{0,\backface^{h}}(T,y,z)\operatorname{dVol_{Y}}\operatorname{dVol}_{Z_{y}}\frac{dT}{T}
     \end{split}
\end{align}
where $u_{0,\backface^{h}}(T,y,z)$ is the leading coefficient in the asymptotic expansion in $x$ of $u'$ at $\backface$ where $u=:u'x^{kf+1}\frac{dx\operatorname{dVol_{Y}}\operatorname{dVol}_{Z_{y}}dT}{xT}$ (which in this case is the restriction $u'x^{kf+1}$ to $\backface^{d}$.

\section{Index of the spin Dirac operator on non-isolated cusps}

We calculate the contribution at $\operatorname{ff}$ for the spin Dirac operator of a spin incomplete cusp edge space. For now we take a product type metric and discuss the more general case at the end.

We now fix a point $p\in \partial M$ with $\pi(p)=y$. Take a local trivialisation $U\times Z$ of the boundary fibration with $y\in U$ and let $y^{i}$ be normal coordinates about $y$ in $Y$ which we extend to coordinates $x,y^{i},z^{j}$ in a neighbourhood of $p$. Let $\partial_{y_{\alpha}}$ be a local orthonormal frame about $p\in Z_{y}$ and extend $\partial_{y_{\alpha}}$ by parallel transport along geodesics to a frame $U_{\alpha}$ in a neighbourhood of $y$. Extend this to a frame in a neighbourhood of $p$ given by $\partial_{x},\tilde{U}_{\alpha},x^{-k}V_{j}$. We denote the restriction of $x^{k}V_{j}$ to $\pi^{-1}(y)$ by $x^{k}V_{j}^{y}$ We also identify the restrictions $\mathcal{S}_{y'}=\mathcal{S}|_{\pi^{-1}(y')}$ to the fibres in a for $y'$ in neighbourhood of $y$ to $\mathcal{S}_{y}$ by parallel transport along geodesics in $U$ with respect to the metric $g_{Y}$.

\begin{lemma}\label{6.2}
    Let $M$ be an incomplete cusp edge space with spin structure and $\partial_{x},\tilde{U}_{\alpha},V_{j}$ a local orthonormal frame described above. Then given a local section of the spin bundle covering this frame which locally identifies sections of the associated spinor bundle with maps to the spinor space as described above, the action of the lift of $\tau\nabla_{W_{i}}$ to $M_{\text{heat}}^{2}$ at $(y^{\prime})^{i}=0$ is given by
    \begin{align}
        \begin{split}
        \tau&\nabla_{\partial_{x}}=\tilde{T}\partial_{\tilde{s}} \\
        \tau&\nabla_{\tilde{U}_{\gamma}}=\tilde{T}\sigma(\tilde{U}_{\gamma})((x')^{k}\tilde{\eta})+(x')^{k}V'_{\gamma}\\
        &+\frac{\tilde{T}}{4}(x')^{k}\sum_{jl}\left[g_{\partial M/Y}([\tilde{U}_{\gamma},V_{j}],V_{l})-\phi_{Y}^{\star}g_{Y}(\mathcal{S}^{\phi}(V_{j},V_{l}),\tilde{U}_{\gamma})\right]((x')^{k}\tilde{\eta},z)\cl(x^{-k}V_{j}^{y})\cl(x^{-k}V_{l}^{y}) \\
        &+\frac{\tilde{T}}{4}(x')^{2k}[(x')^{k-1}\tilde{s}+1]^{k}\sum_{\alpha j}[g_{\partial M/Y}(\mathcal{R}^{\phi}(\tilde{U}_{\gamma},\tilde{U}_{\alpha}),V_{j})(y^{i},z^{j})]((x')^{k}\tilde{\eta},z)\cl(\tilde{U}_{\alpha})\cl(x^{-k}V_{j}^{y}) \\
        &-\frac{\tilde{T}}{8}(x')^{2k}\sum_{\alpha\beta l}g_{Y}(R_{Y}(U_{\gamma},\tilde{\eta}^{l}\partial_{y^{l}}),U_{\alpha},U_{\beta})((x')^{k}\tilde{\eta})\cl(\tilde{U}_{\alpha})\cl(\tilde{U}_{\beta}) \\
        &+\sum_{j\alpha}O((x')^{2k})\cl(\tilde{U}_{\alpha})\cl(x^{-k}V_{i}^{y})+\sum_{\alpha\beta}O((x')^{3k})\cl(\tilde{U}_{\alpha})\cl(\tilde{U}_{\beta}) \\
        \tau&\nabla_{x^{-k}V_{i}}=\tilde{T}[(x')^{k-1}\tilde{s}+1]^{-k}\nabla_{V_{i}}^{\mathcal{S},y} \\
        &+\tilde{T}(x')^{k-1}[(x')^{k-1}\tilde{s}+1]\frac{1}{4}k(-\cl(x^{-k}V_{i}^{y})\cl(\partial_{x})+\cl(\partial_{x})\cl(x^{-k}V_{i}^{y})) \\
        &\quad+\sum_{\alpha}O((x')^{2k-1})\cl(\tilde{U}_{\alpha})\cl(\partial_{x})\\
        &+\frac{T}{2}(x')^{k}\sum_{j\alpha}\phi_{Y}^{\star}g_{Y}(\mathcal{S}^{\phi}(V_{i},V_{j}),\tilde{U}_{\alpha})((x')^{k}\tilde{\eta},z)\cl(x^{-k}V_{j}^{y})\cl(\tilde{U}_{\alpha}) \\
        &-\frac{\tilde{\tilde{T}}}{8}(x')^{2k}[(x')^{k-1}\tilde{s}+1]^{k}\sum_{\alpha\beta}g_{\partial M/Y}(\mathcal{R}^{\phi}(\tilde{U}_{\alpha},\tilde{U}_{\beta}),V_{i})((x')^{k}\tilde{\eta},z)\cl(\tilde{U}_{\alpha})\cl(\tilde{U}_{\beta}) \\
        &+\sum_{j\alpha}O((x')^{2k})\cl(\tilde{U}_{\alpha})\cl(x^{-k}V_{i}^{y})+\sum_{\alpha\beta}O((x')^{3k})\cl(\tilde{U}_{\alpha})\cl(\tilde{U}_{\beta}).
        \end{split}
    \end{align}
    Here, the $V'_{\gamma}$ are some smooth vertical vector fields. For $\tau\nabla_{x^{-k}V_{i}}$, when acted on by $\operatorname{cl}(x^{-k}V_{i})$ and summed over $i$, the sum in line 7 and 8 reduces to
    \begin{align}
        \tilde{T}(x')^{k-1}[(x')^{k-1}\tilde{s}+1]^{-1}\frac{kf}{2}\cl(\partial_{x})    \end{align}
    and in line 10 to
    \begin{align}
        -\frac{T}{2}(x')^{k}\sum_{\alpha}k(\tilde{U}_{\alpha})\operatorname{cl}(\tilde{U}_{\alpha}).
    \end{align}
    All the error terms $O((x')^{ak})$ are smooth and have expansions in powers of $(x')^{k}$.
\end{lemma}

\begin{proof}
   Since we are using a product type metric, using the asymptotics of the connection $\nabla^{\mathcal{S}}_{\partial_{x}}=\partial_{x}$. So $\tau\nabla_{\partial_{x}}$ lifts to $\tilde{T}\partial_{\tilde{s}}$. For the horizontal derivatives we have
    \begin{align}
        \begin{split}
            \nabla^{\mathcal{S}}_{\tilde{U}_{\gamma}}=\tilde{U}_{\gamma}(y^{i})&+\frac{1}{4}\sum_{jl}\left[g_{\partial M/Y}([\tilde{U}_{\gamma},V_{j}],V_{l})-\phi_{Y}^{\star}g_{Y}(\mathcal{S}^{\phi}(V_{j},V_{l}),\tilde{U}_{\gamma})\right](y^{i},z^{j})\operatorname{cl}(x^{-k}V_{j})\operatorname{cl}(x^{-k}V_{l}) \\
            &+\frac{x^{k}}{4}\sum_{\alpha j}g_{\partial M/Y}(\mathcal{R}^{\phi}(\tilde{U}_{\gamma},\tilde{U}_{\alpha}),V_{j})(y^{i},z^{j})\operatorname{cl}(\tilde{U}_{\alpha})\operatorname{cl}(x^{-k}V_{j}) \\
            &+\frac{1}{4}\sum_{\alpha\beta}g_{Y}(\nabla_{\tilde{U}_{\gamma}}\tilde{U}_{\alpha},\tilde{U}_{\beta})(y^{i})\operatorname{cl}(\tilde{U}_{\alpha})\operatorname{cl}(\tilde{U}_{\beta}).
        \end{split}
    \end{align}
    Here $\tau\tilde{U}_{\gamma}(y^{i})$ lifts to $\tilde{T}\sigma(\tilde{U}_{\gamma})((x')^{k}\eta+y')+(x')^{k}V'_{\gamma}$ for some vertical vector field $V'_{\gamma}$, independent of $x'$ which gives us the second term in the first line. First, parallel transport locally identifies a section $\cl(W)$ with a smooth map $C^{\infty}(U,C^{\infty}(Z_{y},\mathbb{C}l(\ice TM|_{Z_{y}}$)) which is equal to $\cl(W^{y})+O(y^{i})$ and using $y^{i}=(x')^{k}\eta^{i}+y^{i\prime}$ we can replace $\cl(W)$ with $\cl(W^{y})$ getting errors of the form in line 6 and each of the coefficients, when taking $(y')=0$ are evaluated at $(x')^{k}\eta^{i}$. Using $\tau=\tilde{T}(x')^{k}$ and $x=x'[(x')^{k-1}\tilde{s}+1]$, produce the factors in front of the line 3 and 4. For the final term, we use that the $U_{\alpha}$ are a local frame about $y\in Y$ extending an orthonormal frame at $y$ by radial parallel transport and the $y^{i}$ are normal coordinates about $y$ giving 
    \begin{align}
        g_{Y}(\nabla_{\tilde{U}_{\gamma}}\tilde{U}_{\alpha},\tilde{U}_{\beta})(y)=\sum_{j}g(R(\partial_{y_{\gamma}},\partial_{y_{j}})_{p}y^{j}\partial_{y_{\alpha}},\partial_{y_{\beta}})+O(\vert y\vert^{2}).
    \end{align}
    Now using $y^{i}=(x')^{k}\eta^{i}+y^{i\prime}$, we get the fourth line and an error terms which is contained in the line 6.
    
    For the vertical derivative, we have
    \begin{align}\label{verticalderivativeconnection}
        \begin{split}
        \nabla_{V_{i}}^{\mathcal{S}}=V_{i}&+\frac{1}{4}\sum_{jl}g_{\partial M/Y}(\nabla^{\partial M/Y}_{V_{i}}V_{k},V_{l})\operatorname{cl}(x^{-k}V_{j})\operatorname{cl}(x^{-k}V_{l}) \\
        &+\frac{1}{4}\sum_{l}kx^{k-1}(-\operatorname{cl}(x^{-k}V_{l})\operatorname{cl}(\partial_{x})+\operatorname{cl}(\partial_{x})\operatorname{cl}(x^{-k}V_{l})) \\
        &+\frac{1}{2}x^{k}\sum_{j\alpha}\phi_{Y}^{\star}g_{Y}(\mathcal{S}^{\phi}(V_{i},V_{j}),\tilde{U}_{\alpha})(y,z)\operatorname{cl}(x^{-k}V_{j})\operatorname{cl}(\tilde{U}_{\alpha})  \\
        &-\frac{x^{2k}}{8}\sum_{\alpha\beta}g_{\partial M/Y}(\mathcal{R}^{\phi}(\tilde{U}_{\alpha},\tilde{U}_{\beta}),V_{i})(y,z)\operatorname{cl}(\tilde{U}_{\alpha})\operatorname{cl}(\tilde{U}_{\beta}).
        \end{split}
    \end{align}
  Now acting on the first two lines by $x^{-k}\operatorname{cl}(x^{-k}V_{i})$ gives
    \begin{align}
        x^{-k}\operatorname{cl}(V_{i})\nabla_{V_{i}}^{\mathcal{S}}=x^{-k}\operatorname{cl}(V_{i})\nabla_{V_{i}}^{\partial M/Y}+\frac{k}{2x}\operatorname{cl}(\partial_{x}).
    \end{align}
    So summing over $i$ and multiplying by $\tau$, we see that the lift of the resulting expression is exactly line 7 and gives us the fact about line 8 and 9.

    For the third line in \eqref{verticalderivativeconnection}, since $\mathcal{S}^{\phi}$ is symmetric in $V_{i},V_{j}$, after summing over $i$ all the terms with $j\neq i$ cancel leaving only terms of the form
    \begin{align}
        \sum_{j\alpha}\phi_{Y}^{\star}g_{Y}(\mathcal{S}^{\phi}(V_{i},V_{i}),\tilde{U}_{\alpha})(y,z)\operatorname{cl}(\tilde{U}_{\alpha})
    \end{align}
    Summing over $i$, for any fixed $\alpha$, by definition we get $k(\tilde{U}_{\alpha})\operatorname{cl}(\tilde{U}_{\alpha})$ so we get the line 10 after summing over $\alpha$. Finally, line 4 in \eqref{verticalderivativeconnection} gives the second last line.

    All errors have expansion in powers of $(x')^{k}$ because all the coefficients are smooth and $y^{i}=(x')^{k}\eta^{i}+y^{i\prime}$ and $x=x'[(x')^{k-1}\tilde{s}+1]$.
\end{proof}

We define a Getzler rescaling $\delta_{(x')}{\omega}=(x')^{-mk}\omega$ for $\omega\in\Lambda^{m}T_{p}^{H}Z\otimes\mathbb{C}$. Note that under conjugation by the horizontal symbol map that the lift of $\tau\slashed{\partial}$ is singular as $(x')\to 0$ under this Getzler rescaling. However, recall the Lichnerowicz formula for the square of the spin Dirac operator
\begin{align}
    \slashed{\partial}^{2}=\Delta^{\mathcal{S}}+\frac{S}{4}
\end{align}
where $\Delta^{\mathcal{S}}$ is the rough Laplacian of $\nabla^{\mathcal{S}}$ given in a local frame by
\begin{align}
    -\left(\sum_{i}(\nabla^{\mathcal{S}}_{e_{i}})^{2}-\nabla^{\mathcal{s}}_{\nabla_{e_{i}}e_{i}}\right).
\end{align}
From the asymptotics of the connection, we can see that $\tau\nabla^{\mathcal{S}}$ has a limit at $(x')\to 0$ under the rescaling, thus the square $\tau^{2}\slashed{\partial}^{2}$ is in fact not singular and we can calculate the limit as $(x')\to 0$. We can also see from the above lemma that the next highest order terms of the rescaling of $\tau^{2}\slashed{\partial}^{2}$ is at order $k-1$. Recall that under the horizontal symbol map $\operatorname{cl}(\tilde{U}_{i})$ acts as $\epsilon(\tilde{U}_{i}))-\iota(\tilde{U}){i})$.

For the scalar curvature we have the following asymptotics.
\begin{lemma}\label{scalarcurvature}
    We have $x^{2k}S=S_{g_{Z_{y}}}+O(x^{k})$ where $S_{g_{Z_{y}}}$ at a point $(y,z)$ on the boundary is the scalar curvature at $z$ of the fibre with metric $g_{Z,y}$. In particular, $\tau^{2}S$ lifts to $\tilde{T}S_{g_{Z_{y}}}+O((x')^{k})$.
\end{lemma}

\begin{proof}
    Recall the scalar curvature is given by
    \begin{align}
        S=\sum_{ij}g(R(e_{i},e_{j})e_{j},e_{i})
    \end{align}
    where $e_{i}$ is any orthonormal frame. We use an orthonormal frame $\partial_{x},\tilde{U}_{i}$ and $x^{-k}V_{j}$ as above. By the asymptotics of the curvature the only terms which are not $O(x^{k})$ are those which contain at least one vertical term. If only one of the terms is vertical we have
    \begin{equation}
         \begin{split}
            g(R(e_{i},x^{-k}V_{j})x^{-k}V_{j},e_{i})&=x^{-k} g(R(e_{i},V_{j})x^{-k}V_{j},e_{i}) \\
            &=O(x^{-k}).
        \end{split}
    \end{equation}
    So these only contribute $O(x^{k})$ terms to $x^{2k}S$. When there are only vertical terms we have
    \begin{align}
         g(R(x^{-k}V_{i},x^{-k}V_{j})x^{-k}V_{j},x^{-k}V_{i})&=x^{-2k} g(R(V_{i},V_{j})x^{-k}V_{j},x^{-k}V_{i}).
    \end{align}
    By asymptotics of the connection, we have
    \begin{equation}\label{97}
        \begin{split}
            \nabla_{V_{i}}x^{-k}V_{j}=x^{k-1}g_{\partial M/Y}(V_{i},V_{j})\partial_{x}&+\sum_{m}g_{\partial M/Y}(\nabla^{\partial M/Y}_{V_{1}}V_{2},V_{m})x^{-k}V_{m}\\
            &+\sum_{l}x^{k}\phi_{Y}^{\star}g_{Y}(\mathcal{S}^{\phi}(V_{i},V_{j}),\tilde{U}_{l})\tilde{U}_.
        \end{split}
    \end{equation}
    If we take $\nabla_{V_{k}}$ of this, the last term contributes terms that are $O(x^{k})$. The first term, there is a non-zero vertical component only when the $\nabla_{V_{k}}$ hits the $\partial_{x}$ which produces another factor of $x^{k}$ so this term only gives $O(x^{2k-1})$ factors. This leaves the second term whose contribution is exactly the scalar curvature of the fibre with metric $g_{Z}$.
\end{proof}
We now recall some facts about the Bismut superconnection and Bismut-Cheeger eta form. Let $\psi\colon N\to Y$ be a fibre bundle with a metric $g_{T(N/Y)}$ on the vertical tangent bundle $T(N/Y)$ and a splitting $TN=T(N/Y)\oplus T_{H}M$ where we have $T_{H}M\simeq \psi^{\star}TY$. For example, given a submersion metric $g_{N}=g_{N/Y}+\psi^{\star}g_{Y}$ we can take $g_{T(N/Y)}$ to be the restriction of $g_{N/Y}$ to the vertical bundle and $T_{H}M$ to be the orthogonal complement to the vertical bundle.

Choosing a metric on the base $g_{Y}$ we define the family of metrics $g=g_{T(N/Y)}\oplus \psi^{\star}g_{Y}$ on $TN=T(N/Y)\oplus T_{H}M$. Let $\nabla^{g}$ be the Levi-Civita connection for $g_{u}$ then we can define a connection on the vertical bundle which is independent of $g_{Y}$ given by $\nabla^{T(N/Y)}=\boldsymbol{v}\nabla^{g}\boldsymbol{v}$ where $\boldsymbol{v}$ is projection onto the vertical bundle. We can then define a product connection on $TN$ by $\nabla^{\oplus}=\nabla^{T(N/Y)}\oplus \psi^{\star}\nabla^{Y}$ where $\nabla^{Y}$ is the Levi-Civita connection on $(Y,g_{Y})$. This defines a metric connection on $N$ which is in general non-torsion free and we can define
\begin{align}
    \omega=\nabla^{g}-\nabla^{\oplus}.
\end{align}
This is an $\operatorname{End}(TM)$-valued one form which via the metric $g$ can be identified with a $T^{\star}M\otimes T^{\star}M$-valued one form which is antisymmetric thus can be identified with a $\Lambda^{2}T^{\star}M$-valued one form which satisfies
\begin{align}\label{differencetensor}
    \begin{split}
    \omega&(W_{1})(W_{2},W_{3})=-\mathcal{S}^{\phi}(\boldsymbol{v}W_{1},\boldsymbol{v}W_{3})(\boldsymbol{h}W_{2})+\mathcal{S}^{\phi}(\boldsymbol{v}W_{1},\boldsymbol{v}W_{2})(\boldsymbol{h}W_{3}) \\
    &-\frac{1}{2}g_{\partial M/Y}(\mathcal{R}^{\phi}(\boldsymbol{h}W_{1},\boldsymbol{h}W_{3}),\boldsymbol{v}W_{2})+\frac{1}{2}g_{\partial M/Y}(\mathcal{R}^{\phi}(\boldsymbol{h}W_{1},\boldsymbol{h}W_{2}),\boldsymbol{v}W_{3}) \\
    &-\frac{1}{2}g_{\partial M/Y}(\mathcal{R}^{\phi}(\boldsymbol{h}W_{2},\boldsymbol{h}W_{3}),\boldsymbol{v}W_{1}).
    \end{split}
\end{align}
We comment that \cite{bgv} uses the opposite sign convention for $\mathcal{S}^{\phi}$ and both \cite{bgv},\cite{BismutCheegerEta} use the opposite sign convention for $\mathcal{R}^{\phi}$.

Now let $E\to N$ be a vertical Clifford module with vertical Clifford connection $\nabla^{E}$ and denote the vertical Clifford multiplication by $\cl_{E}$. We define $\mathbb{E}=E\otimes \psi^{\star}\Lambda T^{\star}Y$ and define a (degenerate) Clifford action on $\mathbb{E}$ by
\begin{align}
    \cl_{0}(W)(s\otimes\omega)=(\cl_{E}(\boldsymbol{v}W)s)\otimes\omega+(-1)^{\deg{s}}s\otimes(\epsilon(\boldsymbol{h}W^{\musFlat{}})\omega)
\end{align}
where $\epsilon$ is exterior multiplication and $s\otimes\omega$ is a section of $E^{\pm}\otimes \psi^{\star}\Lambda T^{\star}Y$. Then we can define a Clifford connection on $\mathbb{E}$ with respect to the Clifford action $\cl_{0}$ by
\begin{align}
    \nabla^{\mathbb{E}}=\nabla^{E}\otimes\operatorname{Id}+\operatorname{Id}\otimes \psi^{\star}\nabla^{Y}+\frac{1}{2}\cl_{0}(\omega)
\end{align}
where the action of $\cl_{0}(\omega)$ with respect to an orthonormal frame $W_{i}$ is given by
\begin{align}
    \cl_{0}(\omega)(W)=\sum_{i<j}\omega(W)(W_{i},W_{j})\cl_{0}(W_{i}\wedge W_{j})=\frac{1}{2}\sum_{ij}\omega(W)(W_{i},W_{j})\cl_{0}(W_{i})\cl_{0}(W_{j}).
\end{align}

From \eqref{differencetensor}, the restriction of this connection to each fibre is independent of $g_{b}$ and in a local orthonormal frame $U_{\alpha},V_{i}$ the action of a vertical derivative is given explicitly by
\begin{align}
    \begin{split}
    \nabla^{\mathbb{E}}_{V}&=\nabla^{E}_{V}+\frac{1}{4}\sum_{ij}\omega(V)(W_{i},W_{j})\cl_{0}(W_{i})\cl_{0}(W_{j}) \\
    &=\nabla^{E}_{V}+\frac{1}{2}\sum_{ij}S(V,V_{i})(\tilde{U}_{\alpha})\cl_{0}(V_{i})\cl_{0}(\tilde{U}_{\alpha}) \\
    &\qquad -\frac{1}{8}\sum_{ij}g_{\partial M/Y}(R(\tilde{U}_{\alpha},\tilde{U}_{\beta}),V)\cl_{0}(\tilde{U}_{\alpha})\cl_{0}(\tilde{U}_{\beta}).
    \end{split}
\end{align}

\begin{defn}
    The Bismut superconnection $\mathbb{B}$ is the Dirac operator on the Clifford module $\mathbb{E}$ with Clifford connection $\nabla^{\mathbb{E}}$
    \begin{align}
        \mathbb{B}=\sum_{i}\cl(W_{i})\nabla^{\mathbb{E}}_{W_{i}}.
    \end{align}
\end{defn}
Define $\psi_{\star}E\to Y$ to be the infinite dimensional $\mathbb{Z}_{2}$-graded Hermitian vector bundle whose fibre over $y$ is the space of smooth sections $C^{\infty}(\psi^{-1}(y),E|_{\psi^{-1}(y)})$. The Bismut superconnection is a Clifford superconnection on $\psi_{\star}E$.

Explicitly, the Bismut superconnection is given by
\begin{align}\label{bismutsc}
    \mathbb{B}=\sum_{i}\cl(V_{i})\nabla^{E}_{V_{i}}+\sum_{\alpha}\epsilon(\tilde{U}_{\alpha})\left(\nabla^{E}_{\tilde{U}_{\alpha}}-\frac{1}{2}k(\tilde{U}_{\alpha})\right)+\frac{1}{4}\sum_{i,\alpha<\beta}g_{\partial M/Y}(R(\tilde{U}_{\alpha},\tilde{U}_{\beta}),V_{i})\epsilon(\tilde{U}_{\alpha})\epsilon(\tilde{U}_{\beta})\cl(V_{i})
\end{align}
where $k(V)$ is the mean curvature
\begin{align}
    k(\tilde{U})=\sum_{i}S^{\phi}(V_{i},V_{i})(\tilde{U}).
\end{align}

From this expression, one can see that the Bismut superconnection depends on the vertical family of metrics and the horizontal bundle and not the metric $g_{Y}$. 

We briefly describe each of the terms which appear in the Bismut superconnection. The first term in \eqref{bismutsc} is exactly the family of Dirac operators on the fibres which we denote $\slashed{\partial}_{Z,y}$. To describe the second term which we will denote $\tilde{\nabla}^{\mathbb{E},u}=\nabla^{E}-\frac{1}{2}k$, we first note that the Hermitian metric on $\psi_{\star}E\to Y$ at $y\in Y$ is given by
\begin{align}
    \langle \tilde{s}_{1},\tilde{s}_{2}\rangle_{\psi_{\star}E,y}=\int_{Z_{y}}\langle s_{1},s_{2}\rangle_{E}\operatorname{dVol}_{Z_{y}}
\end{align}
where $\operatorname{dVol}$ is the $f$-form which restricts to the volume form on each fibre and vanishes on horizontal vectors. Since $\nabla^{E}$ is a Hermitian connection and $\mathcal{L}_{\tilde{U}}\operatorname{dVol}|_{Z_{y}}=-k(\tilde{U}\operatorname{dVol}$, we have
\begin{align}
    \tilde{U}\langle \tilde{s}_{1},\tilde{s}_{2}\rangle_{\psi_{\star}E,y}&=i_{\tilde{U}}d\int_{Z_{y}}\langle s_{1},s_{2}\rangle_{E}\operatorname{dVol}_{Z_{y}} \\
    &=\int_{Z_{y}}i_{\tilde{U}}d\left(\langle s_{1},s_{2}\rangle_{E}\operatorname{dVol}_{Z_{y}}\right) \\
    &=\int_{Z_{y}}\tilde{U}\langle s_{1},s_{2}\rangle_{E}\operatorname{dVol}_{Z_{y}}+\int_{Z_{y}}\langle s_{1},s_{2}\rangle_{E}\mathcal{L}_{\tilde{U}}\left(\operatorname{dVol}_{Z_{y}}\right) \\
    &=\int_{Z_{y}}\langle \nabla_{\tilde{U}}^{E}s_{1},s_{2}\rangle_{E}+\langle s_{1},\nabla_{\tilde{U}}^{E}s_{2}\rangle_{E}-k(\tilde{U})\langle s_{1},s_{2}\rangle_{E}\operatorname{dVol}_{Z_{y}}.
\end{align} 
In the second line, we used that $d$ and $\iota_{\tilde{U}}$ commute with fibre integration up to a factor of $(-1)^{f}$. From this we see that $\nabla^{\mathbb{E},u}$ is a Hermitian connection on $\psi_{\star}E$ naturally induced by the connection $\nabla^{E}$, which we see is not Hermitian considered as a metric on $\psi^{\star}E$.

For the third term, note that, since $\nabla^{g}$ is torsion free, the torsion of $\nabla^{\oplus}$ is given by
\begin{align}
    T^{\oplus}(W_{1},W_{2})^{\musFlat{}}&=-\omega(W_{1})(W_{2},\cdot)+\omega(W_{2})(W_{1},\cdot).
\end{align}
In particular, if we evaluate the torsion of horizontal vectors we have
\begin{align}\label{beforect}
    g_{\partial M/Y}(T^{\oplus}(\tilde{U}_{\alpha},\tilde{U}_{\beta}),V)=-g_{\partial M/Y}(R(\tilde{U}_{\alpha},\tilde{U}_{\beta}),V).
\end{align}
So we see that the third term is also given by the coefficients of the torsion of $\nabla^{\oplus}$ hence we will also denote it by $-c(T)$ (following \cite{BismutCheeger90a}).

The restriction of $\nabla^{\mathbb{E}}$ to each fibre $E|_{\psi^{-1}(y)}$ defines a rough Laplacian $\Delta^{y}=(\nabla^{\mathbb{E}})^{\star}\nabla^{\mathbb{E}}$ and we denote the smooth family of operators by $\Delta^{N/Y}$. Then the Bismut superconnection satisfies a Lichnerowicz formula
\begin{align}\label{bismutlichnerowics}
    \mathbb{B}^{2}=\Delta^{N/Y}+\frac{1}{4}S_{y}+\frac{1}{2}\sum_{ij}F^{E/S}(W_{i},W_{j})\cl(W_{i})\cl(W_{j})
\end{align}
where $S_{y}$ is the scalar curvature of the fibre and $F^{E/S}$ is the twisting curvature of $E$ (relative to the vertical Clifford action).

Let $\delta^{Y}_{t}$ be the rescaling on $\Lambda T^{\star}Y$ by $\delta_{t}^{Y}\omega=t^{-\frac{i}{2}}\omega$ for $\omega\in \Lambda^{i}T^{\star}Y$ then we define the rescaled Bismut superconnection
\begin{align}
    \mathbb{B}_{t}=t^{\frac{1}{2}}\delta_{t}^{Y}\circ\mathbb{B}\circ(\delta^{Y}_{y})^{-1}=t^{\frac{1}{2}}\mathbb{B}_{0}+\mathbb{B}_{1}+t^{-\frac{1}{2}}\mathbb{B}_{2}.
\end{align}
Let $K(E)$ be the infinite dimensional vector bundle on $E$ whose fibre over $y\in Y$ is given by smoothing operators on $E_{y}$, that is, have smooth integral kernel. Since we have compact fibres, these are trace class so there is a fibre supertrace map $\operatorname{Str}_{K(E)}\colon C^{\infty}(Y;K(E))\to C^{\infty}(Y)$ given by $\operatorname{Str}_{K(E)}(K)(y)=\operatorname{Str}_{E_{y}}(K(y))$. This extends to a map $C^{\infty}(Y;K(E)\otimes\Lambda T^{\star}M)\to C^{\infty}(Y;\Lambda T^{\star} M)$ given by
\begin{align}
    \operatorname{Str}_{K(E)}(\omega\otimes K)=\omega\operatorname{Str}_{K(E)}(K).
\end{align}
We also define
\begin{align}
    \operatorname{Tr}^{\text{even}}_{K(E)}(\omega\otimes K)=\omega^{\text{even}}\operatorname{Tr}_{K(E)}(K).
\end{align}
where $\omega^{\text{even}}$ is the projection of $\omega$ onto its even degree component.

It can be shown that when $\operatorname{dim}(N/Y)$ is even or odd respectively then as $t\to 0$ we have
\begin{align}
    \begin{split}
        \operatorname{Str}_{K(E)}\left(\frac{\partial\mathbb{B}_{t}}{\partial t}e^{-\mathbb{B}_{t}^{2}}\right)&=O(t^{\frac{1}{2}}) \\
        \operatorname{Tr}_{K(E)}^{\text{even}}\left(\frac{\partial\mathbb{B}_{t}}{\partial t}e^{-\mathbb{B}_{t}^{2}}\right)&=O(t^{\frac{1}{2}}).
    \end{split}
\end{align}
If the vertical family of Dirac operators is invertible, then both of them also decay exponentially as $t\to\infty$ otherwise they are $O(t^{-\frac{3}{2}})$ thus we have the following definition.

\begin{defn}
    Let $F\to N\to Y$ with odd dimensional total space. If $\dim{N/Y}$ is even the Bismut-Cheeger eta form is
    \begin{align}
        \hat{\eta}=\int_{0}^{\infty}\operatorname{Str}_{K(E)}\left(\frac{\partial\mathbb{B}_{t}}{\partial t}e^{-\mathbb{B}_{t}^{2}}\right)dt.
    \end{align}
    If $\dim{N/Y}$ is odd
    \begin{align}
        \hat{\eta}=\frac{1}{\sqrt{\pi}}\int_{0}^{\infty}\operatorname{Tr}_{K(E)}^{\text{even}}\left(\frac{\partial\mathbb{B}_{t}}{\partial t}e^{-\mathbb{B}_{t}^{2}}\right)dt.
    \end{align}
    The normalised Bismut-Cheeger eta form $\tilde{\eta}$ is 
    \begin{align}
        \tilde{\eta}=
        \begin{dcases}
            \sum_{j}\frac{1}{(2\pi i)^{j}}[\hat{\eta}]_{2j-1}  & \text{ if $\operatorname{dim}(Z/Y)$ even}  \\
            \sum_{j}\frac{1}{(2\pi i)^{j}}[\hat{\eta}]_{2j}  & \text{ if $\operatorname{dim}(Z/Y)$ odd} 
        \end{dcases}
    \end{align}
    where $[\hat{\eta}]_{j}$ is the degree $j$ component of $\hat{\eta}$.
\end{defn}
Explicitly, we have
\begin{align}
    \frac{\partial\mathbb{B}_{t}}{\partial t}e^{-\mathbb{B}_{t}^{2}}=\left(\slashed{\partial}_{Z,y}+\frac{c(T)}{4t}\right)e^{-\mathbb{B}_{t}^{2}}\frac{1}{2t^{\frac{1}{2}}}.
\end{align}
Note that when $Y$ is a point and $\operatorname{dim}(Z)$ is odd $[\hat{\eta}]_{0}$ is the function whose value at $y$ is half the $\eta$ invariant of $\slashed{\partial}_{Z,y}$. Now we determine the asymptotics of the rescaled heat kernel at $\frontface$.

If $N$ is spin the one of $Y$ or $T(N/Y)$ being spin gives the other a unique spin structure compatible with that of $N$. In this case, if $\mathcal{S}_{N},\mathcal{S}_{Y},\mathcal{S}_{Z/Y}$ are their respective spinor bundle then we have a canonical identification $\mathcal{S}_{N}\simeq \psi^{\star}\mathcal{S}_{Y}\otimes \mathcal{S}_{N/Y}$. In this case the vertical spinor bundle $\mathcal{S}_{N/Y}$ inherits a connection $\nabla^{S_{N/Y}}$ from the connection on the vertical tangent bundle $\nabla^{T(N/Y)}=\boldsymbol{v}\nabla^{g}\boldsymbol{v}$ where $\nabla^{g}$ is the Levi-Civita connection on $Z$ ($\nabla^{T(N/Y)}$ defines a connection on the vertical frame bundle which lifts to the vertical spin bundle). Moreover, this connection restricts to the spin connection on each fibre and defines a vertical Clifford connection on $\mathcal{S}_{N/Y}$ so we can define the $\hat{\eta}$ form for this vertical family of operators. Note that in this case that the kernel of $\frac{\partial\mathbb{B}_{t}}{\partial t}e^{-\mathbb{B}_{t}^{2}}$ is a section of $\psi^{\star}\Lambda^{\star}Y\otimes\operatorname{HOM}(\mathcal{S}_{N/Y})$ on the fibre diagonal and restricted to the diagonal is a section of $\psi^{\star}\Lambda^{\star}Y\otimes\operatorname{End}(\mathcal{S}_{N/Y})\simeq\psi^{\star}\Lambda^{\star}Y\otimes\mathbb{C}l(N/Y)$.

Now recall if $(M,g)$ is a spin incomplete cusp edge space then $\partial M$ has an induced spin structure and $\mathcal{S}^{+}|_{\partial M}\to \partial M$ can be identified with the spinor bundle $\mathcal{S}_{\partial M}$ on $\partial M$ with $\mathbb{C}l(\partial M)$ structure given by
\begin{align}\label{boundaryclidentification}
    \cl_{\partial M}(W):=-\cl(\partial_{x})\cl(W')
\end{align}
where $W'=W$ for horizontal vectors and $w'=x^{-k}W$ for vertical vectors. The restriction of the connection $\nabla^{\tilde{S}}|_{\partial M}$ defines a $\mathbb{C}l(\partial M)$ connection on $\mathcal{S}^{+}|_{\partial M}$ and we denote the vertical family of connections defined by the restriction to each fibre by $\nabla^{\mathcal{S},y}$. If the base or fibres are spin then $\mathcal{S}^{+}|_{\partial M}\simeq \mathcal{S}_{Y}\otimes\mathcal{S}_{\partial M/Y}$ and $\nabla^{\mathcal{S},y}$ can be identified with the vertical family of spin connections on $\mathcal{S}_{\partial M/Y}$. In particular, $\nabla^{\mathcal{S},y}$ extends to the $\mathbb{C}l(\partial M/Y)$ connection $\nabla^{\mathcal{S}_{\partial M/Y}}$ on $\mathcal{S}_{\partial M/Y}$ described in the previous paragraph (which is different from the connection $\nabla^{\mathcal{S}}|_{\partial M}$). The associated Bismut superconnection $\mathbb{B}$ acts of the left $\operatorname{HOM}(\mathcal{S}_{\partial M/Y})$ factor of $\psi^{\star}\Lambda^{\star}Y\otimes\operatorname{HOM}(\mathcal{S}_{\partial M/Y})\simeq \psi^{\star}\mathbb{C}l(Y)\otimes\operatorname{HOM}(\mathcal{S}_{\partial M/Y})$ which can be identified with the restriction on $\operatorname{HOM}(\mathcal{S}^{+}|_{\partial M})$ to the fibre diagonal. The action of $\mathbb{B}$ is given by \eqref{bismutsc} where vertical Clifford multiplication acts as \eqref{boundaryclidentification}. In particular, we have 
\begin{align}\label{derivativebismut}
    \frac{\partial\mathbb{B}_{t}}{\partial t}=-\cl(\partial_{x})\left(\slashed{\partial}_{Z,y}+\frac{c(T)}{4t}\right)
\end{align} 
and $\mathbb{B}^{2}$ is given by \eqref{bismutlichnerowics} using the connection $\nabla^{\mathcal{S}_{\partial M/Y}}$. Writing $\slashed{\partial M/Y}$ for the family of Dirac operators on $\mathcal{S}_{\partial M/Y}$, under the identification of Clifford bundle on $\partial M$ with the $\mathbb{C}l(M)^{+}$, \eqref{derivativebismut} is identified with
\begin{align}\label{derivativebismut1}
    \slashed{\partial}_{\partial M/Y,y}+\frac{c(T)}{4t}.
\end{align} 

We can now calculate the rescaling of $\tau\slashed{\partial}$ near $\frontface$. The restriction of $\operatorname{HOM}(\mathcal{S})$ to the fibre diagonal can be identified with $\pi^{\star}\mathbb{C}l(Y)\otimes \operatorname{HOM}_{\mathbb{C}l(Y)}(\mathcal{S})$ and that we have identified $\mathcal{S}_{y'}$ with $\mathcal{S}_{y}$ by parallel transport along geodesics in a neighbourhood $U'$ of $y\in Y$. Thus if we consider $\pi^{-1}(y)$ in the fibre diagonal, it has a neighbourhood in $M\times M$ and thus and its preimage on $M_{\text{heat}^{2}}$ on which we can identify $\operatorname{HOM}(\mathcal{S})$ with $\pi^{\star}\mathbb{C}l(Y)_{y}\otimes \operatorname{HOM}_{\mathbb{C}l(Y)|_{y}}(\mathcal{S}|_{y})\simeq \pi^{\star}\mathbb{C}l(Y)_{y}\otimes \mathbb{C}l(Z_{y})$. Thus in this neighbourhood we can define a horizontal, vertical and total symbol map which maps the factors $\pi^{\star}\mathbb{C}l(Y)_{y}\to\pi^{\star}(\Lambda Y)_{y}$ and similarly for the vertical factor. The lift of the connection by the left projection acts on the left factor of $\operatorname{HOM}(S)$ with its action after these identification given by the above Lemmas. Under the symbol map, left Clifford multiplication by $W$ acts by $\epsilon(W)-\iota(W)$.

Using this together with the expressions for the connection in Lemma \ref{6.2} we get
\begin{lemma}\label{diracsquarelift1}
    Let $M$ be an incomplete cusp edge space with spin structure with the base/fibres of $\partial M$ with compatible spin structures and $\partial_{x},\tilde{U}_{\alpha},V_{j}$ a local orthonormal frame. Then given a local section of the spinor bundle covering this frame and locally identifying sections with maps to the spinor space, the action of the lift of $\tau^{2}\slashed{\partial}^{2}$ to $M_{\text{heat}}^{2}$ at $y^{i}=0$ under the horizontal symbol map and rescaling is given by
    \begin{align}\label{nonisolatedrescaledoperator}
        \begin{split}
            \delta_{(x')}(\tau^{2}\slashed{\partial}^{2})\delta_{(x')}^{-1}&=-\tilde{T}^{2}\partial_{\tilde{s}}^{2}-\tilde{T}^{2}\mathcal{H}+\tilde{T}^{2}\mathbb{B}^{2} \\
             &\quad -x^{k-1}\tilde{T}^{2}k\operatorname{cl}(\partial_{x})\left(\slashed{\partial}_{Z,y}-\frac{1}{8}\sum_{i\alpha\beta}\operatorname{cl}(V_{i})g_{\partial M/Y}(\mathcal{R}^{\phi}(\tilde{U}_{\alpha},\tilde{U}_{\beta}),V)(0,z)\epsilon(\tilde{U}_{\alpha})\epsilon(\tilde{U}_{\beta})\right) \\
             &\quad +(x')^{k-1}\tilde{s}E+O((x')^{k})
        \end{split}     
    \end{align}
    where $\mathbb{B}$ is the Bismut superconnection associated to the $\mathbb{C}l(\partial M/Y)$ connection $\nabla^{\mathcal{S}_{\partial M/Y}}$ on $\mathcal{S}_{\partial M/Y}$ and $E$ is a first order differential operator on $\mathcal{S}_{y}$ restricted to the fibre over $y^{i}=0$ in $\frontface$ and
    \begin{align}
        \mathcal{H}=\sum_{\gamma}\left(\partial_{\tilde{\eta}^{\gamma}}-\frac{1}{8}\sum_{jkl}g(R(\partial_{x^{\gamma}},\tilde{\eta}^{\gamma}\partial_{y^{\alpha}})_{p}\partial_{y^{\alpha}},\partial_{y^{\beta}})\epsilon(\partial_{y^{\alpha}})\epsilon(\partial_{y^{\beta}})\right)^{2}.
    \end{align}
\end{lemma}

\begin{proof}
    From lemma \ref{6.2} we see that the rescaled connection (everything below is after conjugation by the horizontal symbol map) is given by
    \begin{align}
        \begin{split}
         \delta_{(x')}(\tau\nabla_{\partial_{x}})\delta_{(x')}^{-1}&=\tilde{T}\partial_{\tilde{s}} \\
          \delta_{(x')}(\tau\nabla_{\tilde{U}_{\gamma}})\delta_{(x')}^{-1}&=\tilde{T}\sigma(\tilde{U}_{\gamma})((x')^{k}\tilde{\eta})-\frac{\tilde{T}}{8}\sum_{\alpha\beta l}g_{Y}(R_{Y}(\tilde{U}_{\gamma},\eta^{l}\partial_{\eta^{l}}),\tilde{U}_{\alpha},\tilde{U}_{\beta})((x')^{k}\eta)\epsilon(\tilde{U}_{\alpha})\epsilon(\tilde{U}_{\beta}) \\
          &\quad +O((x')^{k}) \\
          \delta_{(x')}(\tau\nabla_{x^{-k}V_{i}})\delta_{(x')}^{-1}&=\tilde{T}\nabla_{V_{i}}^{\mathcal{S},y}+\frac{\tilde{T}}{2}\sum_{j\alpha}\phi_{Y}^{\star}g_{Y}(\mathcal{S}^{\phi}(V_{i},V_{j}),\tilde{U}_{\alpha})(y,z)\operatorname{cl}(x^{-k}V_{j})\epsilon(\tilde{U}_{\alpha})\\
        &\quad -\frac{\tilde{T}}{8}\sum_{\alpha\beta}g_{\partial M/Y}(\mathcal{R}^{\phi}(\tilde{U}_{\alpha},\tilde{U}_{\beta}),V_{i})(y,z)\epsilon(\tilde{U}_{\alpha})\epsilon(\tilde{U}_{\beta})+O((x')^{k-1}). 
        \end{split}
    \end{align}
    The first two terms are the same as in the case of an isolated cusp. The last term is exactly the restriction to the fibre of the connection used to define the Bismut superconnection associated to the $\mathbb{C}l(\partial M/Y)$ connection $\nabla^{\mathcal{S}_{\partial M/Y}}$ so the contribution of these term to the rough Laplacian term in the Lichnerowicz formula is exactly the Laplacian term in the Lichnerowicz formula for the square of the Bismut superconnection. Together with the scalar curvature contribution from the previous lemma, we get the square of the Bismut superconnection $\tilde{T}^{2}\mathbb{B}^{2}$.

    For the higher order term, first we see that the $O((x')^{k-1})$ terms come from the vertical derivative terms in the rough Laplacian. We look at the lift of the Dirac operator
    \begin{align}  \tau\slashed{\partial}=&\cl(\partial_{x})\nabla_{\tau\partial_{x}}+\sum_{j}\cl\tilde{U}_{j})\nabla_{\tilde{U}_{j}}+\sum_{i}\cl(x^{-k}V_{i})\nabla_{x^{-k}V_{i}}
    \end{align}
    where each of the terms lifts at $\frontface$ to
    {\allowdisplaybreaks
    \begin{align*}
            &\cl(\partial_{x})\nabla_{\tau\partial_{x}}=\tilde{T}\operatorname{cl}(\partial_{x})\partial_{\tilde{s}} \\
            \sum_{j}&\cl(\tilde{U}_{j})\nabla_{\tilde{U}_{j}}=\tilde{T}\sum_{i}\operatorname{cl}(\tilde{U}_{i})(\sigma(\tilde{U}_{i})((x')^{k}\tilde{\eta})+(x')^{k}V_{i}')\\
        &+\frac{\tilde{T}}{4}(x')^{k}\sum_{ijl}\operatorname{cl}(\tilde{U}_{i})\left[g_{\partial M/Y}([\tilde{U}_{i},V_{j}],V_{k})-g_{Y}(\mathcal{S}^{\phi}(V_{j},V_{l}),\tilde{U}_{i})\right]((x')^{k}\tilde{\eta},z)\operatorname{cl}(x^{-k}V_{j})\operatorname{cl}(x^{-k}V_{l}) \\
        &+\frac{\tilde{T}}{4}(x')^{2k}[(x')^{k-1}\tilde{s}+1]^{k}\sum_{i\alpha j}\epsilon(\tilde{U}_{i})[g_{\partial M/Y}(\mathcal{R}^{\phi}(\tilde{U}_{i},\tilde{U}_{\alpha}),V_{j})(y^{i},z^{j})]((x')^{k}\tilde{\eta},z)\operatorname{cl}(\tilde{U}_{\alpha})\operatorname{cl}(x^{-k}V_{j}) \\
        &-\frac{\tilde{T}}{8}(x')^{2k}\sum_{i\alpha\beta\gamma}\operatorname{cl}(\tilde{U}_{i})g_{Y}(R_{Y}(\tilde{U}_{i},\eta^{\gamma}\partial_{y^{\gamma}})\tilde{U}_{\alpha},\tilde{U}_{\beta})((x')^{k}\tilde{\eta})\operatorname{cl}(\tilde{U}_{\alpha})\operatorname{cl}(\tilde{U}_{\beta}) \\        &+\sum_{\alpha\beta\gamma}O((x')^{3k})\cl(\tilde{U}_{\alpha})\cl(\tilde{U}_{\beta})\cl(\tilde{U}_{\gamma})+\sum_{j\alpha\beta}O((x')^{2k})\cl(\tilde{U}_{\alpha})\cl(\tilde{U}_{\beta})\cl(x^{-k}V_{i}^{y}) \\      
        \sum_{i}&\cl(x^{-k}V_{i})\nabla_{x^{-k}V_{i}}=\tilde{T}[(x')^{k-1}\tilde{s}+1]^{-k}\slashed{\partial}_{Z,y'}\\
        &+\tilde{T}(x')^{k-1}[(x')^{k-1}\tilde{s}+1]\frac{kf}{2}\operatorname{cl}(\partial_{x})-\frac{T}{2}(x')^{k}\sum_{\alpha}k(\tilde{U}_{\alpha})\operatorname{cl}(\tilde{U}_{\alpha})  \\
        &-\frac{\tilde{T}}{8}(x')^{2k}[(x')^{k-1}\tilde{s}+1]^{k}\sum_{i\alpha\beta}\operatorname{cl}(V_{i})g_{\partial M/Y}(\mathcal{R}^{\phi}(\tilde{U}_{\alpha},\tilde{U}_{\beta}),V)(y,z)\operatorname{cl}(\tilde{U}_{\alpha})\operatorname{cl}(\tilde{U}_{\beta}) \\
        &+\sum_{ij\alpha}O((x')^{2k})\cl(\tilde{U}_{\alpha})\cl(x^{-k}V_{i}^{y})\cl(x^{-k}V_{j}^{y})+\sum_{j\alpha\beta}O((x')^{3k})\cl(\tilde{U}_{\alpha})\cl(\tilde{U}_{\beta})\cl(x^{-k}V_{i}^{y})
    \end{align*}
    }
    where each of the error terms are smooth with expansion in $(x')^{k}$. We write $E_{1}, E_{2}, E_{3}$ for the first three error terms above, the final one being absorbed into $E_{2}$.
    
     First, expanding each of the coefficients at $x'=0$ produces higher order terms which are of the form of the error terms $E_{i}$ except for the term in line 4 since this term has an $(x')^{k-1}$ factor while the error terms are smooth expansions in $(x')^{k}$. Thus we may set $(x')^{k}\tilde{\eta}=0$ for all these coefficients and absorb all the higher order terms into the errors. For the term in line 4, expanding produces an error with leading terms $(x')^{3k}A+(x')^{3k+k-1}\tilde{s}B+O((x')^{4k}$ for some endomorphisms $A$ and $B$. Thus we have the following expression for the lift of the Dirac operator.
     {\allowdisplaybreaks
    \begin{align*}
            \tau\slashed{\partial}=&\tilde{T}\operatorname{cl}(\partial_{x})\partial_{\tilde{s}}+\tilde{T}\sum_{i}\operatorname{cl}(\tilde{U}_{i})(\sigma(\tilde{U}_{i})((x')^{k}\tilde{\eta})\rVert_{(x')^{k}\tilde{\eta}=0}+(x')^{k}V_{i}')\\
        &+\frac{\tilde{T}}{4}(x')^{k}\sum_{ijl}\operatorname{cl}(\tilde{U}_{i})\biggl[g_{\partial M/Y}([\tilde{U}_{i},V_{j}],V_{l}) \\
        &\qquad\qquad-g_{Y}(\mathcal{S}^{\phi}(V_{j},V_{l}),\tilde{U}_{i})\biggr]((x')^{k}\tilde{\eta},z)\rvert_{(x')^{k}\tilde{\eta}=0}\operatorname{cl}(x^{-k}V_{j})\operatorname{cl}(x^{-k}V_{l}) \\
        &+\frac{\tilde{T}}{4}(x')^{2k}[(x')^{k-1}\tilde{s}+1]^{k}\sum_{i\alpha j}\cl(\tilde{U}_{i})g_{\partial M/Y}(\mathcal{R}^{\phi}(\tilde{U}_{i},\tilde{U}_{\alpha}),V_{j})((x')^{k}\tilde{\eta},z)\operatorname{cl}(\tilde{U}_{\alpha})\operatorname{cl}(x^{-k}V_{j}) \\
        &-\frac{\tilde{T}}{8}(x')^{2k}\sum_{i\alpha\beta\gamma}\operatorname{cl}(\tilde{U}_{i})g_{Y}(R_{Y}(\tilde{U}_{i},\eta^{\gamma}\partial_{\tilde{\eta}^{\gamma}}),\tilde{U}_{\alpha},\tilde{U}_{\beta})((x')^{k}\tilde{\eta})\rVert_{(x')^{k}\tilde{\eta}=0}\operatorname{cl}(\tilde{U}_{\alpha})\operatorname{cl}(\tilde{U}_{\beta}) \\
        &+\tilde{T}[(x')^{k-1}\tilde{s}+1]^{-k}\slashed{\partial}_{Z,y'}+\tilde{T}(x')^{k-1}[(x')^{k-1}\tilde{s}+1]\frac{kf}{2}\operatorname{cl}(\partial_{x})-\frac{T}{2}(x')^{k}\sum_{\alpha}k(\tilde{U}_{\alpha})\operatorname{cl}(\tilde{U}_{\alpha})  \\
        &-\frac{\tilde{T}}{8}(x')^{2k}[(x')^{k-1}\tilde{s}+1]^{k}\sum_{i\alpha\beta}\operatorname{cl}(V_{i})g_{\partial M/Y}(\mathcal{R}^{\phi}(\tilde{U}_{\alpha},\tilde{U}_{\beta}),V)(y,z)\operatorname{cl}(\tilde{U}_{\alpha})\operatorname{cl}(\tilde{U}_{\beta}) \\
        &+\sum_{\alpha\beta\gamma}O((x')^{3k})\cl(\tilde{U}_{\alpha})\cl(\tilde{U}_{\beta})\cl(\tilde{U}_{\gamma})+\sum_{j\alpha\beta}O((x')^{2k})\cl(\tilde{U}_{\alpha})\cl(\tilde{U}_{\beta})\cl(x^{-k}V_{i}^{y}) \\
        &+\sum_{ij\alpha}O((x')^{2k})\cl(\tilde{U}_{\alpha})\cl(x^{-k}V_{i}^{y})\cl(x^{-k}V_{j}^{y}).
    \end{align*}
    }
      We only need to look at the terms in square of the above expression whose rescaling will have coefficient $(x')^{k-1}$. Using the fact that $[(x')^{k-1}\tilde{s}+1]^{\pm k}=1+(x')^{k-1}\tilde{s}+O((x')^{2(k-1)})$ we see that these terms come from the terms in lines 3, 5 and 6 where we have used.
    
    In the square, we see the terms containing factors from these three lines come in three types, first we have pairs which cancel as they are the product of two terms which anticommute, which include the terms where $\tilde{T}\operatorname{cl}(\partial_{x})\partial_{\tilde{s}}$ does not act on an $\tilde{s}$ as the $\operatorname{cl}(\partial_{x})$ anticommutes with all the endomorphisms which appear here. This also includes any terms which come from the second term in the line 5 since it also anticommutes with every other endomorphism which appears here.

    The second are those which come from the first term $\tilde{T}\operatorname{cl}(\partial_{x})\partial_{\tilde{s}}$ where the $\partial_{\tilde{s}}$ acts on an $\tilde{s}$. There are three contributions to this type from the third line which produces $\frac{1}{4}$ times the curvature term in line 2 of \eqref{nonisolatedrescaledoperator} while the $(x')^{3k+k-1}$ error term becomes an $O((x')^{k})$ error after rescaling. Next we have contributions from the first term line 5 which give the $\slashed{\partial}_{Z,y'}$ term in \eqref{nonisolatedrescaledoperator} while the second term in line 5 is $O((x')^{2k-2})$ after rescaling. Finally, line 6 contributes $-\frac{1}{8}$ times the curvature in line 2 of \eqref{nonisolatedrescaledoperator} which together with the contribution from line 3 gives exactly the $\frac{1}{8}$ that we have in line 2 of \eqref{nonisolatedrescaledoperator}.
    
    Finally, are those of the form $(x')^{k-1}\tilde{s}B$ for some endomorphism $B$ which include every other term not of the first or second type.
    \end{proof}

Now we calculate the $t=0$ contribution to the supertrace from $\frontface$.

\begin{lemma}
    On the fibre above $y$ in $\frontface$ the rescaled heat kernel, under the horizontal symbol map has asymptotic expansion
    \begin{align}
        \delta_{(x')}\sigma_{B}H=u_{-kn}(x')^{-kn}+u_{-kn+k-1}(x')^{-kn+k-1}+O((x')^{-kn+k})
    \end{align}
    where
    \begin{align}\label{rescaledffasymp1}
        \begin{split}
            u_{-kn}&=h(\tilde{T}^{2},\tilde{s},0)K_{\mathcal{H}}(\tilde{T}^{2},\tilde{\eta})K_{\mathbb{B}^{2}}(\tilde{T}^{2},z,z') \\
             u_{-kn+k-1}&=k\tilde{T}^{2}h(\tilde{T}^{2},\tilde{s},0)K_{\mathcal{H}}(\tilde{T}^{2},\tilde{\eta},0)\cl(\partial_{x})\left(\slashed{\partial}_{Z,y}+\frac{c(T)}{4}\right)K_{\mathbb{B}^{2}}(\tilde{T}^{2},z,z') \\
             &\quad +\int_{0}^{t}\int_{Z_{y'}}\int_{\mathbb{R}^{b}_{\tilde{\eta}}}\int_{-\infty}^{\infty}\biggl[h(t-t',x,x')K_{\mathcal{H}}(t-t',\tilde{\eta},\tilde{\eta}')K_{\mathbb{B}^{2}}(t-t',z,\zeta), \\
            &\quad L_{k-1}\mathcal{N}_{-kn}(\delta_{\tau}H)(t^{'\frac{1}{2}},x',\tilde{\eta}',\zeta,z')\biggr]dx'd\tilde{\eta}'d\zeta dt' +E 
        \end{split}
    \end{align}
    where $h(t,x,x')$ is the Euclidean heat kernel, $K_{\mathcal{H}},K_{\mathbb{B}^{2}}$ are the heat kernels of $\mathcal{H}$ and $\mathbb{B}^{2}$ respectively and $E$ vanishes at $\tilde{s}=0$ and
    \begin{align}
        L_{k-1}=k\operatorname{cl}(\partial_{x})\left(\slashed{\partial}_{Z,y}-\frac{1}{8}\sum_{i\alpha\beta}\operatorname{cl}(V_{i})g_{\partial M/Y}(\mathcal{R}^{\phi}(\tilde{U}_{\alpha},\tilde{U}_{\beta}),V)(0,z)\operatorname{cl}(\tilde{U}_{\alpha})\operatorname{cl}(\tilde{U}_{\beta})+\tilde{s}E\right).
    \end{align}
\end{lemma}

\begin{proof}
The lift of $t\partial_{t}$ to $\frontface$ is $\frac{1}{2}\tilde{T}\partial_{\tilde{T}}$ which is unchanged after rescaling so the normal operator of the rescaled heat kernel satisfies
\begin{align}
            \left(\frac{1}{2}\tilde{T}\partial_{\tilde{T}}-\tilde{T}^{2}\partial_{\tilde{s}}^{2}-\tilde{T}^{2}\mathcal{H}^{2}+\tilde{T}^{2}\mathbb{B}^{2}\right)\mathcal{N}_{-kn}(\delta_{(x')}H)=0
         \end{align}
         where $\mathcal{N}_{-kn}(\delta_{\tau}H)$ vanishes to infinite order as $\tilde{\eta},\tilde{s}\to\infty$ with all $\tilde{\eta}_{i},\tilde{s}$ derivatives and satisfies the initial condition (this is the initial condition for $H$ under the rescaling, there are no terms at order lower than $-kn$ since initial condition for the coefficients below $-kn$ have $0$ on the RHS)
         \begin{align}
             \int_{\beta^{-1}(p)} \left[\mathcal{N}_{-kn}(\delta_{\tau}H)\right]_{0}d\tilde{s}d\tilde{\eta}=\delta(z-z')\otimes\operatorname{Id}_{\mathcal{S}}.
         \end{align}
         This has a unique solution which vanishes to infinite order as $\tilde{s},\tilde{\eta}\to\infty$ given by
         \begin{align}
             \mathcal{N}_{-kn}(\delta_{\tau}H)(\tilde{T},\tilde{s},\tilde{\eta},z,z')=h(\tilde{T}^{2},\tilde{s},0)K_{\mathcal{H}}(\tilde{T}^{2},\tilde{\eta})K_{\mathbb{B}^{2}}(\tilde{T}^{2},z,z').
         \end{align}
         This follows from uniqueness of solutions of the heat equation for $\mathbb{B}^{2}$ on $Z_{y'}$ and $\mathcal{H}^{2}$ on $L^{2}(\mathbb{R}^{b}_{\tilde{\eta}})$ and the fact that these two operators commute. In particular, the kernels themselves commute with each other and $\mathbb{B}^{2}$ and $\mathcal{H}^{2}$.

         From Lemma \ref{diracsquarelift1}, we see the next non-zero term is at order $-kn+k+1$ which is given by the solution to the equation
        \begin{align}\label{rescaledffequationk-1}
            \begin{split}
             \biggl(&\frac{1}{2}\tilde{T}\partial_{\tilde{T}}-\tilde{T}^{2}\partial_{\tilde{s}}^{2}-\tilde{T}^{2}\mathcal{H}^{2}+\tilde{T}^{2}\mathbb{B}^{2}\biggr)\mathcal{N}_{-knk+1}(\delta_{(x')}H) \\
             &=\tilde{T}^{2}k\operatorname{cl}(\partial_{x})\left(\slashed{\partial}_{Z,y}-\frac{1}{8}\sum_{i\alpha\beta}\operatorname{cl}(V_{i})g_{\partial M/Y}(\mathcal{R}^{\phi}(\tilde{U}_{\alpha},\tilde{U}_{\beta}),V)(0,z)\operatorname{cl}(\tilde{U}_{\alpha})\operatorname{cl}(\tilde{U}_{\beta})+\tilde{s}E\right)\mathcal{N}_{-kn}(\delta_{\tau}H)
             \end{split}
         \end{align}
        with vanishing initial condition. The second term in brackets on the second line is exactly $\frac{c(T)}{4}$. Following the case of an isolated cusp, by Duhamel's principle we have
        \begin{align}
            w(t,x,\tilde{\eta}',z,z')=\int_{0}^{t}\int_{Z_{y'}}\int_{\mathbb{R}^{b}_{\tilde{\eta}}}\int_{-\infty}^{\infty}&h(t-t',x,x')K_{\mathcal{H}}(t-t',\tilde{\eta},\tilde{\eta}')K_{\mathbb{B}^{2}}(t-t',z,\zeta) \\
            &L_{k-1}\mathcal{N}_{-kn}(\delta_{\tau}H)(t^{'\frac{1}{2}},x',\tilde{\eta}',\zeta,z')dt'dx'd\tilde{\eta}'d\zeta
        \end{align}
        where $w(\tilde{T}^{2})$ solves equation \eqref{rescaledffequationk-1}. As in the isolated case, the term arising $\tilde{s}E$ vanishes at $x=0$.

        Commuting the $L_{k-1}\mathcal{N}_{-kn}(\delta_{\tau}H)$ to the front of the integrand produces the commutator term in \eqref{rescaledffasymp1} plus $h(t,x,x')$ times the kernel of the operator
        \begin{align}
            \int_{0}^{t}L_{k-1}e^{t'\mathcal{H}}e^{t'\mathbb{B}^{2}}e^{(t-t')\mathcal{H}}e^{(t-t')\mathbb{B}^{2}}dt'=tL_{k-1}e^{t\mathcal{H}}e^{t\mathbb{B}^{2}}
        \end{align}
        since $\mathcal{H}$ and $\mathbb{B}^{2}$ commute. Since $L_{k-1}$ also commutes with $\mathcal{H}$ we get the first term in \eqref{rescaledffasymp1}.
\end{proof}

\begin{proposition}
    The pointwise supertrace $x^{kf+1}\operatorname{str}(H|_{\operatorname{diag}_{H}})$ is smooth up to $\frontface_{h}\subset \operatorname{diag}_{h}$ and its restriction to $\frontface_{h}$ is given by
    \begin{align}\label{leadingsupertraceff}
        \operatorname{str}(H)&x^{kf+1}\operatorname{dVol}_{Y}\operatorname{dVol}_{Z} \\
        &=-2k\tilde{T}\left[a(2\pi i)^{-\lceil\frac{b}{2}\rceil}\tilde{T}^{-b}\mathcal{A}(\tilde{T}^{2}R_{B})\operatorname{str}_{\mathcal{S}_{\partial M/Y}}\left(\left(\slashed{\partial}_{\partial M/Y,y}+\frac{c(T)}{4}\right)e^{-\tilde{T}^{2}\mathbb{B}^{2}}\right)\right]_{b}\operatorname{dVol}_{Z}
    \end{align}
    where $a=\frac{1}{2},\frac{1}{2\sqrt{\pi}}$ is $\operatorname{dim}(Z/Y)$ is even/odd and $\operatorname{str}_{\mathcal{S}_{\partial M/Y}}$ denotes the pointwise supertrace/trace on $\mathcal{S}_{\partial M/Y}$ even/odd.
\end{proposition}

\begin{proof}
    First we consider the case that $\operatorname{dim}(Z/Y)$ is odd. We recall that the pointwise supertrace on $\mathcal{S}$ on all endomorphisms not of full Clifford degree and for     
    \begin{align}
        \begin{split}
            \gamma&=\cl(\partial_{x})\cl(\tilde{U_{1}})\ldots\cl(\tilde{U}_{b})\cl(x^{-k}V_{1})\ldots\cl(x^{-k}V_{f}) \\
            \gamma_{\partial M}&=\cl(\tilde{U_{1}})\ldots\cl(\tilde{U}_{b})\cl(x^{-k}V_{1})\ldots\cl(x^{-k}V_{f}) 
        \end{split}
    \end{align}
    where $\partial_{x},\tilde{U}_{\alpha},x^{-k}V_{i}$ is a an oriented orthonormal frame and restricted to $\mathcal{S}|_{\partial M}$ we have from \eqref{supertraceseplitoddfibre}
    \begin{align}
        \operatorname{str}(\gamma)=-2(-2 i)^{\frac{b}{2}}T_{B}(\operatorname{tr}_{\mathcal{S}_{\partial M/Y}}(\sigma_{B}\gamma_{\partial M})).
    \end{align}
    Thus by the previous Lemma, $H$ has asymptotic expansion
    \begin{align}
        H=\sum_{i=0}^{b} (x')^{-kn+ki}\alpha_{i}+\sum_{i=0}^{b}\alpha_{i} (x')^{-kn+ki+k-1}\beta_{i}+O((x'))^{-kn+kb+k}
    \end{align}
    such that $\alpha_{i},\beta_{i}$ is a section of $\psi^{\star}\mathbb{C}l^{i}(Y)\otimes\operatorname{HOM}(\mathcal{S}_{\partial M/Y})$ and
    \begin{align}
        \begin{split}
            u_{-kn}&=\sum_{i=0}^{b}\sigma_{B}^{i}(\alpha_{i}) \\
            u_{-kn+k-1}&=\sum_{i=0}^{b}\sigma_{B}^{i}(\beta_{i}).
        \end{split}
    \end{align}
    In particular, restricted to the diagonal, the $\alpha_{i}$ do not have full Clifford degree, as they do not have a factor of $\cl(\partial_{x})$ hence the supertrace vanishes on these terms. We see that the first non-vanishing term with comes from the $\beta_{b}$ term so the leading order asymptotics of the pointwise supertrace of the restriction of $H$ to the diagonal is given by
    \begin{align}
        \begin{split}
            \operatorname{str}(H|_{\diag_{h}})&=(x')^{-kn+kb+k-1}\operatorname{str}(\beta_{b}|_{z=z',\tilde{s}=0,\tilde{\eta}=0})+O((x'))^{-kn+kb+k} \\
            &=(x')^{-kf-1}(-2(-2i)^{\frac{b}{2}}T_{B}(\operatorname{tr}_{\mathcal{S}_{\partial M/Y}}(j(u_{-kn-k-1}|_{z=z',\tilde{s}=0,\tilde{\eta}=0}))))+O((x'))^{-kn+kb+k}
        \end{split}
    \end{align}
    where $j$ is the map taking $\cl(\partial_{x})\cl(x^{-k}V_{i}n)\ldots\cl(x^{-k}V_{f})$ to $\cl(x^{-k}V_{i}n)\ldots\cl(x^{-k}V_{f})$ and $0$ on anything of lower vertical Clifford degree.
    
    The coefficient simplifies to
    \begin{align}
        \begin{split}
            -2&(-2i)^{\frac{b}{2}}T_{B}(\operatorname{tr}_{\mathcal{S}_{\partial M/Y}}(u_{-kn-k-1}|_{z=z',\tilde{s}=0,\tilde{\eta}=0}))= \\
            & -2(-2i)^{\frac{b}{2}}k\tilde{T}^{2}h(\tilde{T}^{2},0,0)T_{B}(\operatorname{tr}_{\mathcal{S}_{\partial M/Y}}(K_{\mathcal{H}}(\tilde{T}^{2},\tilde{\eta},0)j(\cl(\partial_{x})\left(\slashed{\partial}_{Z,y}+\frac{c(T)}{4}\right)K_{\mathbb{B}^{2}}(\tilde{T}^{2},z,z))) \\
             &\quad +T_{B}\circ\operatorname{tr}_{\mathcal{S}_{\partial M/Y}}(\int_{0}^{t}\int_{Z_{y'}}\int_{\mathbb{R}^{b}_{\tilde{\eta}}}\int_{-\infty}^{\infty}j(\biggl[h(t-t',0,x')K_{\mathcal{H}}(t-t',0,\tilde{\eta}')K_{\mathbb{B}^{2}}(t-t',z,\zeta), \\
            &\quad L_{k-1}\mathcal{N}_{-kn}(\delta_{\tau}H)(t^{'\frac{1}{2}},x',\tilde{\eta}',\zeta,z)\biggr])dx'd\tilde{\eta}'d\zeta dt' ) \\
            &=-2(-2i)^{\frac{b}{2}}k\tilde{T}^{2}(4\pi\tilde{T}^{2})^{-\frac{1}{2}}T_{B}((4\pi\tilde{T}^{2})^{-\frac{b}{2}}\hat{\mathcal{A}}(\tilde{T}^{2}R_{B})\operatorname{tr}_{\mathcal{S}_{\partial M/Y}}\left(\left(\slashed{\partial}_{Z,y}+\frac{c(T)}{4}\right)K_{\mathbb{B}^{2}}(\tilde{T}^{2},z,z)\right)).
        \end{split}
    \end{align}
    The second term with the integral vanishes interchanging the $\operatorname{tr}$ with the integral since the first term in the commutator has even horizontal form degree so is the pointwise trace of a commutator. For the first term, since $K_{\mathcal{H}}$ has even horizontal form degree, we can move the $\cl(\partial_{x})$ past it without changing sign. Also, by our convention identifying the boundary Clifford bundle with the positive part of the Clifford bundle on $M$, for $A$ so that $j(A)\neq 0$ we have that $j(A)$ acts as $-\cl(\partial_{x})A$. As described before \eqref{derivativebismut1}, the same identification identifies 
    \begin{align}
        -\cl(\partial_{x})\left(\slashed{\partial}_{Z,y}+\frac{c(T)}{4}\right)K_{\mathbb{B}^{2}}(\tilde{T}^{2},z,z).
    \end{align}
    acting on $\psi^{\star}\Lambda^{\star}Y\otimes\operatorname{End}(\mathcal{S}_{\partial M/Y})$ with
    \begin{align}
        \left(\slashed{\partial}_{\partial M/Y,y}+\frac{c(T)}{4}\right)K_{\mathbb{B}^{2}}(\tilde{T}^{2},z,z).
    \end{align}
    Since $T_{B}(a\operatorname{dVol}_{Y})=a$ this gives \eqref{leadingsupertraceff}.
    
    The only difference for the $\operatorname{dim}(Z/Y)$ even case is
    \begin{align}
         \operatorname{str}(\gamma)=(-2 i)^{\frac{b+1}{2}}T_{B}(\operatorname{str}_{\mathcal{S}_{\partial M/Y}}(\sigma_{B}\gamma_{\partial M})).
    \end{align}
    Following through the same steps, leads to the different coefficient in \eqref{leadingsupertraceff}.
\end{proof}

\begin{theorem}\label{nonisolatedspindiract0limit}
    For a spin ice manifold $(M,g)$ with $g$ product type near the boundary such that $Z\to\partial M\to Y$ has a spin structure on the base and fibre compatible with that on $M$, the $t\to 0$ limit of the supertrace of the heat kernel of $\slashed{\partial}^{2}$ is given by
    \begin{align}\label{nonisolatedt0str}
        \lim_{t\to 0}\operatorname{str}(e^{-t\slashed{\partial}^{2}})=\int_{M}\mathcal{A}(M)-\int_{Y}\mathcal{A}(Y)\tilde{\eta}(\slashed{\partial}_{\partial M/Y})
    \end{align}
    where $\tilde{\eta}(\slashed{\partial}_{Z})$ is the Bismut-Cheeger eta form for the family (make explicit somewhere) on $\partial M$.
\end{theorem}

\begin{proof}
    The contribution from the interior of $M$ is the same as in the isolated case. By Lemma \ref{lemma 3.3}, the contribution from $\frontface$ is
    \begin{align}
        \frac{1}{k}\dashint_{0}^{\infty}\int_{\partial M}u_{k(b+1),\operatorname{ff}^{h}}(\tilde{T},y,z)\operatorname{dVol_{Y}}\operatorname{dVol}_{Z_{y}}\frac{d\tilde{T}}{\tilde{T}}
    \end{align}
    where the integrand is given by the previous proposition. Taking $t=T^{2}$, this contribution becomes
    \begin{align}
        -2\dashint_{0}^{\infty}\int_{\partial M}\left[a(2\pi i)^{-\lceil\frac{b}{2}\rceil}t^{-\frac{b}{2}}\mathcal{A}(tR_{B})\operatorname{str}_{\mathcal{S}_{\partial M/Y}}\left(\left(\slashed{\partial}_{\partial M/Y,y}+\frac{c(T)}{4}\right)e^{-t\mathbb{B}^{2}}\right)\right]_{b}\frac{dt}{2t^{\frac{1}{2}}}.
    \end{align}
    Now note that for any form $\omega$ on $B$ we have
    \begin{align}
        \left[(2\pi i t)^{-\frac{b}{2}}\mathcal{A}(tR_{N})\omega\right]_{b}=\left[\mathcal{A}\left(\frac{R_{B}}{2\pi i}\right)\delta_{2\pi i t}\omega\right]_{b}
    \end{align}
    and that
    \begin{align}
        \delta_{2\pi i t}\left(\left(\slashed{\partial}_{\partial M/Y,y}+\frac{c(T)}{4}\right)e^{-t\mathbb{B}^{2}}\right)=\delta_{2\pi i}\left(\left(\slashed{\partial}_{\partial M/Y,y}+\frac{c(T)}{4t}\right)e^{-\mathbb{B}_{t}^{2}}\right).
    \end{align}
    Thus this $t$ dependent part of the integrand is exactly the eta form integrand so the $t$ integral exists and we arrive at \eqref{nonisolatedt0str}.
\end{proof}

Now we consider the case where $g$ has polyhomogeneous error term. Recall that given an invariant formal power series $P\colon \mathfrak{so}(k)\to \mathbb{C}$, if $R^{E}$ is the curvature of a connection $\nabla^{E}$ then $P$ determines a closed differential form $P(R^{E})$ whose cohomology class is independent of the connection. In particular if $\nabla^{E,0},\nabla^{E,1}$ are connections on $E$ then 
\begin{align}
    P(R^{E,1})-P(R^{E,0})=dTP(\nabla^{E,1},\nabla^{E,0})
\end{align}
where $TP(\nabla^{E,1},\nabla^{E,0})$ is the transgression form defined by
\begin{align}
    TP(\nabla^{E,1},\nabla^{E,0})=\int_{0}^{1}\partial_{s}|_{s=0}P(R^{E,t}+s\theta)dt
\end{align}
where $R^{E,t}$ is the curvature of the connection $t\nabla^{E,1}+(1-t)\nabla^{E,0}$ and $\theta=\nabla^{E,1}-\nabla^{E,0}$.

\begin{lemma}
    Let $\hat{\mathcal{A}}_{0}$ and $\hat{\mathcal{A}}_{1}$ be the $\hat{A}$-forms on $M$ for the metric $g_{0}$ and $g$ then
    \begin{align}
        \int_{M}\hat{\mathcal{A}}_{0}=\int_{M}\hat{\mathcal{A}}_{1}.
    \end{align}
\end{lemma}

\begin{proof}
    The difference $\hat{\mathcal{A}}_{1}-\hat{\mathcal{A}}_{0}$ is given by $dTP(\nabla,\nabla^{0})$ so we have
    \begin{align}
        \int_{M}\hat{\mathcal{A}}_{1}-\int_{M}\hat{\mathcal{A}}_{0}=\int_{\partial M}TP(\nabla,\nabla^{0}).
    \end{align}
    From the asymptotics of the connection the form $\theta$ is polyhomogeneous and $O(x^{k-1})$ at the boundary. In particular, the restriction of $\theta$ to the boundary vanishes. Since the pullback by the inclusion $\iota\colon \partial M\to M$ commutes with integration in $t$ and differentiation in $s$ we have
    \begin{align}
        \begin{split}
        \iota^{\star}TP(\nabla^{E,1},\nabla^{E,0})&=\int_{0}^{1}\partial_{s}|_{s=0}\iota^{\star}P(R^{E,t}+s\theta)dt \\
        &=\int_{0}^{1}\partial_{s}|_{s=0}P(\iota^{\star}R^{E,t}+0)dt=0. \\
        \end{split}
    \end{align}
    \end{proof}

    \begin{theorem}\label{spindiracindex}
        Let $(M,g)$ be a spin incomplete cusp edge space such that the base $Z$ or fibres $Y$ of the boundary fibration $Z\xhookrightarrow{}\partial M\to Y$ admit a spin structure. Then the index of the spin Dirac operator is given by
        \begin{align}
            \operatorname{ind}(\slashed{\partial}^{+})=\int_{M}\hat{\mathcal{A}}(M)-\int_{Y}\hat{\mathcal{A}}(Y)\tilde{\eta}({\slashed{\partial}_{\partial M/Y}}).
        \end{align}
        where $\tilde{\eta}({\slashed{\partial}_{\partial M/Y}})$ is the normalised eta form of the family of Dirac operators on the fibres for any choice of compatible spin structures.
         \end{theorem}

\begin{proof}
    By Lemma \ref{deformtoproducttype}, the index of the $\slashed{\partial}^{+}$ is equal to the index of the spin Dirac operator with the associated product-type metric. By the McKean-Singer formula and Theorem \ref{nonisolatedspindiract0limit}, the index of this operator is given by
        \begin{align}
        \int_{M}\mathcal{A}_{0}(M)-\int_{Y}\mathcal{A}(Y)\tilde{\eta}(\slashed{\partial}_{\partial M/Y})
    \end{align}
    where the subscript $0$ denotes the forms with respect to the product type metric. But the above calculation shows that the first integral is unchanged after including the error term.
\end{proof}

\section{Positive scalar curvature}\label{pscalarcurvature}

We can extend some results from \cite{iedge} on the index of the spin Dirac operator as the obstruction to the existence of positive scalar curvature metrics to spin incomplete cusp edge spaces.

\begin{theorem}
    Let (M,g) be a spin incomplete cusp edge space. Suppose either
    \begin{enumerate}
        \item $\operatorname{dim}(\partial M/Y)\geq 2$, the scalar curvature of $g$ is non-negative in a neighbourhood of the boundary and positive at least one point sufficiently close to the boundary.

        \item $\operatorname{dim}(\partial M/Y)=1$ and the spin structure on $M$ is the lift of a spin structure on the associated space $\tilde{M}$ with fibres collapsed at the boundary (given the smooth structure identifying the family of cones as a family of disks).
    \end{enumerate}
    Then the induced boundary family has trivial kernel and the index of the spin Dirac operator vanishes.
\end{theorem}

\begin{proof}
    From Lemma \ref{scalarcurvature}, in a neighbourhood of the boundary the scalar curvature for any product type metric $g_{0}$ satisfies
    \begin{align}
        x^{2k}S_{g_{0}}=S_{g_{Z_{y}}}+O(x^{k}).
    \end{align}
    Since the difference between $g$ and $g_{0}$ is $O(x^{k})$, so this also holds for $g$. It follows that if the scalar curvature of $g$ is non-negative in a neighbourhood of the boundary then the scalar curvature of the fibres are also non-negative. If the scalar curvature is positive at one point sufficiently close to the boundary then $S_{g_{Z_{y}}}$ is also positive at at least one point. Thus the Dirac operator on this fibre has trivial kernel. Since the fibres are all diffeomorphic, they all admit a non-negative scalar curvature metric which is positive at at least one point so the family of Dirac operators have trivial kernel.

    For the second case, locally around any point in the singular stratum, we can trivialise $\tilde{M}$ as a family of disks locally over the base $Y$. Thus locally, the spinor bundle can be identified as the tensor product of the spinor bundle over the base and the disk. The disk has a single spin structure, the trivial one which restricts to the non-trivial spin structure over its boundary. But the restriction of the spin structure at $x=\epsilon$ can be identified with the restriction of the pullback of the spin structure to $M$ at $x=\epsilon$ which can be identified with its restriction at $x=0$. In particular, locally with this trivialisation lifted to $M$, the restriction of the spinor bundle to the boundary can locally be identified with the tensor product of the spinor bundle of the base and the non-trivial spinor bundle over the circle fibres. It follows that this is also true globally in the base since $Y$ is spin with the spin structure induced from $\tilde{M}$. Since the spin Dirac operator on the non-trivial spinor bundle over the circle has spectrum disjoint front $0$ the result holds in this case.

    In either case, since the boundary family has trivial kernel the spin Dirac operator is essentially self-adjoint. By the Lichernowicz formula for any smooth compactly supported section $\varphi$ of $\mathcal{S}$
    \begin{align}
        \begin{split}
        \langle \slashed{\partial}\varphi,\slashed{\partial}\varphi \rangle&=\langle \slashed{\partial}^{2}\varphi,\varphi \rangle \\
        &=\langle (\nabla^{\mathcal{S}})^{\star}\nabla^{\mathcal{S}}\varphi,\varphi \rangle +\frac{S}{4}\langle \varphi,\varphi \rangle \\
        &=\lVert \nabla^{\mathcal{S}}\varphi\rVert^{2}+\frac{S}{4}\lVert \varphi\rVert^{2}.
        \end{split}
    \end{align}
    So if $\varphi_{n}$ converges in $L^{2}$ and $\slashed{\partial}\varphi_{n}$ is Cauchy in $L^{2}$ then so is $\nabla^{\mathcal{S}}\varphi_{n}$ so this equality holds for all $\varphi\in\mathcal{D}$. If $\varphi\in\ker(\slashed{\partial})$ then $\nabla^{\mathcal{S}}\varphi=0$ and $S\varphi=0$. Since $S$ is non-vanishing at at least one point $\varphi$ vanishes at at least one point but then $\nabla^{\mathcal{S}}\varphi=0$ implies 
    \begin{align}
        d\langle \varphi,\varphi\rangle=2\langle \nabla^{\mathcal{S}}\varphi,\varphi\rangle=0.
    \end{align}
    Hence $\varphi=0$ so the kernel of $\slashed{\partial}$ is trivial thus the index vanishes.
\end{proof}

\section{Index theorems for Clifford modules}

Now we consider the case of more general Clifford modules $E$ with Clifford connection which we take to satisfy assumptions \ref{assumption1},\ref{assumption2} and \ref{assumption3}. The proofs are mostly similar to those for the spin Dirac operator so we will just discuss the differences in this case. To obtain a formula on the non-isolated case, we will also make the following assumption. We denote the restriction of the connection to the boundary by $\nabla^{E,\partial M}$.
\begin{manualassumption}{4}\label{assumption4}
There exists a $\mathbb{C}l(Y)$ module $E_{B}$ and a vertical Clifford module $E_{Z}$ such that $E^{+}|_{\partial M}\simeq \phi_{Y}^{\star}E_{B}\otimes E_{Z}$ as $\mathbb{C}l(\partial M)$ modules with a vertical Clifford connection $\nabla^{E_{Z}}$ on $E_{Z}$ such that its restriction to each fibre is equal to the restriction of the $\nabla^{E}$ to each fibre.
\end{manualassumption}
In particular, the restriction of the boundary connection $\nabla^{E,\partial M}$ to each fibre can be extended to the connection $\nabla^{E_{z}}$ which we can think of as acting on $E_{Z}$ so has a well defined Bismut superconnection $\mathbb{B}_{E_{Z}}$. We will denote the vertical family of Dirac operators on $E$ by $\slashed{\partial}_{E,Z,y}$ and the family of operators on $E_{Z}$ by $\slashed{\partial}_{E,\partial M/Y,y}$. Examples of Clifford modules which satisfy this assumption include the exterior bundle with the Hodge-de Rham operator and twistings of the spinor bundle. 

Note that in general, the restriction of the heat kernel to $\frontface$ can be identified as a section of $\mathbb{C}l(Y)\otimes \operatorname{HOM}_{\mathbb{C}l(Y)}(E)$. Unless we make the above assumption, the second factor can not in general be identified with the endomorphism bundle of vector bundle which is globally defined on the boundary hence there it is not clear whether the Bismut superconnection-like terms which appear can be identified with a globally defined family of vertical Clifford connections in this case.

Using the same setup as for the spin Dirac operator, we can generalise Lemma \ref{6.2} with only minor adjustments. The setup for the rescaling is mostly the same as for the spin Dirac operator. The restriction of $\operatorname{HOM}(E)$ to the fibre diagonal can be identified with $\pi^{\star}\mathbb{C}l(Y)\otimes \operatorname{HOM}_{\mathbb{C}l(Y)}(E)$ and everything else can be repeated replacing $\mathcal{S}$ with $E$. We obtain the following analogue of Lemma \ref{diracsquarelift1}.

\begin{lemma}\label{diracsquarelift}
    The action of the lift of $\tau^{2}(\slashed{\partial}^{E})^{2}$ to $M_{\text{heat}}^{2}$ at $y^{i}=0$ under the horizontal symbol map and rescaling is given by
    \begin{align}\label{nonisolatedrescaledoperatortwist}
        \begin{split}
            \delta_{(x')}(\tau^{2}\slashed{\partial}^{2})&\delta_{(x')}^{-1}=-\tilde{T}^{2}\partial_{\tilde{s}}^{2}-\tilde{T}^{2}\mathcal{H}^{2}+\tilde{T}^{2}K'_{B}+\tilde{T}^{2}\mathbb{B}_{E_{Z}}^{2} \\
             &\quad -x^{k-1}\tilde{T}^{2}k\operatorname{cl}(\partial_{x})\left(\slashed{\partial}_{E,Z,y}-\frac{1}{8}\sum_{i\alpha\beta}\operatorname{cl}(V_{i})g_{\partial M/Y}(\mathcal{R}^{\phi}(\tilde{U}_{\alpha},\tilde{U}_{\beta}),V)(0,z)\epsilon(\tilde{U}_{\alpha})\epsilon(\tilde{U}_{\beta})\right) \\
             &\quad +(x')^{k-1}\tilde{T}^{2}\left[\operatorname{cl}(\partial_{x})\slashed{\partial}_{E,Z,y},\rho(0)\right]+(x')^{k-1}\tilde{s}E+O((x')^{k})
        \end{split}     
    \end{align}
    where $\mathbb{B}^{E_{Z}}$ is the Bismut superconnection associated to the $\mathbb{C}l(\partial M/Y)$ connection $\nabla^{E_{Z}}$ on $E_{Z}$ and $E$ is a first order differential operator on $E_{Z,y}$ restricted to the fibre over $y^{i}=0$ in $\frontface$ and
    \begin{align}\label{K'}
        \begin{split}
        \mathcal{H}&=\sum_{\gamma}\left(\partial_{\tilde{\eta}^{\gamma}}-\frac{1}{8}\sum_{jkl}g(R(\partial_{x^{\gamma}},\partial_{y^{\alpha}})_{p}\tilde{\eta}^{\gamma}\partial_{y^{\alpha}},\partial_{y^{\beta}})\epsilon(\partial_{y^{\alpha}})\epsilon(\partial_{y^{\beta}})\right)^{2} \\
        K'_{B}&=\frac{1}{2}\sum_{\alpha\beta}K'_{E}(\partial_{y_{\alpha}},\partial_{y_{\beta}})\epsilon(\partial_{y_{\alpha}})\cl(\partial_{y_{\beta}}).
        \end{split}
    \end{align}
\end{lemma}

\begin{proof}
    The rescaling for the leading order terms except for the twisting connection is the same as for the spin Dirac operator. The twisting connection term is given by
    \begin{align}
        \begin{split}
       \frac{\tau^{2}}{2}\sum_{\alpha\beta}K'_{E}(\tilde{U}_{\alpha},\tilde{U}_{\beta})\cl(\tilde{U}_{\alpha})\cl(\tilde{U}_{\beta})+&\tau^{2}\sum_{i\alpha}K'_{E}(x^{-k}V_{i},\tilde{U}_{\alpha})\cl(x^{-k}V_{i})\cl(\tilde{U}_{\alpha})\\
       &+\frac{\tau^{2}}{2}\sum_{ij}K'_{E}(x^{k}V_{i},x^{-k}V_{j})\cl(x^{-k}V_{i})\cl(x^{-k}V_{j}).       
        \end{split}
    \end{align}
    It is clear that under the horizontal symbol map, lifted to $M^{2}_{\operatorname{heat}}$, this term rescales to give 
    \begin{align}
    \begin{split}
        K'&=\frac{\tilde{T}^{2}}{2}\sum_{ij}K'_{E}(W_{i},W_{j})\epsilon(W_{i})\epsilon(W_{j})+\frac{\tilde{T}^{2}}{2}\sum_{ij}K'_{E}(W_{i},W_{j})\epsilon(W_{i})\cl(W_{j}) \\
        &\quad+\frac{\tilde{T}^{2}}{2}\sum_{ij}K'_{E}(W_{i},W_{j})\cl(W_{i})\cl(W_{j}).
    \end{split}
    \end{align}
    Then using the identification $E\simeq \psi^{\star}E_{B}\otimes E_{Z}$, the twisting curvature can be written as a sum of the twisting curvature from $\nabla^{E_{B}}$ and $\nabla^{E_{Z}}$. The contribution from $\nabla^{E_{B}}$ is the twisting curvature term $K'_{b}$ while the $\nabla^{E_{Z}}$ term contributions the twisting curvature term in the Lichnerowicz formula for $\mathbb{B}_{E_{Z}}^{2}$.
    
    Now we must consider the possible extra terms which appear in $\slashed{\partial}_{E}^{2}$ at order $(x')^{k-1}$ after rescaling. For the Dirac operator, before rescaling the extra terms are
    \begin{align}
        \tilde{T}\cl(\partial_{x})(x')^{k}[(x')^{k-1}\tilde{s}+1]^{k-1}B_{x}+\tilde{T}(x')^{k}\sum_{\gamma}\cl(\tilde{U}_{\gamma})B_{y,\gamma}+\tilde{T}(x')^{k-1}[(x')^{k-1}\tilde{s}+1]\sum_{i}\cl(V_{i})B_{z,i}.
    \end{align}
     where the $B_{x,y,z}$ are sections of $\operatorname{End}_{\mathbb{C}\ice(M)}(E)$ so their products with each other rescale to $0$. Going through the three types of terms as described in the proof for the spin Dirac operator, every term except the $\tilde{T}\cl(\partial_{x})\partial_{\tilde{s}}$, $(x')^{k}V_{i}$ and $\tilde{T}[(x')^{k-1}\tilde{s}+1]^{-k}\slashed{\partial}_{E,Z,y}$ anticommute with these new terms since they are all endomorphisms with odd Clifford degree. 
    
    By definition, we have
    \begin{align}
        \rho(0)=-\sum_{i}\cl(\partial_{x})\cl(V_{i})B_{z,i}
    \end{align}
    so by assumption 3, the third term commutes with $\nabla^{E}_{U_{\gamma}}$ so the $(x')^{k}V_{i}$ term does not contribute anything with the third term. The $(x')^{k-1}$ term does not produce any terms of order $(x')^{k-1}$ the second term and with the first term, any such terms would be of the form $(x')^{k-1}\tilde{e}\tilde{E}$.

    Now considering the $\tilde{T}\cl(\partial_{x})\partial_{\tilde{s}}$ term, if the $\slashed{\partial}_{s}$ hits any of the terms with $\tilde{s}$, the resulting terms are higher order so do not contribute. The only possible contribution is with the third term but $\tilde{T}\cl(\partial_{x})\partial_{\tilde{s}}$ and $\cl(V_{i})B_{z,i}$ anticommute so these terms cancel.

    Finally, the only term which contributes anything with the $\tilde{T}[(x')^{k-1}\tilde{s}+1]^{-k}\slashed{\partial}_{E_{Z},y}$ term is the third term which contributes
    \begin{align}
        -\tilde{T}^{2}(x')^{k-1}(\cl(\partial_{x})\rho(0)\slashed{\partial}_{E,Z,y}+\slashed{\partial}_{E_{Z},y}\cl(\partial_{x})\rho(0)\slashed{\partial}_{E_{Z},y})
    \end{align}
    which is equal to the commutator term appearing in \eqref{nonisolatedrescaledoperatortwist} using the fact that $\cl(\partial_{x})\rho(0)\slashed{\partial}_{E_{Z},y}$ and  $\cl(\partial_{x})$ anticommute and the assumption that $\cl(\partial_{x})$ and $\rho(0)$ commute.
\end{proof} 

\begin{lemma}
    On the fibre above $y$ in $\frontface$ the rescaled heat kernel, under the horizontal symbol map has asymptotic expansion
    \begin{align}
        \delta_{(x')}\sigma_{B}H=u_{-kn}(x')^{-kn}+u_{-kn+k-1}(x')^{-kn+k-1}+O((x')^{-kn+k})
    \end{align}
    where
    \begin{align}\label{rescaledffasymptwist}
        \begin{split}
            &u_{-kn}=h(\tilde{T}^{2},\tilde{s},0)K_{\mathcal{H},K'_{B}}(\tilde{T}^{2},\tilde{\eta})K_{\mathbb{B}_{E_{Z}}^{2}}(\tilde{T}^{2},z,z') \\
             &u_{-kn+k-1}=k\tilde{T}^{2}h(\tilde{T}^{2},\tilde{s},0)K_{\mathcal{H},K'_{B}}(\tilde{T}^{2},\tilde{\eta},0)\cl(\partial_{x})\left(\slashed{\partial}_{E,Z,y}+\frac{c(T)}{4}\right)K_{\mathbb{B}_{E_{Z}}^{2}}(\tilde{T}^{2},z,z') \\
             &\quad +\int_{0}^{t}\int_{Z_{\zeta}}\int_{\mathbb{R}^{b}_{\tilde{\eta}}}\int_{-\infty}^{\infty}\biggl[h(t-t',x,x')K_{\mathcal{H},K'_{B}}(t-t',\tilde{\eta},\tilde{\eta}')K_{\mathbb{B}_{E_{Z}}^{2}}(t-t',z,\zeta), \\
            &\quad L_{k-1}\mathcal{N}_{-kn}(\delta_{\tau}H)(t^{'\frac{1}{2}},x',\tilde{\eta}',\zeta,z')\biggr]dx'd\tilde{\eta}'d\zeta dt' \\
            &\quad +h(t,x,0)K_{\mathcal{H},K'_{B}}(t,\tilde{\eta},0)\int_{0}^{t}\int_{Z_{\zeta}}\left[K_{\mathbb{B}_{E_{Z}}^{2}}(t-t',z,\zeta)\operatorname{cl}(\partial_{x})\slashed{\partial}_{E,Z,y}K_{\mathbb{B}_{E_{Z}}^{2}}(t',\zeta,z'),\rho(0)\right]d\zeta dt'+E.
        \end{split}
    \end{align}
    where $h(t,x,x')$ is the Euclidean heat kernel, $K_{\mathcal{H},K'_{B}},K_{\mathbb{B}^{2}}$ are the heat kernels of $\mathcal{H}+K'_{B}$ and $\mathbb{B}^{2}$ respectively, the integrand in the last line denotes the integral kernel of
    \begin{align}
        \left[e^{(t-t')\mathbb{B}_{E_{Z}}^{2}}\operatorname{cl}(\partial_{x})\slashed{\partial}_{E,Z,y}e^{t'\mathbb{B}_{E_{Z}}^{2}},\rho(0)\right]
    \end{align}
    and $E$ vanishes at $\tilde{s}=0$.
\end{lemma}

\begin{proof}
    The extra base twisting curvature term $K'_{B}$ modifies the heat kernel for $\mathcal{H}$ in the usual way in the Mehler kernel. The only extra term which does not appear for the spin Dirac operator is the commutator.

    Using Duhamel's principle and the fact that $\mathcal{H}+K
    '_{B}$ commutes with $\mathbb{B}_{E_{Z}}^{2}, \cl(\partial_{x})\slashed{\partial}_{E_{Z},y}$ and $\rho(0)$, the commutator term contributes 
    \begin{align}
        \begin{split}
            \int_{0}^{t}\int_{Z_{y'}}\int_{\mathbb{R}^{b}_{\tilde{\eta}}}\int_{-\infty}^{\infty}&h(t-t',x,x')K_{\mathcal{H},K'_{B}}(t-t',\tilde{\eta},\tilde{\eta}')K_{\mathbb{B}_{E_{Z}}^{2}}(t-t',z,\zeta) \\
            &\quad \left[\operatorname{cl}(\partial_{x})\slashed{\partial}_{E_{Z},y},\rho(0)\right]\mathcal{N}_{-kn}(\delta_{\tau}H)(t^{'\frac{1}{2}},x',\tilde{\eta}',\zeta,z')dx'd\tilde{\eta}d\zeta dt' \\
            =h(t,x,0)K_{\mathcal{H},K'_{B}}(t,\tilde{\eta},0)\int_{0}^{t}&\int_{Z_{\zeta}}K_{\mathbb{B}_{E_{Z}}^{2}}(t-t',z,\zeta) \left[\operatorname{cl}(\partial_{x})\slashed{\partial}_{E,Z,y},\rho(0)\right]K_{\mathbb{B}_{E_{Z}}^{2}}(t',\zeta,z')d\zeta dt'.
        \end{split}
    \end{align}
    As usual, the inner integral is the kernel of the operator
    \begin{align}
        e^{(t-t')\mathbb{B}_{E_{Z}}^{2}}\left[\operatorname{cl}(\partial_{x})\slashed{\partial}_{E_{Z},y},\rho(0)\right]e^{t'\mathbb{B}^{2}}=\left[e^{(t-t')\mathbb{B}_{E_{Z}}^{2}}\operatorname{cl}(\partial_{x})\slashed{\partial}_{E,Z,y}e^{t'\mathbb{B}_{E_{Z}}^{2}},\rho(0)\right]
    \end{align} 
    again using the assumption that $B_{E_{Z}}^{2}$ commutes with $\rho(0)$.
\end{proof}

\begin{proposition}
    The pointwise supertrace $x^{kf+1}\operatorname{str}(H|_{\operatorname{diag}_{H}})$ is smooth up to $\frontface_{h}\subset \operatorname{diag}_{h}$ and its restriction to $\frontface_{h}$ is given by
    \begin{align}\label{leadingsupertracefftwist}
        \begin{split}
        \operatorname{str}(H)x^{kf+1}\operatorname{dVol}_{Y}\operatorname{dVol}_{Z}&=2k\tilde{T}\biggl[a(2\pi i)^{-\lceil\frac{b}{2}\rceil}\tilde{T}^{-b}\mathcal{A}(\tilde{T}^{2}R_{B})\operatorname{str}'_{E_{B}}(\exp(-\tilde{T}^{2}K_{B}'))\\
        &\qquad\operatorname{str}_{E_{Z}}\biggl(\biggl(\slashed{\partial}_{E,\partial M/Y,y}+\frac{c(T)}{4}\biggr)e^{-\tilde{T}^{2}\mathbb{B}_{E_{Z}}^{2}}\biggr)\biggr]_{b}\operatorname{dVol}_{Z}
        \end{split}
    \end{align}
    where $a=\frac{1}{2},\frac{1}{2\sqrt{\pi}}$ is $\operatorname{dim}(Z/Y)$ is even/odd and $\operatorname{str}_{E_{Z}}$ denotes the pointwise supertrace for on $E_{Z}$.
\end{proposition}
    \begin{proof}
        The contribution from the Mehler kernel in this case with base twisting curvature $K'_{B}$ gives the $\exp(-\tilde{T}^{2}K'_{B})$ term. The only other extra term is
        \begin{align}
            h(t,x,0)K_{\mathcal{H},K'_{B}}(t,\tilde{\eta},0)\operatorname{tr}_{E_{Z}}(\int_{0}^{t}\int_{Z_{\zeta}}\left[K_{\mathbb{B}_{E_{Z}}^{2}}(t-t',z,\zeta) \operatorname{cl}(\partial_{x})\slashed{\partial}_{E,Z,y}K_{\mathbb{B}_{E_{Z}}^{2}}(t',\zeta,z'),\rho(0)\right]d\zeta dt).
        \end{align}
        This term also vanishes since the first term is even with even horizontal form degree so we have the pointwise trace of a commutator. (Note that this is a commutator of a smooth integral kernel with an endomorphisms so the pointwise trace makes sense.)
    \end{proof}

    The formula for the $t\to 0$ limit of the supertrace of the heat kernel then follows exactly as for the spin Dirac operator.

\begin{theorem}\label{nonisolatedspindiract0limit1}
    For an ice Clifford module $E$ on an ice manifold $(M,g)$ with $g$ product type near the boundary with Clifford connection that satisfies assumptions \ref{assumption1},\ref{assumption2},\ref{assumption3} and \ref{assumption4}, the $t\to 0$ limit of the supertrace of the heat kernel of $\slashed{\partial}^{2}_{E}$ is given by
    \begin{align}\label{nonisolatedt0str2}
        \lim_{t\to 0}\operatorname{str}(e^{-t\slashed{\partial}_{E}^{2}})=\int_{M}\mathcal{A}(M)\operatorname{ch}'(E)-\int_{Y}\mathcal{A}(Y)\operatorname{ch}'(E_{B})\tilde{\eta}(\slashed{\partial}_{E,\partial M/Y})
    \end{align}
    where $\tilde{\eta}(\slashed{\partial}_{E,\partial M/Y})$ is the normalised Bismut-Cheeger eta form for the family $\slashed{\partial}_{E,\partial M/Y}$ on $\partial M$.
\end{theorem}

Finally, we have our index theorem.

    \begin{theorem}\label{twistedtrangression}
        Let $(M,g)$ be incomplete cusp edge space $E$ be an ice Clifford module on M and $\nabla^{E}$ a Clifford connection which satisfies assumptions \ref{assumption2},\ref{assumption3} and \ref{assumption4}. If the Dirac operator $\slashed{\partial}_{E}$ has a boundary family with trivial kernel, then the index of the $\slashed{\partial}_{E}^{+}$ is given by
        \begin{align}
            \operatorname{ind}(\slashed{\partial}_{E}^{+})=\int_{M}\mathcal{A}(M)\operatorname{ch}'(E)-\int_{Y}\mathcal{A}(Y)\operatorname{ch}'(E_{B})\tilde{\eta}(\slashed{\partial}_{E,\partial M/Y})
        \end{align}
        where $\tilde{\eta}(\slashed{\partial}_{E_{Z}})$ is the normalised Bismut-Cheeger eta form for the family $\slashed{\partial}_{E_{Z}}$ on $\partial M$.
         \end{theorem}

         \begin{proof}
             By Lemma \ref{twistedreducetoproductlemma}, $\slashed{\partial}_{E}$ is a Fredholm map $x^{k}H^{1}_{\ce}(M;E)\to L^{2}(M,E)$ and there exists a Clifford module $E_{0}$ on $(M,g_{0})$ for some $g_{0}$ with Clifford connection $\nabla^{E_{0}}$ satisfying assumption 1 (and 2,3,4) with isomorphic restrictions to the boundary and identified boundary families such that the Dirac operator $\slashed{\partial}_{E_{0}}$ is Fredholm on the corresponding domain with equal index. So by the McKean Singer formula and the previous theorem
             \begin{align}
                 \operatorname{ind}(\slashed{\partial}_{E}^{+})=\operatorname{ind}(\slashed{\partial}_{E_{0}}^{+})=\int_{M}\mathcal{A}_{0}(M)\operatorname{ch}'(E)-\int_{Y}\mathcal{A}(Y)\operatorname{ch}'(E_{B})\tilde{\eta}(\slashed{\partial}_{E,\partial M/Y}).
     \end{align}
     Since $\operatorname{ch}'(E)$ is closed
     \begin{align}
        \begin{split}
         \int_{M}(\mathcal{A}_{1}(M)-\mathcal{A}_{0}(M))\operatorname{ch}'(E)&=\int_{M}dTP(\nabla,\nabla^{0})\operatorname{ch}'(E) \\
         &=\int_{M}d(TP(\nabla,\nabla^{0})\operatorname{ch}'(E)) \\
         &=\int_{\partial M}TP(\nabla,\nabla^{0})\operatorname{ch}'(E).
         \end{split}
     \end{align}
     But the transgression form vanishes at the boundary so we have proven our formula.
         \end{proof}
         For a twisting of the spin Dirac operator, as has been mentioned previously, we do not need to make the assumptions 2,3 and 4 and we get the formula
         \begin{align}
            \operatorname{ind}(\slashed{\partial}_{\mathcal{S}\otimes W}^{+})=\int_{M}\mathcal{A}(M)\operatorname{ch}(W)-\int_{Y}\mathcal{A}(Y)\tilde{\eta}(\slashed{\partial}_{W,\partial M/Y})
        \end{align}
        where $\slashed{\partial}_{\partial M/Y,W}$ is the induced family of the spin Dirac operator twisted by $W$. If $W|_{\partial M}$ happens to be of the form $\psi^{\star}W_{B}\otimes W_{Z}$ like we have assumed for our Clifford modules above then this is equal to
        \begin{align}
            \operatorname{ind}(\slashed{\partial}_{\mathcal{S}\otimes W}^{+})=\int_{M}\mathcal{A}(M)\operatorname{ch}(W)-\int_{Y}\mathcal{A}(Y)\operatorname{ch}(W_{B})\tilde{\eta}(\slashed{\partial}_{W_{Z},\partial M/Y})
        \end{align}
        but this assumption is not necessary in this case to get a boundary family globally defined on $Y$.

\section{Signature operator} 

Now we consider the signature operator on a Witt incomplete cusp edge space with $k\geq 3$. First we consider the contribution to the supertrace at $\frontface$. In \cite{ice} the heat kernel for the Hodge Laplacian was constructed and it was shown that the Hodge Laplacian essentially self-adjoint. Although the boundary family in this case has non-trivial kernel and the above calculations for the rescaled heat kernel at $\frontface$ still go through and we get the following.
\begin{proposition}\label{hodgederhamfrontfacesupertrace}
    On a Witt incomplete cusp edge space with product type metric near the boundary, the pointwise supertrace $x^{kf+1}\operatorname{str}(H|_{\operatorname{diag}_{H}})$ of the heat kernel of the signature operator is smooth up to $\frontface_{h}\subset \operatorname{diag}_{h}$ and its restriction to $\frontface_{h}$ is given by
     \begin{align}\label{leadingsupertracesignature}
        \begin{split}
        \operatorname{str}&(H)x^{kf+1}\operatorname{dVol}_{Y}\operatorname{dVol}_{Z} \\
            &=2k\tilde{T}\left[a(\pi i)^{-\lceil\frac{b}{2}\rceil}\tilde{T}^{-b}\mathcal{L}(\tilde{T}^{2}R_{B})\operatorname{str}_{\Lambda(\partial M/Y)}\left(\left(\slashed{\partial}_{\Lambda(\partial M/Y),y}+\frac{c(T)}{4}\right)e^{-\tilde{T}^{2}\mathbb{B}_{\Lambda(\partial M/Y)}^{2}}\right)\right]_{b}\operatorname{dVol}_{Z}
        \end{split}
    \end{align}
     where $a=\frac{1}{2},\frac{1}{2\sqrt{\pi}}$ is $\operatorname{dim}(Z/Y)$ is even/odd and $\operatorname{str}_{E_{Z}}$ denotes the pointwise supertrace for on $E_{Z}$ and $c(T)$ is the third term in the Bismut superconnection \ref{bismutsc} defined earlier (after \ref{beforect}).
\end{proposition}

\begin{proof}
    As we have seen in section 2, the Hodge-de Rham for a product type ice metric satisfies assumption 1 and
    \begin{align}
        \rho(0)=\frac{k}{x}\mathbf{N}+\frac{kf}{2x}
    \end{align} 
    where $\mathbf{N}$ is the vertical number operator so satisfies assumption 2 and 3. The grading for the signature operator is defined by the action of the chirality operator $\Gamma$ which is equal to $i^{k(n-k)}\star$ where $\star$ is the Hodge star operator, $k$ is the form degree and $n$ is the dimension and we write $\ice\Lambda^{\pm}M$ for the odd and even subbundles with respect to this grading.

    Now we work in a neighbourhood of the boundary with the fixed boundary defining function $x$. Given an even form $\omega\in {}\ice\Lambda M$ defining $\omega_{2}=\iota(\partial_{x})\omega$ and $\omega_{1}=\omega-dx\wedge\omega_{2}$ we can write $\omega=\omega_{1}+dx\wedge\omega_{2}$ where $\iota(\partial_{x})\omega_{1,2}=0$. Now given $\eta$ such that $\iota(\partial_{x})\eta=0$, there is a unique section of $\ice\Lambda^{+}M$ given by $\omega=\eta+\Gamma\eta$ such that $\omega_{1}=\eta$. Thus near the boundary, we can make an identification of $\ice\Lambda^{+}M$ with the subbundle of forms with $\iota(\partial_{x})\omega=0$, i.e. generated by $dy^{\alpha},x^{k}dz^{i}$ which we will denote $\ice\Lambda{\partial M}$. Now restricted to the boundary, we can identify $\ice\Lambda\partial M|_{\partial M}$ with $\Lambda\partial M\to\partial M$ via the map
    \begin{align}
        dy^{\alpha}\mapsto dy^{\alpha} \quad x^{k}dz^{i}\to dz^{i}
    \end{align}
    So altogether, we have an identification $\ice\Lambda^{+}M|_{\partial M}\simeq\Lambda\partial M$. Moreover with $\partial M$ given the metric $g_{\partial M}=g_{\partial M/Y}+\phi^{\star}g_{Y}$ so that we have the decomposition $\Lambda\partial M=\Lambda(\partial M/Y)\otimes\phi^{\star}\Lambda Y$ where $\Lambda(\partial M)$ is the exterior bundle on the vertical cotangent bundle, the above map defines an isometry of the vertical ice bundle with $\Lambda(\partial M/Y)$ and similarly for the horizontal bundles. Thus $\ice\Lambda M$ satisfies assumption 4 with $E_{B}=\Lambda Y$ and $E_{Z}=\Lambda(\partial M/Y)$ with the vertical family of connections given by $\boldsymbol(v)\nabla^{\partial M}\boldsymbol{v}$ where $\nabla^{\partial M}$ is the Levi-Civita connection with respect to $g_{\partial M}$.

    We can then repeat the calculation of the asymptotics of the supertrace and the same expression as in proposition \ref{leadingsupertraceff}. We recall, for the Levi-Civita connection, the twisting curvature contribution from $\nabla^{Y}$ is just right Clifford multiplication by Riemann curvature of the base
    \begin{align}
        K_{B}'=R_{B}'=\frac{1}{4}\sum_{\alpha\beta\gamma\delta}g_{Y}(R(U_{\alpha},U_{\beta})U_{\gamma},U_{\delta})\cl_{R}(U_{\gamma})\cl_{R}(U_{\delta}).
    \end{align}
    The relative supertrace $\operatorname{str}'_{\Lambda Y}$ is for the signature just the trace. If the base has even dimension then this becomes
    \begin{align}
        \operatorname{str}_{\Lambda Y}'(\exp(\tilde{T}^{2}R'_{B}))=2^{\frac{b}{2}}{\det}^{\frac{1}{2}}\left(\cosh\left(\frac{R_{B}}{2}\right)\right)
    \end{align}
    which gives us the desired expression in that case while if the dimension of the base is odd 
    \begin{align}
        \operatorname{str}_{\Lambda Y}'(\exp(\tilde{T}^{2}R'_{B}))=2^{\frac{b+1}{2}}{\det}^{\frac{1}{2}}\left(\cosh\left(\frac{R_{B}}{2}\right)\right).
    \end{align}
    \end{proof}

We now determine the contribution to the asymptotics of the supertrace at $\backface$. We first work in the non-isolated case to show the main ideas again without the extra complications of rescaling.

We recall some of the properties of the heat kernel constructed in \cite{ice}. At the back face, the heat kernel vanishes to infinite order away from the fibre harmonic forms and the model problem at $\backface$ restricted the fibre harmonic forms is
\begin{align}
    \mathcal{N}_{\backface}(\Pi_{\mathcal{H}}t(\partial_{t}+\Delta)\Pi_{\mathcal{H}})=T\left(\partial_{T}+
    \begin{bmatrix}
        P_{\alpha(\mathbf{N}),\beta(\mathbf{N}))}+\Delta_{\eta} & 0\\
        0 & P_{\alpha(\mathbf{N}),\gamma(\mathbf{N})}+\Delta_{\eta}
    \end{bmatrix}
    \right)
\end{align}
with respect to the decomposition $\ice\Lambda M=\ice\Lambda\partial M\oplus dx\wedge\ice\Lambda\partial M$ where
\begin{align}
    P_{A,B}=\partial_{s}^{2}-\frac{A}{s}\partial_{s}+\frac{B}{s^{2}}
\end{align}
with
\begin{align}
    \alpha(\mathbf{N})=-kf,\quad \beta(\mathbf{n})=k\mathbf{N}(1-k(f-\mathbf{N})),\quad \gamma(\mathbf{N})=k(f-\mathbf{N})(1-k\mathbf{N}).
\end{align}
For $\mathbf{N}\neq\frac{f}{2}$, the fundamental solution on $L^{2}(\mathbb{R}^{+};s^{A}ds)$ to this heat equation is given by
\begin{align}
    H_{A,B}(s\sigma)=(s\sigma)^{\frac{-(A-1)}{2}}\frac{1}{2t}e^{-\frac{(s^{2}+\sigma^{2})}{4t}}I_{\nu}(\frac{s\sigma}{2t}).
\end{align}
The index set of the heat kernel at the back face satisfies $\inf\mathcal{E}(\backface)=-1-b-kf$ and the normal operator is given by
\begin{align}
    \mathcal{N}_{\backface,-1-b-kf}=\kappa_{\backface,y}\colon=
    \begin{bmatrix}
        H_{\alpha,\beta}(s,1,T) & 0 \\
        0 & H_{\alpha,\gamma}(s,1,T)
    \end{bmatrix}
    (4\pi T)^{-\frac{b}{2}}e^{-\frac{\lvert\eta\rvert_{y^{2}}}{4T}}K_{\Pi_{\mathcal{H}}}(y,z,z')
\end{align}
where $K_{\Pi_{\mathcal{H}}}$ is the kernel of the projection onto the fibre harmonic forms (which in particular is smooth).

When the base is a point, the second factor is just $1$ and $\inf\mathcal{E}(\backface)=-1-kf$ so it is exactly the pointwise supertrace of the restriction of this term to the lifted diagonal which contributes to the $t\to 0$ limit of the supertrace. Thus, we need to consider the form of $\kappa_{\backface,y}$ with respect to the grading for the signature operator.

Suppose that $\omega$ is a section of $\ice\Lambda^{i}\partial M$ in the kernel of the boundary operator $\Delta_{\partial M}$. Since $\star\omega=\pm dx\wedge\star_{\partial M}\omega$ and $\Delta_{\partial M}$ commutes with $\star_{\partial M}$ we have that $\star\omega$ is also in the kernel of $\Delta_{\partial M}$ hence so is $\Gamma\omega$.

Now $\mathbf{N}(\Gamma\omega)=f-i$ and $\iota(\partial_{x})\Gamma\omega\neq 0$ so the action of $\kappa_{\backface,y}$ on $\Gamma\omega$ is given by $H_{\alpha,\gamma}$ where $\gamma(f-i)=ki(1-k(f-i))=\beta(i)$ In particular, the action of $\kappa_{\backface,y}$ is also diagonal with respect to the grading defined by $\Gamma$ and on any element of the kernel with $\mathbf{N}(\omega_{1})=i$, it acts as $H_{\alpha(i),\beta(i)}$.

So we see that $\kappa_{\backface,y}$ has the same action on $\ice\Lambda^{\pm}M$ under their identification with $\ice\Lambda\partial M$. In particular, we see that the pointwise supertrace of the endomorphism $\kappa_{\backface,y}|_{s=1,z=z'}$ on $\mathcal{H}$ vanishes. But the contribution to the constant term in the $t\to 0$ asymptotics of the supertrace is 
\begin{align}
    \dashint_{0}^{\infty}\int_{\partial M}\operatorname{str}_{\ice\Lambda M}(\kappa_{\backface,y}|_{s=1,z=z'})\operatorname{dVol}_{\partial M}\frac{dT}{T}.
\end{align}
So $\backface$ does not contribute to the $t\to 0$ limit of the supertrace and we get the following result.
\begin{proposition}
The limit $\lim_{t\to 0}\operatorname{str}(e^{-t\Delta})$ exists and is given by
    \begin{align}
        \lim_{t\to 0}\operatorname{str}(e^{-t\Delta})=\int_{M}\mathcal{L}(M)-\eta(\slashed{\partial}_{\Lambda,\partial M}).
    \end{align}
\end{proposition}

Now we consider the case of a non-isolated singularity. Since the index set of the heat kernel satisfies $\inf\mathcal{E}(\backface)=-1-b-kf$ where for a finite contribution to the $t\to 0$ supertrace, we would need $\operatorname{str}(H|_{\diag_{H}})$ to satisfy $\inf\mathcal{E}(\backface_{h})=-1-kf$, this suggests we need to do a Getzler rescaling at $\backface$ in the base defined by $\delta_{(x'),\backface}\omega=(x')^{-m}\omega$ for $\omega$ a section of $\Lambda^{M}T^{H}_{p}Z$.

Since the heat kernel vanishes to infinite order away from the fibre harmonic forms, we must consider the rescaling of the operator $\Pi_{\mathcal{H}}t\Delta\Pi_{\mathcal{H}}$ acting on fibre harmonic forms.

We note that if $P$ is some differential operator which maps smooth section of the exterior bundle (with $\iota(\partial_{x})\omega=0$) with fibre degree $i$ to itself then $E_{P}:=\Pi_{\mathcal{H}}P\Pi_{\mathcal{H}}$ is an endomorphism of the kernel bundle which maps fibre harmonic forms with fibre degree $i$ to itself. Moreover if $P$ commutes with $\Gamma$ (equivalently $\star$) then $E(\omega\pm\Gamma\omega)=(E\omega)\pm\tau(E\omega)$ so it has the same action on the even and odd parts of the kernel bundle under the identifications of $\ice\Lambda^{\pm}M$ with $\ice\Lambda\partial M$ for each $i$.
\begin{lemma}
    The action of the lift of $\Pi_{\mathcal{H}}t\Delta\Pi_{\mathcal{H}}$ to near $\backface$ under the rescaling $\delta_{(x')\backface}$ is given by
    \begin{align}
         \delta_{(x')\backface}\Pi_{\mathcal{H}}t(\partial_{t}+\Delta\Pi_{\mathcal{H}} \delta_{(x')\backface}^{-1}=\frac{1}{2}T\partial_{T}+T^{2}(P_{\alpha(\mathbf{N}),\beta(\mathbf{N}))}+\mathcal{H}+R_{B}'+E_{\mathcal{H}}')+O(x')
    \end{align}
    where $P_{\alpha(\mathbf{N}),\beta(\mathbf{N}))}$ acts on $\ice\Lambda^{\pm}M$ under the identification with $\ice\Lambda\partial M$ and $E_{\mathcal{H}}'$ is an endomorphism of the kernel bundle of the form described above. The error is a differential operator which commutes with $x'$ and has coefficients which are $O(x')$.
\end{lemma}

\begin{proof}
    Fix a point $y$ in the base, take and orthonormal frame $U_{\alpha}$ defined by radially parallel transporting an orthonormal coordinate vector field at $y$ in normal coordinates. Then $U_{\gamma}$ are equal to coordinate vector fields to second order at $y$ and $\nabla_{u_{\alpha}}U_{\beta}=0$ at $y$. Take the lift of these vector fields to horizontal vector fields $\tilde{U}_{\alpha}$ on $\partial M$ and extend to and local orthonormal frame with $\partial_{x}$ and a local orthonormal set of vertical vector fields $(x')^{-k}V_{i}$.

    We write the Laplacian as the square of the Dirac operator associated to the Levi-Civita connection.
    
    First we consider the terms contributed by 
    \begin{align}
        \left(\tau\sum_{\alpha}\cl(\tilde{U}_{\alpha})\nabla_{\tilde{U}_{\alpha}}\right)\left(\tau\sum_{\beta}\cl(\tilde{U}_{\beta})\nabla_{\tilde{U}_{\beta}}\right).
    \end{align}
    Using that $\nabla$ is a Clifford connection and commuting $\nabla$ with $\cl$, this reduces to
    \begin{align}
        \begin{split}
        (\tau&\nabla_{\tilde{U}_{\gamma}})^{2}+\frac{\tau^{2}}{2}\sum_{i\alpha}\cl(\tilde{U}_{\alpha})\cl(\tilde{U}_{\gamma})[\nabla_{\tilde{U}_{\gamma}},\nabla_{\tilde{U}_{\alpha}}] \\
        &-\frac{\tau^{2}}{2}\sum_{\alpha\gamma}\cl(\tilde{U}_{\gamma})\cl(\nabla_{\tilde{U}_{\gamma}}\tilde{U}_{\alpha})\nabla_{\tilde{U}_{\alpha}}+\cl(\tilde{U}_{\alpha})\cl(\nabla_{\tilde{U}_{\alpha}}\tilde{U}_{\gamma})\nabla_{\tilde{U}_{\gamma}}.
        \end{split}
    \end{align}
    For the first term we have, the action of the lift of $\tau\nabla_{\tilde{U}_{\gamma}}$ near 
$\backface$ is given by
     \begin{align}
        \begin{split}
        \tau\nabla_{\tilde{U}_{\gamma}}&=T\sigma(\tilde{U}_{\gamma})(x'\eta)+Tx'V'_{\gamma}+T(x')^{k}s^{k-1}B_{y} \\
        &+\frac{\tilde{T}}{4}x'\sum_{jl}\left[g_{\partial M/Y}([\tilde{U}_{\gamma},V_{j}],V_{l})-\phi_{Y}^{\star}g_{Y}(\mathcal{S}^{\phi}(V_{j},V_{l}),\tilde{U}_{\gamma})\right](x'\eta,z)\cl(x^{-k}V_{j}^{y})\cl(x^{-k}V_{l}^{y})  \\
        &-\frac{T}{8}(x')^{2}\sum_{\alpha\beta l}g_{Y}(R_{Y}(\tilde{U}_{\gamma},\eta^{l}\partial_{y^{l}}),\tilde{U}_{\alpha},\tilde{U}_{\beta})(x'\eta)\cl(\tilde{U}_{\alpha})\cl(\tilde{U}_{\beta})+O((x')^{k})
        \end{split}
    \end{align}
    for some endomorphism error term with at most Clifford degree 2. By the asymptotics of the connection, locally the difference $\nabla_{V'_{\gamma}}$ and $V_{\gamma}'$ is an endomorphism of the form $A+(x')^{k}B$ where $A$ has no horizontal Clifford factors. Thus the term $x'V'_{\gamma}$ and the term in the second line contribute terms which are $O((x')^{2})$ in the square of the above term after rescaling so we have
    \begin{align}
         \delta_{(x')\backface}(\tau\nabla_{\tilde{U}_{\gamma}})^{2}\delta_{(x')\backface}^{-1}=T^{2}\mathcal{H}+O((x')^{2}).
    \end{align}
    for the error term is a second order differential operator which commutes with $x'$ and has coefficients which are $O((x')^{2})$.

    Since the horizontal vectors are equal to a coordinate frame to second order at $y$ the commutator term is equal to the curvature with error $\nabla_{W}$ where $W$ is a smooth vector field whose horizontal part vanishes to second order at $y$. So by the asymptotics of the connection and the usual proof of the Lichnerowicz formula this becomes
    \begin{align}
        \begin{split}
        &\frac{\tau^{2}}{2}\sum_{i\alpha}\cl(\tilde{U}_{\alpha})\cl(\tilde{U}_{\gamma})[\nabla_{\tilde{U}_{\gamma}},\nabla_{\tilde{U}_{\alpha}}]=\frac{\tau^{2}}{4}S_{b,y}+\frac{1}{4}\sum_{\alpha\beta\gamma\delta}g_{Y}(R(U_{\alpha},U_{\beta})U_{\gamma},U_{\delta})\cl_{R}(U_{\gamma})\cl_{R}(U_{\delta}) \\
        &+\frac{\tau^{2}}{8}\sum_{ij\alpha\gamma}\cl(\tilde{U}_{\alpha})\cl(\tilde{U}_{\gamma})g(R(\tilde{U}_{\alpha},\tilde{U}_{\gamma})x^{-k}V_{i},x^{-k}V_{j})(\cl(x^{-k}V_{i})\cl(x^{-k}V_{j})+\cl_{R}(x^{-k}V_{i})\cl_{R}(x^{-k}V_{j})) \\
        &+\frac{\tau^{2}}{2}\sum_{\alpha\gamma}\cl(\tilde{U}_{\alpha})\cl(\tilde{U}_{\gamma})\nabla_{W_{\alpha\gamma}}
        \end{split}
    \end{align}
    where $S_{b,y}$ is the scalar curvature of the base. On rescaling, the scalar curvature term is $O((x')^{2})$ while for the other two terms the $\tau^{2}\cl(\tilde{U}_{\alpha})\cl(\tilde{U}_{\gamma})$ terms rescale to $T^{2}\epsilon(\tilde{U}_{\alpha})\epsilon(\tilde{U}_{\alpha\gamma})$ and the curvature term in the first line rescales to $R_{B}'$. For the curvature term in the second line, since the curvature coefficients are smooth and the term commutes with $\Gamma$ and maps sections of fixed fibre degree forms to itself so after acting on it by $\Pi_{\mathcal{H}}\cdot\Pi_{\mathcal{H}}$ this term contributes an endomorphism of form $E_{\mathcal{H}}$. Finally since the horizontal part $W_{\alpha\gamma}$ vanishes to second order in $y$, the rescaling of the lift of this term is given by
    \begin{align}
        \Pi_{\mathcal{H}}\sum_{i\alpha}\epsilon(\tilde{U}_{\alpha})\epsilon(\tilde{U}_{\gamma})\nabla_{\mathbf{v}W_{\alpha\gamma}}^{\partial M}\Pi_{\mathcal{H}}.
    \end{align}
    Since the $\mathbf{v}W_{\alpha\gamma}$ are vertical vector fields which are globally defined in the fibre above $y$, the differential operator $\nabla_{V}^{\partial M}$ is also globally defined on the fibre and maps sections of fixed fibre degree to itself and commutes with $\star$, this term is of the form $E_{\mathcal{H}}$.

    For the $\tau^{2}\cl(\tilde{U}_{\gamma})\cl(\nabla_{\tilde{U}_{\gamma}}\tilde{U}_{\alpha})\nabla_{\tilde{U}_{\alpha}}$ terms, by the asymptotics of the connection, the horizontal part of $\nabla_{\tilde{U}_{\gamma}}$ contributes a term which is $O((x')^{k})$ after rescaling which leaves the horizontal parts for which we have
    \begin{align}
        \frac{\tau^{2}}{2}\sum_{\alpha\gamma}\cl(\tilde{U}_{\gamma})\cl(\nabla_{\tilde{U}_{\gamma}}\tilde{U}_{\alpha})\nabla_{\tilde{U}_{\alpha}}=\frac{\tau^{2}}{2}\sum_{\alpha\gamma\delta}\cl(\tilde{U}_{\gamma})\cl(\tilde{U}_{\delta})g(\nabla_{\tilde{U}_{\gamma}}\tilde{U}_{\alpha},\tilde{U}_{\delta})\nabla_{\tilde{U}_{\alpha}}.
    \end{align}
    For $\gamma=\alpha$, these terms rescale to be $O(x')$ and for the cross terms we have
    \begin{align}
        \begin{split}
            \sum_{\alpha\gamma\delta}\cl(\tilde{U}_{\gamma})\cl(\tilde{U}_{\delta})g(\nabla_{\tilde{U}_{\gamma}}\tilde{U}_{\alpha},\tilde{U}_{\delta})\nabla_{\tilde{U}_{\alpha}}&=\sum_{\alpha(\gamma<\delta)}\cl(\tilde{U}_{\gamma})\cl(\tilde{U}_{\delta})(g(\nabla_{\tilde{U}_{\gamma}}\tilde{U}_{\alpha},\tilde{U}_{\delta})-g(\nabla_{\tilde{U}_{\delta}}\tilde{U}_{\alpha},\tilde{U}_{\gamma}))\nabla_{\tilde{U}_{\alpha}} \\
            &=-\sum_{\alpha(\gamma<\delta)}\cl(\tilde{U}_{\gamma})\cl(\tilde{U}_{\delta})g([\tilde{U}_{\gamma},\tilde{U}_{\delta}],\tilde{U}_{\alpha})\nabla_{\tilde{U}_{\alpha}}.
        \end{split}
    \end{align}
    Since $g([\tilde{U}_{\gamma},\tilde{U}_{\delta}],\tilde{U}_{\alpha})=g_{Y}([U_{\gamma},U_{\delta}],U_{\alpha})$ and the $\tilde{U}_{\gamma}$ are equal to a coordinate frame to second order at $y$, this vanishes to second order in $y$. Thus the lift of the above term multiplied by $\tau^{2}$ is $O(x')$ after rescaling.
    
    Next we consider the terms contributed by
    \begin{align}
        \left(\tau\sum_{i}\cl(x^{-k}V_{i})\nabla_{x^{-k}V_{i}}\right)^{2}.
    \end{align}
    Here we have
    \begin{align}\label{hodgederhamverticalpart}
        \begin{split}
        \sum_{i}\cl(x^{-k}V_{i})\nabla_{x^{-k}V_{i}}&=T(x')^{-k+1}s^{-k}\slashed{\partial}_{Z,y'}+Ts^{-1}\operatorname{cl}(\partial_{x})(\frac{kf}{2}-\rho(0)) \\
        &\quad-\frac{T}{2}x'\sum_{\alpha}k(\tilde{U}_{\alpha})((x')\eta)\operatorname{cl}(\tilde{U}_{\alpha})  \\
        &\quad-\frac{\tilde{T}}{8}(x')^{k+1}s^{k}\sum_{i\alpha\beta}\operatorname{cl}(V_{i})g_{\partial M/Y}(\mathcal{R}^{\phi}(\tilde{U},\tilde{U}_{3}),V)(y,z)\operatorname{cl}(\tilde{U}_{\alpha})\operatorname{cl}(\tilde{U}_{\beta}) \\
        &\quad+\sum_{i\beta\gamma}O((x')^{k+2})\cl(x^{-k}_{i})\cl(\tilde{U}_{\beta})\cl(\tilde{U}_{\gamma}). \\
        \end{split}
    \end{align}
    So in the square of this expression, the terms which contain $\slashed{\partial}_{Z,y'}$ are of the form
    \begin{align}
        T^{2}(x')^{-2k+2}s^{-k}\slashed{\partial}_{Z,y'}^{2}+x^{-k+1}\slashed{\partial}_{Z,y'}P+P'\slashed{\partial}_{Z,y'}.
    \end{align}
    Since on each fibre image of $\slashed{\partial}_{Z,y'}$ is the orthogonal complement to the kernel all these terms vanish inside $\Pi_{\mathcal{K}}$. The $\cl(\partial_{x})$ term anticommutes with every other term and preserves the fibre harmonic forms so this term contributes $-T^{2}s^{2}\frac{k^{2}f^{2}}{4}$. Finally all other terms rescale to terms which are $O((x')^{2})$.

    Next we have the mixed terms
    \begin{align}
        \left[\tau\sum_{i}\cl(\tilde{U}_{\alpha})\nabla_{\tilde{U}_{\alpha}},\tau\sum_{i}\cl(x^{-k}V_{i})\nabla_{x^{-k}V_{i}}\right].
    \end{align}
    We write the first term as $\tau D$. As in the previous case the terms of the form $\tau D\slashed{\partial}_{Z,y}$ and $\slashed{\partial}_{Z,y}\tau D$ vanish inside $\Pi_{\mathcal{H}}$ so we only need to consider the other terms in \eqref{hodgederhamverticalpart}. Since $\cl(\partial_{x})$ anticommutes with $\tau D$ all terms containing it cancel. Next, for the $k(\tilde{U}_{\alpha})\cl(\tilde{U}_{\alpha})$ term we have
    \begin{align}
        \begin{split}
        Dk(\tilde{U}_{\alpha})\cl(\tilde{U}_{\alpha})+k(\tilde{U}_{\alpha})\cl(\tilde{U}_{\alpha})D&=\sum_{\alpha}(\cl(\tilde{U}_{\gamma})\cl(\tilde{U}_{\alpha})+\cl(\tilde{U}_{\alpha})\cl(\tilde{U}_{\gamma}))k(\tilde{U}_{\alpha})\nabla_{\tilde{U}_{\gamma}} \\
        &+\sum_{i}\cl(\tilde{U}_{\gamma})\nabla_{\tilde{U}_{\gamma}}(k(\tilde{U}_{\alpha})\cl(\tilde{U}_{\alpha})).
        \end{split}
    \end{align}
    If $\gamma\neq\alpha$, first summand vanishes otherwise we are left with $-\nabla_{\tilde{U}_{\alpha}}$ which after multiplying by $\tau x'=T(x')^{2}$ is $O(x')$ after rescaling. For the last term, after multiplying by $\tau x'$, the terms of the form $\cl(U)\cl(V)$ are $O(x')$ while those of the form $\cl(U_{\gamma})\cl(U_{\beta})$ rescale to act by $\epsilon(\omega)\otimes_{\operatorname{Id}_{\mathcal{H}}}$ for some horizontal 2-form $\omega$. It is clear that the remaining terms in \eqref{hodgederhamverticalpart} contribute terms which are at least $O(x')$ after rescaling. Note that $\sum_{\alpha}k(\tilde{U}_{\alpha})cl(\tilde{U}_{\alpha})=\cl(k^{\musSharp})$ where $k$ is the mean curvature one-form, so this endomorphism is globally defined on the fibre.

    Finally we come to the terms contributed by $\tau\cl(\partial_{x})\nabla_{\partial_{x}}$. In the local frame, lift is given by $T\cl(\partial_{x})\partial_{x}$ whose square is $-T^{2}\partial_{s}^{2}$. Since $T\cl(\partial_{x})\partial_{x}$ anticommutes with all the terms in $\cl(\tilde{U}_{\gamma})\tau\nabla_{\tilde{U}_{\gamma}}$ up to the terms which are $O((x')^{k})$, summing over these terms and rescaling leaves something which is at least $O(x')$.

    For the mixed terms with $\tau\sum_{i}\cl(x^{-k}V_{i})\nabla_{x^{-k}V_{i}}$, the terms containing $\slashed{\partial}_{Z,y}$ vanish on the fibre harmonic forms while $T\cl(\partial_{x})\partial_{s}$ anticommutes with every term up to those which are $O((x')^{k})$ except for $Ts^{-1}\operatorname{cl}(\partial_{x})(\frac{kf}{2}-\rho(0))$ so in all the $\tau\cl(\partial_{x})\nabla_{\partial_{x}}$ together with the only non-vanishing term from the square of the vertical terms contributes
    \begin{align}
        \left(T\cl(\partial_{x})\partial_{s}+Ts^{-1}\cl(\partial_{x})\left(\frac{kf}{2}-\rho(0)\right)\right)^{2}.
    \end{align}
    But this is equal to
    \begin{align}
        \begin{bmatrix}
        P_{\alpha(\mathbf{N}),\beta(\mathbf{N}))}+\Delta_{\eta} & 0\\
        0 & P_{\alpha(\mathbf{N}),\gamma(\mathbf{N})}+\Delta_{\eta}
    \end{bmatrix}
    \end{align}
    with respect to he decomoposition $\ice\Lambda M=\ice\Lambda\partial M\oplus dx\wedge\ice\Lambda\partial M$. As for the case of an isolated singularity, this is then equal to  $P_{\alpha(\mathbf{N}),\beta(\mathbf{N}))}$ acting on $\ice\Lambda^{\pm}M$ under their identification with $\ice\Lambda\partial M$.
\end{proof}

Now we can show that the contribution to the $t\to 0$ limit of the supertrace of the heat kernel from $\backface$ vanishes.

\begin{proposition}\label{hodgederhambackfacesupertrace}
    At $\backface$ we have $\operatorname{str}_{\ice\Lambda M}(H|_{\diag^{d}})|_{\backface^{d}}x^{1+kf}\operatorname{dVol}_{Y}\operatorname{dVol}_{Z}=0$.
\end{proposition}

\begin{proof}
    By the previous lemma the restriction of $u=(x')^{1+b+kf}\delta_{(x')\backface}H|_{\backface}$ exists and is equal at $y$ to the solution to the equation
    \begin{align}
        \left(\frac{1}{2}T\partial_{T}+T^{2}(P_{\alpha(\mathbf{N}),\beta(\mathbf{N}))}+\mathcal{H}+R_{B}'+E_{\mathcal{H}}')\right)u=0
    \end{align}
    with 
    \begin{align}
             \int_{\beta^{-1}(y)} ud\eta=\delta_{s=1}\delta_{z=z'}\otimes\operatorname{Id}_{\mathcal{H}}.
    \end{align}
    The solution acting on $\ice\Lambda\partial M\simeq{}\ice\Lambda^{\pm}$ is given by
    \begin{align}
        u=H_{\alpha,\beta}(s,1,T^{2})K_{\mathcal{H},R'_{B}}(T^{2},\eta,0)K_{E'_{\mathcal{H}}}(z,z',T^{2})
    \end{align}
    where $K_{E'_{\mathcal{H}}}$ is the kernel of $e^{-tE'_{\mathcal{H}}}$. Since $E'_{\mathcal{H}}$ has the same action on both $\ice\Lambda^{\pm}$ under their identification with $\ice\Lambda\partial M$ so does $K_{E'_{\mathcal{H}}}$ and thus so does $u$. Hence the pointwise supertrace of the restriction of $u$ to $s=1,z=z',\eta=0$ vanishes.

    In particular, this implies that $x^{1+kf}\operatorname{str}_{\ice\Lambda M}(H|_{\diag^{d}})|_{\backface^{d}}=0$ so the contribution from $\backface$ vanishes.
\end{proof}

\begin{theorem}\label{nonisolatedspindiract0limit2}
    Let $(M,g)$ be a Witt incomplete cusp edge space with $g$ product type near the boundary. The $t\to 0$ limit of the supertrace of the heat kernel of the Hodge Laplacian $\Delta$ is given by
    \begin{align}
        \lim_{t\to 0}\operatorname{str}(e^{-t\Delta})=\int_{M}\mathcal{L}(M)-\int_{Y}\mathcal{L}(Y)\tilde{\eta}(\slashed{\partial}_{\Lambda(\partial M/Y)})
    \end{align}
    where $\tilde{\eta}(\slashed{\partial}_{\Lambda(\partial M/Y)})$ is the Bismut-Cheeger eta form for the induced family of operators $\slashed{\partial}_{\Lambda(\partial M/Y)}$ on $\partial M$.
\end{theorem}

\begin{proof}
    The contribution from the interior is standard while the contribution from the boundary follows from Lemmas \ref{hodgederhamfrontfacesupertrace} and \ref{hodgederhambackfacesupertrace}.
\end{proof}

Since the heat kernel $e^{-t\Delta}$ is a self-adjoint compact operator on $L^{2}(M,g)$, we have that $\ker(\Delta))$ is finite dimensional so the proof of the McKean-Singer formula also goes through in this case to give us
\begin{align}
    \lim_{t\to 0}\operatorname{str}(e^{-t\Delta})=\operatorname{dim}(\ker(\Delta^{+}))-\operatorname{dim}(\ker(\Delta^{-})).
\end{align}
Now we assume that $n=\dim(M)$ is divisible by $4$ and can use the standard argument to show that this only has non-zero contribution in middle degree equal to the signature of the intersection form on $\ker(\Delta_{\frac{n}{2}})$. If $\omega_{i}$ is a basis of the kernel of $\Delta$ in degree $j<\frac{n}{2}$ then $\omega_{j}\pm\Gamma\omega_{j}$ is a basis for $\dim(\ker(\Delta^{\pm}))$ away from middle degree so the dimensions cancel leaving $\operatorname{dim}(\ker(\Delta^{+}_{\frac{n}{2}}))-\operatorname{dim}(\ker(\Delta^{-}_{\frac{n}{2}}))$. If $\omega^{\pm}_{i}\in\operatorname{dim}(\ker(\Delta^{\pm}_{\frac{n}{2}}))$ are an orthonormal basis then
\begin{align}
    \int_{M}\omega^{\pm}\wedge\omega^{\pm}=\pm\int_{M}\omega^{\pm}\wedge\star\omega^{\pm}=\pm 1.
\end{align}
The space of $L^{2}$ harmonic forms, also known as the Hodge cohomology, is defined as
\begin{align}
    \mathcal{H}^{\star}_{L^{2}}(M,g)=\{\alpha\in L^{2}\Omega^{\star}(M,g)|d\omega=\delta\omega=0\}.
\end{align}
We can consider the signature form on the middle degree $L^{2}$ harmonic forms and we get the following result.

\begin{proposition}
    Let $(M,g)$ be a Witt incomplete cusp edge space with product type metric. Then the signature on the middle degree $L^{2}$ harmonic forms is 
    \begin{align}\label{harmonicsignatureformula}
        \operatorname{sgn}_{\mathcal{H}_{L^{2}}}(M,g)=\int_{M}\mathcal{L}(M)-\int_{Y}\mathcal{L}(Y)\tilde{\eta}(\slashed{\partial}_{Z}).
    \end{align}
\end{proposition}
We note that, as shown in \cite{ice}, the Hodge cohomology is isomorphic to the $L^{2}$ cohomology
which is defined to be the cohomology of the chain complex
\begin{align}
    \ldots\to L^{2}_{d}\Omega^{p-1}\to L^{2}_{d}\Omega^{p}\to L^{2}_{d}\Omega^{p+1}\to\ldots
\end{align}
where
\begin{align}
    L^{2}_{d}\Omega^{p}=\mathcal{D}_{\max}(d_{p})=\{\omega\in L^{2}\Omega^{p}|d\omega\in L^{2}_{d}\Omega^{p+1}\}.
\end{align}
The spaces $L^{2}\Omega^{\star}(M,g)$ is an invariant of the quasi-isometry class of a metric $g$, so in particular the $L^{2}$-cohomology of $(M,g)$ and $(M,g_{0})$ are the same for any product type metric $g_{0}$ associated to $g$.

From \cite{grieserlesch}[Proposition 2.3.2], since the Hodge Laplacian on an incomplete cusp edge is self-adjoint we have that $\mathcal{D}_{\max}(d)=\mathcal{D}_{\min}(d)$ and $\mathcal{D}_{\max}(\delta)=\mathcal{D}_{\min}(\delta)$. In particular, for all $\omega\in  L^{2}_{d}\Omega^{\frac{n}{2}-1}$ and $\eta\in L^{2}_{\frac{n}{2}}\Omega^{p}$ we have
\begin{align}
    \int_{M}d\omega\wedge\eta=\pm\int_{M}\omega\wedge d\eta.
\end{align}
Thus we see that the intersection form is well defined on the middle degree $L^{2}$ cohomology and the $L^{2}$-signature is the same for quasi-isometric metrics.

From \cite{ice}, the $L^{2}$ cohomology satisfies the Hodge-Kodaira decomposition so on a space with product type metric, the $L^{2}$-signature is equal to the signature on the Hodge cohomology. By the same calculation as in Theorem \ref{twistedtrangression} the integral \eqref{harmonicsignatureformula} is equal for an incomplete cusp edge metric and an associated product type metric so we arrive at the following $L^{2}$-signature formula.
\begin{theorem}\label{signature}
    Let $(M,g)$ be a Witt incomplete cusp edge space with $k\geq 3$. Then $L^{2}$-signature is given by
    \begin{align}
        \operatorname{sgn}_{L^{2}}(M,g)=\int_{M}\mathcal{L}(M)-\int_{Y}\mathcal{L}(Y)\tilde{\eta}(\slashed{\partial}_{\Lambda(\partial M/Y)}).
    \end{align}
\end{theorem}

\section*{Appendix}

\renewcommand{\thesubsection}{\Alph{subsection}}

\subsection{Clifford modules, spin representations and the supertrace}

We use the following convention for the Clifford algebra
\begin{align}
    \operatorname{cl}(V_{i})\operatorname{cl}(V_{j})+\operatorname{cl}(V_{j})\operatorname{cl}(V_{i})=-2g(V_{i},V_{j}).
\end{align}

We define the volume element by
\begin{align}
        \gamma_{n}&=e_{1}\ldots e_{n}   
\end{align}
For $n=2k$, it satisfies $(i^{k}\Gamma)^{2}=1$. For each $k$, the isomorphism $\mathbb{C}l(2k)\simeq \operatorname{End}(\mathbb{C}^{2^{k}})$ gives an irreducible Clifford module $\mathcal{S}_{2k}$ with $\operatorname{dim}(\mathcal{S}_{2k})=2^{k}$ whose even and odd subspaces $\mathcal{S}^{\pm}$ with $\operatorname{dim}(\mathcal{S})=2^{k-1}$ are the $\pm 1$ eigenspaces of $\Gamma_{2k}$. This is by definition the \textbf{spinor module} and is the unique irreducible Clifford module up to isomorphism. Every Clifford module $E$ is isomorphic to $\mathcal{S}\otimes W$ for some $\mathbb{Z}_{2}$-graded vector space $W$ where the Clifford action on $W$ is trivial.

For any $n$ we have an isomorphism $\mathbb{C}l(n-1)\simeq \mathbb{C}l(n)^{+}$ which, taking $\mathbb{R}^{n}=\mathbb{R}\oplus\mathbb{R}^{n-1}$ where the first summand is spanned by $e_{0}$ and the second by the oriented basis $e_{1},\ldots,e_{n-1}$ is given by
\begin{align}
    e_{i}\mapsto -e_{0}e_{i}.
\end{align}
Since $\mathbb{C}l(2k)^{+}$ can be identified with the even endomorphisms of $\mathcal{S}_{2k}$ which are $\operatorname{End}(\mathcal{S}^{+}_{2k})\oplus\operatorname{End}(\mathcal{S}^{-}_{2k})$, the restrictions of $\mathbb{C}l(2k-1)$ to $\mathcal{S}^{+}_{2k}$ and $\mathcal{S}^{-}_{2k}$ are isomorphic to the two irreducible modules of $\mathbb{C}l(2k-1)\simeq \operatorname{End}(\mathbb{C}^{2^{k}})\oplus \operatorname{End}(\mathbb{C}^{2^{k}})$. Note that under this isomorphism, the volume element $\gamma_{2k-1}$ is mapped to $-e_{0}e_{1}\ldots e_{2k-1}e_{2k}$ which is minus the volume element for $\mathbb{R}\oplus\mathbb{R}^{n-1}$ with the standard orientation. For $n=2k-1$, we define the \textbf{spinor module} to be the restriction of $\mathbb{C}l(2k-1)$ to $\mathcal{S}^{+}_{2k}$. Note that for odd dimensions, our choice of representation to define as the spinor module may differ from other sources which use the other representation and leads to a difference in sign in the trace below.

 We define the \textbf{spinor representation} to be the restriction of $\mathcal{S}$ to $\operatorname{Spin}(n)$. Thus for $n=2k$, this is the sum of two irreducible representations $\mathcal{S}^{+}\oplus\mathcal{S}^{-}$ and for $n=2k-1$, this itself is irreducible.

For a $\mathbb{Z}_{2}$-graded vector space $E$ with involution $R$ given by $R|_{E^{\pm}}=\pm 1$ the \textbf{supertrace} of $A\in \operatorname{End}(E)$ is defined $\operatorname{str}_{E}(A)=\operatorname{tr}(RA)$. We will use the notation $\operatorname{str}$, without subscript, to denote the supertrace on the spinor representations $\mathcal{S}$ which satisfies
\begin{align}
        \operatorname{str}(\gamma_{2k})&=(-2i)^{k}.
\end{align}
Every element not proportional to $\gamma$ can be written as a supercommutator so the supertrace vanishes on them. 

The symbol map $\sigma\colon \mathbb{C}l(n)\to\Lambda\mathbb{R}^{n}\otimes\mathbb{C}$ is defined to be the isomorphism of $\mathbb{Z}_{2}$-graded vector spaces given by $\sigma(e_{I})=e_{i_{1}}\wedge\ldots\wedge e_{i_{j}}$.
 We define the Berezin integral $T\colon\Lambda \mathbb{R}^{n}\to\mathbb{R}$ to be the linear map such that $T(e_{1}\wedge\ldots\wedge e_{n})=1$ and $0$ on any element not of degree $n$. Thus in terms of the symbol map, for $n=2k$ we have
\begin{align}\label{supertracespinor}
    \operatorname{str}(A)=(-2i)^{k}T\circ\sigma(A).
\end{align}
For $n=2k-1$ and the spinor representation we have
\begin{align}\label{tensorproductaction}
    \operatorname{tr}(A)=-\frac{1}{2}(-2i)^{k}T\circ\sigma(A)
\end{align}
with the opposite sign for the other representation. 

We have $\mathbb{C}l(n+m)\simeq\mathbb{C}l(n)\hat{\otimes} \mathbb{C}l(m)$ where $\hat{\otimes}$ denotes the $\mathbb{Z}_{2}$-graded tensor product. For $n+m=2k-1$ with $n$ odd, the tensor product $\mathcal{S}_{n}\otimes \mathcal{S}_{m}$ with the action of   $\mathbb{C}l(n)\hat{\otimes} \mathbb{C}l(m)$  given on homogeneous elements $t\in \mathcal{S}_{m}$ by
\begin{align}
    (\alpha\hat{\otimes}\beta)(s\otimes t)=(\alpha s)\otimes (-1)^{|t||\alpha|}(\beta t)
\end{align}
defines a $\mathbb{C}l(n)\hat{\otimes} \mathbb{C}l(m)$ module structure such that the trace of $\gamma_{n}\hat{\otimes}\gamma_{m}$ is
\begin{align}
        \operatorname{tr}(\gamma_{n}\hat{\otimes}\gamma_{m})=-\frac{1}{2}(-2i)^{k}
\end{align}
Under the isomorphism $\mathbb{C}l(n+m)\simeq\mathbb{C}l(n)\hat{\otimes} \mathbb{C}l(m)$ which identifies $\gamma_{n+m}$ with $\gamma_{n}\otimes\gamma_{m}$, it follows that $\mathcal{S}_{n}\otimes\mathcal{S}_{m}\simeq \mathcal{S}_{n+m}$ as Clifford modules.

We define partial symbol maps $\sigma_{n,1}\colon \mathbb{C}l(n+m)\to\Lambda\mathbb{C}^{n}\hat{\otimes}\mathbb{C}l(m)$ and $\sigma_{m,2}\colon\mathbb{C}l(n+m)\to\mathbb{C}l(n)\hat{\otimes}\Lambda\mathbb{C}^{m}$ given by the composition of $\mathbb{C}l(n+m)\simeq\mathbb{C}l(n)\hat{\otimes} \mathbb{C}l(m)$ with the symbol map in the first factor and second factor. These maps commute and compose to give the total symbol map $\sigma$.  We then have for $A\in\mathbb{C}l(n+m)$ and $n$ odd
\begin{align}\label{supertraceseplitevenfibre}
        \operatorname{tr}(A)=-\frac{1}{2}(-2i)^{\frac{n+1}{2}}T(\operatorname{str}_{m}(\sigma_{n,1}(A))).
\end{align}
Similarly, for $n$ even we have
\begin{align}\label{supertraceseplitoddfibre}
        \operatorname{tr}(A)=(-2i)^{\frac{n}{2}}T(\operatorname{tr}_{m}(\sigma_{n,1}(A))) 
\end{align}
where, we use $\operatorname{tr}_{n},\operatorname{str}_{n}$ to denote the supertrace of elements of $\mathbb{C}l(m)$ on $\mathcal{S}$ extended in the obvious way to $\Lambda\mathbb{C}^{n}\hat{\otimes}\mathbb{C}l(m)$. 

Finally let $E$ be a $\mathbb{C}l(2k)$-module with involution $R_{E}$. The involution $R_{E}$ commutes with $\mathbb{C}l(2k)^{+}$ hence with $\Gamma$ and anticommutes with $\mathbb{C}l(2k)^{-}$, thus $R'_{E}=\Gamma R_{E}$ commutes with $\mathbb{C}l(2k)$. If $E=\mathcal{S}\otimes W$ then $R'_{E}$ is exactly the involution for the $\mathbb{Z}_{2}$-grading on $W$.

The relative supertrace $\operatorname{str}'_{E}$ is defined on $A\in \operatorname{End}_{\mathbb{C}l(2k)}(E)$ by $\operatorname{str}'_{E}(A)=2^{-k}\operatorname{tr}(R'_{E}A)$. If $E=\mathcal{S}\otimes W$ then for $\alpha\otimes A\in\mathbb{C}l(2k)\otimes  \operatorname{End}_{\mathbb{C}l(2k)}(E)$
\begin{align}
    \operatorname{str}_{E}(\alpha\otimes A)=(-2i)^{k}T\circ\sigma(\alpha)\operatorname{str}'_{E}(A).
\end{align}
And similar to the spinor modules, for $\alpha\in\mathbb{C}l(2k)^{-}$ which does not contain $e_{0}$
\begin{align}
     \operatorname{str}_{E}(e_{0}\alpha\otimes A)=-2\operatorname{tr}_{E^{+}}(\alpha\otimes A).
\end{align}

\subsection{Triple space \texorpdfstring{$M^{3}_{\ce}$}{M3}}

We construct the triple space to study the composition of polyhomogeneous distributions on the double space. First we need some facts about b-fibrations and blowups.

We say two p-submanifolds $W,Y\subset M$ \textbf{intersect cleanly} if $W\cap Y$ is a p-submanifold and for all $p\in W\cap Y$ we have $T_{p}W\cap T_{p}Y= T_{p}(W\cap Y)$.

A map $f\colon X\to Y$ between manifolds with corners is a \textbf{b-map} if for each boundary hypersurface $H\subset Y$, for any boundary defining function $\rho_{H}$ we have 
\begin{align}
    f^{\star}\rho_{H}=a\prod_{H_{i}'\in\mathcal{M}(X)}\rho_{H_{i}}^{e(H_{i}',H)}
\end{align}
for some non-negative integers $e(H'_{i},H)$ and non-vanishing smooth function $a$. The function $e$ is called the exponent matrix of $f$. Given a b-map, for any boundary face $F$ of $X$, there is a unique boundary face $\overline{f}(F)$ of $Y$ (possibly $Y$ itself) such that $f(F^{\circ})\subset\overline{f}(F)^{\circ}$. A b-map is called a \textbf{b-fibration} if
\begin{enumerate}
    \item for all faces $F$ of $X$ we have $\operatorname{codim}( F)\leq\operatorname{codim}(\overline{f}(F))$,
    \item the restriction of $f$ to $F^{\circ}\to\overline{f}(F)^{\circ}$ is a fibration. 
\end{enumerate}

The importance of b-fibrations is due to Melrose's pushforward theorem which describes the pushforwards of polyhomogeneous densities for b-fibrations.

Now we construct the cusp edge triple space and its projections to the cusp edge double space which closely follows \cite{grieserhunsicker1}. The first step is to construct the edge triple space as in \cite{mazzeo} which we briefly describe. Since the construction is symmetric in the three variables, we will only describe one of the projections. We start with the regular triple space $M^{3}$ and will consider the projection $\pi_{CR}\colon M^{3}\to M^{2}$ given by $\pi_{CR}(\zeta_{1},\zeta_{2},\zeta_{3})=(\zeta_{2},\zeta_{3})$.

For each of the $3$ codimension $2$ faces, there is the corresponding fibre diagonal
\begin{align}
    \begin{split}
    S_{LC}&=\{(\zeta_{1}.\zeta_{2}\,\zeta_{3})\in\partial M\times\partial M\times M|\phi_{Y}(\zeta_{1})=\phi_{Y}(\zeta_{2})\} \\
    S_{LR}&=\{(\zeta_{1}.\zeta_{2}\,\zeta_{3})\in\partial M\times M\times \partial M|\phi_{Y}(\zeta_{1})=\phi_{Y}(\zeta_{3})\} \\
    S_{CR}&=\{(\zeta_{1}.\zeta_{2}\,\zeta_{3})\in M\times\partial M\times \partial M|\phi_{Y}(\zeta_{2})=\phi_{Y}(\zeta_{3})\} 
    \end{split}
\end{align}
and there is the fibre diagonal at the corner which is the intersection of these submanifolds
\begin{align}
    T=\{(\zeta_{1}.\zeta_{2}\,\zeta_{3})\in (\partial M)^{3}|\phi_{Y}(\zeta_{1})=\phi_{Y}(\zeta_{2})=\phi_{Y}(\zeta_{3})\}.
\end{align}
These are all p-submanifolds and the edge triple space is constructed by first blowing $T$ up radially. The closure of the lifts of the interiors of the $S_{\cdot,\cdot}$ which we denote $S_{\cdot,\cdot}$ are p-submanifolds which disjoint from each other so we can blow them up radially in any order and we obtain the edge triple space
\begin{align}
    M^{3}_{e}=[M;T;\tilde{S}_{LC};\tilde{S}_{LR};\tilde{S}_{CR}].
\end{align}
We denote the face obtained by the first blowup by $\cornerbackface$ and the faces obtained by the other 3 blowups by $\edgebackface_{\cdot,\cdot}$. To obtain the projection $\pi_{CR}^{3}\colon M^{3}_{e}\to M^{2}_{e}$ we use the facts
\begin{enumerate}
    \item If $T\subset S$ are p-submanifolds then $[M;S;T]\simeq [M;T;S]$ are naturally diffeomorphic.
    \item If $\beta\colon [X;Y]\to X$ is the blowdown map for the blowup of a boundary p-submanifold  and $f\colon X\to Z$ a b-fibration the $\beta\circ f$ is a b-fibration.
\end{enumerate}
We the have the following sequence of maps 
\begin{align}
    M^{3}_{e}\to [M;T;\tilde{S}_{CR}]\simeq [M;S_{CR};\tilde{T}]\to [M;S_{CR}]=M\times M^{2}_{e}\to M^{2}_{e}.
\end{align}
This is a composition of blow down maps and the last map is a b-fibration. Moreover, this map is the unique smooth extension from the interior of the projection $M^{3}\to M^{2}$. The same holds for the other two projections.

For the cusp edge triple space, fix a boundary defining function $x$ for $\partial M$. Recall for the double space in the neighbourhood $\mathcal{U}\times\mathcal{U}$ of the corner (where $\tilde{U}=\partial M\times [0,\epsilon)$ is a collar neighbourhood of $\partial M\subset M$) defined by
\begin{align}
    U=\{(x,p,x,q)\in\mathcal{U}\times\mathcal{U}|\phi_{Y}(p)=\phi_{Y}(q))\}.    
\end{align}
We have the preimage of this space under each of the three projections defined in a neighbouhood of each of the three codimension 2 faces which we denote $U_{\cdot,\cdot}$. We have $U_{LC}\cap \partial M\times\partial M\times M=S_{LC}\subset U_{LC}$ and similarly for the other 2 and the intersection of all three with the corner is $T$. Thus each of the $U_{\cdot,\cdot}$ are not p-submanifolds of $M^{3}$ but the closure of the lifts of the interiors $\tilde{U}_{\cdot,\cdot}$ are p-submanifolds in $M^{3}_{e}$.

Now we define $V=\tilde{U}_{LC}\cap\tilde{U}_{LR}\cap\tilde{U}_{CR}\cap\cornerbackface$ which is a boundary p-submanifold contained in $\cornerbackface$. Note that $V'=\tilde{U}_{LC}\cap\tilde{U}_{LR}\cap\tilde{U}_{CR}$ defines an interior extension of $V$ so we can take the quasihomogeneous blow up $[M^{3}_{e};V']_{k-1}$. We denote the new face by $\cornerfrontface$. From now on, we will denote the closures of the lifts of the interiors of a submanifold $Y$ by the same symbol to avoid overloading the notation.

Next we have the intersection of the lifts of each the the $\tilde{U}_{\cdot,\cdot}$ with $\cornerbackface$ which we denote by $W_{\cdot,\cdot}$ and the corresponding interior extension defined by $\tilde{U}_{\cdot,\cdot}$ by $W'_{\cdot,\cdot}$. Now we take the quasihomogeneous blowup $[[M^{3}_{e};V']_{k-1};W_{LC}';W_{LR}';W_{CR}']_{k-1}$ which we can take in order since the $W_{\cdot,\cdot}$ are disjoint. We denote the new faces by $\corneredgeface_{\cdot,\cdot}$.

Finally we have the intersections of the lifts of the $\tilde{U}_{\cdot,\cdot}$ with the $\edgebackface$ which we denote by $Y_{\cdot,\cdot}$ and their interior extensions by $Y_{\cdot,\cdot}'$. The cusp edge double space is then defined to be
\begin{align}
    M^{3}_{\ce}=[[[M^{3}_{e};V']_{k-1};W_{LC}';W_{LR}';W_{CR}']_{k-1};Y_{LC}';Y_{LR}';Y_{CR}']_{k-1}.
\end{align}
We denote the new faces by $\edgefrontface_{\cdot,\cdot}$.

To obtain the projection $\pi_{CR}^{3}\colon M^{3}_{\ce}\to M^{2}_{\ce}$ and show that it is a b-fibration we use the following facts in addition to those above.
\begin{enumerate}
    \item If $T\subset S$ are p-submanifolds then $[[M;S]_{a};T]_{a}\simeq [M;T;S]$ are naturally diffeomorphic.
    \item If $Y=T\cap S$ intersect cleanly then $[[[M;T]_{a};Y]_{b};S]_{b}\simeq [[[M;S]_{b};Y]_{a};T]_{a}$ where $T,S,Y$ denote themselves or their lifts under respective blowdown maps wherever it makes sense.
\end{enumerate}
First, since the $W_{CR}$ and $Y_{CR}$ are disjoint from the rest of the $W_{\cdot,\cdot}$and $Y_{\cdot,\cdot}$ we can move them to the front the blowups and obtain the blowdown
\begin{align}
    M^{3}_{\ce}=[[[M^{3}_{e};V']_{k-1};W_{CR}';;Y_{CR}';W_{LC}';W_{LR}';Y_{LC}';Y_{LR}']_{k-1}\to[[[M^{3}_{e};V']_{k-1};W_{CR}';Y_{CR}']_{k-1}.
\end{align}
Similarly, the faces $\edgebackface_{\cdot,\cdot}$ so we can move the blowups which products the faces in $M^{3}_{e}$ to the end as well so we have
\begin{align}
    [[[M^{3}_{e};V']_{k-1};W_{CR}';Y_{CR}']_{k-1}\simeq[[[M;T;\tilde{S}_{CR}];V';W_{CR}';Y_{CR}']_{k-1}\tilde{S}_{LR};\tilde{S}_{LC}].
\end{align}
And as described for the edge triple space projection, we have $[M;T;\tilde{S}_{CR}]\simeq [M;S_{CR};\tilde{T}]\simeq[M\times M^{2}_{e};\tilde{T}]$ so we have
\begin{align}
   [[[M;T;\tilde{S}_{CR}];V';W_{CR}';Y_{CR}']_{k-1}\tilde{S}_{LR};\tilde{S}_{LC}]\to [[M\times M^{2}_{e};\tilde{T}];V';W_{CR}';Y_{CR}']_{k-1}.
\end{align}
Since $V\subset W_{CR}$ and is disjoint from $Y_{CR}$ we can also move this to the end.
\begin{align}
    \begin{split}
    [[M\times M^{2}_{e};\tilde{T}];V';W_{CR}';Y_{CR}']_{k-1}\simeq &[[M\times M^{2}_{e};\tilde{T}];W_{CR}';Y_{CR}';V']_{k-1} \\
    &\quad\to[[M\times M^{2}_{e};\tilde{T}];W_{CR}';Y_{CR}']_{k-1}.
    \end{split}
\end{align}
Finally, $\tilde{T}$ and $Y_{CR}$ intersect cleanly at $Z\subset M\times M^{2}_{e}$ which lifts to $W'_{CR}$ under the blowdown map $[M\times M^{2}_{e},\tilde{Y}]\to M\times M^{2}_{e}$ so we have
\begin{align}
    [[M\times M^{2}_{e};\tilde{T}];W_{CR}';Y_{CR}']_{k-1}\simeq[[M\times M^{2}_{e};Y_{CR}]_{k-1};\tilde{Z};\tilde{Y}]\to[M\times M^{2}_{e};Y_{CR}]_{k-1}.
\end{align}
The last term here is equal to $M\times M^{2}_{\ce}$ so finally composing this with the b-fibration $M\times M^{2}_{\ce}\to M^{2}_{\ce }$ this defines the projection $\pi^{3}_{CR}\colon M^{3}_{\ce}\to M^{2}_{\ce}$ which is a composition of blowdown maps and a b-fibration hence is itself a b-fibration. Again, this is the unique smooth extension from the interior of the projection $M^{3}\to M^{2}$ and we can do the same for the other projections.

Now that we have constructed the cusp edge triple space and the b-fibrations $\pi_{\cdot,\cdot}^{3}\colon M^{3}_{\ce}\to M^{2}_{\ce}$ which smoothly extend the usual projections $\pi_{\cdot,\cdot}\colon M^{3}\to M^{2}$ from the interior, we can study the composition of cusp edge pseudodifferential operators.

Let us first recall a few definitions and theorems that we will need. For a boundary hypersurface $F\subset M$ on a manifold with corners, we define an \textbf{index set} $\mathcal{E}(F)$ to be a subset of $\mathbb{C}\times\mathbb{N}$ such that
\begin{enumerate}
    \item for all $C\in\mathbb{R}$, $(\{\operatorname{Re}(z)<C\}\times\mathbb{N})\cap\mathcal{E}(F)$ is a finite set
    \item if $(z,p)\in\mathcal{E}(F)$ then $(z,q)\in \mathcal{E}(F)$ for all $0\leq q\leq p$
    \item if $(z,p)\in\mathcal{E}(F)$ then $(z+1,p)\in\mathcal{E}(F)$.
\end{enumerate}
A collection of index sets $\mathcal{E}$ is called an \textbf{index family}.

We define the infimum of the empty index set to be $\infty$ and otherwise
\begin{align}
    \inf\mathcal{E}(F)=\inf\{\operatorname{Re}(z)|(z,p)\in\mathcal{E}(F)\}.
\end{align}
We define the \textbf{extended union} of index sets $\mathcal{E}$ and $\mathcal{F}$ by
\begin{align}
    \mathcal{E}\overline{\cup}\mathcal{F}=\mathcal{E}\cup\mathcal{F}\cup \{(z,p+p'+1)|(z,p)\in\mathcal{E}, (z,p')\in\mathcal{F}\}.
\end{align}
The sum of two index sets is defined to be the empty set if at least on of them is empty and otherwise
\begin{align}
    \mathcal{E}(F)+\mathcal{F}(F)=\{(z+z',p+p')|(z,p)\in\mathcal{E}(F),(z',p')\in\mathcal{F}(F)\}.
\end{align}
We define a function $f$ to be polyhomogeneous with index family $\mathcal{E}$ inductively on the condimension of the boundary faces as follows. For a boundary hypersurface $F$, let $\mathcal{E}_{F}$ be the index family on $F$ given by the index sets for the boundary hypersurfaces with non-empty intersection with $F$ and $\rho_{F}$ a boundary defining function for $F$. Then $f$ is polyhomogeneous with index family $\mathcal{E}$ at $F$ if there exists functions $a_{z,p}$ such that
\begin{align}
    f\sim\sum_{(z,p)\in\mathcal{E}(F)}a_{z,p}\rho_{F}^{z}\log^{p}\rho_{F}
\end{align}
where $a_{z,p}$ are polyhomogeneous conormal on $F$ with index family $\mathcal{E}(F)$ (where we identify $a_{z,p}$ with a function in a collar neighbourhood of $F$ using $\rho_{F}$) and the asymptotic sum means that for all $N$, the error
\begin{align}
    E_{N}=f-\sum_{\substack{(z,p)\in\mathcal{E}(F)\\ \operatorname{Re}(z)<N}}a_{z,p}\rho_{F}^{z}\log^{p}\rho_{F}
\end{align}
in the collar neighbourhood of $F$ satisfies $\rho^{M}\rho^{-N}_{F}\mathcal{V}_{b}^{k}E$ is bounded for sufficiently large $M$ where $\rho$ is the product of all boundary hypersurfaces with non-empty intersection with $F$.

Let $Y\subset M$ be an embedded submanifold, then a distribution $u$ is a \textbf{classical conormal distribution} of order $m$ at $Y$ if it is smooth away from $Y$ and locally about any point on $Y$ in coordinates $y_{i},z_{j}$ such that $Y$ is given by $z_{j}=0$  we have
\begin{align}
    u=\int e^{iz\zeta}\sigma(y,\zeta)d\zeta
\end{align}
where $\sigma(y,\zeta)$ is a classical symbol (in $\zeta$) of order $m'=m+\frac{1}{4}\dim(M)-\frac{1}{2}\operatorname{codim} (Y)$. As with the classical pseudodifferential calculus, this condition is coordinate invariant and in coordinates symbols admit unique expansions as asymptotic sums in homogeneous symbols with leading order term giving an invariantly defined fibre density $\sigma_{m}(u)$ on the $N^{\star}Y$ called the \textbf{principal symbol}. We can extend this definition to distributions conormal to an interior p-submanifold $Y$ and polyhomogeneous with index set $\mathcal{E}$ by requiring that the coefficients in the asymptotics expansion be conormal at $F\cap Y$ to order $m+\frac{1}{4}$ and an analogous condition on the remainder.

Next, we have the pullback and pushforward theorems. Let $f\colon M\to N$ be a b-map and $\mathcal{E}$ an index set on $N$, we define the pullback index set $f^{\#}(\mathcal{E})$ to be the index set on $M$ given on a boundary hypersurface $G\subset M$ to be $\mathbb{N}$ if $e_{f}(G,H)=0$ for all boundary hypersurfaces $H\subset N$ and otherwise
\begin{align}
    f^{\#}(\mathcal{E})(G)=\left\{(S,Z)|\exists \{(z_{H},p_{H})\in \mathcal{E}(H) : e_{f}(F,G)\neq 0)\}\text{ with } S=\sum_{H}e(G,H)z_{H}, Z=\sum_{H}p_{H}\right\}.
\end{align}
\begin{theorem}[Pullback theorem]
    Let $f:M\to N$ be a b-map between manifolds with corners, $Y\subset N$ an interior p-submanifold. If $f$ is transversal to $Y$ (for all $p\in f^{-1}(Y)$, $^{b}f_{\star}(^{b}T_{p}M)+^{b}T_{f(p)}Y=^{b}T_{f(p))}N$), then the pullback of smooth functions extends to a continuous map
    \begin{align}
        f^{\star}\colon I^{m,\mathcal{E}}(N,Y)\to I^{m-\frac{1}{4}(\dim(M)-\dim(N)),f^{\#}\mathcal{E}}(M,f^{-1}(Y)).
    \end{align}
\end{theorem}
Given a map $f\colon M\to N$ between compact manifolds with corners and  a distributional (b-)density $u$ on $M$, the pushforward of $u$ is the distributional (b-)density on $N$ defined by $f_{\star}u(\varphi)=u(\varphi^{\star}u)$ for a compactly supported smooth function $\varphi$. 

\begin{theorem}[Pushforward theorem]
    Let $f\colon M\to N$ be a b-fibration between compact manifolds with corners and $W,Y\subset M$ interior p-submanifolds which intersect transversally and $f$ restricts to a diffeomorphism of $W,Y$ with $N$. 
    
    If $u\in I^{m,\mathcal{E}}(M,W)$ and $v\in I^{n,\mathcal{F}}(M,Y)$ and $\mu_{b}(M)$ a non-vanishing b-density and for all boundary hypersurfaces $F$ with $\overline{f}(F)=N$ we have $\inf \mathcal{E}(F)+\inf \mathcal{F}(F)>0$ then $f_{\star}(uv\mu_{b}(M))$ exists and is in $I^{m+n+\frac{1}{4}\dim(N),\mathcal{G}}(N,f(W\cap Y))$ with $\mathcal{G}=f_{\#}(\mathcal{E}+\mathcal{F})$ where
    \begin{align}
    f_{\#}(\mathcal{E})(H)=\overline{\bigcup}_{G:e(G,H)>0}\{(\frac{z}{e(G,H)},p)|(z,p)\in\mathcal{G}\}.
\end{align}
\end{theorem}

We define the \textbf{large cusp edge calculus} of pseudodifferential operators $\Psi_{\ce}^{\star,\mathcal{E}}(M;E,F)$ with index set $\mathcal{E}$ between sections of vector bundles $E$ and $F$ by
\begin{align}
    \Psi_{\ce}^{\star}(M;E,F)=I^{\star,\mathcal{E}}(M^{2}_{\ce},\diag_{\ce};\operatorname{HOM}(E,F)\otimes\Omega_{R,\ce})
\end{align}
where $\Omega_{R,ce}:=\rho^{-k(b+1)}\beta^{\star}_{R}\Omega(M)$. If $A\in\Psi_{\ce}^{\star,\mathcal{E}}(M;E,F)$ and $f\in\mathcal{A}^{\mathcal{F}}(M;F)$ polyhomogeneous on $M$ with index set $\mathcal{F}$, fixing any non-vanishing density $\mu(M)$ on $M$ and defining $\mu_{L,R}=\beta^{\star}_{L,R}\mu$ and writing the kernel of $A$ as $K_{A}\rho^{-k(b+1)}_{\frontface}\mu_{R}$, we define the action of $A$ by the formula
\begin{align}
    Af\mu(M)=(\beta_{L})_{\star}(K_{A}\rho_{\frontface}^{-k(b+1)}\beta^{\star}_{R}f\mu_{R}\mu_{L}).
\end{align}
Note that since we have the canonical identification $\beta^{\star}_{R}\Omega(M)\otimes\beta^{\star}_{L}\Omega(M)\simeq\beta^{\star}\Omega(M^{2})$, we can identify $\mu_{R}\mu_{L}$ with the pullback $\beta^{\star}(\mu(M^{2}))$ of some non-vanishing density $\mu(M^{2})$ on $M^{2}$ and moreover, using projective coordinates at $\backface,\frontface$, we see that
\begin{align}
    \beta^{\star}(\mu(M^{2}))=\rho^{k(b+1)}_{\frontface}\rho^{b+1}_{\backface}\mu(M^{2}_{\ce})
\end{align}
for some non-vanishing density $\mu(M^{2}_{\ce})$. Note that this definition is independent on the choice of $\mu(M)$ and reduces to the usual action for an operator with kernel given by the pullback of a smooth kernel on $M^{2}$.

We can use the pullback and pushforward theorem to determine when $Af$ is well define and the index set of the resulting polyhomogeneous function
\begin{proposition}
    Let $A\in\Psi_{\ce}^{\star,\mathcal{E}}(M;E,F)$ and $f\in\mathcal{A}^{\mathcal{F}}(M;F)$. If the index set $\mathcal{E}$ and $\mathcal{F}$ satisfy
    \begin{align}\label{actiononfunctioncondition1}
        \inf \mathcal{E}(\rightface)+\inf \mathcal{F}(\partial M)>-1
    \end{align}
    then $Af\in\mathcal{A}^{\mathcal{G}}(M;E)$ where
    \begin{align}
        \mathcal{G}=(\mathcal{E}(\leftface)+\mathbb{N})\overline{\cup}(\mathcal{E}(\backface)+\mathcal{F}(\partial M)+b+1)\overline{\cup}(\mathcal{E}(\frontface)+\mathcal{F}(\partial M)).
    \end{align}
\end{proposition}

\begin{proof}
    By definition and the above discussion, we have
    \begin{align}
    Af\rho_{M}\mu_{b}(M)=(\beta_{L})_{\star}(K_{A}\rho_{\frontface}^{-k(b+1)}\beta^{\star}_{R}f\rho^{k(b+1)}_{\frontface}\rho^{b+1}_{\backface}\rho_{M^{2}_{\ce}}\mu_{b}(M^{2}_{\ce})).
\end{align}
    where we use the notation $\rho_{\cdot}$ to denote the product of all boundary defining functions and $\mu_{b}(\cdot)$ the resulting non-vanishing b-density.

    By the pullback theorem, $\beta^{\star}_{R}f$ has index set $f^{\#}(\mathcal{F})$ which is given by
    \begin{align}
        \begin{split}
             f^{\#}(\mathcal{F})(\leftface)&=\mathbb{N} \\
             f^{\#}(\mathcal{F})(\backface)=f^{\#}(\mathcal{F})(\frontface)&=f^{\#}(\mathcal{F})(\rightface)=\mathcal{F}(\partial M).
        \end{split}
    \end{align}
    Since $\beta_{L}$ is a b-fibration and $\beta^{\star}_{R}f$ is smooth in the interior so the other conditions of the pushforward theorem are trivially satisfied. The index set of the distribution $K_{A}\beta^{\star}_{R}f\rho^{b+1}_{\backface}\rho_{M^{2}_{\ce}}\mu_{b}(M^{2}_{\ce})$ as a b-density at $\rightface$ is given by
    \begin{align}\label{actiononfunctioncondition2}
        \mathcal{E}(\rightface)+f^{\#}\mathcal{F}(\rightface)+1=\mathcal{E}(\rightface)+\mathcal{F}(\rightface)(\partial M)+1.
    \end{align}
    So by the pushforward theorem, the pushforward exists if \eqref{actiononfunctioncondition2} is $>0$ which is exactly the condition \eqref{actiononfunctioncondition1}. Moreover defining $w(\backface)=(b+1)+1$ and $w(\cdot)=1$ for the other faces, the index set $\mathcal{G}'$ of the pushforward is given by 
    \begin{align}
        \begin{split}
        f_{\#}(\mathcal{E}+f^{\#}\mathcal{F}+w)(\partial M)=&(\mathcal{E}(\leftface)+f^{\#}\mathcal{F}(\leftface)+1)\overline{\cup}(\mathcal{E}(\backface)+f^{\#}\mathcal{F}(\backface)+(b+1)+1) \\
        &\overline{\cup}(\mathcal{E}(\frontface)+f^{\#}\mathcal{F}(\frontface)+1).
        \end{split}
    \end{align}
    The index set $\mathcal{G}=\mathcal{G}-1$ which gives us the result.
\end{proof}
If $\mathcal{G}$ had different index sets at the two components of $\backface$ then we would get a factor of $(\mathcal{E}(\backface_{i})+\mathcal{F}(\partial M)+b+1$ for each component.

We now discuss the composition of cusp edge pseudodifferential operators. We define the composition $C=A\circ B$ of two operators $A\in \Psi^{r_{1},\mathcal{E}}(M;F,G)$ and $B\in\Psi^{r_{e},\mathcal{F}}(M;E,F)$ to be the operator with integral kernel defined by the formula
\begin{align}
    K_{C}\rho_{ff}^{-k(b+1)}\mu_{R}\mu_{L}=(\pi^{3}_{LR})_{\star}((\pi^{3}_{LC})^{\star}(K_{A}\rho_{\frontface}^{-k(b+1)}\mu_{R})(\pi^{3}_{CR})^{\star}(K_{B}\rho_{\frontface}^{-k(b+1)}\mu_{R})(\pi^{3}_{LR})^{\star}\mu_{L}).
\end{align}
As above, we use the identification of density bundles which gives
\begin{align}
    \mu_{R}\mu_{L}=\rho^{k(b+1)}_{\frontface}\rho^{b+1}_{\backface}\rho\mu_{b}(M^{2}_{\ce}).
\end{align}
Similarly on the triple space, denoting the projection $M^{3}\to M^{2}$ by $\pi_{\cdot,\cdot}$ so that $\pi^{3}_{\cdot,\cdot}=\beta^{(3)}\circ\pi_{\cdot,\cdot}$ we have the canonical identifications
\begin{align}
    \begin{split}
    (\pi^{3}_{LC})^{\star}\Omega_{R}\otimes(\pi^{3}_{CR})^{\star}\Omega_{R}\otimes(\pi^{3}_{LR})^{\star}\Omega_{L}&\simeq(\beta^{(3)})^{\star}(\pi^{3}_{LC})^{\star}\Omega_{R}\otimes(\pi^{3}_{CR})^{\star}\Omega_{R}\otimes(\pi^{3}_{LR})^{\star}\Omega_{L}) \\
    &\simeq (\beta^{(3)})^{\star}\Omega(M^{3}).
    \end{split}
\end{align}
Under these identifications, we have that $(\pi^{3}_{LC})^{\star}\mu_{R}(\pi^{3}_{CR})^{\star}\mu_{R}(\pi^{3}_{LR})^{\star}\mu_{LR}=(\beta^{(3)})^{\star}\mu(M^{3})$ for some smooth non-vanishing density $\mu(M^{3})$. In coordinates, this is given by $adxdydz,dx'dy'dz'dx''dy''dz''$ for some smooth function $a$ and we can use projective coordinates at each face on $M^{3}_{\ce}$ to determine the order of vanishing of the lift if this density at each face. We will do the $LC$ face as the others are the same.

At $\cornerbackface$, we have the projective coordinates
\begin{align}
    s_{1}=\frac{x'}{x},s_{2}=\frac{x''}{x},x,\eta_{1}=\frac{y'-y}{x},\eta_{2}=\frac{y''-y}{x},y,z,z',z''.
\end{align}
So the density lifts to $ax^{2(b+1)}dxdydzds_{1}d\eta_{2}dz'ds_{2}d\eta_{2}dz'$ and similarly in other projective coordinates at $\cornerfrontface$.

Similarly at $\cornerfrontface$, we have the projective coordinates
\begin{align}
    \tilde{s}_{1}=\frac{x'-x}{x^{k}},\tilde{s}_{2}=\frac{x''-x}{x^{k}},x,\tilde{\eta_{1}}=\frac{y'-y}{x^{k}},\tilde{\eta}_{2}=\frac{y''-y}{x^{k}},y,z,z',z''.
\end{align}
So the density lifts to $ax^{2k(b+1)}dxdydzd\tilde{s}_{1}d\tilde{\eta}_{2}dz'd\tilde{s}_{2}d\tilde{\eta}_{2}dz'$. 

At $\corneredgeface_{LC}$ we can use the coordinates
\begin{align}
    S_{1}=\frac{s_{1}-s_{2}}{x^{k-1}},s_{2},E_{1}=\frac{\eta_{1}-\eta_{2}}{x^{k-1}},\eta_{2},x,y,z,z',z''
\end{align}
and we get $$ax^{(k+1)(b+1)}dxdydzdS_{1}dE_{1}dz'ds_{2}d\eta_{2}dz'$$.

At $\edgebackface$ and $\edgefrontface$ we can use the same coordinates as those on the double space with extra coordinates $x'',y'',z''$ as parameters, so overall we get
\begin{align}
    (\beta^{(3)})^{\star}(\mu(M^{3}))=\rho_{\cornerfrontface}^{2k(b+1)}\rho_{\cornerbackface}^{2(b+1)}\rho_{\corneredgeface}^{(k+1)(b+1)}\rho_{\edgefrontface}^{k(b+1)}\rho_{\edgebackface}^{b+1}\mu(M^{3}_{\ce})
\end{align}
where $\mu(M^{3}_{\ce})$ is some smooth non-vanishing density and we write $\rho_{F}$ for the product the three boundary defining functions when $F=\corneredgeface_{\cdot,\cdot},\edgebackface_{\cdot,\cdot},\edgefrontface_{\cdot,\cdot}$.

Next for the pullbacks of $\rho_{\frontface}$ we have
\begin{align}
    \begin{split}
        (\pi^{3}_{LC})^{\star}\rho_{\frontface}&=b\rho_{\cornerfrontface_{LC}}\rho_{\corneredgeface_{LC}}\rho_{\edgefrontface_{LC}} \\
        (\pi^{3}_{CR})^{\star}\rho_{\frontface}&=c\rho_{\cornerfrontface_{CR}}\rho_{\corneredgeface_{CR}}\rho_{\edgefrontface_{CR}}
    \end{split}
\end{align}
for some non-vanishing smooth functions $b,c$. Putting this together we have the following for the composition of cusp edge pseudodifferential operators. 

\begin{proposition}
    Let $A\in \Psi^{r_{1},\mathcal{E}}_{\ce}(M;F,G)$ and $B\in \Psi^{r_{2},\mathcal{F}}_{\ce}(M;E,F)$. If the index sets satisfy
    \begin{align}
        \inf \mathcal{E}(\rightface)+\inf \mathcal{F}(\leftface)>-1
    \end{align}
    then the composition $C=A\circ B\in\Psi_{\ce}^{r_{1}+r_{2},\mathcal{G}}(M;E;G)$ with index set $\mathcal{G}$ given by
    \begin{align}
        \begin{split}
            \mathcal{G}(\frontface)&=(\mathcal{E}(\frontface)+\mathcal{F}(\frontface))\overline{\cup}(\mathcal{E}(\leftface)+\mathcal{F}(\rightface)+k(b+1))\overline{\cup}(\mathcal{E}(\backface)+\mathcal{F}(\backface)+(k+1)(b+1)) \\
            \mathcal{G}(\backface)&=(\mathcal{E}(\backface)+\mathcal{F}(\backface)+b+1)\overline{\cup}(\mathcal{E}(\frontface)+\mathcal{F}(\backface))\overline{\cup}(\mathcal{E}(\backface)+\mathcal{F}(\frontface)) \\
            & \quad \quad \quad\overline{\cup}(\mathcal{E}(\leftface)+\mathcal{F}(\rightface)) \\
            \mathcal{G}(\leftface)&=(\mathcal{E}(\backface)+\mathcal{F}(\leftface)+b+1)\overline{\cup}(\mathcal{E}(\frontface)+\mathcal{F}(\leftface))\overline{\cup}(\mathcal{E}(\leftface)+\mathbb{N}) \\
            \mathcal{G}(\rightface)&= (\mathcal{E}(\rightface)+\mathcal{F}(\backface)+b+1)\overline{\cup}(\mathcal{E}(\rightface)+\mathcal{F}(\frontface)+k(b+1))\overline{\cup}(\mathcal{F}(\rightface)+\mathbb{N}).
        \end{split}
    \end{align}
\end{proposition}

\begin{proof}
We have
\begin{align}
    \begin{split}
       (\pi^{3}_{LC})^{\star}(&K_{A}\rho_{\frontface}^{-k(b+1)}\mu_{R})(\pi^{3}_{CR})^{\star}(K_{B}\rho_{\frontface}^{-k(b+1)}\mu_{R})(\pi^{3}_{LR})^{\star}\mu_{L} \\
       &=(\pi^{3}_{LC})^{\star}(K_{A})(\pi^{3}_{RC})^{\star}(K_{B})a\rho_{\corneredgeface_{LC}}^{b+1}\rho_{\corneredgeface_{CR}}^{b+1}\rho_{\corneredgeface_{LR}}^{(k+1)(b+1)}\rho_{\cornerbackface}^{2(b+1)}\rho_{\edgebackface}^{b+1}\rho_{\edgefrontface}^{k(b+1)}\rho\mu_{b}(M^{3}_{\ce})
    \end{split}
\end{align}
for some non-vanishing smooth function $a$. By the pullback theorem, 
\begin{align}\label{triplespacecompositiondistribution}
    \begin{split}
(\pi^{3}_{LC})^{\star}(K_{A})&\in I^{r_{1}-\frac{n}{4},(\pi^{3}_{LC})^{\#}(\mathcal{E})}(M^{3}_{\ce},(\pi^{3}_{LC})^{-1}(\diag_{\ce})) \\
(\pi^{3}_{CR})^{\star}(K_{B})&\in I^{r_{2}-\frac{n}{4},(\pi^{3}_{LC})^{\#}(\mathcal{F})}(M^{3}_{\ce},(\pi^{3}_{CR})^{-1}(\diag_{\ce})).
    \end{split}
\end{align}
The two lifts of the diagonals $(\pi^{3}_{LC})^{-1}(\diag_{\ce})$ and $(\pi^{3}_{CR})^{-1}(\diag_{\ce})$ satisfy the conditions of the pushforward theorem and the image of their intersection under $\pi^{3}_{LR}$ is $\diag_{\ce}$. The index sets of the distribution \eqref{triplespacecompositiondistribution} as a b-density which we denote by $\mathcal{L}$ are given by
\begin{align}
    \begin{split}
        \mathcal{L}(\cornerfrontface)&=(\pi^{3}_{LC})^{\#}(\mathcal{E})(\cornerfrontface)+(\pi^{3}_{LC})^{\#}(\mathcal{F})(\cornerfrontface)+1 \\
        &=\mathcal{E}(\frontface)+\mathcal{F}(\frontface)+1 \\
        \mathcal{L}(\cornerbackface)&=\mathcal{E}(\backface)+\mathcal{F}(\backface)+2(b+1)+1 \\
        \mathcal{L}(\corneredgeface_{LC})&=\mathcal{E}(\frontface)+\mathcal{F}(\backface)+b+1+1 \\
        \mathcal{L}(\corneredgeface_{CR})&=\mathcal{E}(\backface)+\mathcal{F}(\frontface)+b+1+1 \\
        \mathcal{L}(\corneredgeface_{LR})&=\mathcal{E}(\backface)+\mathcal{F}(\backface)+(k+1)(b+1)+1 \\
        \mathcal{L}(\edgefrontface_{LC})&=\mathcal{E}(\frontface)+\mathcal{F}(\leftface)+1 \\
        \mathcal{L}(\edgefrontface_{CR})&=\mathcal{E}(\rightface)+\mathcal{F}(\frontface)+1 \\
        \mathcal{L}(\edgefrontface_{LR})&=\mathcal{E}(\leftface)+\mathcal{F}(\rightface)+k(b+1)+1 \\
        \mathcal{L}(\edgebackface_{LC})&=\mathcal{E}(\backface)+\mathcal{F}(\leftface)+b+1+1 \\
        \mathcal{L}(\edgebackface_{CR})&=\mathcal{E}(\rightface)+\mathcal{F}(\backface)+b+1+1 \\
        \mathcal{L}(\edgebackface_{LR})&=\mathcal{E}(\leftface)+\mathcal{F}(\rightface)+b+1+1 \\
        \mathcal{L}(LC)&=\mathcal{F}(\rightface)+\mathbb{N}+1 \\
        \mathcal{L}(CR)&=\mathcal{E}(\leftface)+\mathbb{N}+1.
    \end{split}
\end{align}
where we denote by $LC$ are $CR$ the boundary hypersurfaces on $M^{3}_{\ce}$ which are the closure of the lifts of the interior of $M^{2}_{\ce}$ by the maps $\pi^{3}_{LC}$ and $\pi^{3}_{CR}$.

Then by the pushforward theorem, the pushforward of \eqref{triplespacecompositiondistribution} exists whenever
\begin{align}
    \inf\mathcal{L}(LR)=\inf\mathcal{E}(\rightface)+\inf\mathcal{F}(\leftface)+1>0
\end{align}
and is a distribution on $M^{2}_{\ce}$ conormal at $\diag_{\ce}$ of order $r_{1}-\frac{n}{4}+r_{2}-\frac{n}{4}+\frac{2n}{4}=r_{1}+r_{3}$ with index sets
\begin{align}
    \begin{split}
        (\pi^{3}_{LR})_{\#}(\mathcal{L})(\frontface)&=\mathcal{L}(\cornerfrontface)\overline{\cap}\mathcal{L}(\corneredgeface_{LR})\overline{\cap}\mathcal{L}(\edgefrontface_{LR}) \\
        (\pi^{3}_{LR})_{\#}(\mathcal{L})(\backface)&=\mathcal{L}(\cornerbackface)\overline{\cap}\mathcal{L}(\corneredgeface_{LC})\overline{\cap}\mathcal{L}(\corneredgeface_{CR})\overline{\cap}\mathcal{L}(\edgebackface_{LR}) \\
        (\pi^{3}_{LR})_{\#}(\mathcal{L})(\leftface)&=\mathcal{L}(\edgebackface_{LC})\overline{\cap}\mathcal{L}(\edgefrontface_{LC})\overline{\cap}\mathcal{L}(RC) \\
        (\pi^{3}_{LR})_{\#}(\mathcal{L})(\rightface)&=\mathcal{L}(\edgebackface_{RC})\overline{\cap}\mathcal{L}(\edgefrontface_{RC})\overline{\cap}\mathcal{L}(LC)
    \end{split}
\end{align}
Finally, the index set of $K_{C}$ is given by $\mathcal{G}(\backface)= (\pi^{3}_{LR})_{\#}(\mathcal{L})(\backface)-(b+1+1)$ and $\mathcal{G}(\cdot)=(\pi^{3}_{LR})_{\#}(\mathcal{L})(\cdot)-1$ for the other faces which gives us the result. 
\end{proof}
The small cusp edge pseudodifferential calculus is given by the cusp edge pseudodifferential operators with index set $\mathcal{E}(\frontface)=\mathbb{N}$ and empty at all other faces. The composition properties of the small cusp edge pseudodifferential calculus immmediately follow from the above proposition. We refer to \cite{grieserhunsicker1} for the proof that the symbol map and the normal operator map are algebra homomorphisms which can be adapted to this setting.


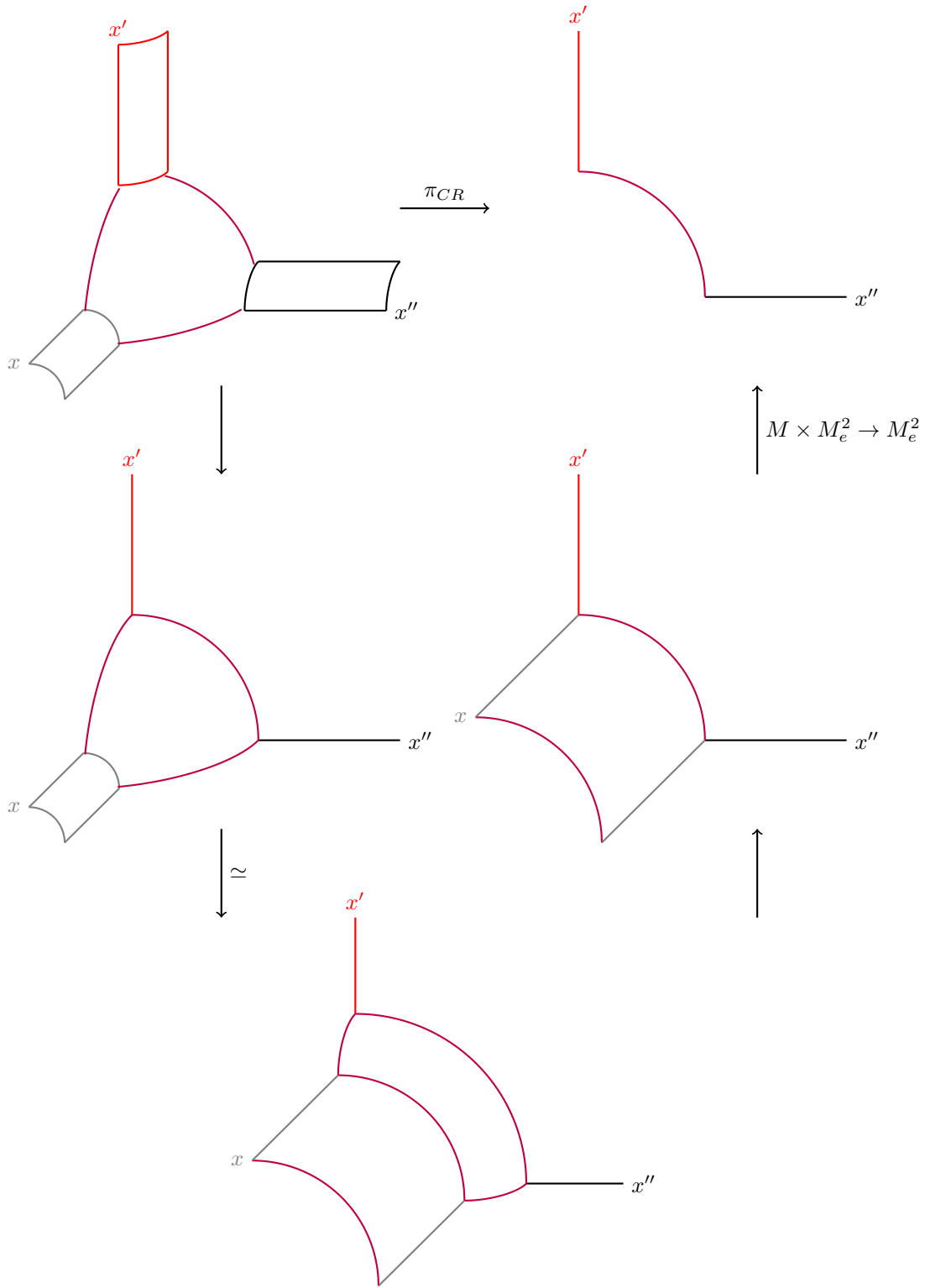
\begin{figure}[H]
\centering
    \begin{tikzpicture}[scale=1.4,rotate around x=90,rotate around z=90,z={(0,0,-1)},y={(0,-1,0)}] 



        \draw[gray, thick] ({sqrt(2)},0,-4.6+5) -- (3,0,-4.6+5) node[pos=1,left] {$x$};
        \draw[gray, thick] ({sqrt(2)},0.4,-5+5) -- (3,0.4,-5+5) ;
        \draw[domain=0:90, smooth, variable=\x, gray, thick] plot ({sqrt(2)}, {0.4*cos(\x)}, {0.4*sin(\x)}); 
        \draw[domain=0:90, smooth, variable=\x, gray, thick] plot ({3}, {0.4*cos(\x)}, {0.4*sin(\x)}); 
        
        \draw[black, thick] (0.4,{sqrt(2)},-5+5) -- (0.4,3,-5+5) node[pos=1,right] {$x''$};
        \draw[black, thick] (0,{sqrt(2)},-4.6+5) -- (0,3,-4.6+5) ;
        \draw[domain=0:90, smooth, variable=\x, black, thick] plot ({0.4*cos(\x)}, {sqrt(2)}, {0.4*sin(\x)}); 
         \draw[domain=0:90, smooth, variable=\x, black, thick] plot ({0.4*cos(\x)}, {3}, {0.4*sin(\x)}); 
        
        \draw[red, thick] (0.4,0,{sqrt(2)}) -- (0.4,0,-2+5) node[pos=1,above] {$x'$};
        \draw[red, thick] (0,0.4,{sqrt(2)}) -- (0,0.4,-2+5);
        \draw[domain=0:90, smooth, variable=\x, red, thick] plot ({0.4*cos(\x)}, {0.4*sin(\x)}, {sqrt(2)}); 
         \draw[domain=0:90, smooth, variable=\x, red, thick] plot ({0.4*cos(\x)}, {0.4*sin(\x)}, {-2+5}); 
        \draw[domain=15:75, smooth, variable=\x, purple, thick] plot ({sqrt(2)*cos(\x)}, {sqrt(2)*(sin(\x)}, {0}); 
        \draw[domain=15:75, smooth, variable=\x, purple, thick] plot ({sqrt(2)*cos(\x)}, {0}, {(sqrt(2)*(sin(\x))}); 
        \draw[domain=15:75, smooth, variable=\x, purple, thick] plot ({0}, {sqrt(2)*cos(\x)}, {(sqrt(2)*(sin(\x))}); 



        \draw[black, thick] (0,{5+sqrt(2)},0) -- (0,8,0) node[pos=1,right] {$x''$};
        \draw[red, thick] (0,{5)},{sqrt(2)}) -- (0,5,3) node[pos=1,above] {$x'$};
         \draw[domain=0:90, smooth, variable=\x, purple, thick] plot ({0}, {5+sqrt(2)*cos(\x)}, {(sqrt(2)*(sin(\x))}); 


         \draw[gray, thick] ({sqrt(2)},0,-4.6) -- (3,0,-4.6) node[pos=1,left] {$x$};
        \draw[gray, thick] ({sqrt(2)},0.4,-5) -- (3,0.4,-5) ;
        \draw[domain=0:90, smooth, variable=\x, gray, thick] plot ({sqrt(2)}, {0.4*cos(\x)}, {0.4*sin(\x)-5}); 
        \draw[domain=0:90, smooth, variable=\x, gray, thick] plot ({3}, {0.4*cos(\x)}, {0.4*sin(\x)-5}); 
        
        \draw[black, thick] (0,{sqrt(2)},-5) -- (0,3,-5) node[pos=1,right] {$x''$};
        
        \draw[red, thick] (0,0,{sqrt(2)-5}) -- (0,0,-2) node[pos=1,above] {$x'$};

        \draw[domain=15:90, smooth, variable=\x, purple, thick] plot ({sqrt(2)*cos(\x)}, {sqrt(2)*(sin(\x)}, {-5}); 
        \draw[domain=15:90, smooth, variable=\x, purple, thick] plot ({sqrt(2)*cos(\x)}, {0}, {(sqrt(2)*(sin(\x))-5}); 
        \draw[domain=0:90, smooth, variable=\x, purple, thick] plot ({0}, {sqrt(2)*cos(\x)}, {(sqrt(2)*(sin(\x))-5});



         \draw[black, thick] (0,{5+sqrt(2)},-5) -- (0,8,0-5) node[pos=1,right] {$x''$};
        \draw[red, thick] (0,{5},{sqrt(2)-5}) -- (0,5,3-5) node[pos=1,above] {$x'$};
         \draw[domain=0:90, smooth, variable=\x, purple, thick] plot ({0}, {5+sqrt(2)*cos(\x)}, {(sqrt(2)*(sin(\x))-5}); 
         \draw[domain=0:90, smooth, variable=\x, purple, thick] plot ({3}, {5+sqrt(2)*cos(\x)}, {(sqrt(2)*(sin(\x))-5});
          \draw[gray, thick] (0,5,{sqrt(2)-5}) -- (3,5,{sqrt(2)-5}) node[pos=1,left] {$x$};
        \draw[gray, thick] (0,{5+sqrt(2)},-5) -- (3,{5+sqrt(2)},-5) ;


         \draw[black, thick] (0,{5+sqrt(2)+0.5-2.5},-10) -- (0,8-2.5,-10) node[pos=1,right] {$x''$};
        \draw[red, thick] (0,{5-2.5},{sqrt(2)-10+0.5}) -- (0,5-2.5,3-10) node[pos=1,above] {$x'$};
         \draw[domain=0:90, smooth, variable=\x, purple, thick] plot ({0.5}, {-2.5+5+sqrt(2)*cos(\x)}, {(sqrt(2)*(sin(\x))-10}); 
          \draw[domain=0:90, smooth, variable=\x, purple, thick] plot ({0}, {-2.5+5+(0.5+sqrt(2))*cos(\x)}, {((0.5+sqrt(2))*(sin(\x))-10}); 
         \draw[domain=0:90, smooth, variable=\x, purple, thick] plot ({3}, {-2.5+5+sqrt(2)*cos(\x)}, {(sqrt(2)*(sin(\x))-10});
         \draw[domain=0:90, smooth, variable=\x, purple, thick] plot ({0.5*cos(\x)}, {-2.5+5}, {(sqrt(2)-10+0.5*sin(\x)});
         \draw[domain=0:90, smooth, variable=\x, purple, thick] plot ({0.5*cos(\x)}, {-2.5+5+sqrt(2)+0.5*sin(\x)}, {-10});
          \draw[gray, thick] (0.5,{-2.5+5},{sqrt(2)-10}) -- (3,{-2.5+5},{sqrt(2)-10}) node[pos=1,left] {$x$};
        \draw[gray, thick] (0.5,{-2.5+5+sqrt(2)},-10) -- (3,{-2.5+5+sqrt(2)},-10) ;

        \draw[->,black,thick] (0,3,1) -- (0,4,1) node[pos=0.5,above] {$\pi_{CR}$};
        \draw[->,black,thick] (0,1,-1) -- (0,1,-2) node[pos=0.5,right] {${}$};
        \draw[->,black,thick] (0,7,-2) -- (0,7,-1) node[pos=0.5,right] {$M\times M^{2}_{e}\to M^{2}_{e}$};
        \draw[->,black,thick] (0,1,-6) -- (0,1,-7) node[pos=0.5,right] {$\simeq$};
        \draw[->,black,thick] (0,7,-7) -- (0,7,-6) node[pos=0.5,right] {${}$};

    \end{tikzpicture}
    \caption{The sequence of blow-downs, projections and identifications used to define the projection $\pi_{CR}$ from the edge triple space to the edge double space. The projections from the cusp edge triple space to the cusp edge double space are similar but more complicated.} \label{fig:blowuptriple}
\end{figure}


\begin{figure}[H]
\centering
    \begin{tikzpicture}[scale=1.4,rotate around x=90,rotate around z=90,z={(0,0,-1)},y={(0,-1,0)}] 
        \draw[gray, thick] (0,0,0) -- (3,0,0) node[pos=1,left] {$x$};
        \draw[black, thick] (0,0,0) -- (0,3,0) node[pos=1,right] {$x''$};
        \draw[red, thick] (0,0,0) -- (0,0,3) node[pos=1,above] {$x'$};


    \draw[->,black,thick] (0,3,1) -- (0,4,1) node[pos=0.5,above] {$\pi_{CR}$};



        \draw[black, thick] (0,5,0) -- (0,8,0) node[pos=1,right] {$x''$};
        \draw[red, thick] (0,5,0) -- (0,5,3) node[pos=1,above] {$x'$};


        \draw[gray, thick] ({sqrt(2)},0,-4.6) -- (3,0,-4.6) node[pos=1,left] {$x$};
        \draw[gray, thick] ({sqrt(2)},0.4,-5) -- (3,0.4,-5) ;
        \draw[domain=0:90, smooth, variable=\x, gray, thick] plot ({sqrt(2)}, {0.4*cos(\x)}, {-5+0.4*sin(\x)}); 
        \draw[domain=0:90, smooth, variable=\x, gray, thick] plot ({3}, {0.4*cos(\x)}, {-5+0.4*sin(\x)}); 
        
        \draw[black, thick] (0.4,{sqrt(2)},-5) -- (0.4,3,-5) node[pos=1,right] {$x''$};
        \draw[black, thick] (0,{sqrt(2)},-4.6) -- (0,3,-4.6) ;
        \draw[domain=0:90, smooth, variable=\x, black, thick] plot ({0.4*cos(\x)}, {sqrt(2)}, {-5+0.4*sin(\x)}); 
         \draw[domain=0:90, smooth, variable=\x, black, thick] plot ({0.4*cos(\x)}, {3}, {-5+0.4*sin(\x)}); 
        
        \draw[red, thick] (0.4,0,{sqrt(2)-5}) -- (0.4,0,-2) node[pos=1,above] {$x'$};
        \draw[red, thick] (0,0.4,{sqrt(2)-5}) -- (0,0.4,-2);
        \draw[domain=0:90, smooth, variable=\x, red, thick] plot ({0.4*cos(\x)}, {0.4*sin(\x)}, {sqrt(2)-5}); 
         \draw[domain=0:90, smooth, variable=\x, red, thick] plot ({0.4*cos(\x)}, {0.4*sin(\x)}, {-2}); 
        \draw[domain=15:75, smooth, variable=\x, purple, thick] plot ({sqrt(2)*cos(\x)}, {sqrt(2)*(sin(\x)}, {-5}); 
        \draw[domain=15:75, smooth, variable=\x, purple, thick] plot ({sqrt(2)*cos(\x)}, {0}, {(sqrt(2)*(sin(\x))-5}); 
        \draw[domain=15:75, smooth, variable=\x, purple, thick] plot ({0}, {sqrt(2)*cos(\x)}, {(sqrt(2)*(sin(\x))-5}); 

        \draw[->,black,thick] (0,3,-3.5) -- (0,4,-3.5) node[pos=0.5,above] {$\pi_{e,CR}$};



        \draw[black, thick] (0,{5+sqrt(2)},-5) -- (0,8,-5) node[pos=1,right] {$x''$};
        \draw[red, thick] (0,{5)},{-5+sqrt(2)}) -- (0,5,-2) node[pos=1,above] {$x'$};
         \draw[domain=0:90, smooth, variable=\x, purple, thick] plot ({0}, {5+sqrt(2)*cos(\x)}, {(sqrt(2)*(sin(\x))-5}); 

        \draw[domain=15:40, smooth, variable=\x, purple, thick] plot ({sqrt(2)*cos(\x)}, {sqrt(2)*(sin(\x)}, {-5-5}); 
        \draw[domain=50:75, smooth, variable=\x, purple, thick] plot ({sqrt(2)*cos(\x)}, {sqrt(2)*(sin(\x)}, {-5-5}); 
        \draw[domain=15:40, smooth, variable=\x, purple, thick] plot ({sqrt(2)*cos(\x)}, {0}, {(sqrt(2)*(sin(\x))-5-5}); 
        \draw[domain=50:75, smooth, variable=\x, purple, thick] plot ({sqrt(2)*cos(\x)}, {0}, {(sqrt(2)*(sin(\x))-5-5}); 
        \draw[domain=15:40, smooth, variable=\x, purple, thick] plot ({0}, {sqrt(2)*cos(\x)}, {(sqrt(2)*(sin(\x))-5-5}); 
        \draw[domain=50:75, smooth, variable=\x, purple, thick] plot ({0}, {sqrt(2)*cos(\x)}, {(sqrt(2)*(sin(\x))-5-5}); 

        \draw[domain=0:360, smooth, variable=\x, red, thick] plot ({sqrt(2)*cos(45)*cos(35)+0.2*cos(45)*cos(45)*cos(\x)-0.3*sin(45)*sin(\x)}, {sqrt(2)*(sin(45)*cos(35)+0.2*sin(45)*cos(45)*cos(\x)+0.3*cos(45)*sin(\x)}, {-10+sqrt(2)*sin(35)-0.3*sin(45)*cos(\x)}); 
        

        
  \draw[domain=-45:135, smooth, variable=\x, purple, thick] plot ( {sqrt(2)*cos(45)+0.1*cos(\x)},{(sqrt(2)*(sin(45))+0.1*sin(\x)}, {-5-5}); 
        \draw[domain=0:23, smooth, variable=\x, purple, thick] plot ( {sqrt(2)*cos(45)*cos(\x)+0.1*cos(-45)}, {(sqrt(2)*(sin(45))*cos(\x)+0.1*sin(-45)}, {sqrt(2)*sin(\x)-5-5});
        \draw[domain=46:76, smooth, variable=\x, purple, thick] plot ( {sqrt(2)*cos(45)*cos(\x)+0.1*cos(-45)}, {(sqrt(2)*(sin(45))*cos(\x)+0.1*sin(-45)}, {sqrt(2)*sin(\x)-5-5});
        \draw[domain=0:24, smooth, variable=\x, purple, thick] plot ( {sqrt(2)*cos(45)*cos(\x)+0.1*cos(135)}, {(sqrt(2)*(sin(45))*cos(\x)+0.1*sin(135)}, {sqrt(2)*sin(\x)-5-5});
        \draw[domain=45:76, smooth, variable=\x, purple, thick] plot ( {sqrt(2)*cos(45)*cos(\x)+0.1*cos(135)}, {(sqrt(2)*(sin(45))*cos(\x)+0.1*sin(135)}, {sqrt(2)*sin(\x)-5-5});


        \draw[domain=-45:135, smooth, variable=\x, purple, thick] plot ( {sqrt(2)*cos(45)+0.1*cos(\x)},{0}, {(sqrt(2)*(sin(45))+0.1*sin(\x)-5-5}); 
        \draw[domain=0:20, smooth, variable=\x, purple, thick] plot ( {sqrt(2)*cos(45)*cos(\x)+0.1*cos(-45)}, {sqrt(2)*sin(\x)}, {(sqrt(2)*(sin(45))*cos(\x)+0.1*sin(-45)-5-5});
        \draw[domain=48:77, smooth, variable=\x, purple, thick] plot ( {sqrt(2)*cos(45)*cos(\x)+0.1*cos(-45)}, {sqrt(2)*sin(\x)}, {(sqrt(2)*(sin(45))*cos(\x)+0.1*sin(-45)-5-5});
        \draw[domain=0:23, smooth, variable=\x, purple, thick] plot ( {sqrt(2)*cos(45)*cos(\x)+0.1*cos(135)}, {sqrt(2)*sin(\x)}, {(sqrt(2)*(sin(45))*cos(\x)+0.1*sin(135)-5-5});
         \draw[domain=52:77, smooth, variable=\x, purple, thick] plot ( {sqrt(2)*cos(45)*cos(\x)+0.1*cos(135)}, {sqrt(2)*sin(\x)}, {(sqrt(2)*(sin(45))*cos(\x)+0.1*sin(135)-5-5});

        \draw[domain=-45:135, smooth, variable=\x, purple, thick] plot ({0}, {sqrt(2)*cos(45)+0.1*cos(\x)}, {(sqrt(2)*(sin(45))+0.1*sin(\x)-5-5}); 
        \draw[domain=0:20, smooth, variable=\x, purple, thick] plot ({sqrt(2)*sin(\x)}, {sqrt(2)*cos(45)*cos(\x)+0.1*cos(-45)}, {(sqrt(2)*(sin(45))*cos(\x)+0.1*sin(-45)-5-5});
        \draw[domain=46:75, smooth, variable=\x, purple, thick] plot ({sqrt(2)*sin(\x)}, {sqrt(2)*cos(45)*cos(\x)+0.1*cos(-45)}, {(sqrt(2)*(sin(45))*cos(\x)+0.1*sin(-45)-5-5});
        \draw[domain=0:25, smooth, variable=\x, purple, thick] plot ({sqrt(2)*sin(\x)}, {sqrt(2)*cos(45)*cos(\x)+0.1*cos(135)}, {(sqrt(2)*(sin(45))*cos(\x)+0.1*sin(135)-5-5});
        \draw[domain=48:75, smooth, variable=\x, purple, thick] plot ({sqrt(2)*sin(\x)}, {sqrt(2)*cos(45)*cos(\x)+0.1*cos(135)}, {(sqrt(2)*(sin(45))*cos(\x)+0.1*sin(135)-5-5});

        \draw[gray, thick] ({sqrt(2)},0,-4.6-5) -- (3,0,-4.6-5) node[pos=1,left] {$x$};
        \draw[gray, thick] ({sqrt(2)},0.4,-5-5) -- (3,0.4,-5-5) ;
        \draw[domain=0:30, smooth, variable=\x, gray, thick] plot ({sqrt(2)}, {0.4*cos(\x)}, {-5-5+0.4*sin(\x)}); 
        \draw[domain=60:90, smooth, variable=\x, gray, thick] plot ({sqrt(2)}, {0.4*cos(\x)}, {-5-5+0.4*sin(\x)}); 
        \draw[domain=60:90, smooth, variable=\x, gray, thick] plot ({sqrt(2)}, {0.4*cos(\x)}, {-5-5+0.4*sin(\x)}); 
        \draw[domain=0:30, smooth, variable=\x, gray, thick] plot ({3}, {0.4*cos(\x)}, {-5-5+0.4*sin(\x)}); 
        \draw[domain=60:90, smooth, variable=\x, gray, thick] plot ({3}, {0.4*cos(\x)}, {-5-5+0.4*sin(\x)}); 
        \draw[gray, thick] ({sqrt(2)},{0.4*cos(30)},{-5-5+0.4*sin(30)}) -- (3,{0.4*cos(30)},{-5-5+0.4*sin(30)}) ;
        \draw[gray, thick] ({sqrt(2)},{0.4*cos(60)},{-5-5+0.4*sin(60)}) -- (3,{0.4*cos(60)},{-5-5+0.4*sin(60)}) ;

        \draw[domain=-50:140, smooth, variable=\x, gray, thick] plot ({sqrt(2)}, {0.4*cos(45)+0.1*cos(\x)}, {-5-5+0.4*sin(45)+0.1*sin(\x)}); 
        \draw[domain=-50:140, smooth, variable=\x, gray, thick] plot ({3}, {0.4*cos(45)+0.1*cos(\x)}, {-5-5+0.4*sin(45)+0.1*sin(\x)});

        \draw[black, thick] (0.4,{sqrt(2)},-5-5) -- (0.4,3,-5-5) node[pos=1,right] {$x''$};
        \draw[black, thick] (0,{sqrt(2)},-4.6-5) -- (0,3,-4.6-5) ;
        \draw[domain=0:29, smooth, variable=\x, black, thick] plot ({0.4*cos(\x)}, {sqrt(2)}, {-5-5+0.4*sin(\x)});
        \draw[domain=61:90, smooth, variable=\x, black, thick] plot ({0.4*cos(\x)}, {sqrt(2)}, {-5-5+0.4*sin(\x)});
         \draw[domain=0:29, smooth, variable=\x, black, thick] plot ({0.4*cos(\x)}, {3}, {-5-5+0.4*sin(\x)}); 
         \draw[domain=61:90, smooth, variable=\x, black, thick] plot ({0.4*cos(\x)}, {3}, {-5-5+0.4*sin(\x)}); 
         \draw[black, thick] ({0.4*cos(30)},{sqrt(2)},{-5-5+0.4*sin(30)}) -- ({0.4*cos(30)},3,{-5-5+0.4*sin(30)}) ;
        \draw[black, thick] ({0.4*cos(60)},{sqrt(2)},{-5-5+0.4*sin(60)}) -- ({0.4*cos(60)},3,{-5-5+0.4*sin(60)}) ;

        \draw[domain=-50:140, smooth, variable=\x, black, thick] plot ({0.4*cos(45)+0.1*cos(\x)},{sqrt(2)}, {-5-5+0.4*sin(45)+0.1*sin(\x)}); 
        \draw[domain=-50:140, smooth, variable=\x, black, thick] plot ({0.4*cos(45)+0.1*cos(\x)}, {3}, {-5-5+0.4*sin(45)+0.1*sin(\x)}); 
        
        \draw[red, thick] (0.4,0,{sqrt(2)-5-5}) -- (0.4,0,-2-5) node[pos=1,above] {$x'$};
        \draw[red, thick] (0,0.4,{sqrt(2)-5-5}) -- (0,0.4,-2-5);
        \draw[domain=0:27, smooth, variable=\x, red, thick] plot ({0.4*cos(\x)}, {0.4*sin(\x)}, {sqrt(2)-5-5}); 
        \draw[domain=63:90, smooth, variable=\x, red, thick] plot ({0.4*cos(\x)}, {0.4*sin(\x)}, {sqrt(2)-5-5}); 
         \draw[domain=0:30, smooth, variable=\x, red, thick] plot ({0.4*cos(\x)}, {0.4*sin(\x)}, {-2-5}); 
         \draw[domain=60:90, smooth, variable=\x, red, thick] plot ({0.4*cos(\x)}, {0.4*sin(\x)}, {-2-5}); 
         \draw[red, thick] ({0.4*cos(30)},{0.4*sin(30)},{-5-5+sqrt(2)}) -- ({0.4*cos(30)},{0.4*sin(30)},{-5-5+3}) ;
        \draw[red, thick] ({0.4*cos(60)},{0.4*sin(60)},{-5-5+sqrt(2)}) -- ({0.4*cos(60)},{0.4*sin(60)},{-5-5+3}) ;

        \draw[domain=-45:135, smooth, variable=\x, red, thick] plot ({0.4*cos(45)+0.1*cos(\x)},{0.4*sin(45)+0.1*sin(\x)}, {-5-5+sqrt(2)}); 
        \draw[domain=-45:135, smooth, variable=\x, red, thick] plot ({0.4*cos(45)+0.1*cos(\x)}, {0.4*sin(45)+0.1*sin(\x)}, {-5-5+3});


        \draw[->,black,thick] (0,3,1.-8.5) -- (0,4,1.-8.5) node[pos=0.5,below] {$\pi_{\ce,CR}$};



        \draw[black, thick] (0,{5+sqrt(2)},-10) -- (0,8,-10) node[pos=1,right] {$x''$};
        \draw[red, thick] (0,5,{-10+sqrt(2)}) -- (0,5,-7) node[pos=1,above] {$x'$};
         \draw[domain=0:37, smooth, variable=\x, purple, thick] plot ({0}, {5+sqrt(2)*cos(\x)}, {(sqrt(2)*(sin(\x))-10}); 
         \draw[domain=53:90, smooth, variable=\x, purple, thick] plot ({0}, {5+sqrt(2)*cos(\x)}, {(sqrt(2)*(sin(\x))-10}); 
         \draw[domain=-45:135, smooth, variable=\x, red, thick] plot ({0}, {5+sqrt(2)*cos(45)+0.2*cos(\x)}, {(sqrt(2)*(sin(45))-10+0.2*sin(\x)}); 


        \draw[->,black,thick] (0,1,-2) -- (0,1,-1) node[pos=0.5,right] {$\beta^{3}_{e}$};
        \draw[->,black,thick] (0,7,-2) -- (0,7,-1) node[pos=0.5,right] {$\beta^{2}_{e}$};
        \draw[->,black,thick] (0,1,-7) -- (0,1,-6) node[pos=0.5,right] {$\beta^{3}_{\ce\to e}$};
        \draw[->,black,thick] (0,7,-7) -- (0,7,-6) node[pos=0.5,right] {$\beta^{2}_{\ce\to e}$};

    \end{tikzpicture}
    \caption{The first row shows is the unblownup triple space with the $\pi_{CR}$ projection to the unblown-up double space. The second row shows that edge triple space and its projection $\pi_{e,CR}$ to the edge double space. The final row shows that cusp edge triple space and its projection $\pi_{\ce,CR}$ to the cusp edge double space.} \label{fig:blowuptriple2}
\end{figure}
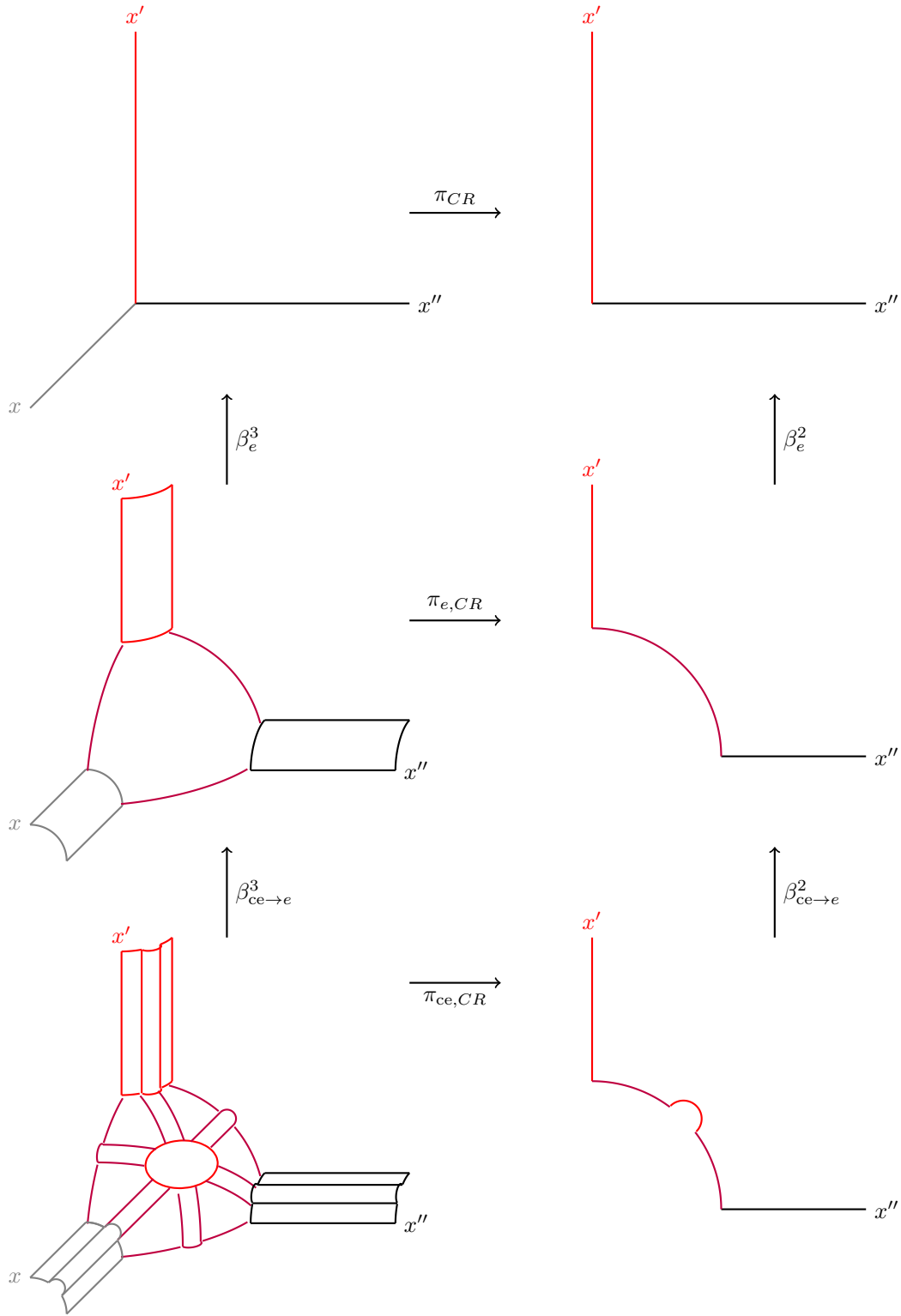

\clearpage


\bibliographystyle{alpha}
\bibliography{bibfinal}

\end{document}